\documentclass[10pt,reqno]{amsart}

\usepackage[utf8x]{inputenc}
\usepackage[english]{babel}
\usepackage{amsmath,amsfonts,amssymb,amsthm,shuffle}
\usepackage[T1]{fontenc}
\usepackage[math]{anttor}

\usepackage[top=3.5cm,bottom=3.5cm,left=3.6cm,right=3.6cm]{geometry}

\usepackage[dvipsnames]{xcolor}
\usepackage[hyperindex=true,frenchlinks=true,colorlinks=true,
citecolor=Mahogany,linkcolor=DarkOrchid,urlcolor=Plum,linktocpage,
pagebackref=true]{hyperref}

\usepackage{tikz}
\usetikzlibrary{shapes}
\usetikzlibrary{fit}
\usetikzlibrary{decorations.pathmorphing}

\usepackage{mathtools}
\usepackage{dsfont}
\usepackage{wasysym}
\usepackage{stmaryrd}
\usepackage{cite}
\usepackage{subfig}
\usepackage{multirow}
\usepackage{enumitem}
\usepackage{multicol}

\title{Operads of decorated cliques}
\date{\today}
\author{Samuele Giraudo}
\address{\scriptsize Université Paris-Est, LIGM (UMR $8049$), CNRS,
    ENPC, ESIEE Paris, UPEM, F-$77454$, Marne-la-Vallée, France}
\email{samuele.giraudo@u-pem.fr}

\numberwithin{equation}{subsection}
\renewcommand{\leq}{\leqslant}
\renewcommand{\geq}{\geqslant}

\newtheorem{Theorem}{Theorem}[subsection]
\newtheorem{Proposition}[Theorem]{Proposition}
\newtheorem{Lemma}[Theorem]{Lemma}

\newcommand{\N}{\mathbb{N}}
\newcommand{\Z}{\mathbb{Z}}

\newcommand{\K}{\mathbb{K}}

\newcommand{\Aca}{\mathcal{A}}

\newcommand{\Oca}{\mathcal{O}}
\newcommand{\Mca}{\mathcal{M}}

\newcommand{\Hsf}{\mathsf{H}}
\newcommand{\Ksf}{\mathsf{K}}

\newcommand{\Dbb}{\mathbb{D}}
\newcommand{\Ebb}{\mathbb{E}}
\newcommand{\Sbb}{\mathbb{S}}
\newcommand{\Ubb}{\mathbb{U}}

\newcommand{\Bfr}{\mathfrak{b}}
\newcommand{\Cfr}{\mathfrak{c}}
\newcommand{\Dfr}{\mathfrak{d}}

\newcommand{\Tfr}{\mathfrak{t}}
\newcommand{\Sfr}{\mathfrak{s}}
\newcommand{\Pfr}{\mathfrak{p}}
\newcommand{\Qfr}{\mathfrak{q}}
\newcommand{\Rfr}{\mathfrak{r}}
\newcommand{\Ufr}{\mathfrak{u}}
\newcommand{\Vfr}{\mathfrak{v}}

\newcommand{\Att}{\mathtt{a}}
\newcommand{\Btt}{\mathtt{b}}
\newcommand{\Ctt}{\mathtt{c}}
\newcommand{\Dtt}{\mathtt{d}}
\newcommand{\Ett}{\mathtt{e}}

\newcommand{\As}{\mathsf{As}}
\newcommand{\BNC}{\mathsf{BNC}}

\newcommand{\Cli}{\mathsf{C}}
\newcommand{\TDendr}{\mathsf{TDendr}}
\newcommand{\RatFct}{\mathsf{RatFct}}
\newcommand{\Mould}{\mathsf{Mould}}
\newcommand{\Dendr}{\mathsf{Dendr}}
\newcommand{\NCP}{\mathsf{NCP}}
\newcommand{\LOp}{\mathsf{L}}
\newcommand{\MT}{\mathsf{MT}}
\newcommand{\DMT}{\mathsf{DMT}}
\newcommand{\Lab}{\mathsf{Lab}}
\newcommand{\Bub}{\mathsf{Bub}}
\newcommand{\Deg}{\mathsf{Deg}}
\newcommand{\Cro}{\mathsf{Cro}}
\newcommand{\Acy}{\mathsf{Acy}}
\newcommand{\Whi}{\mathrm{Whi}}
\newcommand{\NC}{\mathrm{NC}}
\newcommand{\Inf}{\mathrm{Inf}}
\newcommand{\Paths}{\mathrm{Pat}}
\newcommand{\Forests}{\mathrm{For}}
\newcommand{\Motzkin}{\mathrm{Mot}}
\newcommand{\Diss}{\mathrm{Dis}}
\newcommand{\WNC}{\mathrm{WNC}}
\newcommand{\Luc}{\mathsf{Luc}}
\newcommand{\FF}{\mathcal{F}\mathcal{F}}
\newcommand{\Motz}{\mathrm{Motz}}

\newcommand{\Unit}{\mathds{1}}
\newcommand{\Hilbert}{\mathcal{H}}
\newcommand{\Op}{\star}
\newcommand{\OpAssoc}{\odot}
\newcommand{\Leaf}{\perp}
\newcommand{\Nar}{\mathrm{nar}}
\newcommand{\Vect}{\mathrm{Vect}}
\newcommand{\Rew}{\to}
\newcommand{\RewTrans}{\overset{*}{\Rew}}
\newcommand{\RewTransSym}{\overset{*}{\leftrightarrow}}
\newcommand{\CRew}{\Rightarrow}
\newcommand{\CRewTrans}{\overset{*}{\CRew}}
\newcommand{\CRewTransSym}{\overset{*}{\Leftrightarrow}}
\newcommand{\Free}{\mathrm{Free}}
\newcommand{\Gen}{\mathfrak{G}}
\newcommand{\Rel}{\mathfrak{R}}
\newcommand{\Arcs}{\mathcal{A}}
\newcommand{\Diagonals}{\mathcal{D}}
\newcommand{\Edges}{\mathcal{E}}
\newcommand{\Eval}{\mathrm{ev}}
\newcommand{\Corolla}{\mathrm{c}}
\newcommand{\RelEq}{\equiv}
\newcommand{\Bubbles}{\mathcal{B}}
\newcommand{\Triangles}{\mathcal{T}}
\newcommand{\Primes}{\mathcal{P}}
\newcommand{\Cliques}{\mathcal{C}}
\newcommand{\BubbleTree}{\mathrm{bt}}
\newcommand{\Returned}{\mathrm{ret}}
\newcommand{\Shift}{\mathrm{sh}}
\newcommand{\Alg}{\Aca}
\newcommand{\Id}{\mathrm{Id}}
\newcommand{\SeriesBubbles}{\mathrm{B}}
\newcommand{\SeriesElements}{\mathrm{F}}
\newcommand{\Hamming}{\mathrm{h}}
\newcommand{\Border}{\mathrm{bor}}
\newcommand{\Cros}{\mathrm{cros}}
\newcommand{\Degr}{\mathrm{degr}}
\newcommand{\OrdBE}{\preceq_{\mathrm{be}}}
\newcommand{\OrdD}{\preceq_{\mathrm{d}}}
\newcommand{\Del}{\mathrm{d}}
\newcommand{\Skel}{\mathrm{skel}}
\newcommand{\Frac}{\mathrm{F}}
\newcommand{\GDendr}{\prec}
\newcommand{\DDendr}{\succ}
\newcommand{\BinRel}{\,\mathfrak{R}\,}

\newcommand{\Hide}[1]{}
\newcommand{\Def}[1]{\textcolor{MidnightBlue}{\em #1}}
\newcommand{\OEIS}[1]{\href{http://oeis.org/#1}{{\bf #1}}}

\newcommand{\UnitClique}{
\begin{tikzpicture}[scale=.6,Centering]
    \node[CliquePoint](1)at(0,0){};
    \node[CliquePoint](2)at(.75,0){};
    \draw[CliqueEmptyEdge](1)edge[]node[CliqueLabel]{}(2);
\end{tikzpicture}}

\newcommand{\Triangle}[3]{
\begin{tikzpicture}[scale=.42,Centering]
    \node[CliquePoint](1)at(0,1){};
    \node[CliquePoint](2)at(0.87,-0.5){};
    \node[CliquePoint](3)at(-0.87,-0.5){};
    \draw[CliqueEdge](1)edge[]node[CliqueLabel]{\begin{math}#3\end{math}}(2);
    \draw[CliqueEdge](1)edge[]node[CliqueLabel]{\begin{math}#2\end{math}}(3);
    \draw[CliqueEdge](2)edge[]node[CliqueLabel]{\begin{math}#1\end{math}}(3);
\end{tikzpicture}}

\newcommand{\TriangleEXX}[3]{
\begin{tikzpicture}[scale=.42,Centering]
    \node[CliquePoint](1)at(0,1){};
    \node[CliquePoint](2)at(0.87,-0.5){};
    \node[CliquePoint](3)at(-0.87,-0.5){};
    \draw[CliqueEdge](1)edge[]node[CliqueLabel]{\begin{math}#3\end{math}}(2);
    \draw[CliqueEdge](1)edge[]node[CliqueLabel]{\begin{math}#2\end{math}}(3);
    \draw[CliqueEmptyEdge](2)edge[]node[CliqueLabel]{}(3);
\end{tikzpicture}}

\newcommand{\TriangleXEX}[3]{
\begin{tikzpicture}[scale=.42,Centering]
    \node[CliquePoint](1)at(0,1){};
    \node[CliquePoint](2)at(0.87,-0.5){};
    \node[CliquePoint](3)at(-0.87,-0.5){};
    \draw[CliqueEdge](1)edge[]node[CliqueLabel]{\begin{math}#3\end{math}}(2);
    \draw[CliqueEmptyEdge](1)edge[]node[CliqueLabel]{}(3);
    \draw[CliqueEdge](2)edge[]node[CliqueLabel]{\begin{math}#1\end{math}}(3);
\end{tikzpicture}}

\newcommand{\TriangleXXE}[3]{
\begin{tikzpicture}[scale=.42,Centering]
    \node[CliquePoint](1)at(0,1){};
    \node[CliquePoint](2)at(0.87,-0.5){};
    \node[CliquePoint](3)at(-0.87,-0.5){};
    \draw[CliqueEmptyEdge](1)edge[]node[CliqueLabel]{}(2);
    \draw[CliqueEdge](1)edge[]node[CliqueLabel]{\begin{math}#2\end{math}}(3);
    \draw[CliqueEdge](2)edge[]node[CliqueLabel]{\begin{math}#1\end{math}}(3);
\end{tikzpicture}}

\newcommand{\TriangleXEE}[3]{
\begin{tikzpicture}[scale=.42,Centering]
    \node[CliquePoint](1)at(0,1){};
    \node[CliquePoint](2)at(0.87,-0.5){};
    \node[CliquePoint](3)at(-0.87,-0.5){};
    \draw[CliqueEmptyEdge](1)edge[]node[CliqueLabel]{}(2);
    \draw[CliqueEmptyEdge](1)edge[]node[CliqueLabel]{}(3);
    \draw[CliqueEdge](2)edge[]node[CliqueLabel]{\begin{math}#1\end{math}}(3);
\end{tikzpicture}}

\newcommand{\TriangleEEX}[3]{
\begin{tikzpicture}[scale=.42,Centering]
    \node[CliquePoint](1)at(0,1){};
    \node[CliquePoint](2)at(0.87,-0.5){};
    \node[CliquePoint](3)at(-0.87,-0.5){};
    \draw[CliqueEdge](1)edge[]node[CliqueLabel]{\begin{math}#3\end{math}}(2);
    \draw[CliqueEmptyEdge](1)edge[]node[CliqueLabel]{}(3);
    \draw[CliqueEmptyEdge](2)edge[]node[CliqueLabel]{}(3);
\end{tikzpicture}}

\newcommand{\TriangleEXE}[3]{
\begin{tikzpicture}[scale=.42,Centering]
    \node[CliquePoint](1)at(0,1){};
    \node[CliquePoint](2)at(0.87,-0.5){};
    \node[CliquePoint](3)at(-0.87,-0.5){};
    \draw[CliqueEmptyEdge](1)edge[]node[CliqueLabel]{}(2);
    \draw[CliqueEdge](1)edge[]node[CliqueLabel]{\begin{math}#2\end{math}}(3);
    \draw[CliqueEmptyEdge](2)edge[]node[CliqueLabel]{}(3);
\end{tikzpicture}}

\newcommand{\TriangleEEE}[3]{
\begin{tikzpicture}[scale=.42,Centering]
    \node[CliquePoint](1)at(0,1){};
    \node[CliquePoint](2)at(0.87,-0.5){};
    \node[CliquePoint](3)at(-0.87,-0.5){};
    \draw[CliqueEmptyEdge](1)edge[]node[CliqueLabel]{}(2);
    \draw[CliqueEmptyEdge](1)edge[]node[CliqueLabel]{}(3);
    \draw[CliqueEmptyEdge](2)edge[]node[CliqueLabel]{}(3);
\end{tikzpicture}}

\newcommand{\SquareN}[4]{
\begin{tikzpicture}[scale=.6,Centering]
    \node[CliquePoint](1)at(-0.71,0.71){};
    \node[CliquePoint](2)at(0.71,0.71){};
    \node[CliquePoint](3)at(0.71,-0.71){};
    \node[CliquePoint](4)at(-0.71,-0.71){};
    \draw[CliqueEdge](1)edge[]node[CliqueLabel]{\begin{math}#2\end{math}}(2);
    \draw[CliqueEdge](1)edge[]node[CliqueLabel]{\begin{math}#1\end{math}}(4);
    \draw[CliqueEdge](2)edge[]node[CliqueLabel]{\begin{math}#3\end{math}}(3);
    \draw[CliqueEdge](3)edge[]node[CliqueLabel]{\begin{math}#4\end{math}}(4);
\end{tikzpicture}}

\newcommand{\SquareLeft}[5]{
\begin{tikzpicture}[scale=.6,Centering]
    \node[CliquePoint](1)at(-0.71,0.71){};
    \node[CliquePoint](2)at(0.71,0.71){};
    \node[CliquePoint](3)at(0.71,-0.71){};
    \node[CliquePoint](4)at(-0.71,-0.71){};
    \draw[CliqueEdge](1)edge[]node[CliqueLabel]{\begin{math}#2\end{math}}(2);
    \draw[CliqueEdge](1)edge[]node[CliqueLabel]{\begin{math}#1\end{math}}(4);
    \draw[CliqueEdge](2)edge[]node[CliqueLabel]{\begin{math}#3\end{math}}(3);
    \draw[CliqueEdge](3)edge[]node[CliqueLabel]{\begin{math}#4\end{math}}(4);
    \draw[CliqueEdge](1)edge[]node[CliqueLabel]{\begin{math}#5\end{math}}(3);
\end{tikzpicture}}

\newcommand{\SquareRight}[5]{
\begin{tikzpicture}[scale=.6,Centering]
    \node[CliquePoint](1)at(-0.71,0.71){};
    \node[CliquePoint](2)at(0.71,0.71){};
    \node[CliquePoint](3)at(0.71,-0.71){};
    \node[CliquePoint](4)at(-0.71,-0.71){};
    \draw[CliqueEdge](1)edge[]node[CliqueLabel]{\begin{math}#2\end{math}}(2);
    \draw[CliqueEdge](1)edge[]node[CliqueLabel]{\begin{math}#1\end{math}}(4);
    \draw[CliqueEdge](2)edge[]node[CliqueLabel]{\begin{math}#3\end{math}}(3);
    \draw[CliqueEdge](3)edge[]node[CliqueLabel]{\begin{math}#4\end{math}}(4);
    \draw[CliqueEdge](2)edge[]node[CliqueLabel]{\begin{math}#5\end{math}}(4);
\end{tikzpicture}}

\newcommand{\SquareMotz}{
\begin{tikzpicture}[scale=.6,Centering]
    \node[CliquePoint](1)at(-0.71,0.71){};
    \node[CliquePoint](2)at(0.71,0.71){};
    \node[CliquePoint](3)at(0.71,-0.71){};
    \node[CliquePoint](4)at(-0.71,-0.71){};
    \draw[CliqueEmptyEdge](2)edge[]node[CliqueLabel]{}(3);
    \draw[CliqueEmptyEdge](1)edge[]node[CliqueLabel]{}(4);
    \draw[CliqueEmptyEdge](3)edge[]node[CliqueLabel]{}(4);
    \draw[CliqueEdge](1)edge[]node[CliqueLabel]{\begin{math}0\end{math}}(2);
\end{tikzpicture}}

\newcommand{\TriangleOp}[3]{\;
\begin{tikzpicture}[scale=.35,Centering]
    \node[shape=coordinate](1)at(0,1){};
    \node[shape=coordinate](2)at(0.87,-0.5){};
    \node[shape=coordinate](3)at(-0.87,-0.5){};
    \draw[draw=Sepia!90](1)edge[]node[CliqueLabel,font=\tiny]
        {\begin{math}#3\end{math}}(2);
    \draw[draw=Sepia!90](1)edge[]node[CliqueLabel,font=\tiny]
        {\begin{math}#2\end{math}}(3);
    \draw[draw=Sepia!90](2)edge[]node[CliqueLabel,font=\tiny]
        {\begin{math}#1\end{math}}(3);
\end{tikzpicture}\;}

\tikzstyle{Centering}=[{baseline={([yshift=-0.5ex]current
    bounding box.center)}}]
\tikzstyle{CliqueEdge}=[draw=Cerulean!90,thick]
\tikzstyle{CliqueEmptyEdge}=[draw=Gray!90,thick,densely dashed]
\tikzstyle{CliqueLabel}=[midway,inner sep=1pt,fill=White!0,
font=\scriptsize]
\tikzstyle{CliquePoint}=[circle,inner sep=1pt,fill=BrickRed!25,
draw=BrickRed!70]
\tikzstyle{CliqueEdgeGray}=[Black!30,draw,cap=round]
\tikzstyle{CliqueEdgeBlue}=[RoyalBlue!80,thick,draw,cap=round]
\tikzstyle{CliqueEdgeRed}=[BrickRed!80,thick,draw,cap=round,dotted]
\tikzstyle{Node}=[circle,draw=RoyalBlue!80,fill=RoyalBlue!8,inner sep=1pt,
minimum size=2.0mm,thick,font=\scriptsize]
\tikzstyle{Edge}=[draw=BrickRed!80,cap=round,thick]
\tikzstyle{Leaf}=[rectangle,draw=Black!70,fill=Black!16,
inner sep=0pt,minimum size=1mm,thick]
\tikzstyle{EdgeLabel}=[regular polygon,regular polygon sides=6,
draw=ForestGreen!80,fill=ForestGreen!2,inner sep=0pt,line width=.3pt,
minimum size=2.3mm,midway,text=Black,font=\tiny]
\tikzstyle{Subtree}=[regular polygon,regular polygon sides=3,
draw=Apricot!80,fill=Apricot!20,thick,minimum size=5mm,font=\scriptsize]
\tikzstyle{Injection}=[Black!100,draw,>->]
\tikzstyle{Surjection}=[Black!100,draw,->>]
\tikzstyle{Bijection}=[Black!100,draw,<->]

\begin{document}

\begin{abstract}
    The vector space of all polygons with configurations of diagonals
    is endowed with an operad structure. This is the consequence of
    a functorial construction $\Cli$ introduced here, which takes
    unitary magmas $\Mca$ as input and produces operads. The obtained
    operads involve regular polygons with configurations of arcs labeled
    on $\Mca$, called $\Mca$-decorated cliques and generalizing usual
    polygons with configurations of diagonals. We provide here a
    complete study of the operads $\Cli\Mca$. By considering
    combinatorial subfamilies of $\Mca$-decorated cliques defined, for
    instance, by limiting the maximal number of crossing diagonals or
    the maximal degree of the vertices, we obtain suboperads and
    quotients of $\Cli\Mca$. This leads to a new hierarchy of operads
    containing, among others, operads on noncrossing configurations,
    Motzkin configurations, forests, dissections of polygons, and
    involutions. We show that the suboperad of noncrossing
    configurations is Koszul and exhibit its presentation by generators
    and relations. Besides, the construction $\Cli$ leads to alternative
    definitions of several operads, like the operad of bicolored
    noncrossing configurations and the operads of simple and double
    multi-tildes.
\end{abstract}

\maketitle
\vspace{-2.5em}

\begin{footnotesize}
\tableofcontents
\end{footnotesize}

\section*{Introduction}
Regular polygons endowed with configurations of diagonals are very
classical combinatorial objects. Up to some restrictions or enrichments,
sets of these polygons can be put in bijection with several
combinatorial families. For instance, it is well-known that
triangulations~\cite{DRS10}, forming a particular subset of the set
of all polygons, are in one-to-one correspondence with binary trees,
and a lot of structures and operations on binary trees translate nicely
on triangulations. Indeed, among others, the rotation operation on
binary trees~\cite{Knu98} is the covering relation of the Tamari
order~\cite{HT72} and this operation translates as a diagonal flip in
triangulations. Also, noncrossing configurations~\cite{FN99} form an
other interesting subfamily of such polygons. Natural generalizations of
noncrossing configurations consist in allowing, with more or less
restrictions, some crossing diagonals. One of these families is formed
by the multi-triangulations~\cite{CP92}, that are polygons wherein the
number of mutually crossing diagonal is bounded. Besides, let us remark
that the class of combinatorial objects in bijection with sets of
polygons with configurations of diagonals is large enough in order to
contain, among others, dissections of polygons, noncrossing partitions,
permutations, and involutions.
\smallskip

On the other hand, coming historically from algebraic
topology~\cite{May72,BV73}, operads provide an abstraction of the notion
of operators (of any arities) and their compositions. In more concrete
terms, operads are algebraic structures abstracting the notion of planar
rooted trees and their grafting operations (see~\cite{LV12} for a
complete exposition of the theory and~\cite{Men15} for an exposition
focused on symmetric set-operads). The modern treatment of operads in
algebraic combinatorics consists in regarding combinatorial objects like
operators endowed with gluing operations mimicking the composition of
operators. In the last years, a lot of combinatorial sets and
combinatorial spaces have been endowed fruitfully with a structure of an
operad (see for instance~\cite{Cha08} for an exposition of known
interactions between operads and combinatorics, focused on trees,
\cite{LMN13,GLMN16} where operads abstracting operations in language
theory are introduced, \cite{CG14} for the study of an operad involving
particular noncrossing configurations, \cite{Gir15} for a general
construction of operads on many combinatorial sets, \cite{Gir16b} where
operads are constructed from posets, and \cite{CHN16} where operads on
various species of trees are introduced). In most of the cases, this
approach brings results about enumeration, helps to discover new
statistics, and leads to establish new links (by morphisms) between
different combinatorial sets or spaces. We can observe that most of the
subfamilies of polygons endowed with configurations of diagonals
discussed above are stable for several natural composition operations.
Even better, some of these can be described as the cloture with respect
to these composition operations of small sets of polygons. For this
reason, operads are very promising candidates, among the modern
algebraic structures, to study such objects under an algebraic and
combinatorial flavor.
\smallskip

The purpose of this work is twofold. First, we are concerned in endowing
the linear span of the polygons with configurations of arcs with a
structure of an operad. This leads to see these objects under a new
light, stressing some of their combinatorial and algebraic properties.
Second, we would provide a general construction of operads of polygons
rich enough so that it includes some already known operads. As a
consequence, we obtain alternative definitions of existing operads and
new interpretations of these. For this aim, we work here with
$\Mca$-decorated cliques (or $\Mca$-cliques for short), that are
complete graphs whose arcs are labeled on $\Mca$, where $\Mca$ is a
unitary magma. These objects are natural generalizations of polygons
with configurations of arcs since the arcs of any $\Mca$-clique labeled
by the unit of $\Mca$ are considered as missing. The elements of $\Mca$
different from the unit allow moreover to handle polygons with arcs of
different colors. For instance, each usual noncrossing configuration
$\Cfr$ can be encoded by an $\N_2$-clique $\Pfr$, where $\N_2$ is the
cyclic additive unitary magma $\Z/_{2\Z}$, wherein each arc labeled by
$1 \in \N_2$ in $\Pfr$ denotes the presence of the same arc in $\Cfr$,
and each arc labeled by $0 \in \N_2$ in $\Pfr$ denotes its absence
in~$\Cfr$. Our construction is materialized by a functor $\Cli$ from
the category of unitary magmas to the category of operads. It builds,
from any unitary magma $\Mca$, an operad $\Cli\Mca$ on $\Mca$-cliques.
The partial composition $\Pfr \circ_i \Qfr$ of two $\Mca$-cliques $\Pfr$
and $\Qfr$ of $\Cli\Mca$ consists in gluing the $i$th edge of $\Pfr$
(with respect to a precise indexation) and a special arc of $\Qfr$,
called the base, together to form a new $\Mca$-clique. The magmatic
operation of $\Mca$ explains how to relabel the two overlapping arcs.
\smallskip

This operad $\Cli\Mca$ has a lot of properties, which can be apprehended
both under a combinatorial and an algebraic point of view. First, many
families of particular polygons with configurations of arcs form
quotients or suboperads of $\Cli\Mca$. We can for instance control the
degrees of the vertices or the crossings between diagonals to obtain new
operads. We can also forbid all diagonals, or some labels for the
diagonals or the edges, or all inclusions of diagonals, or even all
cycles formed by arcs. All these combinatorial particularities and
restrictions on $\Mca$-cliques behave well algebraically. Moreover, by
using the fact that the direct sum of two ideals of an operad $\Oca$ is
still an ideal of $\Oca$, these constructions can be mixed to get even
more operads. For instance, it is well-known that Motzkin configurations,
that are polygons with disjoint noncrossing diagonals, are enumerated by
Motzkin numbers~\cite{Mot48}. Since a Motzkin configuration can be
encoded by an $\Mca$-clique where all vertices are of degrees at most
$1$ and no diagonal crosses another one, we obtain an operad
$\Motzkin\Mca$ on colored Motzkin configurations which is both a
quotient of $\Deg_1\Mca$, the quotient of $\Cli\Mca$ consisting in all
$\Mca$-cliques such that all vertices are of degrees at most $1$, and of
$\NC\Mca$, the quotient (and suboperad) of $\Cli\Mca$ consisting in all
noncrossing $\Mca$-cliques. We also get quotients of $\Cli\Mca$
involving, among others, Schröder trees, forests of paths, forests of
trees, dissections of polygons, Lucas configurations, with colored
versions for each of these. This leads to a new hierarchy of operads,
wherein links between its components appear as surjective or injective
operad morphisms. Table~\ref{tab:constructed_operads} lists the main
operads constructed in this work and gathers some information about
these.
\begin{table}[ht]
    \centering
    \begin{small}
    \begin{tabular}{c|c|c|c}
        Operad & Objects & Status with respect to $\Cli\Mca$
            & Place \\ \hline \hline
        $\Cli\Mca$ & $\Mca$-cliques & --- &
        Section~\ref{sec:construction_Cli} \\ \hline
        $\Lab_{B, E, D}\Mca$ & $\Mca$-cliques with restricted labels
            & Suboperad
            & Section~\ref{subsubsec:suboperad_Cli_M_labels}\\
        $\Whi\Mca$ & White $\Mca$-cliques & Suboperad
            & Section~\ref{subsubsec:suboperad_Cli_M_white} \\
        $\Cro_k\Mca$ & $\Mca$-cliques of crossings at most $k$
            & Suboperad and quotient
            & Section~\ref{subsubsec:quotient_Cli_M_crossings} \\
        $\Bub\Mca$ & $\Mca$-bubbles
            & Quotient
            & Section~\ref{subsubsec:quotient_Cli_M_bubbles} \\
        $\Deg_k\Mca$ & $\Mca$-cliques of degrees at most $k$
            & Quotient
            & Section~\ref{subsubsec:quotient_Cli_M_degrees} \\
        $\Inf\Mca$ & Inclusion-free $\Mca$-cliques
            & Quotient
            & Section~\ref{subsubsec:quotient_Cli_M_Inf} \\
        $\Acy\Mca$ & Acyclic $\Mca$-cliques
            & Quotient
            & Section~\ref{subsubsec:quotient_Cli_M_acyclic} \\ \hline
        $\NC\Mca$ & noncrossing $\Mca$-cliques & Suboperad and quotient
            & Section~\ref{sec:operad_noncrossing} \\
    \end{tabular}
    \end{small}
    \smallskip

    \caption{\footnotesize
    The main operads defined in this paper. All these operads depend
    on a unitary magma $\Mca$ which has, in some cases, to satisfy
    some precise conditions. Some of these operads depend also on a
    nonnegative integer~$k$ or subsets $B$, $E$, and $D$ of~$\Mca$.}
    \label{tab:constructed_operads}
\end{table}
One of the most notable of these is built by considering the
$\Dbb_0$-cliques that have vertices of degrees at most $1$, where
$\Dbb_0$ is the multiplicative unitary magma on $\{0, 1\}$. This is in
fact the quotient $\Deg_1\Dbb_0$ of $\Cli\Dbb_0$ and involves
involutions (or equivalently, standard Young tableaux by the
Robinson-Schensted correspondence~\cite{Lot02}). To the best of our
knowledge, $\Deg_1\Dbb_0$ is the first nontrivial operad on these
objects. As an important remark at this stage, let us highlight that
when $\Mca$ is nontrivial, $\Cli\Mca$ is not a binary operad. Indeed,
all its minimal generating sets are infinite and its generators have
arbitrary high arities. Nevertheless, the biggest binary suboperad of
$\Cli\Mca$ is the operad $\NC\Mca$ of noncrossing configurations and
this operad is quadratic and Koszul, regardless of $\Mca$. Furthermore,
the construction $\Cli$ maintains some links with the operad $\RatFct$
of rational functions introduced by Loday~\cite{Lod10}. In fact,
provided that $\Mca$ satisfies some conditions, each $\Mca$-clique
encodes a rational function. This defines an operad morphism from
$\Cli\Mca$ to $\RatFct$. Moreover, the construction $\Cli$ allows to
construct already known operads in original ways. For instance, for
well-chosen unitary magmas $\Mca$, the operads $\Cli\Mca$ contain
$\FF_4$, a suboperad of the operad of formal fractions
$\FF$~\cite{CHN16}, $\BNC$, the operad of bicolored noncrossing
configurations~\cite{CG14}, and $\MT$ and $\DMT$, two operads respectively
defined in~\cite{LMN13} and~\cite{GLMN16} that involve multi-tildes and
double multi-tildes, operators coming from formal language
theory~\cite{CCM11}.
\smallskip

This text is organized as follows. Section~\ref{sec:definitions_tools}
sets our notations, general definitions, and tools about nonsymmetric
operads (since we deal only with nonsymmetric operads here, we call
these simply operads).
\smallskip

In Section~\ref{sec:construction_Cli}, we introduce $\Mca$-cliques and
some definitions about these. Then, the construction $\Cli$ is
described and the fact that $\Cli$ is a functor from the category of
unitary magmas to the category of set-operads (treated here as operads
in the category of vector spaces) is established
(Theorem~\ref{thm:clique_construction}). We show that the Hadamard
product of two operads obtained as images of $\Cli$ is isomorphic to an
operad in the image of $\Cli$
(Proposition~\ref{prop:Cli_M_Cartesian_product}). We then investigate
the general properties of the operads $\Cli\Mca$. We compute their
dimensions (Proposition~\ref{prop:dimensions_Cli_M}), describe one of
their minimal generating sets
(Proposition~\ref{prop:generating_set_Cli_M}), and describe all their
associative elements (Proposition~\ref{prop:associative_elements_Cli_M}).
Notice that the description of the associative elements of $\Cli\Mca$
relies on its linear structure. Indeed, even if $\Cli\Mca$ is
well-defined in the category of sets, some of its associative elements
are nontrivial sums of $\Mca$-cliques. We also explicit some symmetries
of $\Cli\Mca$ (Proposition~\ref{prop:symmetries_Cli_M}), show that, as a
set-operad, $\Cli\Mca$ is basic~\cite{Val07} if and only if $\Mca$ is
right cancellable (Proposition~\ref{prop:basic_Cli_M}), and show that
$\Cli\Mca$ is a cyclic operad~\cite{GK95}
(Proposition~\ref{prop:cyclic_Cli_M}). Next, two additional bases of
$\Cli\Mca$ are defined, the $\Hsf$-basis and the $\Ksf$-basis, defined
from two natural partial order relations on $\Mca$-cliques and their
Möbius functions. We give expressions for the partial composition of
$\Cli\Mca$ over these two bases
(Propositions~\ref{prop:composition_Cli_M_basis_H}
and~\ref{prop:composition_Cli_M_basis_K}). This section ends by
explaining how any $\Mca$-clique $\Pfr$ encodes a rational function
$\Frac_\theta(\Pfr)$ of $\K(\Ubb)$, $\Ubb$ being a commutative alphabet,
when $\Mca$ is $\Z$-graded. We show that $\Frac_\theta$ is an operad
morphism from $\Cli\Mca$ to $\RatFct$, the operad of rational
functions~\cite{Lod10} (Theorem~\ref{thm:rat_fct_cliques}). This
morphism is not injective but its image contains all Laurent polynomials
on $\Ubb$
(Proposition~\ref{prop:rat_fct_cliques_map_Laurent_polynomials}).
\smallskip

Then, Section~\ref{sec:quotients_suboperads} is devoted to define
several suboperads and quotients of $\Cli\Mca$. All the quotients we
consider are of the form $\Cli\Mca/_{\Rel}$, where $\Rel$ is an operad
ideal of $\Cli\Mca$ generated by a family of $\Mca$-cliques we want to
discard. For instance, the quotient $\Bub\Mca$ of $\Cli\Mca$ on
$\Mca$-bubbles (that are $\Mca$-cliques without diagonals) is defined
from the operad ideal $\Rel_{\Bub}$ generated by all the $\Mca$-cliques
having at least one diagonal
(Proposition~\ref{prop:quotient_Cli_M_bubbles}). Some of the
constructions presented here require that $\Mca$ as no nontrivial unit
divisors. This is the case for instance for the quotient $\Acy\Mca$
involving acyclic $\Mca$-cliques
(Proposition~\ref{prop:quotient_Cli_M_acyclic}). We also construct and
briefly study the suboperad $\Lab_{B, E, D}\Mca$ involving
$\Mca$-cliques with restrictions for its labels
(Proposition~\ref{prop:suboperad_Cli_M_labels}), the suboperad
$\Whi\Mca$ involving $\Mca$-cliques with unlabeled edges, the quotient
and suboperad $\Cro_k\Mca$ involving $\Mca$-cliques such that each
diagonal crosses at most $k$ other diagonals
(Proposition~\ref{prop:quotient_Cli_M_crossings}), the quotient
$\Deg_k\Mca$ involving $\Mca$-cliques such that each vertex is of degree
at most $k$ (Proposition~\ref{prop:quotient_Cli_M_degrees}), and
$\Inf\Mca$ involving $\Mca$-cliques without inclusions between arcs
(Proposition~\ref{prop:quotient_Cli_M_inclusion_free}).
\smallskip

We focus next, in Section~\ref{sec:operad_noncrossing}, on the study of
the suboperad $\NC\Mca$ of $\Cli\Mca$ on the noncrossing $\Mca$-cliques.
We first show that $\NC\Mca$ inherits from a lot of properties of
$\Cli\Mca$, as the same description for its associative elements and the
fact that it is a cyclic operad
(Proposition~\ref{prop:inherited_properties_NC_M}). The operad $\NC\Mca$
admits an alternative realization in terms of $\Mca$-Schröder trees,
that are Schröder trees with edges labeled on $\Mca$ satisfying some
conditions. This is the consequence of the fact that the set of all
noncrossing $\Mca$-cliques of a given arity is in one-to-one
correspondence with a set of trees whose internal nodes are labeled by
$\Mca$-bubbles (Proposition~\ref{prop:map_NC_M_bubble_tree}). The
partial composition of $\NC\Mca$ on $\Mca$-Schröder trees is a grafting
of trees together with a relabeling of edges or a contraction of edges.
To continue the study of $\NC\Mca$, we describe one of its minimal
generating families (Proposition~\ref{prop:generating_set_NC_M}),
provide an algebraic equation for its Hilbert series
(Proposition~\ref{prop:Hilbert_series_NC_M}), and give a formula for its
dimensions involving Narayana numbers~\cite{Nar55}
(Proposition~\ref{prop:dimensions_NC_M}). In order to compute the space
of the relations of $\NC\Mca$, we use techniques of rewrite systems of
trees~\cite{BN98}. Thus, we define a convergent rewrite rule $\Rew$ and
show that the space induced by $\Rew$ is the space of relations of
$\NC\Mca$, leading to a presentation by generators and relations of
$\NC\Mca$ (Theorem~\ref{thm:presentation_NC_M}). The existence of a
convergent orientation of the space of the relations of $\NC\Mca$
implies by~\cite{Hof10} that this operad is Koszul
(Theorem~\ref{thm:Koszul_NC_M}). Then, we turn our attention on
suboperads of $\NC\Mca$ generated by some finite families of
bubbles. Under some conditions on the considered sets of bubbles, we can
describe the Hilbert series of these suboperads of $\NC\Mca$ by a system
of algebraic equations
(Proposition~\ref{prop:suboperads_NC_M_triangles_dimensions}). We give
two examples of suboperads of $\NC\Mca$ generated by some subsets of
triangles, including one which is a suboperad of $\NC\Dbb_0$ isomorphic
to the operad $\Motz$ of Motzkin paths defined in~\cite{Gir15}. From the
presentation of $\NC\Mca$, we list the relations between the operations
of the algebras over $\NC\Mca$. We describe the free algebra over one
generator over $\NC\Mca$, and a general way to construct algebras over
$\NC\Mca$ from associative algebras endowed with some linear maps
(Theorem~\ref{thm:NC_M_algebras}). Moreover, when $\Mca$ is a monoid,
there is a simple way to endow the space $\K \langle \Mca^* \rangle$ of
all noncommutative polynomials on $\Mca$ with the structure of an
algebra over $\NC\Mca$
(Proposition~\ref{prop:NC_M_algebras_monoid_polynomials}). We ends this
section by considering the Koszul dual $\NC\Mca^!$ of $\NC\Mca$ (which
is well-defined since $\NC\Mca$ is a binary and quadratic operad). We
compute its presentation
(Proposition~\ref{prop:presentation_dual_NC_M}), express an algebraic
equation for its Hilbert series
(Proposition~\ref{prop:Hilbert_series_NC_M_dual}), give a formula for
its dimensions (Proposition~\ref{prop:dimensions_NC_M_dual}), and
establish a combinatorial realization of $\NC\Mca^!$ as a graded space
involving dual $\Mca$-cliques
(Proposition~\ref{prop:elements_NC_M_dual}), that are $\Mca^2$-cliques
with some constraints for the labels of their arcs.
\smallskip

This works ends with Section~\ref{sec:concrete_constructions}, where we
use the construction $\Cli$ to provide alternative definitions of some
known operads. We hence construct the operad $\NCP$ of based noncrossing
trees~\cite{Cha07,Ler11} (Proposition~\ref{prop:construction_NCP}), the
suboperad $\FF_4$ of the operad of formal fractions $\FF$~\cite{CHN16}
(Proposition~\ref{prop:construction_FF4}), and the operad of bicolored
noncrossing configurations $\BNC$~\cite{CG14}
(Proposition~\ref{prop:construction_BNC}). For this reason,
in particular, all the suboperads of $\BNC$ can be obtained from the
construction $\Cli$. This includes for example the dipterous
operad~\cite{LR03,Zin12}. We also construct some versions of the operads
$\MT$ of multi-tildes~\cite{CCM11,LMN13}
(Proposition~\ref{prop:construction_MT}) and $\DMT$ of double
multi-tildes~\cite{GLMN16} (Proposition~\ref{prop:construction_DMT})
which are trivial in arity~$1$.
\medskip

\subsubsection*{General notations and conventions}
All the algebraic structures of this article have a field of
characteristic zero $\K$ as ground field. For any set $S$, $\Vect(S)$
denotes the linear span of the elements of $S$. For any integers $a$ and
$c$, $[a, c]$ denotes the set $\{b \in \N : a \leq b \leq c\}$ and $[n]$,
the set $[1, n]$. The cardinality of a finite set $S$ is denoted
by~$\# S$. For any set $A$, $A^*$ denotes the set of all finite
sequences, called words, of elements of~$A$. For any $n \geq 0$, $A^n$
(resp. $A^{\geq n}$) is the set of all words on $A$ of length $n$ (resp.
at least~$n$). The word of length $0$ is the empty word denoted by
$\epsilon$. If $u$ is a word, its letters are indexed from left to right
from $1$ to its length $|u|$. For any $i \in [|u|]$, $u_i$ is the letter
of $u$ at position $i$. If $a$ is a letter and $n$ is a nonnegative
integer, $a^n$ denotes the word consisting in $n$ occurrences of $a$.
For any letter $a$, $|u|_a$ denotes the number of occurrences of $a$
in~$u$.
\medskip

\section{Elementary definitions and tools}
\label{sec:definitions_tools}
We set here our notations and recall some definitions about operads and
the related structures. The main purposes are to provide tools to
compute presentations and prove Koszulity of operads. For this, it is
important to handle precise definitions about free operads, trees,
and rewrite rules on trees. This is the starting point of this section.
\medskip

\subsection{Trees and rewrite rules}
Unless otherwise specified, we use in the sequel the standard
terminology ({\em i.e.}, \Def{node}, \Def{edge}, \Def{root}, \Def{child},
{\em etc.}) about planar rooted trees~\cite{Knu97}. For the sake of
completeness, let us recall the most important definitions and set our
notations.
\medskip

\subsubsection{Trees}
Let $\Tfr$ be a planar rooted tree. The \Def{arity} of a node of $\Tfr$
is its number of children. An \Def{internal node} (resp. a \Def{leaf})
of $\Tfr$ is a node with a nonzero (resp. null) arity. Internal nodes
can be \Def{labeled}, that is, each internal node of a tree is associated
with an element of a certain set. Given an internal node $x$ of $\Tfr$,
due to the planarity of $\Tfr$, the children of $x$ are totally ordered
from left to right and are thus indexed from $1$ to the arity $\ell$ of
$x$. For any $i \in [\ell]$, the \Def{$i$th subtree} of $\Tfr$ is the
tree rooted at the $i$th child of $\Tfr$. Similarly, the leaves of
$\Tfr$ are totally ordered from left to right and thus are indexed from
$1$ to the number of its leaves. A tree $\Sfr$ is a \Def{subtree} of
$\Tfr$ if it possible to fit $\Sfr$ at a certain place of $\Tfr$, by
possibly superimposing leaves of $\Sfr$ and internal nodes of $\Tfr$. In
this case, we say that $\Tfr$ \Def{admits an occurrence} of (the
\Def{pattern}) $\Sfr$. Conversely, we say that $\Tfr$ \Def{avoids}
$\Sfr$ if there is no occurrence of $\Sfr$ in $\Tfr$. In our graphical
representations, each planar rooted tree is depicted so that its root is
the uppermost node. Since we consider in the sequel only planar rooted
trees, we shall call these simply \Def{trees}.
\medskip

\subsubsection{Rewrite rules}
Let $S$ be a set of trees. A \Def{rewrite rule} on $S$ is a binary
relation $\Rew$ on $S$ such that for all trees $\Sfr$ and $\Sfr'$ of
$S$, $\Sfr \Rew \Sfr'$ only if $\Sfr$ and $\Sfr'$ have the same number
of leaves. We say that a tree $\Tfr$ is \Def{rewritable in one step}
into $\Tfr'$ by $\Rew$ if there exist two trees $\Sfr$ and $\Sfr'$
satisfying $\Sfr \Rew \Sfr'$ and $\Tfr$ has a subtree $\Sfr$ such that,
by replacing $\Sfr$ by $\Sfr'$ in $\Tfr$, we obtain $\Tfr'$. We denote
by $\Tfr \CRew \Tfr'$ this property, so that $\CRew$ is a binary
relation on $S$. When $\Tfr = \Tfr'$ or when there exists a sequence of
trees $(\Tfr_1, \dots, \Tfr_{k - 1})$ with $k \geq 1$ such that
\begin{math}
    \Tfr \CRew \Tfr_1 \CRew \cdots \CRew \Tfr_{k - 1} \CRew \Tfr',
\end{math}
we say that $\Tfr$ is \Def{rewritable} by $\CRew$ into $\Tfr'$ and we
denote this property by $\Tfr \CRewTrans \Tfr'$. In other words,
$\CRewTrans$ is the reflexive and transitive closure of $\CRew$.
We denote by $\RewTrans$ the reflexive and transitive closure of $\Rew$
and by $\RewTransSym$ (resp. $\CRewTransSym$) the reflexive, transitive,
and symmetric closure of $\Rew$ (resp. $\CRew$). The \Def{vector space
induced} by $\Rew$ is the subspace of the linear span $\Vect(S)$ of all
trees of $S$ generated by the family of all $\Tfr - \Tfr'$ such
that~$\Tfr \RewTransSym \Tfr'$.
\medskip

For instance, let $S$ be the set of all trees where internal nodes are
labeled on $\{\Att, \Btt, \Ctt\}$ and consider the rewrite rule $\Rew$
on $S$ satisfying
\vspace{-1.75em}
\begin{multicols}{2}
\begin{subequations}
\begin{equation}
    \begin{tikzpicture}[xscale=.4,yscale=.23,Centering]
        \node(0)at(0.00,-2.00){};
        \node(2)at(1.00,-2.00){};
        \node(3)at(2.00,-2.00){};
        \node(1)at(1.00,0.00){\begin{math}\Btt\end{math}};
        \draw(0)--(1);
        \draw(2)--(1);
        \draw(3)--(1);
        \node(r)at(1.00,1.75){};
        \draw(r)--(1);
    \end{tikzpicture}
    \enspace \Rew \enspace
    \begin{tikzpicture}[xscale=.27,yscale=.25,Centering]
        \node(0)at(0.00,-3.33){};
        \node(2)at(2.00,-3.33){};
        \node(4)at(4.00,-1.67){};
        \node(1)at(1.00,-1.67){\begin{math}\Att\end{math}};
        \node(3)at(3.00,0.00){\begin{math}\Att\end{math}};
        \draw(0)--(1);
        \draw(1)--(3);
        \draw(2)--(1);
        \draw(4)--(3);
        \node(r)at(3.00,1.5){};
        \draw(r)--(3);
    \end{tikzpicture}\,,
\end{equation}

\begin{equation}
    \begin{tikzpicture}[xscale=.27,yscale=.25,Centering]
        \node(0)at(0.00,-3.33){};
        \node(2)at(2.00,-3.33){};
        \node(4)at(4.00,-1.67){};
        \node(1)at(1.00,-1.67){\begin{math}\Att\end{math}};
        \node(3)at(3.00,0.00){\begin{math}\Ctt\end{math}};
        \draw(0)--(1);
        \draw(1)--(3);
        \draw(2)--(1);
        \draw(4)--(3);
        \node(r)at(3.00,1.5){};
        \draw(r)--(3);
    \end{tikzpicture}
    \enspace \Rew \enspace
    \begin{tikzpicture}[xscale=.27,yscale=.25,Centering]
        \node(0)at(0.00,-1.67){};
        \node(2)at(2.00,-3.33){};
        \node(4)at(4.00,-3.33){};
        \node(1)at(1.00,0.00){\begin{math}\Att\end{math}};
        \node(3)at(3.00,-1.67){\begin{math}\Ctt\end{math}};
        \draw(0)--(1);
        \draw(2)--(3);
        \draw(3)--(1);
        \draw(4)--(3);
        \node(r)at(1.00,1.5){};
        \draw(r)--(1);
    \end{tikzpicture}\,.
\end{equation}
\end{subequations}
\end{multicols}
\noindent We then have the following steps of rewritings by $\Rew$:
\begin{equation}
    \begin{tikzpicture}[xscale=.2,yscale=.16,Centering]
        \node(0)at(0.00,-6.50){};
        \node(10)at(8.00,-6.50){};
        \node(12)at(10.00,-6.50){};
        \node(2)at(1.00,-9.75){};
        \node(4)at(3.00,-9.75){};
        \node(5)at(4.00,-9.75){};
        \node(7)at(5.00,-9.75){};
        \node(8)at(6.00,-9.75){};
        \node[text=BrickRed](1)at(2.00,-3.25){\begin{math}\Btt\end{math}};
        \node(11)at(9.00,-3.25){\begin{math}\Att\end{math}};
        \node(3)at(2.00,-6.50){\begin{math}\Ctt\end{math}};
        \node(6)at(5.00,-6.50){\begin{math}\Btt\end{math}};
        \node(9)at(7.00,0.00){\begin{math}\Ctt\end{math}};
        \draw[draw=BrickRed](0)--(1);
        \draw[draw=BrickRed](1)--(9);
        \draw(10)--(11);
        \draw(11)--(9);
        \draw(12)--(11);
        \draw(2)--(3);
        \draw[draw=BrickRed](3)--(1);
        \draw(4)--(3);
        \draw(5)--(6);
        \draw[draw=BrickRed](6)--(1);
        \draw(7)--(6);
        \draw(8)--(6);
        \node(r)at(7.00,2.5){};
        \draw(r)--(9);
    \end{tikzpicture}
    \enspace \CRew \enspace
    \begin{tikzpicture}[xscale=.18,yscale=.18,Centering]
        \node(0)at(0.00,-8.40){};
        \node(11)at(10.00,-5.60){};
        \node(13)at(12.00,-5.60){};
        \node(2)at(2.00,-11.20){};
        \node(4)at(4.00,-11.20){};
        \node(6)at(6.00,-8.40){};
        \node(8)at(7.00,-8.40){};
        \node(9)at(8.00,-8.40){};
        \node(1)at(1.00,-5.60){\begin{math}\Att\end{math}};
        \node[text=BrickRed](10)at(9.00,0.00){\begin{math}\Ctt\end{math}};
        \node(12)at(11.00,-2.80){\begin{math}\Att\end{math}};
        \node(3)at(3.00,-8.40){\begin{math}\Ctt\end{math}};
        \node[text=BrickRed](5)at(5.00,-2.80){\begin{math}\Att\end{math}};
        \node(7)at(7.00,-5.60){\begin{math}\Btt\end{math}};
        \draw(0)--(1);
        \draw[draw=BrickRed](1)--(5);
        \draw(11)--(12);
        \draw[draw=BrickRed](12)--(10);
        \draw(13)--(12);
        \draw(2)--(3);
        \draw(3)--(1);
        \draw(4)--(3);
        \draw[draw=BrickRed](5)--(10);
        \draw(6)--(7);
        \draw[draw=BrickRed](7)--(5);
        \draw(8)--(7);
        \draw(9)--(7);
        \node(r)at(9.00,2.25){};
        \draw[draw=BrickRed](r)--(10);
    \end{tikzpicture}
    \enspace \CRew \enspace
    \begin{tikzpicture}[xscale=.18,yscale=.15,Centering]
        \node(0)at(0.00,-7.00){};
        \node(11)at(10.00,-10.50){};
        \node(13)at(12.00,-10.50){};
        \node(2)at(2.00,-10.50){};
        \node(4)at(4.00,-10.50){};
        \node(6)at(6.00,-10.50){};
        \node(8)at(7.00,-10.50){};
        \node(9)at(8.00,-10.50){};
        \node(1)at(1.00,-3.50){\begin{math}\Att\end{math}};
        \node(10)at(9.00,-3.50){\begin{math}\Ctt\end{math}};
        \node(12)at(11.00,-7.00){\begin{math}\Att\end{math}};
        \node(3)at(3.00,-7.00){\begin{math}\Ctt\end{math}};
        \node(5)at(5.00,0.00){\begin{math}\Att\end{math}};
        \node[text=BrickRed](7)at(7.00,-7.00){\begin{math}\Btt\end{math}};
        \draw(0)--(1);
        \draw(1)--(5);
        \draw(10)--(5);
        \draw(11)--(12);
        \draw(12)--(10);
        \draw(13)--(12);
        \draw(2)--(3);
        \draw(3)--(1);
        \draw(4)--(3);
        \draw[draw=BrickRed](6)--(7);
        \draw[draw=BrickRed](7)--(10);
        \draw[draw=BrickRed](8)--(7);
        \draw[draw=BrickRed](9)--(7);
        \node(r)at(5.00,2.75){};
        \draw(r)--(5);
    \end{tikzpicture}
    \enspace \CRew \enspace
    \begin{tikzpicture}[xscale=.2,yscale=.16,Centering]
        \node(0)at(0.00,-6.00){};
        \node(10)at(10.00,-9.00){};
        \node(12)at(12.00,-9.00){};
        \node(14)at(14.00,-9.00){};
        \node(2)at(2.00,-9.00){};
        \node(4)at(4.00,-9.00){};
        \node(6)at(6.00,-12.00){};
        \node(8)at(8.00,-12.00){};
        \node(1)at(1.00,-3.00){\begin{math}\Att\end{math}};
        \node(11)at(11.00,-3.00){\begin{math}\Ctt\end{math}};
        \node(13)at(13.00,-6.00){\begin{math}\Att\end{math}};
        \node(3)at(3.00,-6.00){\begin{math}\Ctt\end{math}};
        \node(5)at(5.00,0.00){\begin{math}\Att\end{math}};
        \node(7)at(7.00,-9.00){\begin{math}\Att\end{math}};
        \node(9)at(9.00,-6.00){\begin{math}\Att\end{math}};
        \draw(0)--(1);
        \draw(1)--(5);
        \draw(10)--(9);
        \draw(11)--(5);
        \draw(12)--(13);
        \draw(13)--(11);
        \draw(14)--(13);
        \draw(2)--(3);
        \draw(3)--(1);
        \draw(4)--(3);
        \draw(6)--(7);
        \draw(7)--(9);
        \draw(8)--(7);
        \draw(9)--(11);
        \node(r)at(5.00,2.25){};
        \draw(r)--(5);
    \end{tikzpicture}\,.
\end{equation}
\medskip

We shall use the standard terminology (\Def{terminating},
\Def{normal form}, \Def{confluent}, \Def{convergent}, {\em etc.}) about
rewrite rules~\cite{BN98}. Let us recall the most important definitions.
Let $\Rew$ be a rewrite rule on a set $S$ of trees. We say that $\Rew$
is \Def{terminating} if there is no infinite chain
$\Tfr \CRew \Tfr_1 \CRew \Tfr_2 \CRew \cdots$. In this case, any tree
$\Tfr$ of $S$ that cannot be rewritten by $\Rew$ is a \Def{normal form}
for $\Rew$. We say that $\Rew$ is \Def{confluent} if for any trees
$\Tfr$, $\Rfr_1$, and $\Rfr_2$ such that $\Tfr \CRewTrans \Rfr_1$ and
$\Tfr \CRewTrans \Rfr_2$, there exists a tree $\Tfr'$ such that
$\Rfr_1 \CRewTrans \Tfr'$ and $\Rfr_2 \CRewTrans \Tfr'$. When $\Rew$ is
both terminating and confluent, $\Rew$ is \Def{convergent}.
\medskip

\subsection{Operads and Koszulity}
We adopt most of notations and conventions of~\cite{LV12} about operads.
For the sake of completeness, we recall here the elementary notions
about operads employed thereafter.
\medskip

\subsubsection{Nonsymmetric operads}
A \Def{nonsymmetric operad in the category of vector spaces}, or a
\Def{nonsymmetric operad} for short, is a graded vector space
\begin{equation}
    \Oca := \bigoplus_{n \geq 1} \Oca(n)
\end{equation}
together with linear maps
\begin{equation}
    \circ_i : \Oca(n) \otimes \Oca(m) \to \Oca(n + m - 1),
    \qquad n, m \geq 1, i \in [n],
\end{equation}
called \Def{partial compositions}, and a distinguished element
$\Unit \in \Oca(1)$, the \Def{unit} of $\Oca$. This data has to satisfy
the three relations
\begin{subequations}
\begin{equation} \label{equ:operad_axiom_1}
    (x \circ_i y) \circ_{i + j - 1} z = x \circ_i (y \circ_j z),
    \qquad x \in \Oca(n), y \in \Oca(m),
    z \in \Oca(k), i \in [n], j \in [m],
\end{equation}
\begin{equation} \label{equ:operad_axiom_2}
    (x \circ_i y) \circ_{j + m - 1} z = (x \circ_j z) \circ_i y,
    \qquad x \in \Oca(n), y \in \Oca(m),
    z \in \Oca(k), i < j \in [n],
\end{equation}
\begin{equation} \label{equ:operad_axiom_3}
    \Unit \circ_1 x = x = x \circ_i \Unit,
    \qquad x \in \Oca(n), i \in [n].
\end{equation}
\end{subequations}
Since we consider in this paper only nonsymmetric operads, we shall call
these simply \Def{operads}. Moreover, in this work, we shall only
consider operads $\Oca$ for which $\Oca(1)$ has dimension~$1$.
\medskip

When $\Oca$ is such that all $\Oca(n)$ have finite dimensions for all
$n \geq 1$, the \Def{Hilbert series} of~$\Oca$ is the series
$\Hilbert_\Oca(t)$ defined by
\begin{equation}
    \Hilbert_\Oca(t) := \sum_{n \geq 1} \dim \Oca(n)\, t^n.
\end{equation}
If $x$ is an element of $\Oca$ such that $x \in \Oca(n)$ for a
$n \geq 1$, we say that $n$ is the \Def{arity} of $x$ and we denote it
by $|x|$. The \Def{complete composition map} of $\Oca$ is the linear map
\begin{equation}
    \circ : \Oca(n) \otimes \Oca(m_1) \otimes \dots \otimes \Oca(m_n)
    \to \Oca(m_1 + \dots + m_n),
\end{equation}
defined, for any $x \in \Oca(n)$ and $y_1, \dots, y_n \in \Oca$, by
\begin{equation}
    x \circ [y_1, \dots, y_n]
    := (\dots ((x \circ_n y_n) \circ_{n - 1} y_{n - 1}) \dots)
    \circ_1 y_1.
\end{equation}
If $\Oca_1$ and $\Oca_2$ are two operads, a linear map
$\phi : \Oca_1 \to \Oca_2$ is an \Def{operad morphism} if it respects
arities, sends the unit of $\Oca_1$ to the unit of $\Oca_2$, and
commutes with partial composition maps. We say that $\Oca_2$ is a
\Def{suboperad} of $\Oca_1$ if $\Oca_2$ is a graded subspace of $\Oca_1$,
$\Oca_1$ and $\Oca_2$ have the same unit, and the partial compositions
of $\Oca_2$ are the ones of $\Oca_1$ restricted on $\Oca_2$. For any
subset $G$ of $\Oca$, the \Def{operad generated} by $G$ is the smallest
suboperad $\Oca^G$ of $\Oca$ containing $G$. When $\Oca^G = \Oca$ and
$G$ is minimal with respect to the inclusion among the subsets of $G$
satisfying this property, $G$ is a \Def{minimal generating set} of
$\Oca$ and its elements are \Def{generators} of $\Oca$. An \Def{operad
ideal} of $\Oca$ is a graded subspace $I$ of $\Oca$ such that, for any
$x \in \Oca$ and $y \in I$, $x \circ_i y$ and $y \circ_j x$ are in $I$
for all valid integers $i$ and $j$. Given an operad ideal $I$ of $\Oca$,
one can define the \Def{quotient operad} $\Oca/_I$ of $\Oca$ by $I$ in
the usual way.
\medskip

Let us recall and set some more definitions about operads. The
\Def{Hadamard product} between the two operads $\Oca_1$ and $\Oca_2$ is
the operad $\Oca_1 * \Oca_2$ satisfying
$(\Oca_1 * \Oca_2)(n) = \Oca_1(n) \otimes \Oca_2(n)$, and its partial
composition is defined component-wise from the partial compositions of
$\Oca_1$ and~$\Oca_2$. An element $x$ of $\Oca(2)$ is \Def{associative}
if $x \circ_1 x = x \circ_2 x$. A \Def{symmetry} of $\Oca$ is either an
automorphism or an antiautomorphism of $\Oca$. The set of all symmetries
of $\Oca$ forms a group for the map composition, called the \Def{group
of symmetries} of~$\Oca$. A basis $B := \sqcup_{n \geq 1} B(n)$ of
$\Oca$ is a \Def{set-operad basis} if all partial compositions of
elements of $B$ belong to $B$. In this case, we say that $\Oca$ is a
\Def{set-operad} with respect to the basis $B$. Moreover, when all the
maps
\begin{equation}
    \circ_i^y : B(n) \to B(n + m - 1),
    \qquad n, m \geq 1, i \in [n], y \in B(m),
\end{equation}
defined by
\begin{equation}
    \circ_i^y(x) = x \circ_i y,
    \qquad x \in B(n),
\end{equation}
are injective, we say that $B$ is a \Def{basic set-operad basis} of
$\Oca$. This notion is a slightly modified version of the original
notion of basic set-operads introduced by Vallette\cite{Val07}. Finally,
$\Oca$ is \Def{cyclic} (see~\cite{GK95}) if there is a map
\begin{equation} \label{equ:rotation_map_0}
    \rho : \Oca(n) \to \Oca(n), \qquad n \geq 1,
\end{equation}
satisfying, for all $x \in \Oca(n)$, $y \in \Oca(m)$, and $i \in [n]$,
\begin{subequations}
\begin{equation} \label{equ:rotation_map_1}
    \rho(\Unit) = \Unit,
\end{equation}
\begin{equation} \label{equ:rotation_map_2}
    \rho^{n + 1}(x) = x,
\end{equation}
\begin{equation} \label{equ:rotation_map_3}
    \rho(x \circ_i y) =
    \begin{cases}
        \rho(y) \circ_m \rho(x) & \mbox{if } i = 1, \\
        \rho(x) \circ_{i - 1} y & \mbox{otherwise}.
    \end{cases}
\end{equation}
\end{subequations}
We call such a map $\rho$ a \Def{rotation map}.
\medskip

\subsubsection{Syntax trees and free operads}
Let $G := \sqcup_{n \geq 1} G(n)$ be a graded set. The \Def{arity} of an
element $x$ of $G$ is $n$ provided that $x \in G(n)$. A
\Def{syntax tree} on $G$ is a planar rooted tree such that its internal
nodes of arity $n$ are labeled by elements of arity $n$ of~$G$. The
\Def{degree} (resp. \Def{arity}) of a syntax tree is its number of
internal nodes (resp. leaves). For instance, if $G := G(2) \sqcup G(3)$
with $G(2) := \{\Att, \Ctt\}$ and $G(3) := \{\Btt\}$,
\begin{equation}
    \begin{tikzpicture}[xscale=.26,yscale=.15,Centering]
        \node(0)at(0.00,-6.50){};
        \node(10)at(8.00,-9.75){};
        \node(12)at(10.00,-9.75){};
        \node(2)at(2.00,-6.50){};
        \node(4)at(3.00,-3.25){};
        \node(5)at(4.00,-9.75){};
        \node(7)at(5.00,-9.75){};
        \node(8)at(6.00,-9.75){};
        \node(1)at(1.00,-3.25){\begin{math}\Ctt\end{math}};
        \node(11)at(9.00,-6.50){\begin{math}\Att\end{math}};
        \node(3)at(3.00,0.00){\begin{math}\Btt\end{math}};
        \node(6)at(5.00,-6.50){\begin{math}\Btt\end{math}};
        \node(9)at(7.00,-3.25){\begin{math}\Att\end{math}};
        \node(r)at(3.00,2.75){};
        \draw(0)--(1); \draw(1)--(3); \draw(10)--(11); \draw(11)--(9);
        \draw(12)--(11); \draw(2)--(1); \draw(4)--(3); \draw(5)--(6);
        \draw(6)--(9); \draw(7)--(6); \draw(8)--(6); \draw(9)--(3);
        \draw(r)--(3);
    \end{tikzpicture}
\end{equation}
is a syntax tree on $G$ of degree $5$ and arity $8$. Its root is
labeled by $\Btt$ and has arity $3$.
\medskip

Let $\Gen := \oplus_{n \geq 1} \Gen(n)$ be a graded vector space. In
particular, $\Gen$ is a graded set so that we can consider syntax trees
on $\Gen$. The \Def{free operad} over $\Gen$ is the operad $\Free(\Gen)$
wherein for any $n \geq 1$, $\Free(\Gen)(n)$ is the linear span of the
syntax trees on $\Gen$ of arity $n$. The labeling of the internal nodes
of the trees of $\Free(\Gen)$ is linear in the sense that if $\Tfr$ is a
syntax tree on $\Gen$ having an internal node labeled by
$x + \lambda y \in \Gen$, $\lambda \in \K$, then, in $\Free(\Gen)$, we
have $\Tfr = \Tfr_x + \lambda \Tfr_y$, where $\Tfr_x$ (resp. $\Tfr_y$)
is the tree obtained by labeling by $x$ (resp. $y$) the considered node
labeled by $x + \lambda y$ in $\Tfr$. The partial composition
$\Sfr \circ_i \Tfr$ of $\Free(\Gen)$ of two syntax trees $\Sfr$ and
$\Tfr$ on $\Gen$ consists in grafting the root of $\Tfr$ on the $i$th
leaf of $\Sfr$. The unit $\Leaf$ of $\Free(\Gen)$ is the tree consisting
in one leaf. For instance, by setting $\Gen := \Vect(G)$ where $G$ is
the graded set defined in the previous example, one has in~$\Free(\Gen)$,
\begin{equation}
    \begin{tikzpicture}[xscale=.28,yscale=.17,Centering]
        \node(0)at(0.00,-5.33){};
        \node(2)at(2.00,-5.33){};
        \node(4)at(4.00,-5.33){};
        \node(6)at(5.00,-5.33){};
        \node(7)at(6.00,-5.33){};
        \node[text=RoyalBlue](1)at(1.00,-2.67){\begin{math}\Att\end{math}};
        \node[text=RoyalBlue](3)at(3.00,0.00){\begin{math}\Att\end{math}};
        \node[text=RoyalBlue](5)at(5.00,-2.67){\begin{math}\Btt\end{math}};
        \node(r)at(3.00,2.25){};
        \draw[draw=RoyalBlue](0)--(1); \draw[draw=RoyalBlue](1)--(3);
        \draw[draw=RoyalBlue](2)--(1); \draw[draw=RoyalBlue](4)--(5);
        \draw[draw=RoyalBlue](5)--(3); \draw[draw=RoyalBlue](6)--(5);
        \draw[draw=RoyalBlue](7)--(5); \draw[draw=RoyalBlue](r)--(3);
    \end{tikzpicture}
    \enspace \circ_3 \enspace
    \begin{tikzpicture}[xscale=.28,yscale=.28,Centering]
        \node(0)at(0.00,-1.67){};
        \node(2)at(2.00,-3.33){};
        \node(4)at(4.00,-3.33){};
        \node[text=BrickRed](1)at(1.00,0.00){\begin{math}\Ctt\end{math}};
        \node[text=BrickRed](3)at(3.00,-1.67)
            {\begin{math}\Att + \Ctt\end{math}};
        \node(r)at(1.00,1.5){};
        \draw[draw=BrickRed](0)--(1); \draw[draw=BrickRed](2)--(3);
        \draw[draw=BrickRed](3)--(1); \draw[draw=BrickRed](4)--(3);
        \draw[draw=BrickRed](r)--(1);
    \end{tikzpicture}
    \enspace = \enspace
    \begin{tikzpicture}[xscale=.25,yscale=.18,Centering]
        \node(0)at(0.00,-4.80){};
        \node(10)at(9.00,-4.80){};
        \node(11)at(10.00,-4.80){};
        \node(2)at(2.00,-4.80){};
        \node(4)at(4.00,-7.20){};
        \node(6)at(6.00,-9.60){};
        \node(8)at(8.00,-9.60){};
        \node[text=RoyalBlue](1)at(1.00,-2.40){\begin{math}\Att\end{math}};
        \node[text=RoyalBlue](3)at(3.00,0.00){\begin{math}\Att\end{math}};
        \node[text=BrickRed](5)at(5.00,-4.80){\begin{math}\Ctt\end{math}};
        \node[text=BrickRed](7)at(7.00,-7.20){\begin{math}\Att\end{math}};
        \node[text=RoyalBlue](9)at(9.00,-2.40){\begin{math}\Btt\end{math}};
        \node(r)at(3.00,2.40){};
        \draw[draw=RoyalBlue](0)--(1);
        \draw[draw=RoyalBlue](1)--(3);
        \draw[draw=RoyalBlue](10)--(9);
        \draw[draw=RoyalBlue](11)--(9);
        \draw[draw=RoyalBlue](2)--(1);
        \draw[draw=BrickRed](4)--(5);
        \draw[draw=RoyalBlue!50!BrickRed!50](5)--(9);
        \draw[draw=BrickRed](6)--(7);
        \draw[draw=BrickRed](7)--(5);
        \draw[draw=BrickRed](8)--(7);
        \draw[draw=RoyalBlue](9)--(3);
        \draw[draw=RoyalBlue](r)--(3);
    \end{tikzpicture}
    \enspace + \enspace
    \begin{tikzpicture}[xscale=.25,yscale=.18,Centering]
        \node(0)at(0.00,-4.80){};
        \node(10)at(9.00,-4.80){};
        \node(11)at(10.00,-4.80){};
        \node(2)at(2.00,-4.80){};
        \node(4)at(4.00,-7.20){};
        \node(6)at(6.00,-9.60){};
        \node(8)at(8.00,-9.60){};
        \node[text=RoyalBlue](1)at(1.00,-2.40){\begin{math}\Att\end{math}};
        \node[text=RoyalBlue](3)at(3.00,0.00){\begin{math}\Att\end{math}};
        \node[text=BrickRed](5)at(5.00,-4.80){\begin{math}\Ctt\end{math}};
        \node[text=BrickRed](7)at(7.00,-7.20){\begin{math}\Ctt\end{math}};
        \node[text=RoyalBlue](9)at(9.00,-2.40){\begin{math}\Btt\end{math}};
        \node(r)at(3.00,2.40){};
        \draw[draw=RoyalBlue](0)--(1);
        \draw[draw=RoyalBlue](1)--(3);
        \draw[draw=RoyalBlue](10)--(9);
        \draw[draw=RoyalBlue](11)--(9);
        \draw[draw=RoyalBlue](2)--(1);
        \draw[draw=BrickRed](4)--(5);
        \draw[draw=RoyalBlue!50!BrickRed!50](5)--(9);
        \draw[draw=BrickRed](6)--(7);
        \draw[draw=BrickRed](7)--(5);
        \draw[draw=BrickRed](8)--(7);
        \draw[draw=RoyalBlue](9)--(3);
        \draw[draw=RoyalBlue](r)--(3);
    \end{tikzpicture}\,.
\end{equation}
\medskip

We denote by $\Corolla : \Gen \to \Free(\Gen)$ the inclusion map,
sending any $x$ of $\Gen$ to the \Def{corolla} labeled by $x$, that
is the syntax tree consisting in a single internal node labeled by $x$
attached to a required number of leaves. In the sequel, if required by
the context, we shall implicitly see any element $x$ of $\Gen$ as the
corolla $\Corolla(x)$ of $\Free(\Gen)$. For instance, when $x$ and $y$
are two elements of $\Gen$, we shall simply denote by $x \circ_i y$ the
syntax tree $\Corolla(x) \circ_i \Corolla(y)$ for all valid integers~$i$.
\medskip

\subsubsection{Evaluations and treelike expressions}
For any operad $\Oca$, by seeing $\Oca$ as a graded vector space,
$\Free(\Oca)$ is by definition the free operad on $\Oca$. The
\Def{evaluation map} of $\Oca$ is the map
\begin{equation}
    \Eval : \Free(\Oca) \to \Oca,
\end{equation}
defined linearly by induction, for any syntax tree $\Tfr$ on $\Oca$, by
\begin{equation}
    \Eval(\Tfr) :=
    \begin{cases}
        \Unit \in \Oca
            & \mbox{if } \Tfr = \Leaf, \\
        x \circ \left[\Eval\left(\Tfr_1\right), \dots,
        \Eval\left(\Tfr_k\right)\right]
            & \mbox{otherwise},
    \end{cases}
\end{equation}
where $x$ is the label of the root of $\Tfr$ and $\Tfr_1$, \dots,
$\Tfr_k$ are, from left to right, the subtrees of $\Tfr$. This map is
the unique surjective operad morphism from $\Free(\Oca)$ to $\Oca$
satisfying $\Eval(\Corolla(x)) = x$ for all $x \in \Oca$. If $S$ is a
subspace of $\Oca$, a \Def{treelike expression} on $S$ of $x \in \Oca$
is a tree $\Tfr$ of $\Free(\Oca)$ such that $\Eval(\Tfr) = x$ and all
internal nodes of $\Tfr$ are labeled on~$S$.
\medskip

\subsubsection{Presentations by generators and relations}
A \Def{presentation} of an operad $\Oca$ consists in a pair $(G, \Rel)$
such that $G := \sqcup_{n \geq 1} G(n)$ is a graded set, $\Rel$ is a
subspace of $\Free(\Gen)$, where $\Gen := \Vect(G)$, and $\Oca$ is
isomorphic to $\Free(\Gen)/_{\langle \Rel \rangle}$,  where
$\langle \Rel \rangle$ is the operad ideal of $\Free(\Gen)$ generated by
$\Rel$. We call $\Gen$ the \Def{space of generators} and $\Rel$ the
\Def{space of relations} of $\Oca$. We say that $\Oca$ is \Def{quadratic}
if there is a presentation $(G, \Rel)$ of $\Oca$ such that $\Rel$ is a
homogeneous subspace of $\Free(\Gen)$ consisting in syntax trees of
degree $2$. Besides, we say that $\Oca$ is \Def{binary} if there is a
presentation $(G, \Rel)$ of $\Oca$ such that $\Gen$ is concentrated in
arity~$2$. Furthermore, if $\Oca$ admits a presentation $(G, \Rel)$ and
$\Rew$ is a rewrite rule on $\Free(\Gen)$ such that the space induced by
$\Rew$ is $\Rel$, we say that $\Rew$ is an \Def{orientation} of~$\Rel$.
\medskip

\subsubsection{Koszul duality and Koszulity}%
\label{subsubsec:koszul_duality_koszulity_criterion}
In~\cite{GK94}, Ginzburg and Kapranov extended the notion of Koszul
duality of quadratic associative algebras to quadratic operads. Starting
with an operad $\Oca$ admitting a binary and quadratic presentation
$(G, \Rel)$ where $G$ is finite, the \Def{Koszul dual} of $\Oca$ is the
operad $\Oca^!$, isomorphic to the operad admitting the presentation
$\left(G, \Rel^\perp\right)$ where $\Rel^\perp$ is the annihilator of
$\Rel$ in $\Free(\Gen)$, $\Gen$ being the space $\Vect(G)$, with respect
to the scalar product
\begin{equation}
    \langle -, - \rangle :
    \Free(\Gen)(3) \otimes \Free(\Gen)(3) \to \K
\end{equation}
linearly defined, for all $x, x', y, y' \in \Gen(2)$, by
\begin{equation} \label{equ:scalar_product_koszul}
    \left\langle x \circ_i y, x' \circ_{i'} y' \right\rangle :=
    \begin{cases}
        1 & \mbox{if }
            x = x', y = y', \mbox{ and } i = i' = 1, \\
        -1 & \mbox{if }
            x = x', y = y', \mbox{ and } i = i' = 2, \\
        0 & \mbox{otherwise}.
    \end{cases}
\end{equation}
Then, with knowledge of a presentation of $\Oca$, one can compute a
presentation of~$\Oca^!$.
\medskip

Recall that a quadratic operad $\Oca$ is \Def{Koszul} if its Koszul
complex is acyclic~\cite{GK94,LV12}. Furthermore, when $\Oca$ is Koszul
and admits an Hilbert series, the Hilbert series of $\Oca$ and of its
Koszul dual $\Oca^!$ are related~\cite{GK94} by
\begin{equation} \label{equ:Hilbert_series_Koszul_operads}
    \Hilbert_\Oca\left(-\Hilbert_{\Oca^!}(-t)\right) = t.
\end{equation}
Relation~\eqref{equ:Hilbert_series_Koszul_operads} can be used either to
prove that an operad is not Koszul (it is the case when the coefficients
of the hypothetical Hilbert series of the Koszul dual admits
coefficients that are not negative integers) or to compute the Hilbert
series of the Koszul dual of a Koszul operad.
\medskip

In this work, to prove the Koszulity of an operad $\Oca$, we shall make
use of a tool introduced by Dotsenko and Khoroshkin~\cite{DK10} in the
context of Gröbner bases for operads, which reformulates in our context,
by using rewrite rules on syntax trees, in the following way.
\begin{Lemma} \label{lem:koszulity_criterion_pbw}
    Let $\Oca$ be an operad admitting a quadratic presentation
    $(G, \Rel)$. If there exists an orientation $\Rew$ of
    $\Rel$ such that $\Rew$ is a convergent rewrite rule, then
    $\Oca$ is Koszul.
\end{Lemma}
\medskip

When $\Rew$ satisfies the conditions contained in the statement of
Lemma~\ref{lem:koszulity_criterion_pbw}, the set of the normal forms of
$\Rew$ forms a basis of $\Oca$, called \Def{Poincaré-Birkhoff-Witt basis}.
These bases arise from the work of Hoffbeck~\cite{Hof10} (see
also~\cite{LV12}).
\medskip

\subsubsection{Algebras over operads}
Any operad $\Oca$ encodes a category of algebras whose objects are
called \Def{$\Oca$-algebras}. An $\Oca$-algebra $\Alg_\Oca$ is a vector
space endowed with a linear left action
\begin{equation}
    \cdot : \Oca(n) \otimes \Alg_\Oca^{\otimes n} \to \Alg_\Oca,
    \qquad n \geq 1,
\end{equation}
satisfying the relations imposed by the structure of $\Oca$, that are
\begin{multline} \label{equ:algebra_over_operad}
    (x \circ_i y) \cdot
    \left(a_1 \otimes \dots \otimes a_{n + m - 1}\right)
    = \\
    x \cdot \left(a_1 \otimes \dots
        \otimes a_{i - 1} \otimes
        y \cdot \left(a_i \otimes \dots \otimes a_{i + m - 1}\right)
        \otimes a_{i + m} \otimes
        \dots \otimes a_{n + m - 1}\right),
\end{multline}
for all $x \in \Oca(n)$, $y \in \Oca(m)$, $i \in [n]$, and
\begin{math}
    a_1 \otimes \dots \otimes a_{n + m - 1}
    \in \Alg_\Oca^{\otimes {n + m - 1}}.
\end{math}
\medskip

Notice that, by~\eqref{equ:algebra_over_operad}, if $G$ is a generating
set of $\Oca$, it is enough to define the action of each $x \in G$ on
$\Alg_\Oca^{\otimes |x|}$ to wholly define~$\cdot$. In other words, any
element $x$ of $\Oca$ of arity $n$ plays the role of a linear operation
\begin{equation}
    x : \Alg_\Oca^{\otimes n}  \to \Alg_\Oca,
\end{equation}
taking $n$ elements of $\Alg_\Oca$ as inputs and computing an element of
$\Alg_\Oca$. By a slight but convenient abuse of notation, for any
$x \in \Oca(n)$, we shall denote by $x(a_1, \dots, a_n)$, or by
$a_1 \, x \, a_2$ if $x$ has arity $2$, the element
$x \cdot (a_1 \otimes \dots \otimes a_n)$ of $\Alg_\Oca$, for any
$a_1 \otimes \dots \otimes a_n \in \Alg_\Oca^{\otimes n}$. Observe that
by~\eqref{equ:algebra_over_operad}, any associative element of $\Oca$
gives rise to an associative operation on~$\Alg_\Oca$.
\medskip

\section{From unitary magmas to operads}\label{sec:construction_Cli}
We describe in this section our construction from unitary magmas to
operads and study its main algebraic and combinatorial properties.
\medskip

\subsection{Cliques, unitary magmas, and operads}%
\label{subsec:decorated_cliques}
We present here our main combinatorial objects, the decorated cliques.
The construction $\Cli$, which takes a unitary magma as input and
produces an operad, is defined.
\medskip

\subsubsection{Cliques}
\label{subsec:cliques}
A \Def{clique} of \Def{size} $n \geq 1$ is a complete graph $\Pfr$ on
the set of vertices $[n + 1]$. An \Def{arc} of $\Pfr$ is a pair of
integers $(x, y)$ with $1 \leq x < y \leq n + 1$, a \Def{diagonal} is
an arc $(x, y)$ different from $(x, x + 1)$ and $(1, n + 1)$, and an
\Def{edge} is an arc of the form $(x, x + 1)$ and different from
$(1, n + 1)$. We denote by $\Arcs_\Pfr$ (resp. $\Diagonals_\Pfr$,
$\Edges_\Pfr$) the set of all arcs (resp. diagonals, edges) of $\Pfr$.
For any $i \in [n]$, the \Def{$i$th edge} of $\Pfr$ is the edge
$(i, i + 1)$, and the arc $(1, n + 1)$ is the \Def{base} of~$\Pfr$.
\medskip

In our graphical representations, each clique is depicted so that its
base is the bottommost segment, vertices are implicitly numbered from
$1$ to $n + 1$ in the clockwise direction, and the diagonals are not
drawn. For example,
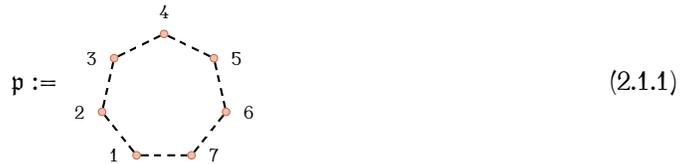
\begin{equation}
    \Pfr :=
    \begin{tikzpicture}[scale=.85,Centering]
        \node[CliquePoint](1)at(-0.43,-0.90){};
        \node[CliquePoint](2)at(-0.97,-0.22){};
        \node[CliquePoint](3)at(-0.78,0.62){};
        \node[CliquePoint](4)at(-0.00,1.00){};
        \node[CliquePoint](5)at(0.78,0.62){};
        \node[CliquePoint](6)at(0.97,-0.22){};
        \node[CliquePoint](7)at(0.43,-0.90){};
        \draw[CliqueEmptyEdge](1)edge[]node[CliqueLabel]{}(2);
        \draw[CliqueEmptyEdge](1)edge[]node[CliqueLabel]{}(7);
        \draw[CliqueEmptyEdge](2)edge[]node[CliqueLabel]{}(3);
        \draw[CliqueEmptyEdge](3)edge[]node[CliqueLabel]{}(4);
        \draw[CliqueEmptyEdge](4)edge[]node[CliqueLabel]{}(5);
        \draw[CliqueEmptyEdge](5)edge[]node[CliqueLabel]{}(6);
        \draw[CliqueEmptyEdge](6)edge[]node[CliqueLabel]{}(7);
        \node[left of=1,node distance=3mm,font=\scriptsize]
            {\begin{math}1\end{math}};
        \node[left of=2,node distance=3mm,font=\scriptsize]
            {\begin{math}2\end{math}};
        \node[left of=3,node distance=3mm,font=\scriptsize]
            {\begin{math}3\end{math}};
        \node[above of=4,node distance=3mm,font=\scriptsize]
            {\begin{math}4\end{math}};
        \node[right of=5,node distance=3mm,font=\scriptsize]
            {\begin{math}5\end{math}};
        \node[right of=6,node distance=3mm,font=\scriptsize]
            {\begin{math}6\end{math}};
        \node[right of=7,node distance=3mm,font=\scriptsize]
            {\begin{math}7\end{math}};
    \end{tikzpicture}
\end{equation}
is a clique of size $6$. Its set of all diagonals satisfies
\begin{equation}
    \Diagonals_\Pfr =
    \{(1, 3), (1, 4), (1, 5), (1, 6), (2, 4), (2, 5), (2, 6), (2, 7),
    (3, 5), (3, 6), (3, 7), (4, 6), (4, 7), (5, 7)\},
\end{equation}
its set of all edges satisfies
\begin{equation}
    \Edges_\Pfr = \{(1, 2), (2, 3), (3, 4), (4, 5), (5, 6), (6, 7)\},
\end{equation}
and its set of all arcs satisfies
\begin{equation}
    \Arcs_\Pfr = \Diagonals_\Pfr \sqcup \Edges_\Pfr \sqcup \{(1, 7)\}.
\end{equation}
\medskip

\subsubsection{Unitary magmas and decorated cliques}
Recall first that a unitary magma is a set endowed with a binary
operation $\Op$ admitting a left and right unit $\Unit_\Mca$. For
convenience, we denote by $\bar{\Mca}$ the set
$\Mca \setminus \{\Unit_\Mca\}$. To explore some examples in this
article, we shall mostly consider four sorts of unitary magmas: the
additive unitary magma on all integers denoted by $\Z$, the cyclic
additive unitary magma on $\Z/_{\ell \Z}$ denoted by $\N_\ell$, the
unitary magma
\begin{equation}
    \Dbb_\ell := \{\Unit, 0, \Dtt_1, \dots, \Dtt_\ell\}
\end{equation}
where $\Unit$ is the unit of $\Dbb_\ell$, $0$ is absorbing, and
$\Dtt_i \Op \Dtt_j = 0$ for all $i, j \in [\ell]$, and the unitary magma
\begin{equation}
    \Ebb_\ell := \{\Unit, \Ett_1, \dots, \Ett_\ell\}
\end{equation}
where $\Unit$ is the unit of $\Ebb_\ell$ and $\Ett_i \Op \Ett_j = \Unit$
for all $i, j \in [\ell]$. Observe that since
\begin{equation}
    \Ett_1 \Op (\Ett_1 \Op \Ett_2) = \Ett_1 \Op \Unit = \Ett_1
    \ne
    \Ett_2 = \Unit \Op \Ett_2 = (\Ett_1 \Op \Ett_1) \Op \Ett_2,
\end{equation}
all unitary magmas $\Ebb_\ell$, $\ell \geq 2$, are not monoids.
\medskip

An \Def{$\Mca$-decorated clique} (or an \Def{$\Mca$-clique} for short)
is a clique $\Pfr$ endowed with a map
\begin{equation}
    \phi_\Pfr : \Arcs_\Pfr \to \Mca.
\end{equation}
For convenience, for any arc $(x, y)$ of $\Pfr$, we shall denote by
$\Pfr(x, y)$ the value $\phi_\Pfr((x, y))$. Moreover, we say that the
arc $(x, y)$ is \Def{labeled} by $\Pfr(x, y)$. When the arc $(x, y)$ is
labeled by an element different from  $\Unit_\Mca$, we say that the arc
$(x, y)$ is \Def{solid}. Again for convenience, we denote by $\Pfr_0$
the label $\Pfr(1, n + 1)$ of the base of $\Pfr$, where $n$ is the size
of $\Pfr$. Moreover, we denote by $\Pfr_i$, $i \in [n]$, the label
$\Pfr(i, i + 1)$ of the $i$th edge of $\Pfr$. By convention, we
require that the $\Mca$-clique of size $1$ having its base labeled by
$\Unit_\Mca$ is the only such object of size $1$. The set of all
$\Mca$-cliques is denoted by $\Cliques_\Mca$.
\medskip

In our graphical representations, we shall represent any $\Mca$-clique
$\Pfr$ by drawing a clique of the same size as the one of $\Pfr$
following the conventions explained before, and by labeling some of its
arcs in the following way. If $(x, y)$ is a solid arc of $\Pfr$, we
represent it by a line decorated by $\Pfr(x, y)$. If $(x, y)$ is not
solid and is an edge or the base of $\Pfr$, we represent it as a dashed
line. In the remaining case, when $(x, y)$ is a diagonal of $\Pfr$ and
is not solid, we do not draw it. For instance,
\begin{equation}
    \Pfr :=
    \begin{tikzpicture}[scale=.85,Centering]
        \node[CliquePoint](1)at(-0.43,-0.90){};
        \node[CliquePoint](2)at(-0.97,-0.22){};
        \node[CliquePoint](3)at(-0.78,0.62){};
        \node[CliquePoint](4)at(-0.00,1.00){};
        \node[CliquePoint](5)at(0.78,0.62){};
        \node[CliquePoint](6)at(0.97,-0.22){};
        \node[CliquePoint](7)at(0.43,-0.90){};
        \draw[CliqueEdge](1)edge[]node[CliqueLabel]
            {\begin{math}-1\end{math}}(2);
        \draw[CliqueEdge](1)edge[bend left=30]node[CliqueLabel,near start]
            {\begin{math}2\end{math}}(5);
        \draw[CliqueEdge](1)edge[]node[CliqueLabel]
            {\begin{math}1\end{math}}(7);
        \draw[CliqueEmptyEdge](2)edge[]node[CliqueLabel]{}(3);
        \draw[CliqueEmptyEdge](3)edge[]node[CliqueLabel]{}(4);
        \draw[CliqueEdge](3)edge[bend left=30]node[CliqueLabel,near start]
            {\begin{math}-1\end{math}}(7);
        \draw[CliqueEdge](4)edge[]node[CliqueLabel]
            {\begin{math}3\end{math}}(5);
        \draw[CliqueEdge](5)edge[]node[CliqueLabel]
            {\begin{math}2\end{math}}(6);
        \draw[CliqueEdge](5)edge[]node[CliqueLabel]
            {\begin{math}1\end{math}}(7);
        \draw[CliqueEmptyEdge](6)edge[]node[CliqueLabel]{}(7);
    \end{tikzpicture}
\end{equation}
is a $\Z$-clique such that, among others $\Pfr(1, 2) = -1$,
$\Pfr(1, 5) = 2$, $\Pfr(3, 7) = -1$, $\Pfr(5, 7) = 1$, $\Pfr(2, 3) = 0$
(because $0$ is the unit of $\Z$), and $\Pfr(2, 6) = 0$ (for the same
reason).
\medskip

Let us now provide some definitions and statistics on
$\Mca$-cliques. Let $\Pfr$ be an $\Mca$-clique of size $n$. The
\Def{skeleton} of $\Pfr$ is the undirected graph $\Skel(\Pfr)$ on the
set of vertices $[n + 1]$ and such that for any $x < y \in [n + 1]$,
there is an arc $\{x, y\}$ in $\Skel(\Pfr)$ if $(x, y)$ is a solid arc
of $\Pfr$. The \Def{degree} of a vertex $x$ of $\Pfr$ is the number of
solid arcs adjacent to $x$. The \Def{degree} $\Degr(\Pfr)$ of $\Pfr$ is
the maximal degree among its vertices. Two (non-necessarily solid)
diagonals $(x, y)$ and $(x', y')$ of $\Pfr$ are \Def{crossing} if
$x < x' < y < y'$ or $x' < x < y' < y$. The \Def{crossing} of a solid
diagonal $(x, y)$ of $\Pfr$ is the number of solid diagonals $(x', y')$
such that $(x, y)$ and $(x', y')$ are crossing. The \Def{crossing}
$\Cros(\Pfr)$ of $\Pfr$ is the maximal crossing among its solid
diagonals. When $\Cros(\Pfr) = 0$, there are no crossing diagonals in
$\Pfr$ and in this case, $\Pfr$ is \Def{noncrossing}. A (non-necessarily
solid) arc $(x, y)$ \Def{includes} a (non-necessarily solid) arc
$(x', y')$ of $\Pfr$ if $x \leq x' < y' \leq y$. We say that $\Pfr$ is
\Def{inclusion-free} if for any solid arcs $(x, y)$ and $(x', y')$ of
$\Pfr$ such that $(x, y)$ includes $(x', y')$, $(x, y) = (x', y')$.
Besides, $\Pfr$ is \Def{acyclic} if the subgraph of $\Pfr$ consisting in
all its solid arcs is acyclic.  When $\Pfr$ has no solid edges nor solid
base, $\Pfr$ is \Def{white}. The \Def{border} of $\Pfr$ is the word
$\Border(\Pfr)$ of length $n$ such that for any $i \in [n]$,
$\Border(\Pfr)_i = \Pfr_i$. If $\Pfr$ has no solid diagonals, $\Pfr$ is
an \Def{$\Mca$-bubble}. An $\Mca$-bubble is hence specified by a pair
$(b, u)$ where $b$ is the label of its base and $u$ is its border.
Obviously, all $\Mca$-bubbles are noncrossing $\Mca$-cliques. The set of
all $\Mca$-bubbles is denoted by $\Bubbles_\Mca$. We denote by
$\Triangles_\Mca$ the set of all \Def{$\Mca$-triangles}, that are
$\Mca$-cliques of size~$2$. Notice that any $\Mca$-triangle is also an
$\Mca$-bubble.
\medskip

\subsubsection{Partial composition of $\Mca$-cliques}
From now, the \Def{arity} of an $\Mca$-clique $\Pfr$ is its size and
is denoted by $|\Pfr|$. For any unitary magma $\Mca$, we define the
vector space
\begin{equation}
    \Cli\Mca := \bigoplus_{n \geq 1} \Cli\Mca(n),
\end{equation}
where $\Cli\Mca(n)$ is the linear span of all $\Mca$-cliques of arity
$n$, $n \geq 1$. The set $\Cliques_\Mca$ forms hence a basis of
$\Cli\Mca$ called \Def{fundamental basis}. Observe that the space
$\Cli\Mca(1)$ has dimension $1$ since it is the linear span of the
$\Mca$-clique $\UnitClique$. We endow $\Cli\Mca$ with partial
composition maps
\begin{equation}
    \circ_i : \Cli\Mca(n) \otimes \Cli\Mca(m) \to
    \Cli\Mca(n + m - 1),
    \qquad n, m \geq 1, i \in [n],
\end{equation}
defined linearly, in the fundamental basis, in the following way. Let
$\Pfr$ and $\Qfr$ be two $\Mca$-cliques of respective arities $n$ and
$m$, and $i \in [n]$ be an integer. We set $\Pfr \circ_i \Qfr$ as the
$\Mca$-clique of arity $n + m - 1$ such that, for any arc $(x, y)$ where
$1 \leq x < y \leq n + m$,
\begin{equation} \label{equ:partial_composition_Cli_M}
    (\Pfr \circ_i \Qfr)(x, y) :=
    \begin{cases}
        \Pfr(x, y)
            & \mbox{if } y \leq i, \\
        \Pfr(x, y - m + 1)
            & \mbox{if } x \leq i < i + m \leq y
            \mbox{ and } (x, y) \ne (i, i + m), \\
        \Pfr(x - m + 1, y - m + 1)
            & \mbox{if } i + m \leq x, \\
        \Qfr(x - i + 1, y - i + 1)
            & \mbox{if } i \leq x < y \leq i + m
              \mbox{ and } (x, y) \ne (i, i + m), \\
        \Pfr_i \Op \Qfr_0
            & \mbox{if } (x, y) = (i, i + m), \\
        \Unit_\Mca
            & \mbox{otherwise}.
    \end{cases}
\end{equation}
We recall that $\Op$ denotes the operation of $\Mca$ and $\Unit_\Mca$
its unit. In a geometric way, $\Pfr \circ_i \Qfr$ is obtained by gluing
the base of $\Qfr$ onto the $i$th edge of $\Pfr$, by relabeling the
common arcs between $\Pfr$ and $\Qfr$, respectively the arcs
$(i, i + 1)$ and $(1, m + 1)$, by $\Pfr_i \Op \Qfr_0$, and by adding all
required non solid diagonals on the graph thus obtained to become a
clique (see Figure~\ref{fig:composition_Cli_M}).
\begin{figure}[ht]
    \centering
    \begin{equation*}
        \begin{tikzpicture}[scale=.55,Centering]
            \node[shape=coordinate](1)at(-0.50,-0.87){};
            \node[shape=coordinate](2)at(-1.00,-0.00){};
            \node[CliquePoint](3)at(-0.50,0.87){};
            \node[CliquePoint](4)at(0.50,0.87){};
            \node[shape=coordinate](5)at(1.00,0.00){};
            \node[shape=coordinate](6)at(0.50,-0.87){};
            \draw[CliqueEdge](1)edge[]node[CliqueLabel]{}(2);
            \draw[CliqueEdge](1)edge[]node[CliqueLabel]{}(6);
            \draw[CliqueEdge](2)edge[]node[CliqueLabel]{}(3);
            \draw[CliqueEdge](3)edge[]node[CliqueLabel]
                {\begin{math}\Pfr_i\end{math}}(4);
            \draw[CliqueEdge](4)edge[]node[CliqueLabel]{}(5);
            \draw[CliqueEdge](5)edge[]node[CliqueLabel]{}(6);
            \node[left of=3,node distance=2mm,font=\scriptsize]
                {\begin{math}i\end{math}};
            \node[right of=4,node distance=4mm,font=\scriptsize]
                {\begin{math}i\!+\!1\end{math}};
            \node[font=\footnotesize](name)at(0,0)
                {\begin{math}\Pfr\end{math}};
        \end{tikzpicture}
        \enspace \circ_i \enspace\enspace
        \begin{tikzpicture}[scale=.55,Centering]
            \node[CliquePoint](1)at(-0.50,-0.87){};
            \node[shape=coordinate](2)at(-1.00,-0.00){};
            \node[shape=coordinate](3)at(-0.50,0.87){};
            \node[shape=coordinate](4)at(0.50,0.87){};
            \node[shape=coordinate](5)at(1.00,0.00){};
            \node[CliquePoint](6)at(0.50,-0.87){};
            \draw[CliqueEdge,RawSienna!80](1)edge[]node[CliqueLabel]{}(2);
            \draw[CliqueEdge,draw=RawSienna!80](1)edge[]node[CliqueLabel]
                {\begin{math}\Qfr_0\end{math}}(6);
            \draw[CliqueEdge,RawSienna!80](2)edge[]node[CliqueLabel]{}(3);
            \draw[CliqueEdge,RawSienna!80](3)edge[]node[CliqueLabel]{}(4);
            \draw[CliqueEdge,RawSienna!80](4)edge[]node[CliqueLabel]{}(5);
            \draw[CliqueEdge,RawSienna!80](5)edge[]node[CliqueLabel]{}(6);
            \node[font=\footnotesize](name)at(0,0)
                {\begin{math}\Qfr\end{math}};
        \end{tikzpicture}
        \quad = \quad
        \begin{tikzpicture}[scale=.55,Centering]
            \begin{scope}
            \node[shape=coordinate](1)at(-0.50,-0.87){};
            \node[shape=coordinate](2)at(-1.00,-0.00){};
            \node[CliquePoint](3)at(-0.50,0.87){};
            \node[CliquePoint](4)at(0.50,0.87){};
            \node[shape=coordinate](5)at(1.00,0.00){};
            \node[shape=coordinate](6)at(0.50,-0.87){};
            \draw[CliqueEdge](1)edge[]node[CliqueLabel]{}(2);
            \draw[CliqueEdge](1)edge[]node[CliqueLabel]{}(6);
            \draw[CliqueEdge](2)edge[]node[CliqueLabel]{}(3);
            \draw[CliqueEdge](3)edge[]node[CliqueLabel]
                {\begin{math}\Pfr_i\end{math}}(4);
            \draw[CliqueEdge](4)edge[]node[CliqueLabel]{}(5);
            \draw[CliqueEdge](5)edge[]node[CliqueLabel]{}(6);
            \node[left of=3,node distance=2mm,font=\scriptsize]
                {\begin{math}i\end{math}};
            \node[right of=4,node distance=4mm,font=\scriptsize]
                {\begin{math}i\!+\!1\end{math}};
            \node[font=\footnotesize](name)at(0,0)
                {\begin{math}\Pfr\end{math}};
            \end{scope}
            \begin{scope}[yshift=2.1cm]
            \node[CliquePoint](1)at(-0.50,-0.87){};
            \node[shape=coordinate](2)at(-1.00,-0.00){};
            \node[shape=coordinate](3)at(-0.50,0.87){};
            \node[shape=coordinate](4)at(0.50,0.87){};
            \node[shape=coordinate](5)at(1.00,0.00){};
            \node[CliquePoint](6)at(0.50,-0.87){};
            \draw[CliqueEdge,RawSienna!80](1)edge[]node[CliqueLabel]{}(2);
            \draw[CliqueEdge,draw=RawSienna!80](1)edge[]node[CliqueLabel]
                {\begin{math}\Qfr_0\end{math}}(6);
            \draw[CliqueEdge,RawSienna!80](2)edge[]node[CliqueLabel]{}(3);
            \draw[CliqueEdge,RawSienna!80](3)edge[]node[CliqueLabel]{}(4);
            \draw[CliqueEdge,RawSienna!80](4)edge[]node[CliqueLabel]{}(5);
            \draw[CliqueEdge,RawSienna!80](5)edge[]node[CliqueLabel]{}(6);
            \node[font=\footnotesize](name)at(0,0)
                {\begin{math}\Qfr\end{math}};
            \end{scope}
        \end{tikzpicture}
        \quad = \quad
        \begin{tikzpicture}[scale=.75,Centering]
            \node[shape=coordinate](1)at(-0.31,-0.95){};
            \node[shape=coordinate](2)at(-0.81,-0.59){};
            \node[CliquePoint](3)at(-1.00,-0.00){};
            \node[shape=coordinate](4)at(-0.81,0.59){};
            \node[shape=coordinate](5)at(-0.31,0.95){};
            \node[shape=coordinate](6)at(0.31,0.95){};
            \node[shape=coordinate](7)at(0.81,0.59){};
            \node[CliquePoint](8)at(1.00,0.00){};
            \node[shape=coordinate](9)at(0.81,-0.59){};
            \node[shape=coordinate](10)at(0.31,-0.95){};
            \draw[CliqueEdge](1)edge[]node[]{}(2);
            \draw[CliqueEdge](1)edge[]node[]{}(10);
            \draw[CliqueEdge](2)edge[]node[]{}(3);
            \draw[CliqueEdge,RawSienna!80](3)edge[]node[]{}(4);
            \draw[CliqueEdge,RawSienna!80](4)edge[]node[]{}(5);
            \draw[CliqueEdge,RawSienna!80](5)edge[]node[]{}(6);
            \draw[CliqueEdge,RawSienna!80](6)edge[]node[]{}(7);
            \draw[CliqueEdge,RawSienna!80](7)edge[]node[]{}(8);
            \draw[CliqueEdge](8)edge[]node[]{}(9);
            \draw[CliqueEdge](9)edge[]node[]{}(10);
            \node[left of=3,node distance=2mm,font=\scriptsize]
                {\begin{math}i\end{math}};
            \node[right of=8,node distance=5mm,font=\scriptsize]
                {\begin{math}i\!+\!m\end{math}};
            \draw[CliqueEdge,draw=RoyalBlue!50!RawSienna!50]
                (3)edge[]node[CliqueLabel]
                {\begin{math}\Pfr_i \Op \Qfr_0\end{math}}(8);
        \end{tikzpicture}
    \end{equation*}
    \caption{\footnotesize
    The partial composition of $\Cli\Mca$, described in geometric terms.
    Here, $\Pfr$ and $\Qfr$ are two $\Mca$-cliques. The arity of $\Qfr$
    is~$m$ and $i$ is an integer between $1$ and~$|\Pfr|$.}
    \label{fig:composition_Cli_M}
\end{figure}
For example, in $\Cli\Z$, one has the two partial compositions
\begin{subequations}
\begin{equation}
    \begin{tikzpicture}[scale=.85,Centering]
        \node[CliquePoint](1)at(-0.50,-0.87){};
        \node[CliquePoint](2)at(-1.00,-0.00){};
        \node[CliquePoint](3)at(-0.50,0.87){};
        \node[CliquePoint](4)at(0.50,0.87){};
        \node[CliquePoint](5)at(1.00,0.00){};
        \node[CliquePoint](6)at(0.50,-0.87){};
        \draw[CliqueEdge](1)edge[]node[CliqueLabel]
            {\begin{math}1\end{math}}(2);
        \draw[CliqueEdge](1)edge[bend left=30]node[CliqueLabel]
            {\begin{math}-2\end{math}}(5);
        \draw[CliqueEmptyEdge](1)edge[]node[CliqueLabel]{}(6);
        \draw[CliqueEdge](2)edge[]node[CliqueLabel]
            {\begin{math}-2\end{math}}(3);
        \draw[CliqueEmptyEdge](3)edge[]node[CliqueLabel]{}(4);
        \draw[CliqueEdge](3)edge[bend right=30]node[CliqueLabel]
            {\begin{math}1\end{math}}(5);
        \draw[CliqueEmptyEdge](4)edge[]node[CliqueLabel]{}(5);
        \draw[CliqueEmptyEdge](5)edge[]node[CliqueLabel]{}(6);
    \end{tikzpicture}
    \enspace \circ_2 \enspace
    \begin{tikzpicture}[scale=.65,Centering]
        \node[CliquePoint](1)at(-0.71,-0.71){};
        \node[CliquePoint](2)at(-0.71,0.71){};
        \node[CliquePoint](3)at(0.71,0.71){};
        \node[CliquePoint](4)at(0.71,-0.71){};
        \draw[CliqueEmptyEdge](1)edge[]node[CliqueLabel]{}(2);
        \draw[CliqueEdge](1)edge[]node[CliqueLabel,near end]
            {\begin{math}1\end{math}}(3);
        \draw[CliqueEdge](1)edge[]node[CliqueLabel]
            {\begin{math}3\end{math}}(4);
        \draw[CliqueEmptyEdge](2)edge[]node[CliqueLabel]{}(3);
        \draw[CliqueEdge](2)edge[]node[CliqueLabel,near start]
            {\begin{math}1\end{math}}(4);
        \draw[CliqueEdge](3)edge[]node[CliqueLabel]
            {\begin{math}2\end{math}}(4);
    \end{tikzpicture}
    \enspace = \enspace
    \begin{tikzpicture}[scale=1.1,Centering]
        \node[CliquePoint](1)at(-0.38,-0.92){};
        \node[CliquePoint](2)at(-0.92,-0.38){};
        \node[CliquePoint](3)at(-0.92,0.38){};
        \node[CliquePoint](4)at(-0.38,0.92){};
        \node[CliquePoint](5)at(0.38,0.92){};
        \node[CliquePoint](6)at(0.92,0.38){};
        \node[CliquePoint](7)at(0.92,-0.38){};
        \node[CliquePoint](8)at(0.38,-0.92){};
        \draw[CliqueEdge](1)edge[]node[CliqueLabel]
            {\begin{math}1\end{math}}(2);
        \draw[CliqueEdge](1)edge[bend left=30]node[CliqueLabel]
            {\begin{math}-2\end{math}}(7);
        \draw[CliqueEmptyEdge](1)edge[]node[CliqueLabel]{}(8);
        \draw[CliqueEmptyEdge](2)edge[]node[CliqueLabel]{}(3);
        \draw[CliqueEdge](2)edge[bend right=30]node[CliqueLabel]
            {\begin{math}1\end{math}}(4);
        \draw[CliqueEdge](2)edge[bend right=30]node[CliqueLabel]
            {\begin{math}1\end{math}}(5);
        \draw[CliqueEmptyEdge](3)edge[]node[CliqueLabel]{}(4);
        \draw[CliqueEdge](3)edge[bend right=30]node[CliqueLabel]
            {\begin{math}1\end{math}}(5);
        \draw[CliqueEdge](4)edge[]node[CliqueLabel]
            {\begin{math}2\end{math}}(5);
        \draw[CliqueEmptyEdge](5)edge[]node[CliqueLabel]{}(6);
        \draw[CliqueEdge](5)edge[bend right=30]node[CliqueLabel]
            {\begin{math}1\end{math}}(7);
        \draw[CliqueEmptyEdge](6)edge[]node[CliqueLabel]{}(7);
        \draw[CliqueEmptyEdge](7)edge[]node[CliqueLabel]{}(8);
    \end{tikzpicture}\,,
\end{equation}
\begin{equation}
    \begin{tikzpicture}[scale=.85,Centering]
        \node[CliquePoint](1)at(-0.50,-0.87){};
        \node[CliquePoint](2)at(-1.00,-0.00){};
        \node[CliquePoint](3)at(-0.50,0.87){};
        \node[CliquePoint](4)at(0.50,0.87){};
        \node[CliquePoint](5)at(1.00,0.00){};
        \node[CliquePoint](6)at(0.50,-0.87){};
        \draw[CliqueEdge](1)edge[]node[CliqueLabel]
            {\begin{math}1\end{math}}(2);
        \draw[CliqueEdge](1)edge[bend left=30]node[CliqueLabel]
            {\begin{math}-2\end{math}}(5);
        \draw[CliqueEmptyEdge](1)edge[]node[CliqueLabel]{}(6);
        \draw[CliqueEdge](2)edge[]node[CliqueLabel]
            {\begin{math}-2\end{math}}(3);
        \draw[CliqueEmptyEdge](3)edge[]node[CliqueLabel]{}(4);
        \draw[CliqueEdge](3)edge[bend right=30]node[CliqueLabel]
            {\begin{math}1\end{math}}(5);
        \draw[CliqueEmptyEdge](4)edge[]node[CliqueLabel]{}(5);
        \draw[CliqueEmptyEdge](5)edge[]node[CliqueLabel]{}(6);
    \end{tikzpicture}
    \enspace \circ_2 \enspace
    \begin{tikzpicture}[scale=.65,Centering]
        \node[CliquePoint](1)at(-0.71,-0.71){};
        \node[CliquePoint](2)at(-0.71,0.71){};
        \node[CliquePoint](3)at(0.71,0.71){};
        \node[CliquePoint](4)at(0.71,-0.71){};
        \draw[CliqueEmptyEdge](1)edge[]node[CliqueLabel]{}(2);
        \draw[CliqueEdge](1)edge[]node[CliqueLabel,near end]
            {\begin{math}1\end{math}}(3);
        \draw[CliqueEdge](1)edge[]node[CliqueLabel]
            {\begin{math}2\end{math}}(4);
        \draw[CliqueEmptyEdge](2)edge[]node[CliqueLabel]{}(3);
        \draw[CliqueEdge](2)edge[]node[CliqueLabel,near start]
            {\begin{math}1\end{math}}(4);
        \draw[CliqueEdge](3)edge[]node[CliqueLabel]
            {\begin{math}2\end{math}}(4);
    \end{tikzpicture}
    \enspace = \enspace
    \begin{tikzpicture}[scale=1.1,Centering]
        \node[CliquePoint](1)at(-0.38,-0.92){};
        \node[CliquePoint](2)at(-0.92,-0.38){};
        \node[CliquePoint](3)at(-0.92,0.38){};
        \node[CliquePoint](4)at(-0.38,0.92){};
        \node[CliquePoint](5)at(0.38,0.92){};
        \node[CliquePoint](6)at(0.92,0.38){};
        \node[CliquePoint](7)at(0.92,-0.38){};
        \node[CliquePoint](8)at(0.38,-0.92){};
        \draw[CliqueEdge](1)edge[]node[CliqueLabel]
            {\begin{math}1\end{math}}(2);
        \draw[CliqueEdge](1)edge[bend left=30]node[CliqueLabel]
            {\begin{math}-2\end{math}}(7);
        \draw[CliqueEmptyEdge](1)edge[]node[CliqueLabel]{}(8);
        \draw[CliqueEmptyEdge](2)edge[]node[CliqueLabel]{}(3);
        \draw[CliqueEdge](2)edge[bend right=30]node[CliqueLabel]
            {\begin{math}1\end{math}}(4);
        \draw[CliqueEmptyEdge](3)edge[]node[CliqueLabel]{}(4);
        \draw[CliqueEdge](3)edge[bend right=30]node[CliqueLabel]
            {\begin{math}1\end{math}}(5);
        \draw[CliqueEdge](4)edge[]node[CliqueLabel]
            {\begin{math}2\end{math}}(5);
        \draw[CliqueEmptyEdge](5)edge[]node[CliqueLabel]{}(6);
        \draw[CliqueEdge](5)edge[bend right=30]node[CliqueLabel]
            {\begin{math}1\end{math}}(7);
        \draw[CliqueEmptyEdge](6)edge[]node[CliqueLabel]{}(7);
        \draw[CliqueEmptyEdge](7)edge[]node[CliqueLabel]{}(8);
    \end{tikzpicture}\,.
\end{equation}
\end{subequations}
\medskip

\subsubsection{Functorial construction from unitary magmas to operads}
If $\Mca_1$ and $\Mca_2$ are two unitary magmas and
$\theta : \Mca_1 \to \Mca_2$ is a unitary magma morphism, we define
\begin{equation}
    \Cli\theta : \Cli\Mca_1 \to \Cli\Mca_2
\end{equation}
as the linear map sending any $\Mca_1$-clique $\Pfr$ of arity $n$ to the
$\Mca_2$-clique $(\Cli\theta)(\Pfr)$ of the same arity such that, for
any arc $(x, y)$ where $1 \leq x < y \leq  n + 1$,
\begin{equation} \label{equ:morphism_Cli_M}
    ((\Cli\theta)(\Pfr))(x, y) := \theta(\Pfr(x, y)).
\end{equation}
In a geometric way, $(\Cli\theta)(\Pfr)$ is the $\Mca_2$-clique obtained
by relabeling each arc of $\Pfr$ by the image of its label by~$\theta$.
\medskip

\begin{Theorem} \label{thm:clique_construction}
    The construction $\Cli$ is a functor from the category of unitary
    magmas to the category of operads. Moreover, $\Cli$ respects
    injections and surjections.
\end{Theorem}
\begin{proof}
    Let $\Mca$ be a unitary magma. The fact that $\Cli\Mca$ endowed with
    the partial composition~\eqref{equ:partial_composition_Cli_M} is an
    operad can be established by showing that the two associativity
    relations~\eqref{equ:operad_axiom_1} and~\eqref{equ:operad_axiom_2}
    of operads are satisfied. This is a technical but a simple
    verification. Since $\Cli\Mca(1)$ contains $\UnitClique$ and this
    element is the unit for this partial composition,
    \eqref{equ:operad_axiom_3} holds. Moreover, let $\Mca_1$ and
    $\Mca_2$ be two unitary magmas and $\theta : \Mca_1 \to \Mca_2$ be a
    unitary magma morphism. The fact that the map $\Cli\theta$ defined
    in~\eqref{equ:morphism_Cli_M} is an operad morphism is a
    straightforward checking. All this imply that $\Cli$ is a functor.
    Finally, the fact that $\Cli$ respects injections and surjections is
    also a straightforward verification.
\end{proof}
\medskip

We name the construction $\Cli$ as the \Def{clique construction} and
$\Cli\Mca$ as the \Def{$\Mca$-clique operad}. Observe that the
fundamental basis of $\Cli\Mca$ is a set-operad basis of $\Cli\Mca$.
Besides, when $\Mca$ is the trivial unitary magma $\{\Unit_\Mca\}$,
$\Cli\Mca$ is the linear span of all decorated cliques having only
non-solid arcs. Thus, each space $\Cli\Mca(n)$, $n \geq 1$, is of
dimension $1$ and it follows from the definition of the partial
composition of $\Cli\Mca$ that this operad is isomorphic to the
associative operad $\As$. The next result shows that the clique
construction is compatible with the Cartesian product of unitary magmas.
\medskip

\begin{Proposition} \label{prop:Cli_M_Cartesian_product}
    Let $\Mca_1$ and $\Mca_2$ be two unitary magmas. Then, the operads
    $(\Cli\Mca_1) * (\Cli\Mca_2)$ and $\Cli(\Mca_1 \times \Mca_2)$ are
    isomorphic.
\end{Proposition}
\begin{proof}
    Let
    $\phi : (\Cli\Mca_1) * (\Cli\Mca_2) \to \Cli(\Mca_1 \times \Mca_2)$
    be the map defined linearly as follows. For any $\Mca_1$-clique
    $\Pfr$ of $\Cli\Mca_1$ and any $\Mca_2$-clique $\Qfr$ of $\Cli\Mca_2$
    both of arity $n$, $\phi(\Pfr \otimes \Qfr)$ is the
    $\Mca_1 \times \Mca_2$-clique defined, for any
    $1 \leq x < y \leq n + 1$, by
    \begin{equation}
        \left(\phi(\Pfr \otimes \Qfr)\right)(x, y)
        := (\Pfr(x, y), \Qfr(x, y)).
    \end{equation}
    Let the map
    $\psi : \Cli(\Mca_1 \times \Mca_2) \to (\Cli\Mca_1) * (\Cli\Mca_2)$
    defined linearly, for any $\Mca_1 \times \Mca_2$-clique $\Rfr$ of
    $\Cli(\Mca_1 \times \Mca_2)$ of arity $n$, as  follows. The
    $\Mca_1$-clique $\Pfr$ and the $\Mca_2$-clique $\Qfr$ of arity $n$
    of the tensor $\Pfr \otimes \Qfr := \psi(\Rfr)$ are defined, for any
    $1 \leq x < y \leq n + 1$, by $\Pfr(x, y) := a$ and
    $\Qfr(x, y) := b$ where $(a, b) = \Rfr(x, y)$. Since we observe
    immediately that $\psi$ is the inverse of $\phi$, $\phi$ is a
    bijection. Moreover, it follows from the definition of the partial
    composition of clique operads that $\phi$ is an operad morphism. The
    statement of the proposition follows.
\end{proof}
\medskip

\subsection{General properties}
We investigate here some properties of clique operads, as their
dimensions, their minimal generating sets, the fact that they admit a
cyclic operad structure, and describe their partial compositions over
two alternative bases.
\medskip

\subsubsection{Binary relations}
Let us start by remarking that, depending on the cardinality $m$ of
$\Mca$, the set of all $\Mca$-cliques can be interpreted as particular
binary relations. When $m \geq 4$, let us set
$\Mca = \{\Unit_\Mca, \Att, \Btt, \Ctt, \dots\}$ so that $\Att$, $\Btt$,
and $\Ctt$ are distinguished pairwise distinct elements of $\Mca$
different from $\Unit_\Mca$. Given an $\Mca$-clique $\Pfr$ of arity
$n \geq 2$, we build a binary relation $\BinRel$ on $[n + 1]$
satisfying, for all $x < y \in [n + 1]$,
\begin{equation}\begin{split}
    x \BinRel y & \quad \mbox{ if } \Pfr(x, y) = \Att, \\
    y \BinRel x & \quad \mbox{ if } \Pfr(x, y) = \Btt, \\
    x \BinRel y \mbox{ and } y \BinRel x
        & \quad \mbox{ if } \Pfr(x, y) = \Ctt.
\end{split}\end{equation}
In particular, when $m = 2$ (resp. $m = 3$, $m = 4$),
$\Mca = \{\Unit, \Ctt\}$ (resp. $\Mca = \{\Unit, \Att, \Btt\}$,
$\Mca = \{\Unit, \Att, \Btt, \Ctt\}$) and the set of all $\Mca$-cliques
of arities $n \geq 2$ is in one-to-one correspondence with the set of all
irreflexive and symmetric (resp. irreflexive and antisymmetric,
irreflexive) binary relations on $[n + 1]$. Therefore, the operads
$\Cli\Mca$ can be interpreted as operads involving binary relations with
more or less properties.
\medskip

\subsubsection{Dimensions and minimal generating set}

\begin{Proposition} \label{prop:dimensions_Cli_M}
    Let $\Mca$ be a finite unitary magma. For all $n \geq 2$,
    \begin{equation} \label{equ:dimensions_Cli_M}
        \dim \Cli\Mca(n) = m^{\binom{n + 1}{2}},
    \end{equation}
    where $m := \# \Mca$.
\end{Proposition}
\begin{proof}
    By definition of the clique construction and of $\Mca$-cliques, the
    dimension of $\Cli\Mca(n)$ is the number of maps from the set
    $\left\{(x, y) \in [n + 1]^2 : x < y\right\}$ to $\Mca$. Therefore,
    when $n \geq 2$, this implies~\eqref{equ:dimensions_Cli_M}.
\end{proof}
\medskip

From Proposition~\ref{prop:dimensions_Cli_M}, the first dimensions of
$\Cli\Mca$ depending on $m := \# \Mca$ are
\begin{subequations}
\begin{equation}
    1, 1, 1, 1, 1, 1, 1, 1,
    \qquad m = 1,
\end{equation}
\begin{equation}
    1, 8, 64, 1024, 32768, 2097152, 268435456, 68719476736,
    \qquad m = 2,
\end{equation}
\begin{multline}
    1, 27, 729, 59049, 14348907, 10460353203, 22876792454961, \\
    150094635296999121,
    \qquad m = 3,
\end{multline}
\begin{multline}
    1, 64, 4096, 1048576, 1073741824, 4398046511104, 72057594037927936, \\
    4722366482869645213696,
    \qquad m = 4.
\end{multline}
\end{subequations}
Except for the first terms, the second one forms
Sequence~\OEIS{A006125}, the third one forms Sequence~\OEIS{A047656},
and the last one forms Sequence~\OEIS{A053763} of~\cite{Slo}.
\medskip

\begin{Lemma} \label{lem:decomposition_clique_diagonal}
    Let $\Mca$ be a unitary magma, $\Pfr$ be an $\Mca$-clique of arity
    $n \geq 2$, and $(x, y)$ be a diagonal of $\Pfr$. Then, the
    following two assertions are equivalent:
    \begin{enumerate}[fullwidth,label={(\it\roman*)}]
        \item \label{item:decomposition_clique_diagonal_1}
        the crossing of $(x, y)$ is $0$;
        \item \label{item:decomposition_clique_diagonal_2}
        the $\Mca$-clique $\Pfr$ expresses as $\Pfr = \Qfr \circ_x \Rfr$,
        where $\Qfr$ is an $\Mca$-clique of arity $n + x - y + 1$ and
        $\Rfr$ is an $\Mca$-clique or arity $y - x$.
    \end{enumerate}
\end{Lemma}
\begin{proof}
    Assume first that~\ref{item:decomposition_clique_diagonal_1} holds.
    Set $\Qfr$ as the $\Mca$-clique of arity $n + x - y + 1$ defined,
    for any arc $(z, t)$ where $1 \leq z < t \leq n + x - y + 2$, by
    \begin{equation}
        \Qfr(z, t) :=
        \begin{cases}
            \Pfr(z, t) & \mbox{if } t \leq x, \\
            \Pfr(z, t + y - x - 1) & \mbox{if } x + 1 \leq t, \\
            \Pfr(z + y - x - 1, t + y - x - 1) & \mbox{otherwise},
        \end{cases}
    \end{equation}
    and $\Rfr$ as the $\Mca$-clique of arity $y - x$ defined, for any
    arc $(z, t)$ where $1 \leq z < t \leq y - x + 1$, by
    \begin{equation}
        \Rfr(z, t) :=
        \begin{cases}
            \Pfr(z + x - 1, t + x - 1)
                & \mbox{if } (z, t) \ne (1, y - x + 1), \\
            \Unit_\Mca & \mbox{otherwise}.
        \end{cases}
    \end{equation}
    By following the definition of the partial composition of
    $\Cli\Mca$, one obtains $\Pfr = \Qfr \circ_x \Rfr$,
    whence~\ref{item:decomposition_clique_diagonal_2} holds.
    \smallskip

    Assume conversely that~\ref{item:decomposition_clique_diagonal_2}
    holds. By definition of the partial composition of $\Cli\Mca$, the
    fact that $\Pfr = \Qfr \circ_x \Rfr$ implies that
    $\Pfr(x', y') = \Unit_\Mca$ for any arc $(x', y')$ such that
    $(x, y)$ and $(x', y')$ are crossing. Therefore,
    \ref{item:decomposition_clique_diagonal_1} holds.
\end{proof}
\medskip

Let $\Primes_\Mca$ be the set of all $\Mca$-cliques $\Pfr$ or arity
$n \geq 2$ that do not satisfy the property of the statement of
Lemma~\ref{lem:decomposition_clique_diagonal}. In other words,
$\Primes_\Mca$ is the set of all $\Mca$-cliques such that, for any
(non-necessarily solid) diagonal $(x, y)$ of $\Pfr$, there is at least
one solid diagonal $(x', y')$ of $\Pfr$ such that $(x, y)$ and
$(x', y')$ are crossing. We call $\Primes_\Mca$ the set of all
\Def{prime $\Mca$-cliques}. Observe that, according to this description,
all $\Mca$-triangles are prime.
\medskip

\begin{Proposition} \label{prop:generating_set_Cli_M}
    Let $\Mca$ be a unitary magma. The set $\Primes_\Mca$ is a minimal
    generating set of~$\Cli\Mca$.
\end{Proposition}
\begin{proof}
    We show by induction on the arity that $\Primes_\Mca$ is a
    generating set of $\Cli\Mca$. Let $\Pfr$ be an $\Mca$-clique. If
    $\Pfr$ is of arity $1$, $\Pfr = \UnitClique$ and hence $\Pfr$
    trivially belongs to $(\Cli\Mca)^{\Primes_\Mca}$. Let us assume that
    $\Pfr$ is of arity $n \geq 2$. First, if $\Pfr \in \Primes_\Mca$,
    then $\Pfr \in (\Cli\Mca)^{\Primes_\Mca}$. Otherwise, $\Pfr$ is an
    $\Mca$-clique which satisfies the description of the statement of
    Lemma~\ref{lem:decomposition_clique_diagonal}. Therefore, by this
    lemma, there are two $\Mca$-cliques $\Qfr$ and $\Rfr$ and an integer
    $x \in [|\Pfr|]$ such that $|\Qfr| < |\Pfr|$, $|\Rfr| < |\Pfr|$, and
    $\Pfr = \Qfr \circ_x \Rfr$. By induction hypothesis, $\Qfr$ and
    $\Rfr$ belong to $(\Cli\Mca)^{\Primes_\Mca}$ and hence, $\Pfr$
    also belongs to~$(\Cli\Mca)^{\Primes_\Mca}$.
    \smallskip

    Finally, by Lemma~\ref{lem:decomposition_clique_diagonal}, if $\Pfr$
    is a prime $\Mca$-clique, $\Pfr$ cannot be expressed as a partial
    composition of prime $\Mca$-cliques. Moreover, since the space
    $\Cli\Mca(1)$ is trivial, these arguments imply that $\Primes_\Mca$
    is a minimal generating set of~$\Cli\Mca$.
\end{proof}
\medskip

\subsubsection{Associative elements}

\begin{Proposition} \label{prop:associative_elements_Cli_M}
    Let $\Mca$ be a unitary magma and $f$ be an element of $\Cli\Mca(2)$
    of the form
    \begin{equation} \label{equ:associative_elements_Cli_M_0}
        f :=
        \sum_{\Pfr \in \Triangles_\Mca}
        \lambda_\Pfr \Pfr,
    \end{equation}
    where the $\lambda_\Pfr$, $\Pfr \in \Triangles_\Mca$, are
    coefficients of $\K$. Then, $f$ is associative if and only if
    \begin{subequations}
    \begin{equation} \label{equ:associative_elements_Cli_M_1}
        \sum_{\substack{
            \Pfr_1, \Qfr_0 \in \Mca \\
            \delta =  \Pfr_1 \Op \Qfr_0
        }}
        \lambda_{\Triangle{\Pfr_0}{\Pfr_1}{\Pfr_2}}
        \lambda_{\Triangle{\Qfr_0}{\Qfr_1}{\Qfr_2}}
        = 0,
        \qquad
        \Pfr_0, \Pfr_2, \Qfr_1, \Qfr_2 \in \Mca,
        \delta \in \bar{\Mca},
    \end{equation}
    \begin{equation} \label{equ:associative_elements_Cli_M_2}
        \sum_{\substack{
            \Pfr_1, \Qfr_0 \in \Mca \\
            \Pfr_1 \Op \Qfr_0 = \Unit_\Mca
        }}
        \lambda_{\Triangle{\Pfr_0}{\Pfr_1}{\Pfr_2}}
        \lambda_{\Triangle{\Qfr_0}{\Qfr_1}{\Qfr_2}}
        -
        \lambda_{\Triangle{\Pfr_0}{\Qfr_1}{\Pfr_1}}
        \lambda_{\Triangle{\Qfr_0}{\Qfr_2}{\Pfr_2}}
        = 0,
        \qquad
        \Pfr_0, \Pfr_2, \Qfr_1, \Qfr_2 \in \Mca,
    \end{equation}
    \begin{equation} \label{equ:associative_elements_Cli_M_3}
        \sum_{\substack{
            \Pfr_2, \Qfr_0 \in \Mca \\
            \delta =  \Pfr_2 \Op \Qfr_0
        }}
        \lambda_{\Triangle{\Pfr_0}{\Pfr_1}{\Pfr_2}}
        \lambda_{\Triangle{\Qfr_0}{\Qfr_1}{\Qfr_2}}
        = 0,
        \qquad
        \Pfr_0, \Pfr_1, \Qfr_1, \Qfr_2 \in \Mca,
        \delta \in \bar{\Mca}.
    \end{equation}
    \end{subequations}
\end{Proposition}
\begin{proof}
    The element $f$ defined in~\eqref{equ:associative_elements_Cli_M_0}
    is associative if and only if $f \circ_1 f - f \circ_2 f = 0$.
    Therefore, this property is equivalent to the fact that
    \begin{equation}\begin{split}
        \label{equ:associative_elements_Cli_M_demo}
        f \circ_1 f - f \circ_2 f
        & =
        \left(
        \sum_{\substack{
            \Pfr, \Qfr \in \Triangles_\Mca \\
            \delta := \Pfr_1 \Op \Qfr_0 \ne \Unit_\Mca
        }}
        \lambda_\Pfr \lambda_\Qfr
        \SquareRight{\Qfr_1}{\Qfr_2}{\Pfr_2}{\Pfr_0}{\delta}
        \right)
        +
        \left(
        \sum_{\substack{
            \Pfr, \Qfr \in \Triangles_\Mca \\
            \Pfr_1 \Op \Qfr_0 = \Unit_\Mca
        }}
        \lambda_\Pfr \lambda_\Qfr
        \SquareN{\Qfr_1}{\Qfr_2}{\Pfr_2}{\Pfr_0}
        \right) \\
        & \quad -
        \left(
        \sum_{\substack{
            \Pfr, \Qfr \in \Triangles_\Mca \\
            \delta := \Pfr_2 \Op \Qfr_0 \ne \Unit_\Mca
        }}
        \lambda_\Pfr \lambda_\Qfr
        \SquareLeft{\Pfr_1}{\Qfr_1}{\Qfr_2}{\Pfr_0}{\delta}
        \right)
        -
        \left(
        \sum_{\substack{
            \Pfr, \Qfr \in \Triangles_\Mca \\
            \Pfr_2 \Op \Qfr_0 = \Unit_\Mca
        }}
        \lambda_\Pfr \lambda_\Qfr
        \SquareN{\Pfr_1}{\Qfr_1}{\Qfr_2}{\Pfr_0}
        \right) \\
        & =
        \left(
        \sum_{\substack{
            \Pfr_0, \Pfr_2, \Qfr_1, \Qfr_2 \in \Mca \\
            \delta \in \bar{\Mca}
        }}
        \left(
        \sum_{\substack{
            \Pfr_1, \Qfr_0 \in \Mca \\
            \delta =  \Pfr_1 \Op \Qfr_0
        }}
        \lambda_{\Triangle{\Pfr_0}{\Pfr_1}{\Pfr_2}}
        \lambda_{\Triangle{\Qfr_0}{\Qfr_1}{\Qfr_2}}
        \right)
        \SquareRight{\Qfr_1}{\Qfr_2}{\Pfr_2}{\Pfr_0}{\delta}
        \right) \\
        & \quad +
        \left(
        \sum_{
            \Pfr_0, \Pfr_2, \Qfr_1, \Qfr_2 \in \Mca
        }
        \left(
        \sum_{\substack{
            \Pfr_1, \Qfr_0 \in \Mca \\
            \Pfr_1 \Op \Qfr_0 = \Unit_\Mca
        }}
        \lambda_{\Triangle{\Pfr_0}{\Pfr_1}{\Pfr_2}}
        \lambda_{\Triangle{\Qfr_0}{\Qfr_1}{\Qfr_2}}
        -
        \lambda_{\Triangle{\Pfr_0}{\Qfr_1}{\Pfr_1}}
        \lambda_{\Triangle{\Qfr_0}{\Qfr_2}{\Pfr_2}}
        \right)
        \SquareN{\Qfr_1}{\Qfr_2}{\Pfr_2}{\Pfr_0}
        \right) \\
        & \quad -
        \left(
        \sum_{\substack{
            \Pfr_0, \Pfr_1, \Qfr_1, \Qfr_2 \in \Mca \\
            \delta \in \bar{\Mca}
        }}
        \left(
        \sum_{\substack{
            \Pfr_2, \Qfr_0 \in \Mca \\
            \delta =  \Pfr_2 \Op \Qfr_0
        }}
        \lambda_{\Triangle{\Pfr_0}{\Pfr_1}{\Pfr_2}}
        \lambda_{\Triangle{\Qfr_0}{\Qfr_1}{\Qfr_2}}
        \right)
        \SquareLeft{\Pfr_1}{\Qfr_1}{\Qfr_2}{\Pfr_0}{\delta}
        \right) \\
        & = 0,
    \end{split}\end{equation}
    and hence, is equivalent to the fact
    that~\eqref{equ:associative_elements_Cli_M_1},
    \eqref{equ:associative_elements_Cli_M_2},
    and~\eqref{equ:associative_elements_Cli_M_3} hold.
\end{proof}
\medskip

For instance, by Proposition~\ref{prop:associative_elements_Cli_M}, the
binary elements
\begin{subequations}
\begin{equation}
    \Triangle{1}{1}{1},
\end{equation}
\begin{equation}
    \TriangleEEE{}{}{}
    +
    \TriangleEXE{}{1}{}
    -
    \TriangleXEE{1}{}{}
    +
    \TriangleEEX{}{}{1}
    -
    \TriangleXXE{1}{1}{}
    +
    \TriangleEXX{}{1}{1}
    -
    \TriangleXEX{1}{}{1}
    -
    \Triangle{1}{1}{1}
\end{equation}
\end{subequations}
of $\Cli\N_2$ are associative, and the binary elements
\begin{subequations}
\begin{equation}
    \TriangleEXX{}{0}{0}
    -
    \Triangle{0}{0}{0},
\end{equation}
\begin{equation}
    \TriangleXEE{0}{}{}
    -
    \TriangleXXE{0}{0}{}
    -
    \TriangleXEX{0}{}{0}
    +
    \Triangle{0}{0}{0}
\end{equation}
\end{subequations}
of $\Cli\Dbb_0$ are associative.
\medskip

\subsubsection{Symmetries}
Let $\Returned : \Cli\Mca \to \Cli\Mca$ be the linear map sending any
$\Mca$-clique $\Pfr$ of arity $n$ to the $\Mca$-clique $\Returned(\Pfr)$
of the same arity such that, for any arc $(x, y)$ where
$1 \leq x < y \leq n + 1$,
\begin{equation} \label{equ:returned_map_Cli_M}
    \left(\Returned(\Pfr)\right)(x, y) := \Pfr(n - y + 2, n - x + 2).
\end{equation}
In a geometric way, $\Returned(\Pfr)$ is the $\Mca$-clique obtained by
applying on $\Pfr$ a reflection trough the vertical line passing by its
base. For instance, one has in $\Cli\Z$,
\begin{equation}
    \Returned\left(
    \begin{tikzpicture}[scale=.85,Centering]
        \node[CliquePoint](1)at(-0.50,-0.87){};
        \node[CliquePoint](2)at(-1.00,-0.00){};
        \node[CliquePoint](3)at(-0.50,0.87){};
        \node[CliquePoint](4)at(0.50,0.87){};
        \node[CliquePoint](5)at(1.00,0.00){};
        \node[CliquePoint](6)at(0.50,-0.87){};
        \draw[CliqueEdge](1)edge[]node[CliqueLabel]
            {\begin{math}1\end{math}}(2);
        \draw[CliqueEdge](1)edge[bend left=30]node[CliqueLabel]
            {\begin{math}-2\end{math}}(5);
        \draw[CliqueEmptyEdge](1)edge[]node[CliqueLabel]{}(6);
        \draw[CliqueEdge](2)edge[]node[CliqueLabel]
            {\begin{math}-2\end{math}}(3);
        \draw[CliqueEmptyEdge](3)edge[]node[CliqueLabel]{}(4);
        \draw[CliqueEdge](3)edge[bend right=30]node[CliqueLabel]
            {\begin{math}1\end{math}}(5);
        \draw[CliqueEmptyEdge](4)edge[]node[CliqueLabel]{}(5);
        \draw[CliqueEmptyEdge](5)edge[]node[CliqueLabel]{}(6);
    \end{tikzpicture}
    \right)
    \enspace = \enspace
    \begin{tikzpicture}[scale=.85,Centering]
        \node[CliquePoint](1)at(-0.50,-0.87){};
        \node[CliquePoint](2)at(-1.00,-0.00){};
        \node[CliquePoint](3)at(-0.50,0.87){};
        \node[CliquePoint](4)at(0.50,0.87){};
        \node[CliquePoint](5)at(1.00,0.00){};
        \node[CliquePoint](6)at(0.50,-0.87){};
        \draw[CliqueEmptyEdge](1)edge[]node[CliqueLabel]{}(2);
        \draw[CliqueEmptyEdge](1)edge[]node[CliqueLabel]{}(6);
        \draw[CliqueEmptyEdge](2)edge[]node[CliqueLabel]{}(3);
        \draw[CliqueEdge](2)edge[bend right=30]node[CliqueLabel]
            {\begin{math}1\end{math}}(4);
        \draw[CliqueEdge](2)edge[bend left=30]node[CliqueLabel]
            {\begin{math}-2\end{math}}(6);
        \draw[CliqueEmptyEdge](3)edge[]node[CliqueLabel]{}(4);
        \draw[CliqueEdge](4)edge[]node[CliqueLabel]
            {\begin{math}-2\end{math}}(5);
        \draw[CliqueEdge](5)edge[]node[CliqueLabel]
            {\begin{math}1\end{math}}(6);
    \end{tikzpicture}\,.
\end{equation}
\medskip

\begin{Proposition} \label{prop:symmetries_Cli_M}
    Let $\Mca$ be a unitary magma. Then, the group of symmetries of
    $\Cli\Mca$ contains the map $\Returned$ and all the maps
    $\Cli\theta$ where $\theta$ are unitary magma automorphisms
    of~$\Mca$.
\end{Proposition}
\begin{proof}
    When $\theta$ is a unitary magma automorphism of $\Mca$, since by
    Theorem~\ref{thm:clique_construction} $\Cli$ is a functor
    respecting bijections, $\Cli\theta$ is an operad automorphism of
    $\Cli\Mca$. Hence, $\Cli\theta$ belongs to the group of symmetries
    of $\Cli\Mca$. Moreover, the fact that $\Returned$ belongs to the
    group of symmetries of $\Cli\Mca$ can be established by showing that
    this map is an antiautomorphism of $\Cli\Mca$, directly from the
    definition of the partial composition of $\Cli\Mca$ and that
    of~$\Returned$.
\end{proof}
\medskip

\subsubsection{Basic set-operad basis}
A unitary magma $\Mca$ is \Def{right cancellable} is for any
$x, y, z \in \Mca$, $y \Op x = z \Op x$ implies $y = z$.
\medskip

\begin{Proposition} \label{prop:basic_Cli_M}
    Let $\Mca$ be a unitary magma. The fundamental basis of $\Cli\Mca$
    is a basic set-operad basis if and only if $\Mca$ is right
    cancellable.
\end{Proposition}
\begin{proof}
    Assume first that $\Mca$ is right cancellable. Let $n \geq 1$,
    $i \in [n]$, and $\Pfr$, $\Pfr'$, and $\Qfr$ be three $\Mca$-cliques
    such that $\Pfr$ and $\Pfr'$ are of arity $n$. If
    $\circ_i^\Qfr(\Pfr) = \circ_i^\Qfr(\Pfr')$, we have
    $\Pfr \circ_i \Qfr = \Pfr' \circ_i \Qfr$. By definition of the
    partial composition map of $\Cli\Mca$, any same arc of $\Pfr$ and
    $\Pfr'$ have the same label, unless possibly the edge $(i, i + 1)$.
    Moreover, we have
    \begin{math}
        \Pfr_i \Op \Qfr_0 = \Pfr'_i \Op \Qfr_0.
    \end{math}
    Since $\Mca$ is right cancellable, this implies that
    $\Pfr_i = \Pfr'_i$, and hence, $\Pfr = \Pfr'$. This shows that the
    maps $\circ_i^\Qfr$ are injective and thus, that the fundamental
    basis of $\Cli\Mca$ is a basic set-operad basis.
    \smallskip

    Conversely, assume that the fundamental basis of $\Cli\Mca$ is a
    basic set-operad basis. Then, in particular, for all $n \geq 1$ and
    all $\Mca$-cliques $\Pfr$, $\Pfr'$, and $\Qfr$ such that $\Pfr$ and
    $\Pfr'$ are of arity $n$, $\circ_1^\Qfr(\Pfr) = \circ_1^\Qfr(\Pfr')$
    implies $\Pfr = \Pfr'$. This is equivalent to state that
    \begin{math}
        \Pfr_1 \Op \Qfr_0 = \Pfr'_1 \Op \Qfr_0
    \end{math}
    implies $\Pfr_1 = \Pfr'_1$. This amount exactly to state that $\Mca$
    is right cancellable.
\end{proof}
\medskip

\subsubsection{Cyclic operad structure}
Let $\rho : \Cli\Mca \to \Cli\Mca$ be the linear map sending any
$\Mca$-clique $\Pfr$ of arity $n$ to the $\Mca$-clique $\rho(\Pfr)$ of
the same arity such that, for any arc $(x, y)$ where
$1 \leq x < y \leq n + 1$,
\begin{equation} \label{equ:rotation_map_Cli_M}
    (\rho(\Pfr))(x, y) :=
    \begin{cases}
        \Pfr(x + 1, y + 1) & \mbox{if } y \leq n, \\
        \Pfr(1, x + 1) & \mbox{otherwise (} y = n + 1 \mbox{)}.
    \end{cases}
\end{equation}
In a geometric way, $\rho(\Pfr)$ is the $\Mca$-clique obtained by
applying a rotation of one step of $\Pfr$ in the counterclockwise
direction. For instance, one has in $\Cli\Z$,
\begin{equation}
    \rho\left(
    \begin{tikzpicture}[scale=.85,Centering]
        \node[CliquePoint](1)at(-0.50,-0.87){};
        \node[CliquePoint](2)at(-1.00,-0.00){};
        \node[CliquePoint](3)at(-0.50,0.87){};
        \node[CliquePoint](4)at(0.50,0.87){};
        \node[CliquePoint](5)at(1.00,0.00){};
        \node[CliquePoint](6)at(0.50,-0.87){};
        \draw[CliqueEdge](1)edge[]node[CliqueLabel]
            {\begin{math}1\end{math}}(2);
        \draw[CliqueEdge](1)edge[bend left=30]node[CliqueLabel]
            {\begin{math}-2\end{math}}(5);
        \draw[CliqueEmptyEdge](1)edge[]node[CliqueLabel]{}(6);
        \draw[CliqueEdge](2)edge[]node[CliqueLabel]
            {\begin{math}-2\end{math}}(3);
        \draw[CliqueEmptyEdge](3)edge[]node[CliqueLabel]{}(4);
        \draw[CliqueEdge](3)edge[bend right=30]node[CliqueLabel]
            {\begin{math}1\end{math}}(5);
        \draw[CliqueEmptyEdge](4)edge[]node[CliqueLabel]{}(5);
        \draw[CliqueEmptyEdge](5)edge[]node[CliqueLabel]{}(6);
    \end{tikzpicture}
    \right)
    \enspace = \enspace
    \begin{tikzpicture}[scale=.85,Centering]
        \node[CliquePoint](1)at(-0.50,-0.87){};
        \node[CliquePoint](2)at(-1.00,-0.00){};
        \node[CliquePoint](3)at(-0.50,0.87){};
        \node[CliquePoint](4)at(0.50,0.87){};
        \node[CliquePoint](5)at(1.00,0.00){};
        \node[CliquePoint](6)at(0.50,-0.87){};
        \draw[CliqueEdge](1)edge[]node[CliqueLabel]
            {\begin{math}-2\end{math}}(2);
        \draw[CliqueEdge](1)edge[]node[CliqueLabel]
            {\begin{math}1\end{math}}(6);
        \draw[CliqueEmptyEdge](2)edge[]node[CliqueLabel]{}(3);
        \draw[CliqueEdge](2)edge[bend right=30]node[CliqueLabel]
            {\begin{math}1\end{math}}(4);
        \draw[CliqueEmptyEdge](3)edge[]node[CliqueLabel]{}(4);
        \draw[CliqueEmptyEdge](4)edge[]node[CliqueLabel]{}(5);
        \draw[CliqueEdge](4)edge[bend right=30]node[CliqueLabel]
            {\begin{math}-2\end{math}}(6);
        \draw[CliqueEmptyEdge](5)edge[]node[CliqueLabel]{}(6);
    \end{tikzpicture}\,.
\end{equation}
\medskip

\begin{Proposition} \label{prop:cyclic_Cli_M}
    Let $\Mca$ be a unitary magma. The map $\rho$ is a rotation map
    of $\Cli\Mca$, endowing this operad with a cyclic operad structure.
\end{Proposition}
\begin{proof}
    The fact that $\rho$ is a rotation map for $\Cli\Mca$ follows from
    a technical but straightforward verification of the fact that
    Relations~\eqref{equ:rotation_map_1}, \eqref{equ:rotation_map_2},
    and~\eqref{equ:rotation_map_3} hold.
\end{proof}
\medskip

\subsubsection{Alternative bases}
If $\Pfr$ and $\Qfr$ are two $\Mca$-cliques of the same arity, the
\Def{Hamming distance} $\Hamming(\Pfr, \Qfr)$ between $\Pfr$ and $\Qfr$
is the number of arcs $(x, y)$ such that $\Pfr(x, y) \ne \Qfr(x, y)$.
Let $\OrdBE$ be the partial order relation on the set of all
$\Mca$-cliques, where, for any $\Mca$-cliques $\Pfr$ and $\Qfr$, one
has $\Pfr \OrdBE \Qfr$ if $\Qfr$ can be obtained from $\Pfr$ by
replacing some labels $\Unit_\Mca$ of its edges or its base by other
labels of $\Mca$. In the same way, let $\OrdD$ be the partial order
on the same set where $\Pfr \OrdD \Qfr$ if $\Qfr$ can be obtained from
$\Pfr$ by replacing some labels $\Unit_\Mca$ of its diagonals by other
labels of $\Mca$.
\medskip

For all $\Mca$-cliques $\Pfr$, let the elements of $\Cli\Mca$ defined by
\begin{subequations}
\begin{equation}
    \Hsf_\Pfr :=
    \sum_{\substack{
        \Pfr' \in \Cliques_\Mca \\
        \Pfr' \OrdBE \Pfr
    }}
    \Pfr',
\end{equation}
and
\begin{equation}
    \Ksf_\Pfr :=
    \sum_{\substack{
        \Pfr' \in \Cliques_\Mca \\
        \Pfr' \OrdD \Pfr
    }}
    (-1)^{\Hamming(\Pfr', \Pfr)}
    \Pfr'.
\end{equation}
\end{subequations}
For instance, in~$\Cli\Z$,
\begin{subequations}
\begin{equation}
    \Hsf_{
    \begin{tikzpicture}[scale=.6,Centering]
        \node[CliquePoint](1)at(-0.59,-0.81){};
        \node[CliquePoint](2)at(-0.95,0.31){};
        \node[CliquePoint](3)at(-0.00,1.00){};
        \node[CliquePoint](4)at(0.95,0.31){};
        \node[CliquePoint](5)at(0.59,-0.81){};
        \draw[CliqueEmptyEdge](1)edge[]node[]{}(2);
        \draw[CliqueEmptyEdge](1)edge[]node[]{}(5);
        \draw[CliqueEmptyEdge](2)edge[]node[]{}(3);
        \draw[CliqueEdge](2)edge[bend left=30]node[CliqueLabel,near start]
            {\begin{math}1\end{math}}(5);
        \draw[CliqueEdge](3)edge[]node[CliqueLabel]
            {\begin{math}1\end{math}}(4);
        \draw[CliqueEdge](4)edge[]node[CliqueLabel]
            {\begin{math}2\end{math}}(5);
        \draw[CliqueEdge](1)edge[bend right=30]node[CliqueLabel,near end]
            {\begin{math}2\end{math}}(3);
    \end{tikzpicture}}
    =
    \begin{tikzpicture}[scale=.6,Centering]
        \node[CliquePoint](1)at(-0.59,-0.81){};
        \node[CliquePoint](2)at(-0.95,0.31){};
        \node[CliquePoint](3)at(-0.00,1.00){};
        \node[CliquePoint](4)at(0.95,0.31){};
        \node[CliquePoint](5)at(0.59,-0.81){};
        \draw[CliqueEmptyEdge](1)edge[]node[]{}(2);
        \draw[CliqueEmptyEdge](1)edge[]node[]{}(5);
        \draw[CliqueEmptyEdge](2)edge[]node[]{}(3);
        \draw[CliqueEdge](2)edge[bend left=30]node[CliqueLabel,near start]
            {\begin{math}1\end{math}}(5);
        \draw[CliqueEmptyEdge](3)edge[]node[]{}(4);
        \draw[CliqueEmptyEdge](4)edge[]node[]{}(5);
        \draw[CliqueEdge](1)edge[bend right=30]node[CliqueLabel,near end]
            {\begin{math}2\end{math}}(3);
    \end{tikzpicture}
    +
    \begin{tikzpicture}[scale=.6,Centering]
        \node[CliquePoint](1)at(-0.59,-0.81){};
        \node[CliquePoint](2)at(-0.95,0.31){};
        \node[CliquePoint](3)at(-0.00,1.00){};
        \node[CliquePoint](4)at(0.95,0.31){};
        \node[CliquePoint](5)at(0.59,-0.81){};
        \draw[CliqueEmptyEdge](1)edge[]node[]{}(2);
        \draw[CliqueEmptyEdge](1)edge[]node[]{}(5);
        \draw[CliqueEmptyEdge](2)edge[]node[]{}(3);
        \draw[CliqueEdge](2)edge[bend left=30]node[CliqueLabel,near start]
            {\begin{math}1\end{math}}(5);
        \draw[CliqueEmptyEdge](3)edge[]node[]{}(4);
        \draw[CliqueEdge](4)edge[]node[CliqueLabel]
            {\begin{math}2\end{math}}(5);
        \draw[CliqueEdge](1)edge[bend right=30]node[CliqueLabel,near end]
            {\begin{math}2\end{math}}(3);
    \end{tikzpicture}
    +
    \begin{tikzpicture}[scale=.6,Centering]
        \node[CliquePoint](1)at(-0.59,-0.81){};
        \node[CliquePoint](2)at(-0.95,0.31){};
        \node[CliquePoint](3)at(-0.00,1.00){};
        \node[CliquePoint](4)at(0.95,0.31){};
        \node[CliquePoint](5)at(0.59,-0.81){};
        \draw[CliqueEmptyEdge](1)edge[]node[]{}(2);
        \draw[CliqueEmptyEdge](1)edge[]node[]{}(5);
        \draw[CliqueEmptyEdge](2)edge[]node[]{}(3);
        \draw[CliqueEdge](2)edge[bend left=30]node[CliqueLabel,near start]
            {\begin{math}1\end{math}}(5);
        \draw[CliqueEdge](3)edge[]node[CliqueLabel]
            {\begin{math}1\end{math}}(4);
        \draw[CliqueEmptyEdge](4)edge[]node[]{}(5);
        \draw[CliqueEdge](1)edge[bend right=30]node[CliqueLabel,near end]
            {\begin{math}2\end{math}}(3);
    \end{tikzpicture}
    +
    \begin{tikzpicture}[scale=.6,Centering]
        \node[CliquePoint](1)at(-0.59,-0.81){};
        \node[CliquePoint](2)at(-0.95,0.31){};
        \node[CliquePoint](3)at(-0.00,1.00){};
        \node[CliquePoint](4)at(0.95,0.31){};
        \node[CliquePoint](5)at(0.59,-0.81){};
        \draw[CliqueEmptyEdge](1)edge[]node[]{}(2);
        \draw[CliqueEmptyEdge](1)edge[]node[]{}(5);
        \draw[CliqueEmptyEdge](2)edge[]node[]{}(3);
        \draw[CliqueEdge](2)edge[bend left=30]node[CliqueLabel,near start]
            {\begin{math}1\end{math}}(5);
        \draw[CliqueEdge](3)edge[]node[CliqueLabel]
            {\begin{math}1\end{math}}(4);
        \draw[CliqueEdge](4)edge[]node[CliqueLabel]
            {\begin{math}2\end{math}}(5);
        \draw[CliqueEdge](1)edge[bend right=30]node[CliqueLabel,near end]
            {\begin{math}2\end{math}}(3);
    \end{tikzpicture}\,,
\end{equation}
\begin{equation}
    \Ksf_{
    \begin{tikzpicture}[scale=.6,Centering]
        \node[CliquePoint](1)at(-0.59,-0.81){};
        \node[CliquePoint](2)at(-0.95,0.31){};
        \node[CliquePoint](3)at(-0.00,1.00){};
        \node[CliquePoint](4)at(0.95,0.31){};
        \node[CliquePoint](5)at(0.59,-0.81){};
        \draw[CliqueEmptyEdge](1)edge[]node[]{}(2);
        \draw[CliqueEmptyEdge](1)edge[]node[]{}(5);
        \draw[CliqueEmptyEdge](2)edge[]node[]{}(3);
        \draw[CliqueEdge](2)edge[bend left=30]node[CliqueLabel,near start]
            {\begin{math}1\end{math}}(5);
        \draw[CliqueEdge](3)edge[]node[CliqueLabel]
            {\begin{math}1\end{math}}(4);
        \draw[CliqueEdge](4)edge[]node[CliqueLabel]
            {\begin{math}2\end{math}}(5);
        \draw[CliqueEdge](1)edge[bend right=30]node[CliqueLabel,near end]
            {\begin{math}2\end{math}}(3);
    \end{tikzpicture}}
    =
    \begin{tikzpicture}[scale=.6,Centering]
        \node[CliquePoint](1)at(-0.59,-0.81){};
        \node[CliquePoint](2)at(-0.95,0.31){};
        \node[CliquePoint](3)at(-0.00,1.00){};
        \node[CliquePoint](4)at(0.95,0.31){};
        \node[CliquePoint](5)at(0.59,-0.81){};
        \draw[CliqueEmptyEdge](1)edge[]node[]{}(2);
        \draw[CliqueEmptyEdge](1)edge[]node[]{}(5);
        \draw[CliqueEmptyEdge](2)edge[]node[]{}(3);
        \draw[CliqueEdge](2)edge[bend left=30]node[CliqueLabel,near start]
            {\begin{math}1\end{math}}(5);
        \draw[CliqueEdge](3)edge[]node[CliqueLabel]
            {\begin{math}1\end{math}}(4);
        \draw[CliqueEdge](4)edge[]node[CliqueLabel]
            {\begin{math}2\end{math}}(5);
        \draw[CliqueEdge](1)edge[bend right=30]node[CliqueLabel,near end]
            {\begin{math}2\end{math}}(3);
    \end{tikzpicture}
    -
    \begin{tikzpicture}[scale=.6,Centering]
        \node[CliquePoint](1)at(-0.59,-0.81){};
        \node[CliquePoint](2)at(-0.95,0.31){};
        \node[CliquePoint](3)at(-0.00,1.00){};
        \node[CliquePoint](4)at(0.95,0.31){};
        \node[CliquePoint](5)at(0.59,-0.81){};
        \draw[CliqueEmptyEdge](1)edge[]node[]{}(2);
        \draw[CliqueEmptyEdge](1)edge[]node[]{}(5);
        \draw[CliqueEmptyEdge](2)edge[]node[]{}(3);
        \draw[CliqueEdge](3)edge[]node[CliqueLabel]
            {\begin{math}1\end{math}}(4);
        \draw[CliqueEdge](4)edge[]node[CliqueLabel]
            {\begin{math}2\end{math}}(5);
        \draw[CliqueEdge](1)edge[bend right=30]node[CliqueLabel]
            {\begin{math}2\end{math}}(3);
    \end{tikzpicture}
    -
    \begin{tikzpicture}[scale=.6,Centering]
        \node[CliquePoint](1)at(-0.59,-0.81){};
        \node[CliquePoint](2)at(-0.95,0.31){};
        \node[CliquePoint](3)at(-0.00,1.00){};
        \node[CliquePoint](4)at(0.95,0.31){};
        \node[CliquePoint](5)at(0.59,-0.81){};
        \draw[CliqueEmptyEdge](1)edge[]node[]{}(2);
        \draw[CliqueEmptyEdge](1)edge[]node[]{}(5);
        \draw[CliqueEmptyEdge](2)edge[]node[]{}(3);
        \draw[CliqueEdge](2)edge[bend left=30]node[CliqueLabel]
            {\begin{math}1\end{math}}(5);
        \draw[CliqueEdge](3)edge[]node[CliqueLabel]
            {\begin{math}1\end{math}}(4);
        \draw[CliqueEdge](4)edge[]node[CliqueLabel]
            {\begin{math}2\end{math}}(5);
    \end{tikzpicture}
    +
    \begin{tikzpicture}[scale=.6,Centering]
        \node[CliquePoint](1)at(-0.59,-0.81){};
        \node[CliquePoint](2)at(-0.95,0.31){};
        \node[CliquePoint](3)at(-0.00,1.00){};
        \node[CliquePoint](4)at(0.95,0.31){};
        \node[CliquePoint](5)at(0.59,-0.81){};
        \draw[CliqueEmptyEdge](1)edge[]node[]{}(2);
        \draw[CliqueEmptyEdge](1)edge[]node[]{}(5);
        \draw[CliqueEmptyEdge](2)edge[]node[]{}(3);
        \draw[CliqueEdge](3)edge[]node[CliqueLabel]
            {\begin{math}1\end{math}}(4);
        \draw[CliqueEdge](4)edge[]node[CliqueLabel]
            {\begin{math}2\end{math}}(5);
    \end{tikzpicture}\,.
\end{equation}
\end{subequations}

Since by Möbius inversion, one has for any $\Mca$-clique $\Pfr$,
\begin{equation}
    \sum_{\substack{
        \Pfr' \in \Cliques_\Mca \\
        \Pfr' \OrdBE \Pfr
    }}
    (-1)^{\Hamming(\Pfr', \Pfr)}
    \Hsf_{\Pfr'}
    =
    \Pfr
    =
    \sum_{\substack{
        \Pfr' \in \Cliques_\Mca \\
        \Pfr' \OrdD \Pfr
    }}
    \Ksf_{\Pfr'},
\end{equation}
by triangularity, the family of all the $\Hsf_\Pfr$ (resp. $\Ksf_\Pfr$)
forms a  basis of $\Cli\Mca$ called the \Def{$\Hsf$-basis} (resp. the
\Def{$\Ksf$-basis}).
\medskip

If $\Pfr$ is an $\Mca$-clique, $\Del_0(\Pfr)$ (resp. $\Del_i(\Pfr)$) is
the $\Mca$-clique obtained by replacing the label of the base
(resp. $i$th edge) of $\Pfr$ by $\Unit_\Mca$.
\medskip

\begin{Proposition} \label{prop:composition_Cli_M_basis_H}
    Let $\Mca$ be a unitary magma. The partial composition of $\Cli\Mca$
    expresses over the $\Hsf$-basis, for any $\Mca$-cliques $\Pfr$ and
    $\Qfr$ different from $\UnitClique$ and any valid integer $i$, as
    \begin{small}
    \begin{equation}
        \Hsf_\Pfr \circ_i \Hsf_\Qfr
        =
        \begin{cases}
            \Hsf_{\Pfr \circ_i \Qfr}
            + \Hsf_{\Del_i(\Pfr) \circ_i \Qfr}
            + \Hsf_{\Pfr \circ_i \Del_0(\Qfr)}
            + \Hsf_{\Del_i(\Pfr) \circ_i \Del_0(\Qfr)}
                & \mbox{if } \Pfr_i \ne \Unit_\Mca \mbox{ and }
                    \Qfr_0 \ne \Unit_\Mca, \\
            \Hsf_{\Pfr \circ_i \Qfr}
            + \Hsf_{\Del_i(\Pfr) \circ_i \Qfr}
                & \mbox{if } \Pfr_i \ne \Unit_\Mca, \\
            \Hsf_{\Pfr \circ_i \Qfr}
            + \Hsf_{\Pfr \circ_i \Del_0(\Qfr)}
                & \mbox{if } \Qfr_0 \ne \Unit_\Mca, \\
            \Hsf_{\Pfr \circ_i \Qfr} & \mbox{otherwise}.
        \end{cases}
    \end{equation}
    \end{small}
\end{Proposition}
\begin{proof}
    From the definition of the $\Hsf$-basis, we have
    \begin{equation}\begin{split}
        \label{equ:composition_Cli_M_basis_H_demo}
        \Hsf_\Pfr \circ_i \Hsf_\Qfr
        & =
        \sum_{\substack{
            \Pfr', \Qfr' \in \Cliques_\Mca \\
            \Pfr' \OrdBE \Pfr \\
            \Qfr' \OrdBE \Qfr
        }}
        \Pfr' \circ_i \Qfr' \\
        & =
        \sum_{\substack{
            \Pfr', \Qfr' \in \Cliques_\Mca \\
            \Pfr' \OrdBE \Pfr \\
            \Qfr' \OrdBE \Qfr \\
            \Pfr'_i \ne \Unit_\Mca \\
            \Qfr'_0 \ne \Unit_\Mca
        }}
        \Pfr' \circ_i \Qfr'
        +
        \sum_{\substack{
            \Pfr', \Qfr' \in \Cliques_\Mca \\
            \Pfr' \OrdBE \Pfr \\
            \Qfr' \OrdBE \Qfr \\
            \Pfr'_i \ne \Unit_\Mca \\
            \Qfr'_0 = \Unit_\Mca
        }}
        \Pfr' \circ_i \Qfr'
        +
        \sum_{\substack{
            \Pfr', \Qfr' \in \Cliques_\Mca \\
            \Pfr' \OrdBE \Pfr \\
            \Qfr' \OrdBE \Qfr \\
            \Pfr'_i = \Unit_\Mca \\
            \Qfr'_0 \ne \Unit_\Mca
        }}
        \Pfr' \circ_i \Qfr'
        +
        \sum_{\substack{
            \Pfr', \Qfr' \in \Cliques_\Mca \\
            \Pfr' \OrdBE \Pfr \\
            \Qfr' \OrdBE \Qfr \\
            \Pfr'_i = \Unit_\Mca \\
            \Qfr'_0 = \Unit_\Mca
        }}
        \Pfr' \circ_i \Qfr'.
    \end{split}\end{equation}
    Let $s_1$ (resp. $s_2$, $s_3$, $s_4$) be the first (resp. second,
    third, fourth) summand of the last member
    of~\eqref{equ:composition_Cli_M_basis_H_demo}. There are four cases
    to explore depending on whether the $i$th edge of $\Pfr$ and the
    base of $\Qfr$ are solid or not. From the definition of the
    $\Hsf$-basis and of the partial order relation $\OrdBE$, we have
    that
    \begin{enumerate}[fullwidth,label=(\alph*)]
        \item when $\Pfr_i \ne \Unit_\Mca$ and $\Qfr_0 \ne \Unit_\Mca$,
        $s_1 = \Hsf_{\Pfr \circ_i \Qfr}$,
        $s_2 = \Hsf_{\Pfr \circ_i \Del_0(\Qfr)}$,
        $s_3 = \Hsf_{\Del_i(\Pfr) \circ_i \Qfr}$, and
        $s_4 = \Hsf_{\Del_i(\Pfr) \circ_i \Del_0(\Qfr)}$;
        \item when $\Pfr_i \ne \Unit_\Mca$ and $\Qfr_0 = \Unit_\Mca$,
        $s_1 = 0$, $s_2 = \Hsf_{\Pfr \circ_i \Qfr}$,
        $s_3 = 0$, and $s_4 = \Hsf_{\Del_i(\Pfr) \circ_i \Qfr}$;
        \item when $\Pfr_i = \Unit_\Mca$ and $\Qfr_0 \ne \Unit_\Mca$,
        $s_1 = 0$, $s_2 = 0$, $s_3 = \Hsf_{\Pfr \circ_i \Qfr}$, and
        $s_4 = \Hsf_{\Pfr \circ_i \Del_0(\Qfr)}$;
        \item and when $\Pfr_i = \Unit_\Mca$ and $\Qfr_0 = \Unit_\Mca$,
        $s_1 = 0$, $s_2 = 0$, $s_3 = 0$, and
        $s_4 = \Hsf_{\Pfr \circ_i \Qfr}$.
    \end{enumerate}
    By assembling these cases together, we retrieve the stated result.
\end{proof}
\medskip

\begin{Proposition} \label{prop:composition_Cli_M_basis_K}
    Let $\Mca$ be a unitary magma. The partial composition of $\Cli\Mca$
    expresses over the $\Ksf$-basis, for any $\Mca$-cliques $\Pfr$ and
    $\Qfr$ different from $\UnitClique$ and any valid integer $i$, as
    \begin{small}
    \begin{equation}
        \Ksf_\Pfr \circ_i \Ksf_\Qfr
        =
        \begin{cases}
            \Ksf_{\Pfr \circ_i \Qfr}
                & \mbox{if }
                \Pfr_i \Op \Qfr_0 = \Unit_\Mca, \\
            \Ksf_{\Pfr \circ_i \Qfr} +
            \Ksf_{\Del_i(\Pfr) \circ_i \Del_0(\Qfr)}
                & \mbox{otherwise}.
        \end{cases}
    \end{equation}
    \end{small}
\end{Proposition}
\begin{proof}
    Let $m$ be the arity of $\Qfr$. From the definition of the
    $\Ksf$-basis and of the partial order relation $\OrdD$, we have
    \begin{equation}\begin{split}
        \label{equ:composition_Cli_M_basis_K_demo}
        \Ksf_\Pfr \circ_i \Ksf_\Qfr
        & =
        \sum_{\substack{
            \Pfr', \Qfr' \in \Cliques_\Mca \\
            \Pfr' \OrdD \Pfr \\
            \Qfr' \OrdD \Qfr
        }}
        (-1)^{\Hamming(\Pfr', \Pfr) + \Hamming(\Qfr', \Qfr)}
        \Pfr' \circ_i \Qfr' \\
        & =
        \sum_{\substack{
            \Pfr', \Qfr' \in \Cliques_\Mca \\
            \Pfr' \circ_i \Qfr' \OrdD \Pfr \circ_i \Qfr \\
            \Pfr'_i = \Pfr_i \\
            \Qfr'_0 = \Qfr_0
        }}
        (-1)^{\Hamming(\Pfr', \Pfr) + \Hamming(\Qfr', \Qfr)}
        \Pfr' \circ_i \Qfr' \\
        & =
        \sum_{\substack{
            \Rfr \in \Cliques_\Mca \\
            \Rfr \OrdD \Pfr \circ_i \Qfr \\
            \Rfr(i, i + m - 1) = \Pfr_i \Op \Qfr_0
        }}
        (-1)^{\Hamming(\Rfr, \Pfr \circ_i \Qfr)}
        \Rfr.
    \end{split}\end{equation}
    When $\Pfr_i \Op \Qfr_0 = \Unit_\Mca$,
    \eqref{equ:composition_Cli_M_basis_K_demo} is equal to
    $\Ksf_{\Pfr \circ_i \Qfr}$. Otherwise, when
    $\Pfr_i \Op \Qfr_0 \ne \Unit_\Mca$, we have
    \begin{equation}\begin{split}
        \sum_{\substack{
            \Rfr \in \Cliques_\Mca \\
            \Rfr \OrdD \Pfr \circ_i \Qfr \\
            \Rfr(i, i + m - 1) = \Pfr_i \Op \Qfr_0
        }}
        (-1)^{\Hamming(\Rfr, \Pfr \circ_i \Qfr)}
        \Rfr
        & =
         \sum_{\substack{
            \Rfr \in \Cliques_\Mca \\
            \Rfr \OrdD \Pfr \circ_i \Qfr
        }}
        (-1)^{\Hamming(\Rfr, \Pfr \circ_i \Qfr)}
        \Rfr
        \enspace -
        \sum_{\substack{
            \Rfr \in \Cliques_\Mca \\
            \Rfr \OrdD \Pfr \circ_i \Qfr \\
            \Rfr(i, i + m - 1) \ne \Pfr_i \Op \Qfr_0
        }}
        (-1)^{\Hamming(\Rfr, \Pfr \circ_i \Qfr)}
        \Rfr \\
        & =
        \Ksf_{\Pfr \circ_i \Qfr}
        \enspace -
        \sum_{\substack{
            \Rfr \in \Cliques_\Mca \\
            \Rfr \OrdD \Del_i(\Pfr) \circ_i \Del_0(\Qfr) \\
        }}
        (-1)^{\Hamming(\Rfr, \Pfr \circ_i \Qfr)}
        \Rfr \\
        & =
        \Ksf_{\Pfr \circ_i \Qfr}
        \enspace -
        \sum_{\substack{
            \Rfr \in \Cliques_\Mca \\
            \Rfr \OrdD \Del_i(\Pfr) \circ_i \Del_0(\Qfr) \\
        }}
        (-1)^{1 + \Hamming(\Rfr, \Del_i(\Pfr) \circ_i \Del_0(\Qfr))}
        \Rfr \\
        & =
        \Ksf_{\Pfr \circ_i \Qfr} + \Ksf_{\Del_i(\Pfr) \circ_i \Del_0(\Qfr)}.
    \end{split}\end{equation}
    This proves the claimed formula for the partial composition of
    $\Cli\Mca$ over the $\Ksf$-basis.
\end{proof}
\medskip

For instance, in $\Cli\Z$,
\begin{subequations}
\begin{equation}
    \Hsf_{
    \begin{tikzpicture}[scale=0.4,Centering]
        \node[CliquePoint](1)at(-0.87,-0.50){};
        \node[CliquePoint](2)at(-0.00,1.00){};
        \node[CliquePoint](3)at(0.87,-0.50){};
        \draw[CliqueEmptyEdge](1)edge[]node[]{}(2);
        \draw[CliqueEmptyEdge](1)edge[]node[]{}(3);
        \draw[CliqueEdge](2)edge[]node[CliqueLabel]
            {\begin{math}1\end{math}}(3);
    \end{tikzpicture}}
    \circ_2
    \Hsf_{
    \begin{tikzpicture}[scale=0.4,Centering]
        \node[CliquePoint](1)at(-0.87,-0.50){};
        \node[CliquePoint](2)at(-0.00,1.00){};
        \node[CliquePoint](3)at(0.87,-0.50){};
        \draw[CliqueEmptyEdge](1)edge[]node[]{}(2);
        \draw[CliqueEdge](1)edge[]node[CliqueLabel]
            {\begin{math}1\end{math}}(3);
        \draw[CliqueEmptyEdge](2)edge[]node[]{}(3);
    \end{tikzpicture}}
    =
    \Hsf_{
    \begin{tikzpicture}[scale=0.5,Centering]
        \node[CliquePoint](1)at(-0.71,-0.71){};
        \node[CliquePoint](2)at(-0.71,0.71){};
        \node[CliquePoint](3)at(0.71,0.71){};
        \node[CliquePoint](4)at(0.71,-0.71){};
        \draw[CliqueEmptyEdge](1)edge[]node[]{}(2);
        \draw[CliqueEmptyEdge](1)edge[]node[]{}(4);
        \draw[CliqueEmptyEdge](2)edge[]node[]{}(3);
        \draw[CliqueEmptyEdge](3)edge[]node[]{}(4);
    \end{tikzpicture}}
    +
    2\;
    \Hsf_{
    \begin{tikzpicture}[scale=0.5,Centering]
        \node[CliquePoint](1)at(-0.71,-0.71){};
        \node[CliquePoint](2)at(-0.71,0.71){};
        \node[CliquePoint](3)at(0.71,0.71){};
        \node[CliquePoint](4)at(0.71,-0.71){};
        \draw[CliqueEmptyEdge](1)edge[]node[]{}(2);
        \draw[CliqueEmptyEdge](1)edge[]node[]{}(4);
        \draw[CliqueEmptyEdge](2)edge[]node[]{}(3);
        \draw[CliqueEdge](2)edge[]node[CliqueLabel]
            {\begin{math}1\end{math}}(4);
        \draw[CliqueEmptyEdge](3)edge[]node[]{}(4);
    \end{tikzpicture}}
    +
    \Hsf_{
    \begin{tikzpicture}[scale=0.5,Centering]
        \node[CliquePoint](1)at(-0.71,-0.71){};
        \node[CliquePoint](2)at(-0.71,0.71){};
        \node[CliquePoint](3)at(0.71,0.71){};
        \node[CliquePoint](4)at(0.71,-0.71){};
        \draw[CliqueEmptyEdge](1)edge[]node[]{}(2);
        \draw[CliqueEmptyEdge](1)edge[]node[]{}(4);
        \draw[CliqueEmptyEdge](2)edge[]node[]{}(3);
        \draw[CliqueEdge](2)edge[]node[CliqueLabel]
            {\begin{math}2\end{math}}(4);
        \draw[CliqueEmptyEdge](3)edge[]node[]{}(4);
    \end{tikzpicture}}\,,
\end{equation}
\begin{equation}
    \Ksf_{
    \begin{tikzpicture}[scale=0.4,Centering]
        \node[CliquePoint](1)at(-0.87,-0.50){};
        \node[CliquePoint](2)at(-0.00,1.00){};
        \node[CliquePoint](3)at(0.87,-0.50){};
        \draw[CliqueEmptyEdge](1)edge[]node[]{}(2);
        \draw[CliqueEmptyEdge](1)edge[]node[]{}(3);
        \draw[CliqueEdge](2)edge[]node[CliqueLabel]
            {\begin{math}1\end{math}}(3);
    \end{tikzpicture}}
    \circ_2
    \Ksf_{
    \begin{tikzpicture}[scale=0.4,Centering]
        \node[CliquePoint](1)at(-0.87,-0.50){};
        \node[CliquePoint](2)at(-0.00,1.00){};
        \node[CliquePoint](3)at(0.87,-0.50){};
        \draw[CliqueEmptyEdge](1)edge[]node[]{}(2);
        \draw[CliqueEdge](1)edge[]node[CliqueLabel]
            {\begin{math}1\end{math}}(3);
        \draw[CliqueEmptyEdge](2)edge[]node[]{}(3);
    \end{tikzpicture}}
    =
    \Ksf_{
    \begin{tikzpicture}[scale=0.5,Centering]
        \node[CliquePoint](1)at(-0.71,-0.71){};
        \node[CliquePoint](2)at(-0.71,0.71){};
        \node[CliquePoint](3)at(0.71,0.71){};
        \node[CliquePoint](4)at(0.71,-0.71){};
        \draw[CliqueEmptyEdge](1)edge[]node[]{}(2);
        \draw[CliqueEmptyEdge](1)edge[]node[]{}(4);
        \draw[CliqueEmptyEdge](2)edge[]node[]{}(3);
        \draw[CliqueEmptyEdge](3)edge[]node[]{}(4);
    \end{tikzpicture}}
    +
    \Ksf_{
    \begin{tikzpicture}[scale=0.5,Centering]
        \node[CliquePoint](1)at(-0.71,-0.71){};
        \node[CliquePoint](2)at(-0.71,0.71){};
        \node[CliquePoint](3)at(0.71,0.71){};
        \node[CliquePoint](4)at(0.71,-0.71){};
        \draw[CliqueEmptyEdge](1)edge[]node[]{}(2);
        \draw[CliqueEmptyEdge](1)edge[]node[]{}(4);
        \draw[CliqueEmptyEdge](2)edge[]node[]{}(3);
        \draw[CliqueEdge](2)edge[]node[CliqueLabel]
            {\begin{math}2\end{math}}(4);
        \draw[CliqueEmptyEdge](3)edge[]node[]{}(4);
    \end{tikzpicture}}\,,
\end{equation}
\begin{equation}
    \Hsf_{
    \begin{tikzpicture}[scale=0.5,Centering]
        \node[CliquePoint](1)at(-0.71,-0.71){};
        \node[CliquePoint](2)at(-0.71,0.71){};
        \node[CliquePoint](3)at(0.71,0.71){};
        \node[CliquePoint](4)at(0.71,-0.71){};
        \draw[CliqueEmptyEdge](1)edge[]node[]{}(2);
        \draw[CliqueEdge](1)edge[]node[CliqueLabel]
            {\begin{math}2\end{math}}(3);
        \draw[CliqueEmptyEdge](1)edge[]node[]{}(4);
        \draw[CliqueEmptyEdge](2)edge[]node[]{}(3);
        \draw[CliqueEdge](3)edge[]node[CliqueLabel]
            {\begin{math}1\end{math}}(4);
    \end{tikzpicture}}
    \circ_3
    \Hsf_{
    \begin{tikzpicture}[scale=0.4,Centering]
        \node[CliquePoint](1)at(-0.87,-0.50){};
        \node[CliquePoint](2)at(-0.00,1.00){};
        \node[CliquePoint](3)at(0.87,-0.50){};
        \draw[CliqueEdge](1)edge[]node[CliqueLabel]
            {\begin{math}1\end{math}}(2);
        \draw[CliqueEdge](1)edge[]node[CliqueLabel]
            {\begin{math}2\end{math}}(3);
        \draw[CliqueEdge](2)edge[]node[CliqueLabel]
            {\begin{math}2\end{math}}(3);
    \end{tikzpicture}}
    =
    \Hsf_{
    \begin{tikzpicture}[scale=0.6,Centering]
        \node[CliquePoint](1)at(-0.59,-0.81){};
        \node[CliquePoint](2)at(-0.95,0.31){};
        \node[CliquePoint](3)at(-0.00,1.00){};
        \node[CliquePoint](4)at(0.95,0.31){};
        \node[CliquePoint](5)at(0.59,-0.81){};
        \draw[CliqueEmptyEdge](1)edge[]node[]{}(2);
        \draw[CliqueEdge](1)edge[bend right=30]node[CliqueLabel]
            {\begin{math}2\end{math}}(3);
        \draw[CliqueEmptyEdge](1)edge[]node[]{}(5);
        \draw[CliqueEmptyEdge](2)edge[]node[]{}(3);
        \draw[CliqueEdge](3)edge[]node[CliqueLabel]
            {\begin{math}1\end{math}}(4);
        \draw[CliqueEdge](4)edge[]node[CliqueLabel]
            {\begin{math}2\end{math}}(5);
    \end{tikzpicture}}
    +
    \Hsf_{
    \begin{tikzpicture}[scale=0.6,Centering]
        \node[CliquePoint](1)at(-0.59,-0.81){};
        \node[CliquePoint](2)at(-0.95,0.31){};
        \node[CliquePoint](3)at(-0.00,1.00){};
        \node[CliquePoint](4)at(0.95,0.31){};
        \node[CliquePoint](5)at(0.59,-0.81){};
        \draw[CliqueEmptyEdge](1)edge[]node[]{}(2);
        \draw[CliqueEdge](1)edge[]node[CliqueLabel]
            {\begin{math}2\end{math}}(3);
        \draw[CliqueEmptyEdge](1)edge[]node[]{}(5);
        \draw[CliqueEmptyEdge](2)edge[]node[]{}(3);
        \draw[CliqueEdge](3)edge[]node[CliqueLabel]
            {\begin{math}1\end{math}}(4);
        \draw[CliqueEdge](3)edge[]node[CliqueLabel]
            {\begin{math}1\end{math}}(5);
        \draw[CliqueEdge](4)edge[]node[CliqueLabel]
            {\begin{math}2\end{math}}(5);
    \end{tikzpicture}}
    +
    \Hsf_{
    \begin{tikzpicture}[scale=0.6,Centering]
        \node[CliquePoint](1)at(-0.59,-0.81){};
        \node[CliquePoint](2)at(-0.95,0.31){};
        \node[CliquePoint](3)at(-0.00,1.00){};
        \node[CliquePoint](4)at(0.95,0.31){};
        \node[CliquePoint](5)at(0.59,-0.81){};
        \draw[CliqueEmptyEdge](1)edge[]node[]{}(2);
        \draw[CliqueEdge](1)edge[]node[CliqueLabel]
            {\begin{math}2\end{math}}(3);
        \draw[CliqueEmptyEdge](1)edge[]node[]{}(5);
        \draw[CliqueEmptyEdge](2)edge[]node[]{}(3);
        \draw[CliqueEdge](3)edge[]node[CliqueLabel]
            {\begin{math}1\end{math}}(4);
        \draw[CliqueEdge](3)edge[]node[CliqueLabel]
            {\begin{math}2\end{math}}(5);
        \draw[CliqueEdge](4)edge[]node[CliqueLabel]
            {\begin{math}2\end{math}}(5);
    \end{tikzpicture}}
    +
    \Hsf_{
    \begin{tikzpicture}[scale=0.6,Centering]
        \node[CliquePoint](1)at(-0.59,-0.81){};
        \node[CliquePoint](2)at(-0.95,0.31){};
        \node[CliquePoint](3)at(-0.00,1.00){};
        \node[CliquePoint](4)at(0.95,0.31){};
        \node[CliquePoint](5)at(0.59,-0.81){};
        \draw[CliqueEmptyEdge](1)edge[]node[]{}(2);
        \draw[CliqueEdge](1)edge[]node[CliqueLabel]
            {\begin{math}2\end{math}}(3);
        \draw[CliqueEmptyEdge](1)edge[]node[]{}(5);
        \draw[CliqueEmptyEdge](2)edge[]node[]{}(3);
        \draw[CliqueEdge](3)edge[]node[CliqueLabel]
            {\begin{math}1\end{math}}(4);
        \draw[CliqueEdge](3)edge[]node[CliqueLabel]
            {\begin{math}3\end{math}}(5);
        \draw[CliqueEdge](4)edge[]node[CliqueLabel]
            {\begin{math}2\end{math}}(5);
    \end{tikzpicture}}\,,
\end{equation}
\begin{equation}
    \Ksf_{
    \begin{tikzpicture}[scale=0.5,Centering]
        \node[CliquePoint](1)at(-0.71,-0.71){};
        \node[CliquePoint](2)at(-0.71,0.71){};
        \node[CliquePoint](3)at(0.71,0.71){};
        \node[CliquePoint](4)at(0.71,-0.71){};
        \draw[CliqueEmptyEdge](1)edge[]node[]{}(2);
        \draw[CliqueEdge](1)edge[]node[CliqueLabel]
            {\begin{math}2\end{math}}(3);
        \draw[CliqueEmptyEdge](1)edge[]node[]{}(4);
        \draw[CliqueEmptyEdge](2)edge[]node[]{}(3);
        \draw[CliqueEdge](3)edge[]node[CliqueLabel]
            {\begin{math}1\end{math}}(4);
    \end{tikzpicture}}
    \circ_3
    \Ksf_{
    \begin{tikzpicture}[scale=0.4,Centering]
        \node[CliquePoint](1)at(-0.87,-0.50){};
        \node[CliquePoint](2)at(-0.00,1.00){};
        \node[CliquePoint](3)at(0.87,-0.50){};
        \draw[CliqueEdge](1)edge[]node[CliqueLabel]
            {\begin{math}1\end{math}}(2);
        \draw[CliqueEdge](1)edge[]node[CliqueLabel]
            {\begin{math}2\end{math}}(3);
        \draw[CliqueEdge](2)edge[]node[CliqueLabel]
            {\begin{math}2\end{math}}(3);
    \end{tikzpicture}}
    =
    \Ksf_{
    \begin{tikzpicture}[scale=0.6,Centering]
        \node[CliquePoint](1)at(-0.59,-0.81){};
        \node[CliquePoint](2)at(-0.95,0.31){};
        \node[CliquePoint](3)at(-0.00,1.00){};
        \node[CliquePoint](4)at(0.95,0.31){};
        \node[CliquePoint](5)at(0.59,-0.81){};
        \draw[CliqueEmptyEdge](1)edge[]node[]{}(2);
        \draw[CliqueEdge](1)edge[bend right=30]node[CliqueLabel]
            {\begin{math}2\end{math}}(3);
        \draw[CliqueEmptyEdge](1)edge[]node[]{}(5);
        \draw[CliqueEmptyEdge](2)edge[]node[]{}(3);
        \draw[CliqueEdge](3)edge[]node[CliqueLabel]
            {\begin{math}1\end{math}}(4);
        \draw[CliqueEdge](4)edge[]node[CliqueLabel]
            {\begin{math}2\end{math}}(5);
    \end{tikzpicture}}
    +
    \Ksf_{
    \begin{tikzpicture}[scale=0.6,Centering]
        \node[CliquePoint](1)at(-0.59,-0.81){};
        \node[CliquePoint](2)at(-0.95,0.31){};
        \node[CliquePoint](3)at(-0.00,1.00){};
        \node[CliquePoint](4)at(0.95,0.31){};
        \node[CliquePoint](5)at(0.59,-0.81){};
        \draw[CliqueEmptyEdge](1)edge[]node[]{}(2);
        \draw[CliqueEdge](1)edge[]node[CliqueLabel]
            {\begin{math}2\end{math}}(3);
        \draw[CliqueEmptyEdge](1)edge[]node[]{}(5);
        \draw[CliqueEmptyEdge](2)edge[]node[]{}(3);
        \draw[CliqueEdge](3)edge[]node[CliqueLabel]
            {\begin{math}1\end{math}}(4);
        \draw[CliqueEdge](3)edge[]node[CliqueLabel]
            {\begin{math}3\end{math}}(5);
        \draw[CliqueEdge](4)edge[]node[CliqueLabel]
            {\begin{math}2\end{math}}(5);
    \end{tikzpicture}}\,.
\end{equation}
\begin{equation}
    \Hsf_{
    \begin{tikzpicture}[scale=0.6,Centering]
        \node[CliquePoint](1)at(-0.71,-0.71){};
        \node[CliquePoint](2)at(-0.71,0.71){};
        \node[CliquePoint](3)at(0.71,0.71){};
        \node[CliquePoint](4)at(0.71,-0.71){};
        \draw[CliqueEmptyEdge](1)edge[]node[]{}(2);
        \draw[CliqueEmptyEdge](1)edge[]node[]{}(4);
        \draw[CliqueEdge](2)edge[]node[CliqueLabel]
            {\begin{math}-1\end{math}}(3);
        \draw[CliqueEdge](2)edge[]node[CliqueLabel]
            {\begin{math}2\end{math}}(4);
        \draw[CliqueEdge](3)edge[]node[CliqueLabel]
            {\begin{math}1\end{math}}(4);
    \end{tikzpicture}}
    \circ_2
    \Hsf_{
    \begin{tikzpicture}[scale=0.6,Centering]
        \node[CliquePoint](1)at(-0.71,-0.71){};
        \node[CliquePoint](2)at(-0.71,0.71){};
        \node[CliquePoint](3)at(0.71,0.71){};
        \node[CliquePoint](4)at(0.71,-0.71){};
        \draw[CliqueEmptyEdge](1)edge[]node[]{}(2);
        \draw[CliqueEdge](1)edge[]node[CliqueLabel]
            {\begin{math}-1\end{math}}(3);
        \draw[CliqueEdge](1)edge[]node[CliqueLabel]
            {\begin{math}1\end{math}}(4);
        \draw[CliqueEdge](2)edge[]node[CliqueLabel]
            {\begin{math}1\end{math}}(3);
        \draw[CliqueEmptyEdge](3)edge[]node[]{}(4);
    \end{tikzpicture}}
    =
    \Hsf_{
    \begin{tikzpicture}[scale=0.8,Centering]
        \node[CliquePoint](1)at(-0.50,-0.87){};
        \node[CliquePoint](2)at(-1.00,-0.00){};
        \node[CliquePoint](3)at(-0.50,0.87){};
        \node[CliquePoint](4)at(0.50,0.87){};
        \node[CliquePoint](5)at(1.00,0.00){};
        \node[CliquePoint](6)at(0.50,-0.87){};
        \draw[CliqueEmptyEdge](1)edge[]node[]{}(2);
        \draw[CliqueEmptyEdge](1)edge[]node[]{}(6);
        \draw[CliqueEmptyEdge](2)edge[]node[]{}(3);
        \draw[CliqueEdge](2)edge[]node[CliqueLabel]
            {\begin{math}-1\end{math}}(4);
        \draw[CliqueEdge](2)edge[]node[CliqueLabel]
            {\begin{math}-1\end{math}}(5);
        \draw[CliqueEdge](2)edge[]node[CliqueLabel]
            {\begin{math}2\end{math}}(6);
        \draw[CliqueEdge](3)edge[]node[CliqueLabel]
            {\begin{math}1\end{math}}(4);
        \draw[CliqueEmptyEdge](4)edge[]node[]{}(5);
        \draw[CliqueEdge](5)edge[]node[CliqueLabel]
            {\begin{math}1\end{math}}(6);
    \end{tikzpicture}}
     + 2 \;
    \Hsf_{
    \begin{tikzpicture}[scale=0.8,Centering]
        \node[CliquePoint](1)at(-0.50,-0.87){};
        \node[CliquePoint](2)at(-1.00,-0.00){};
        \node[CliquePoint](3)at(-0.50,0.87){};
        \node[CliquePoint](4)at(0.50,0.87){};
        \node[CliquePoint](5)at(1.00,0.00){};
        \node[CliquePoint](6)at(0.50,-0.87){};
        \draw[CliqueEmptyEdge](1)edge[]node[]{}(2);
        \draw[CliqueEmptyEdge](1)edge[]node[]{}(6);
        \draw[CliqueEmptyEdge](2)edge[]node[]{}(3);
        \draw[CliqueEdge](2)edge[bend right=30]node[CliqueLabel]
            {\begin{math}-1\end{math}}(4);
        \draw[CliqueEdge](2)edge[bend left=30]node[CliqueLabel]
            {\begin{math}2\end{math}}(6);
        \draw[CliqueEdge](3)edge[]node[CliqueLabel]
            {\begin{math}1\end{math}}(4);
        \draw[CliqueEmptyEdge](4)edge[]node[]{}(5);
        \draw[CliqueEdge](5)edge[]node[CliqueLabel]
            {\begin{math}1\end{math}}(6);
    \end{tikzpicture}}
     +
    \Hsf_{
    \begin{tikzpicture}[scale=0.8,Centering]
        \node[CliquePoint](1)at(-0.50,-0.87){};
        \node[CliquePoint](2)at(-1.00,-0.00){};
        \node[CliquePoint](3)at(-0.50,0.87){};
        \node[CliquePoint](4)at(0.50,0.87){};
        \node[CliquePoint](5)at(1.00,0.00){};
        \node[CliquePoint](6)at(0.50,-0.87){};
        \draw[CliqueEmptyEdge](1)edge[]node[]{}(2);
        \draw[CliqueEmptyEdge](1)edge[]node[]{}(6);
        \draw[CliqueEmptyEdge](2)edge[]node[]{}(3);
        \draw[CliqueEdge](2)edge[]node[CliqueLabel]
            {\begin{math}-1\end{math}}(4);
        \draw[CliqueEdge](2)edge[]node[CliqueLabel]
            {\begin{math}1\end{math}}(5);
        \draw[CliqueEdge](2)edge[]node[CliqueLabel]
            {\begin{math}2\end{math}}(6);
        \draw[CliqueEdge](3)edge[]node[CliqueLabel]
            {\begin{math}1\end{math}}(4);
        \draw[CliqueEmptyEdge](4)edge[]node[]{}(5);
        \draw[CliqueEdge](5)edge[]node[CliqueLabel]
            {\begin{math}1\end{math}}(6);
    \end{tikzpicture}}\,,
\end{equation}
\begin{equation}
    \Ksf_{
    \begin{tikzpicture}[scale=0.6,Centering]
        \node[CliquePoint](1)at(-0.71,-0.71){};
        \node[CliquePoint](2)at(-0.71,0.71){};
        \node[CliquePoint](3)at(0.71,0.71){};
        \node[CliquePoint](4)at(0.71,-0.71){};
        \draw[CliqueEmptyEdge](1)edge[]node[]{}(2);
        \draw[CliqueEmptyEdge](1)edge[]node[]{}(4);
        \draw[CliqueEdge](2)edge[]node[CliqueLabel]
            {\begin{math}-1\end{math}}(3);
        \draw[CliqueEdge](2)edge[]node[CliqueLabel]
            {\begin{math}2\end{math}}(4);
        \draw[CliqueEdge](3)edge[]node[CliqueLabel]
            {\begin{math}1\end{math}}(4);
    \end{tikzpicture}}
    \circ_2
    \Ksf_{
    \begin{tikzpicture}[scale=0.6,Centering]
        \node[CliquePoint](1)at(-0.71,-0.71){};
        \node[CliquePoint](2)at(-0.71,0.71){};
        \node[CliquePoint](3)at(0.71,0.71){};
        \node[CliquePoint](4)at(0.71,-0.71){};
        \draw[CliqueEmptyEdge](1)edge[]node[]{}(2);
        \draw[CliqueEdge](1)edge[]node[CliqueLabel]
            {\begin{math}-1\end{math}}(3);
        \draw[CliqueEdge](1)edge[]node[CliqueLabel]
            {\begin{math}1\end{math}}(4);
        \draw[CliqueEdge](2)edge[]node[CliqueLabel]
            {\begin{math}1\end{math}}(3);
        \draw[CliqueEmptyEdge](3)edge[]node[]{}(4);
    \end{tikzpicture}}
    =
    \Ksf_{
    \begin{tikzpicture}[scale=0.8,Centering]
        \node[CliquePoint](1)at(-0.50,-0.87){};
        \node[CliquePoint](2)at(-1.00,-0.00){};
        \node[CliquePoint](3)at(-0.50,0.87){};
        \node[CliquePoint](4)at(0.50,0.87){};
        \node[CliquePoint](5)at(1.00,0.00){};
        \node[CliquePoint](6)at(0.50,-0.87){};
        \draw[CliqueEmptyEdge](1)edge[]node[]{}(2);
        \draw[CliqueEmptyEdge](1)edge[]node[]{}(6);
        \draw[CliqueEmptyEdge](2)edge[]node[]{}(3);
        \draw[CliqueEdge](2)edge[bend right=30]node[CliqueLabel]
            {\begin{math}-1\end{math}}(4);
        \draw[CliqueEdge](2)edge[bend left=30]node[CliqueLabel]
            {\begin{math}2\end{math}}(6);
        \draw[CliqueEdge](3)edge[]node[CliqueLabel]
            {\begin{math}1\end{math}}(4);
        \draw[CliqueEmptyEdge](4)edge[]node[]{}(5);
        \draw[CliqueEdge](5)edge[]node[CliqueLabel]
            {\begin{math}1\end{math}}(6);
    \end{tikzpicture}}\,,
\end{equation}
\end{subequations}
and in $\Dbb_1$,
\begin{subequations}
\begin{equation}
    \Hsf_{
    \begin{tikzpicture}[scale=0.6,Centering]
        \node[CliquePoint](1)at(-0.71,-0.71){};
        \node[CliquePoint](2)at(-0.71,0.71){};
        \node[CliquePoint](3)at(0.71,0.71){};
        \node[CliquePoint](4)at(0.71,-0.71){};
        \draw[CliqueEmptyEdge](1)edge[]node[]{}(2);
        \draw[CliqueEmptyEdge](1)edge[]node[]{}(4);
        \draw[CliqueEdge](2)edge[]node[CliqueLabel]
            {\begin{math}0\end{math}}(3);
        \draw[CliqueEdge](2)edge[]node[CliqueLabel]
            {\begin{math}\Dtt_1\end{math}}(4);
        \draw[CliqueEdge](3)edge[]node[CliqueLabel]
            {\begin{math}0\end{math}}(4);
    \end{tikzpicture}}
    \circ_2
    \Hsf_{
    \begin{tikzpicture}[scale=0.6,Centering]
        \node[CliquePoint](1)at(-0.71,-0.71){};
        \node[CliquePoint](2)at(-0.71,0.71){};
        \node[CliquePoint](3)at(0.71,0.71){};
        \node[CliquePoint](4)at(0.71,-0.71){};
        \draw[CliqueEmptyEdge](1)edge[]node[]{}(2);
        \draw[CliqueEdge](1)edge[]node[CliqueLabel]
            {\begin{math}0\end{math}}(3);
        \draw[CliqueEdge](1)edge[]node[CliqueLabel]
            {\begin{math}0\end{math}}(4);
        \draw[CliqueEdge](2)edge[]node[CliqueLabel]
            {\begin{math}0\end{math}}(3);
        \draw[CliqueEmptyEdge](3)edge[]node[]{}(4);
    \end{tikzpicture}}
    =
    3 \;
    \Hsf_{
    \begin{tikzpicture}[scale=0.8,Centering]
        \node[CliquePoint](1)at(-0.50,-0.87){};
        \node[CliquePoint](2)at(-1.00,-0.00){};
        \node[CliquePoint](3)at(-0.50,0.87){};
        \node[CliquePoint](4)at(0.50,0.87){};
        \node[CliquePoint](5)at(1.00,0.00){};
        \node[CliquePoint](6)at(0.50,-0.87){};
        \draw[CliqueEmptyEdge](1)edge[]node[]{}(2);
        \draw[CliqueEmptyEdge](1)edge[]node[]{}(6);
        \draw[CliqueEmptyEdge](2)edge[]node[]{}(3);
        \draw[CliqueEdge](2)edge[]node[CliqueLabel]
            {\begin{math}0\end{math}}(4);
        \draw[CliqueEdge](2)edge[]node[CliqueLabel]
            {\begin{math}0\end{math}}(5);
        \draw[CliqueEdge](2)edge[]node[CliqueLabel]
            {\begin{math}\Dtt_1\end{math}}(6);
        \draw[CliqueEdge](3)edge[]node[CliqueLabel]
            {\begin{math}0\end{math}}(4);
        \draw[CliqueEmptyEdge](4)edge[]node[]{}(5);
        \draw[CliqueEdge](5)edge[]node[CliqueLabel]
            {\begin{math}0\end{math}}(6);
    \end{tikzpicture}}
    +
    \Hsf_{
    \begin{tikzpicture}[scale=0.8,Centering]
        \node[CliquePoint](1)at(-0.50,-0.87){};
        \node[CliquePoint](2)at(-1.00,-0.00){};
        \node[CliquePoint](3)at(-0.50,0.87){};
        \node[CliquePoint](4)at(0.50,0.87){};
        \node[CliquePoint](5)at(1.00,0.00){};
        \node[CliquePoint](6)at(0.50,-0.87){};
        \draw[CliqueEmptyEdge](1)edge[]node[]{}(2);
        \draw[CliqueEmptyEdge](1)edge[]node[]{}(6);
        \draw[CliqueEmptyEdge](2)edge[]node[]{}(3);
        \draw[CliqueEdge](2)edge[bend right=30]node[CliqueLabel]
            {\begin{math}0\end{math}}(4);
        \draw[CliqueEdge](2)edge[bend left=30]node[CliqueLabel]
            {\begin{math}\Dtt_1\end{math}}(6);
        \draw[CliqueEdge](3)edge[]node[CliqueLabel]
            {\begin{math}0\end{math}}(4);
        \draw[CliqueEmptyEdge](4)edge[]node[]{}(5);
        \draw[CliqueEdge](5)edge[]node[CliqueLabel]
            {\begin{math}0\end{math}}(6);
    \end{tikzpicture}}\,,
\end{equation}
\begin{equation}
    \Ksf_{
    \begin{tikzpicture}[scale=0.6,Centering]
        \node[CliquePoint](1)at(-0.71,-0.71){};
        \node[CliquePoint](2)at(-0.71,0.71){};
        \node[CliquePoint](3)at(0.71,0.71){};
        \node[CliquePoint](4)at(0.71,-0.71){};
        \draw[CliqueEmptyEdge](1)edge[]node[]{}(2);
        \draw[CliqueEmptyEdge](1)edge[]node[]{}(4);
        \draw[CliqueEdge](2)edge[]node[CliqueLabel]
            {\begin{math}0\end{math}}(3);
        \draw[CliqueEdge](2)edge[]node[CliqueLabel]
            {\begin{math}\Dtt_1\end{math}}(4);
        \draw[CliqueEdge](3)edge[]node[CliqueLabel]
            {\begin{math}0\end{math}}(4);
    \end{tikzpicture}}
    \circ_2
    \Ksf_{
    \begin{tikzpicture}[scale=0.6,Centering]
        \node[CliquePoint](1)at(-0.71,-0.71){};
        \node[CliquePoint](2)at(-0.71,0.71){};
        \node[CliquePoint](3)at(0.71,0.71){};
        \node[CliquePoint](4)at(0.71,-0.71){};
        \draw[CliqueEmptyEdge](1)edge[]node[]{}(2);
        \draw[CliqueEdge](1)edge[]node[CliqueLabel]
            {\begin{math}0\end{math}}(3);
        \draw[CliqueEdge](1)edge[]node[CliqueLabel]
            {\begin{math}0\end{math}}(4);
        \draw[CliqueEdge](2)edge[]node[CliqueLabel]
            {\begin{math}0\end{math}}(3);
        \draw[CliqueEmptyEdge](3)edge[]node[]{}(4);
    \end{tikzpicture}}
    =
    \Ksf_{
    \begin{tikzpicture}[scale=0.8,Centering]
        \node[CliquePoint](1)at(-0.50,-0.87){};
        \node[CliquePoint](2)at(-1.00,-0.00){};
        \node[CliquePoint](3)at(-0.50,0.87){};
        \node[CliquePoint](4)at(0.50,0.87){};
        \node[CliquePoint](5)at(1.00,0.00){};
        \node[CliquePoint](6)at(0.50,-0.87){};
        \draw[CliqueEmptyEdge](1)edge[]node[]{}(2);
        \draw[CliqueEmptyEdge](1)edge[]node[]{}(6);
        \draw[CliqueEmptyEdge](2)edge[]node[]{}(3);
        \draw[CliqueEdge](2)edge[]node[CliqueLabel]
            {\begin{math}0\end{math}}(4);
        \draw[CliqueEdge](2)edge[]node[CliqueLabel]
            {\begin{math}0\end{math}}(5);
        \draw[CliqueEdge](2)edge[]node[CliqueLabel]
            {\begin{math}\Dtt_1\end{math}}(6);
        \draw[CliqueEdge](3)edge[]node[CliqueLabel]
            {\begin{math}0\end{math}}(4);
        \draw[CliqueEmptyEdge](4)edge[]node[]{}(5);
        \draw[CliqueEdge](5)edge[]node[CliqueLabel]
            {\begin{math}0\end{math}}(6);
    \end{tikzpicture}}
    +
    \Ksf_{
    \begin{tikzpicture}[scale=0.8,Centering]
        \node[CliquePoint](1)at(-0.50,-0.87){};
        \node[CliquePoint](2)at(-1.00,-0.00){};
        \node[CliquePoint](3)at(-0.50,0.87){};
        \node[CliquePoint](4)at(0.50,0.87){};
        \node[CliquePoint](5)at(1.00,0.00){};
        \node[CliquePoint](6)at(0.50,-0.87){};
        \draw[CliqueEmptyEdge](1)edge[]node[]{}(2);
        \draw[CliqueEmptyEdge](1)edge[]node[]{}(6);
        \draw[CliqueEmptyEdge](2)edge[]node[]{}(3);
        \draw[CliqueEdge](2)edge[bend right=30]node[CliqueLabel]
            {\begin{math}0\end{math}}(4);
        \draw[CliqueEdge](2)edge[bend left=30]node[CliqueLabel]
            {\begin{math}\Dtt_1\end{math}}(6);
        \draw[CliqueEdge](3)edge[]node[CliqueLabel]
            {\begin{math}0\end{math}}(4);
        \draw[CliqueEmptyEdge](4)edge[]node[]{}(5);
        \draw[CliqueEdge](5)edge[]node[CliqueLabel]
            {\begin{math}0\end{math}}(6);
    \end{tikzpicture}}\,.
\end{equation}
\end{subequations}
\medskip

\subsubsection{Rational functions}%
\label{subsubsec:rational_functions}
The graded vector space of all commutative rational functions
$\K(\Ubb)$, where $\Ubb$ is the infinite commutative alphabet
$\{u_1, u_2, \dots\}$, has the structure of an operad $\RatFct$
introduced by Loday~\cite{Lod10} and defined as follows. Let
$\RatFct(n)$ be the subspace $\K(u_1, \dots, u_n)$ of $\K(\Ubb)$ and
\begin{equation}
    \RatFct := \bigoplus_{n \geq 1} \RatFct(n).
\end{equation}
Observe that since $\RatFct$ is a graded space, each rational function
has an arity. Hence, by setting $f_1(u_1) := 1$ and $f_2(u_1, u_2) := 1$,
$f_1$ is of arity $1$ while $f_2$ is of arity $2$, so that $f_1$ and
$f_2$ are considered as different rational functions. The partial
composition of two rational functions $f \in \RatFct(n)$ and
$g \in \RatFct(m)$ satisfies
\begin{equation} \label{equ:partial_composition_RatFct}
    f \circ_i g :=
    f\left(u_1, \dots, u_{i - 1}, u_i + \dots + u_{i + m - 1},
        u_{i + m}, \dots, u_{n + m - 1}\right)
    \;
    g\left(u_i, \dots, u_{i + m - 1}\right).
\end{equation}
The rational function $f$ of $\RatFct(1)$ defined by $f(u_1) := 1$ is
the unit of~$\RatFct$. As shown by Loday, this operad is (nontrivially)
isomorphic to the operad $\Mould$ introduced by Chapoton~\cite{Cha07}.
\medskip

Let us assume that $\Mca$ is a \Def{$\Z$-graded unitary magma}, that
is a unitary magma such that there exists a unitary magma morphism
$\theta : \Mca \to \Z$. We say that $\theta$ is a \Def{rank function}
of $\Mca$. In this context, let
\begin{equation}
    \Frac_\theta : \Cli\Mca \to \RatFct
\end{equation}
be the linear map defined, for any $\Mca$-clique $\Pfr$, by
\begin{equation} \label{equ:definition_frac_clique}
    \Frac_\theta(\Pfr) :=
    \prod_{(x, y) \in \Arcs_\Pfr}
    \left(u_x + \dots + u_{y - 1}\right)^{\theta(\Pfr(x, y))}.
\end{equation}
For instance, by considering the unitary magma $\Z$ together with its
identity map $\Id$ as rank function, one has
\begin{equation}
    \Frac_\Id\left(
    \begin{tikzpicture}[scale=.85,Centering]
        \node[CliquePoint](1)at(-0.43,-0.90){};
        \node[CliquePoint](2)at(-0.97,-0.22){};
        \node[CliquePoint](3)at(-0.78,0.62){};
        \node[CliquePoint](4)at(-0.00,1.00){};
        \node[CliquePoint](5)at(0.78,0.62){};
        \node[CliquePoint](6)at(0.97,-0.22){};
        \node[CliquePoint](7)at(0.43,-0.90){};
        \draw[CliqueEdge](1)edge[]node[CliqueLabel]
            {\begin{math}-1\end{math}}(2);
        \draw[CliqueEdge](1)edge[bend left=30]node[CliqueLabel,near start]
            {\begin{math}2\end{math}}(5);
        \draw[CliqueEdge](1)edge[]node[CliqueLabel]
            {\begin{math}1\end{math}}(7);
        \draw[CliqueEmptyEdge](2)edge[]node[CliqueLabel]{}(3);
        \draw[CliqueEmptyEdge](3)edge[]node[CliqueLabel]{}(4);
        \draw[CliqueEdge](3)edge[bend left=30]node[CliqueLabel,near start]
            {\begin{math}-2\end{math}}(7);
        \draw[CliqueEdge](4)edge[]node[CliqueLabel]
            {\begin{math}3\end{math}}(5);
        \draw[CliqueEmptyEdge](5)edge[]node[CliqueLabel]{}(6);
        \draw[CliqueEmptyEdge](6)edge[]node[CliqueLabel]{}(7);
        \draw[CliqueEdge](5)edge[]node[CliqueLabel]
            {\begin{math}-1\end{math}}(7);
    \end{tikzpicture}\right)
    =
    \frac{
        \left(u_1 + u_2 + u_3 + u_4\right)^2
        \left(u_1 + u_2 + u_3 + u_4 + u_5 + u_6\right)
        u_4^3
    }{u_1 \left(u_3 + u_4 + u_5 + u_6\right)^2 \left(u_5 + u_6\right)}.
\end{equation}
\medskip

\begin{Theorem} \label{thm:rat_fct_cliques}
    Let $\Mca$ be a $\Z$-graded unitary magma and $\theta$ be a rank
    function of $\Mca$. The map $\Frac_\theta$ is an operad morphism
    from $\Cli\Mca$ to~$\RatFct$.
\end{Theorem}
\begin{proof}
    To gain concision, for all positive integers $x < y$, we denote
    by $\Ubb_{x, y}$ the sums $u_x + \dots + u_{y - 1}$. Let $\Pfr$
    and $\Qfr$ be two $\Mca$-cliques of respective arities $n$ and $m$,
    and $i \in [n]$ be an integer. From the definition of the partial
    composition of $\Cli\Mca$, the
    one~\eqref{equ:partial_composition_RatFct} of $\RatFct$, and the
    fact that $\theta$ is a unitary magma morphism, we have
    \begin{equation}\begin{split}
        \Frac_\theta(\Pfr) \circ_i \Frac_\theta(\Qfr)
        & =
        \left(\Frac_\theta(\Pfr)\right)
            \left(u_1, \dots,u_{i - 1}, \Ubb_{i, i + m}, u_{i + m},
            \dots, u_{n + m - 1}\right)
        \;
        \left(\Frac_\theta(\Qfr)\right)
            \left(u_i, \dots, u_{i + m - 1}\right) \\
        & =
        \left(
        \prod_{1 \leq x < y \leq i - 1}
        \Ubb_{x, y}^{\theta(\Pfr(x, y))}
        \right)
        \left(
        \prod_{i + 1 \leq x < y \leq n + 1}
        \Ubb_{x + m - 1, y + m - 1}^{\theta(\Pfr(x, y))}
        \right)
        \Ubb_{i, i + m}^{\theta(\Pfr_i)} \\
        & \qquad
        \left(
        \prod_{1 \leq x < y \leq m + 1}
        \Ubb_{x + i - 1, y + i - 1}^{\theta(\Qfr(x, y))}
        \right) \\
        & =
        \left(
        \prod_{1 \leq x < y \leq i - 1}
        \Ubb_{x, y}^{\theta(\Pfr(x, y))}
        \right)
        \left(
        \prod_{i + 1 \leq x < y \leq n + 1}
        \Ubb_{x + m - 1, y + m - 1}^{\theta(\Pfr(x, y))}
        \right)
        \Ubb_{i, i + m}
            ^{\theta(\Pfr_i) + \theta(\Qfr_0)} \\
        & \qquad
        \left(
        \prod_{\substack{
            1 \leq x < y \leq m + 1 \\
            (x, y) \ne (1, m + 1)
        }}
        \Ubb_{x + i - 1, y + i - 1}^{\theta(\Qfr(x, y))}
        \right) \\
        & =
        \left(
        \prod_{1 \leq x < y \leq i - 1}
        \Ubb_{x, y}^{\theta(\Pfr(x, y))}
        \right)
        \left(
        \prod_{i + 1 \leq x < y \leq n + 1}
        \Ubb_{x + m - 1, y + m - 1}^{\theta(\Pfr(x, y))}
        \right)
        \Ubb_{i, i + m}
            ^{\theta(\Pfr_i \Op \Qfr_0)} \\
        & \qquad
        \left(
        \prod_{\substack{
            1 \leq x < y \leq m + 1 \\
            (x, y) \ne (1, m + 1)
        }}
        \Ubb_{x + i - 1, y + i - 1}^{\theta(\Qfr(x, y))}
        \right) \\
        & =
        \prod_{(x, y) \in \Arcs_{\Pfr \circ_i \Qfr}}
        \Ubb_{x, y}^{\theta((\Pfr \circ_i \Qfr)(x, y))} \\
        & =
        \Frac_\theta(\Pfr \circ_i \Qfr).
    \end{split}\end{equation}
    Moreover, since $\theta(\Unit_\Mca) = 0$, one has
    $\Frac_\theta\left(\UnitClique\right) = 1$, so that $\Frac_\theta$
    sends the unit of $\Cli\Mca$ to the unit of $\RatFct$. Therefore,
    $\Frac_\theta$ is an operad morphism.
\end{proof}
\medskip

The operad morphism $\Frac_\theta$ is not injective. Indeed, by
considering the magma $\Z$ together with its identity map $\Id$ as rank
function, one one has for instance
\begin{subequations}
\begin{equation} \label{equ:frac_not_injective_1}
    \Frac_\Id\left(
    \TriangleXEE{1}{}{}
    - \TriangleEXE{}{1}{}
    - \TriangleEEX{}{}{1}
    \right)
    = (u_1 + u_2) - u_1 - u_2
    = 0,
\end{equation}
\begin{equation} \label{equ:frac_not_injective_2}
    \Frac_\Id\left(
    \begin{tikzpicture}[scale=0.6,Centering]
        \node[CliquePoint](1)at(-0.71,-0.71){};
        \node[CliquePoint](2)at(-0.71,0.71){};
        \node[CliquePoint](3)at(0.71,0.71){};
        \node[CliquePoint](4)at(0.71,-0.71){};
        \draw[CliqueEmptyEdge](1)edge[]node[]{}(2);
        \draw[CliqueEmptyEdge](1)edge[]node[]{}(4);
        \draw[CliqueEdge](2)edge[]node[CliqueLabel]
            {\begin{math}-1\end{math}}(3);
        \draw[CliqueEdge](3)edge[]node[CliqueLabel]
            {\begin{math}-1\end{math}}(4);
    \end{tikzpicture}
    -
    \begin{tikzpicture}[scale=0.6,Centering]
        \node[CliquePoint](1)at(-0.71,-0.71){};
        \node[CliquePoint](2)at(-0.71,0.71){};
        \node[CliquePoint](3)at(0.71,0.71){};
        \node[CliquePoint](4)at(0.71,-0.71){};
        \draw[CliqueEmptyEdge](1)edge[]node[]{}(2);
        \draw[CliqueEmptyEdge](1)edge[]node[]{}(4);
        \draw[CliqueEmptyEdge](2)edge[]node[]{}(3);
        \draw[CliqueEdge](2)edge[]node[CliqueLabel]
            {\begin{math}-1\end{math}}(4);
        \draw[CliqueEdge](3)edge[]node[CliqueLabel]
            {\begin{math}-1\end{math}}(4);
    \end{tikzpicture}
    -
    \begin{tikzpicture}[scale=0.6,Centering]
        \node[CliquePoint](1)at(-0.71,-0.71){};
        \node[CliquePoint](2)at(-0.71,0.71){};
        \node[CliquePoint](3)at(0.71,0.71){};
        \node[CliquePoint](4)at(0.71,-0.71){};
        \draw[CliqueEmptyEdge](1)edge[]node[]{}(2);
        \draw[CliqueEmptyEdge](1)edge[]node[]{}(4);
        \draw[CliqueEmptyEdge](3)edge[]node[]{}(4);
        \draw[CliqueEdge](2)edge[]node[CliqueLabel]
            {\begin{math}-1\end{math}}(3);
        \draw[CliqueEdge](2)edge[]node[CliqueLabel]
            {\begin{math}-1\end{math}}(4);
    \end{tikzpicture}
    \right)
    =
    \frac{1}{u_2 u_3}
    -
    \frac{1}{(u_2 + u_3) u_3}
    -
    \frac{1}{u_2 (u_2 + u_3)}
    = 0.
\end{equation}
\end{subequations}
\medskip

\begin{Proposition} \label{prop:rat_fct_cliques_map_Laurent_polynomials}
    The subspace of $\RatFct$ of all Laurent polynomials on~$\Ubb$ is
    the image by $\Frac_\Id : \Cli\Z \to \RatFct$ of the subspace of
    $\Cli\Z$ consisting in the linear span of all $\Z$-bubbles.
\end{Proposition}
\begin{proof}
    First, by Theorem~\ref{thm:rat_fct_cliques}, $\Frac_\Id$ is a
    well-defined operad morphism from $\Cli\Z$ to $\RatFct$. Let
    $u_1^{\alpha_1} \dots u_n^{\alpha_n}$ be a Laurent monomial, where
    $\alpha_1, \dots, \alpha_n \in \Z$ and $n \geq 1$. Consider also the
    $\Z$-clique $\Pfr_\alpha$ of arity $n + 1$ satisfying
    \begin{equation}
        \Pfr_\alpha(x, y) :=
        \begin{cases}
            \alpha_x & \mbox{if } y = x + 1, \\
            0 & \mbox{otherwise}.
        \end{cases}
    \end{equation}
    Observe that $\Pfr_\alpha$ is a $\Z$-bubble. By definition of
    $\Frac_\Id$, we have
    $\Frac_\Id(\Pfr_\alpha) = u_1^{\alpha_1} \dots u_n^{\alpha_n}$.
    Now, since a Laurent polynomial is a linear combination of some
    Laurent monomials, by the linearity of $\Frac_\Id$, the statement of
    the proposition follows.
\end{proof}
\medskip

On each homogeneous subspace $\Cli\Mca(n)$ of the elements of arity
$n \geq 1$ of $\Cli\Mca$, let the product
\begin{equation}
    \Op : \Cli\Mca(n) \otimes \Cli\Mca(n) \to \Cli\Mca(n)
\end{equation}
defined linearly, for each $\Mca$-cliques $\Pfr$ and $\Qfr$ of
$\Cli\Mca(n)$, by
\begin{equation}
    (\Pfr \Op \Qfr)(x, y) := \Pfr(x, y) \Op \Qfr(x, y),
\end{equation}
where $(x, y)$ is any arc such that $1 \leq x < y \leq n + 1$. For
instance, in $\Cli\Z$,
\begin{equation}
    \begin{tikzpicture}[scale=0.8,Centering]
        \node[CliquePoint](1)at(-0.50,-0.87){};
        \node[CliquePoint](2)at(-1.00,-0.00){};
        \node[CliquePoint](3)at(-0.50,0.87){};
        \node[CliquePoint](4)at(0.50,0.87){};
        \node[CliquePoint](5)at(1.00,0.00){};
        \node[CliquePoint](6)at(0.50,-0.87){};
        \draw[CliqueEmptyEdge](1)edge[]node[]{}(2);
        \draw[CliqueEmptyEdge](1)edge[]node[]{}(6);
        \draw[CliqueEmptyEdge](2)edge[]node[]{}(3);
        \draw[CliqueEdge](2)edge[bend right=30]node[CliqueLabel]
            {\begin{math}2\end{math}}(4);
        \draw[CliqueEdge](2)edge[bend left=30]node[CliqueLabel]
            {\begin{math}-1\end{math}}(6);
        \draw[CliqueEdge](3)edge[]node[CliqueLabel]
            {\begin{math}1\end{math}}(4);
        \draw[CliqueEmptyEdge](4)edge[]node[]{}(5);
        \draw[CliqueEdge](5)edge[]node[CliqueLabel]
            {\begin{math}-2\end{math}}(6);
    \end{tikzpicture}
    \enspace \Op \enspace
    \begin{tikzpicture}[scale=0.8,Centering]
        \node[CliquePoint](1)at(-0.50,-0.87){};
        \node[CliquePoint](2)at(-1.00,-0.00){};
        \node[CliquePoint](3)at(-0.50,0.87){};
        \node[CliquePoint](4)at(0.50,0.87){};
        \node[CliquePoint](5)at(1.00,0.00){};
        \node[CliquePoint](6)at(0.50,-0.87){};
        \draw[CliqueEmptyEdge](1)edge[]node[]{}(2);
        \draw[CliqueEmptyEdge](1)edge[]node[]{}(6);
        \draw[CliqueEdge](2)edge[]node[CliqueLabel]
            {\begin{math}3\end{math}}(3);
        \draw[CliqueEdge](2)edge[bend right=30]node[CliqueLabel]
            {\begin{math}1\end{math}}(4);
        \draw[CliqueEdge](1)edge[bend left=30]node[CliqueLabel]
            {\begin{math}-1\end{math}}(5);
        \draw[CliqueEdge](3)edge[]node[CliqueLabel]
            {\begin{math}1\end{math}}(4);
        \draw[CliqueEmptyEdge](4)edge[]node[]{}(5);
        \draw[CliqueEdge](5)edge[]node[CliqueLabel]
            {\begin{math}2\end{math}}(6);
    \end{tikzpicture}
    \enspace = \enspace
    \begin{tikzpicture}[scale=0.8,Centering]
        \node[CliquePoint](1)at(-0.50,-0.87){};
        \node[CliquePoint](2)at(-1.00,-0.00){};
        \node[CliquePoint](3)at(-0.50,0.87){};
        \node[CliquePoint](4)at(0.50,0.87){};
        \node[CliquePoint](5)at(1.00,0.00){};
        \node[CliquePoint](6)at(0.50,-0.87){};
        \draw[CliqueEmptyEdge](1)edge[]node[]{}(2);
        \draw[CliqueEmptyEdge](1)edge[]node[]{}(6);
        \draw[CliqueEmptyEdge](5)edge[]node[]{}(6);
        \draw[CliqueEdge](2)edge[]node[CliqueLabel]
            {\begin{math}3\end{math}}(3);
        \draw[CliqueEdge](2)edge[bend right=30]node[CliqueLabel,near end]
            {\begin{math}3\end{math}}(4);
        \draw[CliqueEdge](2)edge[bend left=30]node[CliqueLabel,near end]
            {\begin{math}-1\end{math}}(6);
        \draw[CliqueEdge](3)edge[]node[CliqueLabel]
            {\begin{math}2\end{math}}(4);
        \draw[CliqueEmptyEdge](4)edge[]node[]{}(5);
        \draw[CliqueEdge](1)edge[bend left=30]node[CliqueLabel,near end]
            {\begin{math}-1\end{math}}(5);
    \end{tikzpicture}\,.
\end{equation}
\medskip

\begin{Proposition} \label{prop:rat_fct_cliques_product}
    Let $\Mca$ be a $\Z$-graded unitary magma and $\theta$ be a rank
    function of~$\Mca$. For any homogeneous elements $f$ and $g$ of
    $\Cli\Mca$ of the same arity,
    \begin{equation} \label{equ:rat_fct_cliques_product}
        \Frac_\theta(f) \Frac_\theta(g) = \Frac_\theta(f \Op g).
    \end{equation}
\end{Proposition}
\begin{proof}
    Let $\Pfr$ and $\Qfr$ be two $\Mca$-cliques of $\Cli\Mca$ of arity
    $n$. By definition of the operation $\Op$ on $\Cli\Mca(n)$ and the
    fact that $\theta$ is a unitary magma morphism,
    \begin{equation}\begin{split}
        \Frac_\theta(\Pfr) \Frac_\theta(\Qfr)
        & =
        \left(
        \prod_{(x, y) \in \Arcs_\Pfr}
        \left(u_x + \dots + u_{y - 1}\right)^{\theta(\Pfr(x, y))}
        \right)
        \left(
        \prod_{(x, y) \in \Arcs_\Qfr}
        \left(u_x + \dots + u_{y - 1}\right)^{\theta(\Qfr(x, y))}
        \right) \\
        & =
        \prod_{1 \leq x < y \leq n + 1}
        \left(u_x + \dots + u_{y - 1}\right)
            ^{\theta(\Pfr(x, y)) + \theta(\Qfr(x, y))} \\
        & =
        \prod_{1 \leq x < y \leq n + 1}
        \left(u_x + \dots + u_{y - 1}\right)
            ^{\theta(\Pfr(x, y) * \Qfr(x, y))} \\
        & = \Frac_\theta(\Pfr \Op \Qfr).
    \end{split}\end{equation}
    By the linearity of $\Frac_\theta$ and of $\Op$,
    \eqref{equ:rat_fct_cliques_product} follows.
\end{proof}
\medskip

\begin{Proposition} \label{prop:rat_fct_cliques_inverse}
    Let $\Pfr$ be an $\Mca$-clique of $\Cli\Z$. Then,
    \begin{equation}
        \frac{1}{\Frac_\Id(\Pfr)} = \Frac_\Id((\Cli \eta)(\Pfr)),
    \end{equation}
    where $\eta : \Z \to \Z$ is the unitary magma morphism defined by
    $\eta(x) := -x$ for all $x \in \Z$.
\end{Proposition}
\begin{proof}
    Observe that $(\Cli \eta)(\Pfr)$ is the $\Mca$-clique obtained by
    relabeling each arc $(x, y)$ of $\Pfr$ by~$-\Pfr(x, y)$. Hence,
    since $\eta$ is a unitary magma morphism, we have
    \begin{equation}\begin{split}
        \Frac_\Id((\Cli \eta)(\Pfr))
        & =
        \prod_{(x, y) \in \Arcs_\Pfr}
        \left(u_x + \dots + u_{y - 1}\right)^{\theta(-\Pfr(x, y))} \\
        & =
        \prod_{(x, y) \in \Arcs_\Pfr}
        \left(u_x + \dots + u_{y - 1}\right)^{-\theta(\Pfr(x, y))} \\
        & =
        \frac{1}{\Frac_\Id(\Pfr)}
    \end{split}\end{equation}
    as expected.
\end{proof}
\medskip

\section{Quotients and suboperads}\label{sec:quotients_suboperads}
We define here quotients and suboperads of $\Cli\Mca$, leading to the
construction of some new operads involving various combinatorial objects
which are, basically, $\Mca$-cliques with some restrictions.
\medskip

\subsection{Main substructures}%
\label{subsec:main_substructures}
Most of the natural subfamilies of $\Mca$-cliques that can be described
by simple combinatorial properties as $\Mca$-cliques with restrained
labels for the bases, edges, and diagonals, white $\Mca$-cliques,
$\Mca$-cliques with a fixed maximal value for their crossings,
$\Mca$-bubbles, $\Mca$-cliques with a fixed maximal value for their
degrees, inclusion-free $\Mca$-cliques, and acyclic $\Mca$-cliques
inherit from the algebraic structure of operad of $\Cli\Mca$ and form
quotients and suboperads of $\Cli\Mca$. We construct and briefly study
here these main substructures of~$\Cli\Mca$.
\medskip

\subsubsection{Restricting the labels}%
\label{subsubsec:suboperad_Cli_M_labels}
In what follows, if $X$ and $Y$ are two subsets of $\Mca$,
$X \Op Y$ denotes the set $\{x \Op y : x \in X \mbox{ and } y \in Y\}$.
\medskip

Let $B$, $E$, and $D$ be three subsets of $\Mca$ and
$\Lab_{B, E, D}\Mca$ be the subspace of $\Cli\Mca$ generated by all
$\Mca$-cliques $\Pfr$ such that the bases of $\Pfr$ are labeled on $B$,
all edges of $\Pfr$ are labeled on $E$, and all diagonals of $\Pfr$ are
labeled on~$D$.
\medskip

\begin{Proposition} \label{prop:suboperad_Cli_M_labels}
    Let $\Mca$ be a unitary magma and $B$, $E$, and $D$ be three subsets
    of $\Mca$. If $\Unit_\Mca \in B$, $\Unit_\Mca \in D$, and
    $E \Op B \subseteq D$, $\Lab_{B, E, D}\Mca$ is a suboperad of~$\Cli\Mca$.
\end{Proposition}
\begin{proof}
    First, since $\Unit_\Mca \in B$, the unit $\UnitClique$ of
    $\Cli\Mca$ belongs to $\Lab_{B, E, D}\Mca$. Consider now two
    $\Mca$-cliques $\Pfr$ and $\Qfr$ of $\Lab^{B, E, D}\Mca$ and a
    partial composition $\Rfr := \Pfr \circ_i \Qfr$ for a valid integer
    $i$. By the definition of the partial composition of $\Cli\Mca$, the
    base of $\Rfr$ has the same label as the base of $\Pfr$, and all
    edges of $\Rfr$ have labels coming from the ones of $\Pfr$ and
    $\Qfr$. Moreover, all diagonals of $\Rfr$ are either non-solid, or
    come from diagonals of $\Pfr$ and $\Qfr$, or are the diagonal
    $\Rfr(i, i + |\Qfr|)$ which is labeled by $\Pfr_i \Op \Qfr_0$. Since
    $\Unit_\Mca \in D$, $\Pfr_i \in E$, $\Qfr_0 \in B$, and
    $E \Op B \subseteq D$, all the labels of these diagonals are in $D$.
    For these reasons, $\Rfr$ is in $\Lab_{B, E, D}\Mca$. Whence the
    statement of the proposition.
\end{proof}
\medskip

\begin{Proposition} \label{prop:suboperad_Cli_M_labels_dimensions}
    Let $\Mca$ be a unitary magma and $B$, $E$, and $D$ be three
    finite subsets of $\Mca$. For all $n \geq 2$,
    \begin{equation} \label{equ:suboperad_Cli_M_labels_dimensions}
        \dim \Lab_{B, E, D}\Mca(n) =
         b e^n d^{(n + 1)(n - 2) / 2},
    \end{equation}
    where $b := \# B$, $e := \# E$, and $d := \# D$.
\end{Proposition}
\begin{proof}
    By Proposition~\ref{prop:dimensions_Cli_M}, there are
    $m^{\binom{n + 1}{2}}$ $\Mca$-cliques of arity $n$, where
    $m := \# \Mca$. Hence, there are
    $m^{\binom{n + 1}{2}} / m^{n + 1}$ $\Mca$-cliques of arity $n$ with
    all edges and the base labeled by $\Unit_\Mca$. This also says that
    there are $d^{\binom{n + 1}{2}} / d^{n + 1}$ $\Mca$-cliques of arity
    $n$ with all diagonals labeled on $D$ and all edges and the base
    labeled by $\Unit_\Mca$. Since an $\Mca$-clique of
    $\Lab_{B, E, D}\Mca(n)$ have its $n$ edges labeled on $E$ and its
    base labeled on $B$, \eqref{equ:suboperad_Cli_M_labels_dimensions}
    follows.
\end{proof}
\medskip

\subsubsection{White cliques}%
\label{subsubsec:suboperad_Cli_M_white}
Let $\Whi\Mca$ be the subspace of $\Cli\Mca$ generated by all white
$\Mca$-cliques. Since, by definition of white $\Mca$-cliques,
\begin{equation}
    \Whi\Mca = \Lab_{\{\Unit_\Mca\}, \{\Unit_\Mca\}, \Mca}\Mca,
\end{equation}
by Proposition~\ref{prop:suboperad_Cli_M_labels}, $\Whi\Mca$ is a
suboperad of $\Cli\Mca$. It follows from
Proposition~\ref{prop:suboperad_Cli_M_labels_dimensions} that when
$\Mca$ is finite, the dimensions of $\Whi\Mca$ satisfy, for any
$n \geq 2$,
\begin{equation}
    \dim \Whi\Mca(n) =
    m^{(n + 1)(n - 2) / 2},
\end{equation}
where $m := \# \Mca$.
\medskip

\subsubsection{Restricting the crossing}%
\label{subsubsec:quotient_Cli_M_crossings}
Let $k \geq 0$ be an integer and $\Rel_{\Cro_k\Mca}$ be the subspace of
$\Cli\Mca$ generated by all $\Mca$-cliques $\Pfr$ such that
$\Cros(\Pfr) \geq k + 1$. As a quotient of graded vector spaces,
\begin{equation}
    \Cro_k\Mca := \Cli\Mca/_{\Rel_{\Cro_k\Mca}}
\end{equation}
is the linear span of all $\Mca$-cliques $\Pfr$ such that
$\Cros(\Pfr) \leq k$.
\medskip

\begin{Proposition} \label{prop:quotient_Cli_M_crossings}
    Let $\Mca$ be a unitary magma and $k \geq 0$ be an integer. Then,
    the space $\Cro_k\Mca$ is both a quotient and a suboperad
    of~$\Cli\Mca$.
\end{Proposition}
\begin{proof}
    Let us first prove that $\Cro_k\Mca$ is a quotient of
    $\Cli\Mca$. For this, observe that if $\Pfr$ and $\Qfr$ are two
    $\Mca$-cliques,
    \begin{equation} \label{equ:quotient_Cli_M_crossings}
        \Cros(\Pfr \circ_i \Qfr) = \max\{\Cros(\Pfr), \Cros(\Qfr)\}
    \end{equation}
    for any valid integer $i$.
    For this reason, if $\Pfr$ is an $\Mca$-clique of $\Rel_{\Cro_k\Mca}$,
    each clique obtained by a partial composition involving $\Pfr$ and
    other $\Mca$-cliques is still in $\Rel_{\Cro_k\Mca}$. This proves
    that $\Rel_{\Cro_k\Mca}$ is an operad ideal of $\Cli\Mca$ and
    hence, that $\Cro_k\Mca$ is a quotient of $\Cli\Mca$.
    \smallskip

    To prove that $\Cro_k\Mca$ is also a suboperad of $\Cli\Mca$,
    consider two $\Mca$-cliques $\Pfr$ and $\Qfr$  of $\Cro_k\Mca$.
    By~\eqref{equ:quotient_Cli_M_crossings}, all $\Mca$-cliques
    $\Pfr \circ_i \Qfr$ are still in $\Cro_k\Mca$, for all valid integer
    $i$. Moreover, the unit $\UnitClique$ of $\Cli\Mca$ belongs to
    $\Cro_k\Mca$. This imply that $\Cro_k\Mca$ is a suboperad
    of~$\Cli\Mca$.
\end{proof}
\medskip

For instance, in the operad $\Cro_2\Z$, we have
\begin{equation}
    \begin{tikzpicture}[scale=0.7,Centering]
        \node[CliquePoint](1)at(-0.59,-0.81){};
        \node[CliquePoint](2)at(-0.95,0.31){};
        \node[CliquePoint](3)at(-0.00,1.00){};
        \node[CliquePoint](4)at(0.95,0.31){};
        \node[CliquePoint](5)at(0.59,-0.81){};
        \draw[CliqueEmptyEdge](1)edge[]node[]{}(2);
        \draw[CliqueEdge](1)edge[bend right=30]node[CliqueLabel,near start]
            {\begin{math}2\end{math}}(3);
        \draw[CliqueEmptyEdge](1)edge[]node[]{}(5);
        \draw[CliqueEmptyEdge](2)edge[]node[]{}(3);
        \draw[CliqueEdge](3)edge[]node[CliqueLabel]
            {\begin{math}1\end{math}}(4);
        \draw[CliqueEdge](4)edge[]node[CliqueLabel]
            {\begin{math}2\end{math}}(5);
        \draw[CliqueEdge](2)edge[]node[CliqueLabel, near end]
            {\begin{math}1\end{math}}(4);
        \draw[CliqueEdge](2)edge[bend left=30]node[CliqueLabel,near end]
            {\begin{math}3\end{math}}(5);
    \end{tikzpicture}
    \circ_3
    \begin{tikzpicture}[scale=0.6,Centering]
        \node[CliquePoint](1)at(-0.71,-0.71){};
        \node[CliquePoint](2)at(-0.71,0.71){};
        \node[CliquePoint](3)at(0.71,0.71){};
        \node[CliquePoint](4)at(0.71,-0.71){};
        \draw[CliqueEmptyEdge](1)edge[]node[]{}(2);
        \draw[CliqueEdge](1)edge[]node[CliqueLabel]
            {\begin{math}2\end{math}}(3);
        \draw[CliqueEmptyEdge](1)edge[]node[]{}(4);
        \draw[CliqueEmptyEdge](2)edge[]node[]{}(3);
        \draw[CliqueEdge](3)edge[]node[CliqueLabel]
            {\begin{math}1\end{math}}(4);
    \end{tikzpicture}
    =
    \begin{tikzpicture}[scale=0.9,Centering]
        \node[CliquePoint](1)at(-0.43,-0.90){};
        \node[CliquePoint](2)at(-0.97,-0.22){};
        \node[CliquePoint](3)at(-0.78,0.62){};
        \node[CliquePoint](4)at(-0.00,1.00){};
        \node[CliquePoint](5)at(0.78,0.62){};
        \node[CliquePoint](6)at(0.97,-0.22){};
        \node[CliquePoint](7)at(0.43,-0.90){};
        \draw[CliqueEmptyEdge](1)edge[]node[]{}(2);
        \draw[CliqueEmptyEdge](1)edge[]node[]{}(7);
        \draw[CliqueEmptyEdge](2)edge[]node[]{}(3);
        \draw[CliqueEmptyEdge](3)edge[]node[]{}(4);
        \draw[CliqueEmptyEdge](4)edge[]node[]{}(5);
        \draw[CliqueEdge](5)edge[]node[CliqueLabel]
            {\begin{math}1\end{math}}(6);
        \draw[CliqueEdge](6)edge[]node[CliqueLabel]
            {\begin{math}2\end{math}}(7);
        \draw[CliqueEdge](1)edge[bend right=30]node[CliqueLabel,near end]
            {\begin{math}2\end{math}}(3);
        \draw[CliqueEdge](2)edge[]node[CliqueLabel,near end]
            {\begin{math}1\end{math}}(6);
        \draw[CliqueEdge](2)edge[bend left=30]node[CliqueLabel,near end]
            {\begin{math}3\end{math}}(7);
        \draw[CliqueEdge](3)edge[]node[CliqueLabel]
            {\begin{math}2\end{math}}(5);
        \draw[CliqueEdge](3)edge[]node[CliqueLabel]
            {\begin{math}1\end{math}}(6);
    \end{tikzpicture}\,.
\end{equation}
\medskip

When $0 \leq k' \leq k$ are integers, by
Proposition~\ref{prop:quotient_Cli_M_crossings}, $\Cro_k\Mca$ and
$\Cro_{k'}\Mca$ are both quotients and suboperads of $\Cli\Mca$. First,
since any $\Mca$-clique of $\Cro_{k'}\Mca$ is also an $\Mca$-clique of
$\Cro_k\Mca$, $\Cro_{k'}\Mca$ is a suboperad of~$\Cro_k\Mca$. Second,
since $\Rel_{\Cro_k\Mca}$ is a subspace of $\Rel_{\Cro_{k'}\Mca}$,
$\Cro_{k'}\Mca$ is a quotient of~$\Cro_k\Mca$.
\medskip

Remark that $\Cro_0\Mca$ is the linear span of all noncrossing
$\Mca$-cliques. We can see these objects as noncrossing
configurations~\cite{FN99} where the edges and bases are colored by
elements of $\Mca$ and the diagonals, by elements of $\bar{\Mca}$. The
operad $\Cro_0\Mca$ has a lot of properties and will be studied in
details in Section~\ref{sec:operad_noncrossing}.
\medskip

\subsubsection{Bubbles}
\label{subsubsec:quotient_Cli_M_bubbles}
Let $\Rel_{\Bub\Mca}$ be the subspace of $\Cli\Mca$ generated by all
$\Mca$-cliques that are not bubbles. As a quotient of graded vector
spaces,
\begin{equation}
    \Bub\Mca := \Cli\Mca/_{\Rel_{\Bub\Mca}}
\end{equation}
is the linear span of all $\Mca$-bubbles.
\medskip

\begin{Proposition} \label{prop:quotient_Cli_M_bubbles}
    Let $\Mca$ be a unitary magma. Then, the space $\Bub_\Mca$ is a
    quotient operad of $\Cli\Mca$.
\end{Proposition}
\begin{proof}
    If $\Pfr$ and $\Qfr$ are two $\Mca$-cliques, all solid diagonals
    of $\Pfr$ and $\Qfr$ appears in $\Pfr \circ_i \Qfr$, for any valid
    integer $i$. For this reason, if $\Pfr$ is an $\Mca$-clique of
    $\Rel_{\Bub\Mca}$, each $\Mca$-clique obtained by a partial
    composition involving $\Pfr$ and other $\Mca$-cliques is still in
    $\Rel_{\Bub\Mca}$. This proves that $\Rel_{\Bub\Mca}$ is an operad
    ideal of $\Cli\Mca$ and implies the statement of the proposition.
\end{proof}
\medskip

For instance, in the operad $\Bub\Z$, we have
\vspace{-1.75em}
\begin{multicols}{2}
\begin{subequations}
\begin{equation}
    \begin{tikzpicture}[scale=0.6,Centering]
        \node[CliquePoint](1)at(-0.59,-0.81){};
        \node[CliquePoint](2)at(-0.95,0.31){};
        \node[CliquePoint](3)at(-0.00,1.00){};
        \node[CliquePoint](4)at(0.95,0.31){};
        \node[CliquePoint](5)at(0.59,-0.81){};
        \draw[CliqueEmptyEdge](1)edge[]node[]{}(2);
        \draw[CliqueEmptyEdge](1)edge[]node[]{}(5);
        \draw[CliqueEmptyEdge](2)edge[]node[]{}(3);
        \draw[CliqueEdge](3)edge[]node[CliqueLabel]
            {\begin{math}1\end{math}}(4);
        \draw[CliqueEdge](4)edge[]node[CliqueLabel]
            {\begin{math}2\end{math}}(5);
    \end{tikzpicture}
    \circ_2
    \begin{tikzpicture}[scale=0.5,Centering]
        \node[CliquePoint](1)at(-0.71,-0.71){};
        \node[CliquePoint](2)at(-0.71,0.71){};
        \node[CliquePoint](3)at(0.71,0.71){};
        \node[CliquePoint](4)at(0.71,-0.71){};
        \draw[CliqueEmptyEdge](1)edge[]node[]{}(2);
        \draw[CliqueEmptyEdge](1)edge[]node[]{}(4);
        \draw[CliqueEmptyEdge](2)edge[]node[]{}(3);
        \draw[CliqueEdge](3)edge[]node[CliqueLabel]
            {\begin{math}1\end{math}}(4);
    \end{tikzpicture}
    =
    \begin{tikzpicture}[scale=0.8,Centering]
        \node[CliquePoint](1)at(-0.43,-0.90){};
        \node[CliquePoint](2)at(-0.97,-0.22){};
        \node[CliquePoint](3)at(-0.78,0.62){};
        \node[CliquePoint](4)at(-0.00,1.00){};
        \node[CliquePoint](5)at(0.78,0.62){};
        \node[CliquePoint](6)at(0.97,-0.22){};
        \node[CliquePoint](7)at(0.43,-0.90){};
        \draw[CliqueEmptyEdge](1)edge[]node[]{}(2);
        \draw[CliqueEmptyEdge](1)edge[]node[]{}(7);
        \draw[CliqueEmptyEdge](2)edge[]node[]{}(3);
        \draw[CliqueEmptyEdge](3)edge[]node[]{}(4);
        \draw[CliqueEmptyEdge](4)edge[]node[]{}(5);
        \draw[CliqueEdge](4)edge[]node[CliqueLabel]
            {\begin{math}1\end{math}}(5);
        \draw[CliqueEdge](5)edge[]node[CliqueLabel]
            {\begin{math}1\end{math}}(6);
        \draw[CliqueEdge](6)edge[]node[CliqueLabel]
            {\begin{math}2\end{math}}(7);
    \end{tikzpicture}\,,
\end{equation}
\begin{equation}
    \begin{tikzpicture}[scale=0.6,Centering]
        \node[CliquePoint](1)at(-0.59,-0.81){};
        \node[CliquePoint](2)at(-0.95,0.31){};
        \node[CliquePoint](3)at(-0.00,1.00){};
        \node[CliquePoint](4)at(0.95,0.31){};
        \node[CliquePoint](5)at(0.59,-0.81){};
        \draw[CliqueEmptyEdge](1)edge[]node[]{}(2);
        \draw[CliqueEmptyEdge](1)edge[]node[]{}(5);
        \draw[CliqueEmptyEdge](2)edge[]node[]{}(3);
        \draw[CliqueEdge](3)edge[]node[CliqueLabel]
            {\begin{math}-1\end{math}}(4);
        \draw[CliqueEdge](4)edge[]node[CliqueLabel]
            {\begin{math}2\end{math}}(5);
    \end{tikzpicture}
    \circ_3
    \begin{tikzpicture}[scale=0.5,Centering]
        \node[CliquePoint](1)at(-0.71,-0.71){};
        \node[CliquePoint](2)at(-0.71,0.71){};
        \node[CliquePoint](3)at(0.71,0.71){};
        \node[CliquePoint](4)at(0.71,-0.71){};
        \draw[CliqueEmptyEdge](1)edge[]node[]{}(2);
        \draw[CliqueEmptyEdge](1)edge[]node[]{}(4);
        \draw[CliqueEmptyEdge](2)edge[]node[]{}(3);
        \draw[CliqueEdge](3)edge[]node[CliqueLabel]
            {\begin{math}1\end{math}}(4);
        \draw[CliqueEdge](1)edge[]node[CliqueLabel]
            {\begin{math}1\end{math}}(4);
    \end{tikzpicture}
    =
    \begin{tikzpicture}[scale=0.8,Centering]
        \node[CliquePoint](1)at(-0.43,-0.90){};
        \node[CliquePoint](2)at(-0.97,-0.22){};
        \node[CliquePoint](3)at(-0.78,0.62){};
        \node[CliquePoint](4)at(-0.00,1.00){};
        \node[CliquePoint](5)at(0.78,0.62){};
        \node[CliquePoint](6)at(0.97,-0.22){};
        \node[CliquePoint](7)at(0.43,-0.90){};
        \draw[CliqueEmptyEdge](1)edge[]node[]{}(2);
        \draw[CliqueEmptyEdge](1)edge[]node[]{}(7);
        \draw[CliqueEmptyEdge](2)edge[]node[]{}(3);
        \draw[CliqueEmptyEdge](3)edge[]node[]{}(4);
        \draw[CliqueEmptyEdge](4)edge[]node[]{}(5);
        \draw[CliqueEdge](5)edge[]node[CliqueLabel]
            {\begin{math}1\end{math}}(6);
        \draw[CliqueEdge](6)edge[]node[CliqueLabel]
            {\begin{math}2\end{math}}(7);
    \end{tikzpicture}\,,
\end{equation}

\begin{equation}
    \begin{tikzpicture}[scale=0.6,Centering]
        \node[CliquePoint](1)at(-0.59,-0.81){};
        \node[CliquePoint](2)at(-0.95,0.31){};
        \node[CliquePoint](3)at(-0.00,1.00){};
        \node[CliquePoint](4)at(0.95,0.31){};
        \node[CliquePoint](5)at(0.59,-0.81){};
        \draw[CliqueEmptyEdge](1)edge[]node[]{}(2);
        \draw[CliqueEmptyEdge](1)edge[]node[]{}(5);
        \draw[CliqueEmptyEdge](2)edge[]node[]{}(3);
        \draw[CliqueEdge](3)edge[]node[CliqueLabel]
            {\begin{math}1\end{math}}(4);
        \draw[CliqueEdge](4)edge[]node[CliqueLabel]
            {\begin{math}2\end{math}}(5);
    \end{tikzpicture}
    \circ_3
    \begin{tikzpicture}[scale=0.5,Centering]
        \node[CliquePoint](1)at(-0.71,-0.71){};
        \node[CliquePoint](2)at(-0.71,0.71){};
        \node[CliquePoint](3)at(0.71,0.71){};
        \node[CliquePoint](4)at(0.71,-0.71){};
        \draw[CliqueEmptyEdge](1)edge[]node[]{}(2);
        \draw[CliqueEmptyEdge](1)edge[]node[]{}(4);
        \draw[CliqueEmptyEdge](2)edge[]node[]{}(3);
        \draw[CliqueEdge](3)edge[]node[CliqueLabel]
            {\begin{math}1\end{math}}(4);
    \end{tikzpicture}
    = 0,
\end{equation}
\begin{equation}
    \begin{tikzpicture}[scale=0.6,Centering]
        \node[CliquePoint](1)at(-0.59,-0.81){};
        \node[CliquePoint](2)at(-0.95,0.31){};
        \node[CliquePoint](3)at(-0.00,1.00){};
        \node[CliquePoint](4)at(0.95,0.31){};
        \node[CliquePoint](5)at(0.59,-0.81){};
        \draw[CliqueEmptyEdge](1)edge[]node[]{}(2);
        \draw[CliqueEmptyEdge](1)edge[]node[]{}(5);
        \draw[CliqueEmptyEdge](2)edge[]node[]{}(3);
        \draw[CliqueEdge](3)edge[]node[CliqueLabel]
            {\begin{math}1\end{math}}(4);
        \draw[CliqueEdge](4)edge[]node[CliqueLabel]
            {\begin{math}2\end{math}}(5);
    \end{tikzpicture}
    \circ_2
    \begin{tikzpicture}[scale=0.5,Centering]
        \node[CliquePoint](1)at(-0.71,-0.71){};
        \node[CliquePoint](2)at(-0.71,0.71){};
        \node[CliquePoint](3)at(0.71,0.71){};
        \node[CliquePoint](4)at(0.71,-0.71){};
        \draw[CliqueEmptyEdge](1)edge[]node[]{}(2);
        \draw[CliqueEmptyEdge](2)edge[]node[]{}(3);
        \draw[CliqueEdge](3)edge[]node[CliqueLabel]
            {\begin{math}1\end{math}}(4);
        \draw[CliqueEdge](1)edge[]node[CliqueLabel]
            {\begin{math}2\end{math}}(4);
    \end{tikzpicture}
    = 0.
\end{equation}
\end{subequations}
\end{multicols}
\medskip

When $\Mca$ is finite, the dimensions of $\Bub\Mca$ satisfy, for any
$n \geq 2$,
\begin{equation}
    \dim \Bub\Mca(n) = m^{n + 1},
\end{equation}
where $m := \# \Mca$.
\medskip

\subsubsection{Restricting the degrees}%
\label{subsubsec:quotient_Cli_M_degrees}
Let $k \geq 0$ be an integer and $\Rel_{\Deg_k\Mca}$ be the subspace of
$\Cli\Mca$ generated by all $\Mca$-cliques $\Pfr$ such that
$\Degr(\Pfr) \geq k + 1$. As a quotient of graded vector spaces,
\begin{equation}
    \Deg_k\Mca := \Cli\Mca/_{\Rel_{\Deg_k\Mca}}
\end{equation}
is the linear span of all $\Mca$-cliques $\Pfr$ such that
$\Degr(\Pfr) \leq k$.
\medskip

\begin{Proposition} \label{prop:quotient_Cli_M_degrees}
    Let $\Mca$ be a unitary magma without nontrivial unit divisors and
    $k \geq 0$ be an integer. Then, the space $\Deg_k\Mca$ is a quotient
    operad of $\Cli\Mca$.
\end{Proposition}
\begin{proof}
    Since $\Mca$ has no nontrivial unit divisors, for any $\Mca$-cliques
    $\Pfr$ and $\Qfr$ of $\Cli\Mca$, each solid arc of $\Pfr$ (resp.
    $\Qfr$) gives rise to a solid arc in $\Pfr \circ_i \Qfr$, for any
    valid integer $i$. Hence,
    \begin{equation}
        \Degr(\Pfr \circ_i \Qfr) \geq \max\{\Degr(\Pfr), \Degr(\Qfr)\},
    \end{equation}
    and then, if $\Pfr$ is an $\Mca$-clique of $\Rel_{\Deg_k\Mca}$,
    each $\Mca$-clique obtained by a partial composition involving
    $\Pfr$ and other $\Mca$-cliques is still in $\Rel_{\Deg_k\Mca}$.
    This proves that $\Rel_{\Deg_k\Mca}$ is an operad ideal of
    $\Cli\Mca$ and implies the statement of the proposition.
\end{proof}
\medskip

For instance, in the operad $\Deg_3\Dbb_2$ (observe that $\Dbb_2$ is
a unitary magma without nontrivial unit divisors), we have
\vspace{-1.75em}
\begin{multicols}{2}
\begin{subequations}
\begin{equation}
    \begin{tikzpicture}[scale=0.6,Centering]
        \node[CliquePoint](1)at(-0.59,-0.81){};
        \node[CliquePoint](2)at(-0.95,0.31){};
        \node[CliquePoint](3)at(-0.00,1.00){};
        \node[CliquePoint](4)at(0.95,0.31){};
        \node[CliquePoint](5)at(0.59,-0.81){};
        \draw[CliqueEmptyEdge](1)edge[]node[]{}(2);
        \draw[CliqueEmptyEdge](1)edge[]node[]{}(5);
        \draw[CliqueEdge](2)edge[]node[CliqueLabel]
            {\begin{math}\Dtt_1\end{math}}(3);
        \draw[CliqueEmptyEdge](3)edge[]node[]{}(4);
        \draw[CliqueEdge](4)edge[]node[CliqueLabel]
            {\begin{math}0\end{math}}(5);
        \draw[CliqueEmptyEdge](1)edge[]node[]{}(2);
        \draw[CliqueEdge](1)edge[bend left=30]node[CliqueLabel,near start]
            {\begin{math}0\end{math}}(4);
        \draw[CliqueEdge](2)edge[bend left=30]node[CliqueLabel,near start]
            {\begin{math}\Dtt_1\end{math}}(5);
    \end{tikzpicture}
    \circ_2
    \begin{tikzpicture}[scale=0.5,Centering]
        \node[CliquePoint](1)at(-0.71,-0.71){};
        \node[CliquePoint](2)at(-0.71,0.71){};
        \node[CliquePoint](3)at(0.71,0.71){};
        \node[CliquePoint](4)at(0.71,-0.71){};
        \draw[CliqueEmptyEdge](1)edge[]node[]{}(2);
        \draw[CliqueEdge](1)edge[]node[CliqueLabel]
            {\begin{math}0\end{math}}(4);
        \draw[CliqueEmptyEdge](2)edge[]node[]{}(3);
        \draw[CliqueEdge](3)edge[]node[CliqueLabel]
            {\begin{math}0\end{math}}(4);
        \draw[CliqueEdge](2)edge[]node[CliqueLabel]
            {\begin{math}\Dtt_1\end{math}}(4);
    \end{tikzpicture}
    =
    \begin{tikzpicture}[scale=0.8,Centering]
        \node[CliquePoint](1)at(-0.43,-0.90){};
        \node[CliquePoint](2)at(-0.97,-0.22){};
        \node[CliquePoint](3)at(-0.78,0.62){};
        \node[CliquePoint](4)at(-0.00,1.00){};
        \node[CliquePoint](5)at(0.78,0.62){};
        \node[CliquePoint](6)at(0.97,-0.22){};
        \node[CliquePoint](7)at(0.43,-0.90){};
        \draw[CliqueEmptyEdge](1)edge[]node[]{}(2);
        \draw[CliqueEmptyEdge](1)edge[]node[]{}(7);
        \draw[CliqueEmptyEdge](2)edge[]node[]{}(3);
        \draw[CliqueEmptyEdge](3)edge[]node[]{}(4);
        \draw[CliqueEmptyEdge](4)edge[]node[]{}(5);
        \draw[CliqueEdge](3)edge[bend right=30]node[CliqueLabel]
            {\begin{math}\Dtt_1\end{math}}(5);
        \draw[CliqueEdge](4)edge[]node[CliqueLabel]
            {\begin{math}0\end{math}}(5);
        \draw[CliqueEmptyEdge](5)edge[]node[CliqueLabel]{}(6);
        \draw[CliqueEdge](6)edge[]node[CliqueLabel]
            {\begin{math}0\end{math}}(7);
        \draw[CliqueEdge](2)edge[bend left=30]node[CliqueLabel,near end]
            {\begin{math}\Dtt_1\end{math}}(7);
        \draw[CliqueEdge](2)edge[bend right=30]node[CliqueLabel]
            {\begin{math}0\end{math}}(5);
    \end{tikzpicture}\,,
\end{equation}

\begin{equation}
    \begin{tikzpicture}[scale=0.6,Centering]
        \node[CliquePoint](1)at(-0.59,-0.81){};
        \node[CliquePoint](2)at(-0.95,0.31){};
        \node[CliquePoint](3)at(-0.00,1.00){};
        \node[CliquePoint](4)at(0.95,0.31){};
        \node[CliquePoint](5)at(0.59,-0.81){};
        \draw[CliqueEmptyEdge](1)edge[]node[]{}(2);
        \draw[CliqueEmptyEdge](1)edge[]node[]{}(5);
        \draw[CliqueEdge](2)edge[]node[CliqueLabel]
            {\begin{math}\Dtt_1\end{math}}(3);
        \draw[CliqueEmptyEdge](3)edge[]node[]{}(4);
        \draw[CliqueEdge](4)edge[]node[CliqueLabel]
            {\begin{math}0\end{math}}(5);
        \draw[CliqueEmptyEdge](1)edge[]node[]{}(2);
        \draw[CliqueEdge](1)edge[bend left=30]node[CliqueLabel,near start]
            {\begin{math}0\end{math}}(4);
        \draw[CliqueEdge](2)edge[bend left=30]node[CliqueLabel,near start]
            {\begin{math}\Dtt_1\end{math}}(5);
    \end{tikzpicture}
    \circ_3
    \begin{tikzpicture}[scale=0.5,Centering]
        \node[CliquePoint](1)at(-0.71,-0.71){};
        \node[CliquePoint](2)at(-0.71,0.71){};
        \node[CliquePoint](3)at(0.71,0.71){};
        \node[CliquePoint](4)at(0.71,-0.71){};
        \draw[CliqueEmptyEdge](1)edge[]node[]{}(2);
        \draw[CliqueEdge](1)edge[]node[CliqueLabel]
            {\begin{math}0\end{math}}(4);
        \draw[CliqueEmptyEdge](2)edge[]node[]{}(3);
        \draw[CliqueEdge](3)edge[]node[CliqueLabel]
            {\begin{math}0\end{math}}(4);
        \draw[CliqueEdge](2)edge[]node[CliqueLabel]
            {\begin{math}\Dtt_1\end{math}}(4);
    \end{tikzpicture}
    = 0.
\end{equation}
\end{subequations}
\end{multicols}
\medskip

When $0 \leq k' \leq k$ are integers, by
Proposition~\ref{prop:quotient_Cli_M_degrees}, $\Deg_k\Mca$ and
$\Deg_{k'}\Mca$ are both quotients operads of $\Cli\Mca$. Moreover,
since $\Rel_{\Deg_k\Mca}$ is a subspace of $\Rel_{\Deg_{k'}\Mca}$,
$\Deg_{k'}\Mca$ is a quotient operad of $\Deg_k\Mca$.
\medskip

Observe that $\Deg_0\Mca$ is the linear span of all $\Mca$-cliques
without solid arcs. If $\Pfr$ and $\Qfr$ are such $\Mca$-cliques, all
partial compositions $\Pfr \circ_i \Qfr$ are equal to the unique
$\Mca$-clique without solid arcs of arity $|\Pfr| + |\Qfr| - 1$. For
this reason, $\Deg_0\Mca$ is the associative operad~$\As$.
\medskip

Any skeleton of an $\Mca$-clique of arity $n$ of $\Deg_1\Mca$ can be
seen as a partition of the set $[n + 1]$ in singletons or pairs.
Therefore, $\Deg_1\Mca$ can be seen as an operad on such colored
partitions, where each pair of the partitions have one color among the
set $\bar{\Mca}$. In the operad $\Deg_1\Dbb_0$ (observe that $\Dbb_0$ is
the only unitary magma without nontrivial unit divisors on two
elements), one has for instance
\vspace{-1.75em}
\begin{multicols}{2}
\begin{subequations}
\begin{equation} \label{equ:example_involutions_1}
    \begin{tikzpicture}[scale=0.6,Centering]
        \node[CliquePoint](1)at(-0.59,-0.81){};
        \node[CliquePoint](2)at(-0.95,0.31){};
        \node[CliquePoint](3)at(-0.00,1.00){};
        \node[CliquePoint](4)at(0.95,0.31){};
        \node[CliquePoint](5)at(0.59,-0.81){};
        \draw[CliqueEmptyEdge](1)edge[]node[]{}(2);
        \draw[CliqueEmptyEdge](1)edge[]node[]{}(5);
        \draw[CliqueEmptyEdge](2)edge[]node[]{}(3);
        \draw[CliqueEmptyEdge](3)edge[]node[]{}(4);
        \draw[CliqueEmptyEdge](4)edge[]node[]{}(5);
        \draw[CliqueEdge](1)edge[bend left=30]node[CliqueLabel]
            {\begin{math}0\end{math}}(4);
    \end{tikzpicture}
    \circ_2
    \begin{tikzpicture}[scale=0.5,Centering]
        \node[CliquePoint](1)at(-0.71,-0.71){};
        \node[CliquePoint](2)at(-0.71,0.71){};
        \node[CliquePoint](3)at(0.71,0.71){};
        \node[CliquePoint](4)at(0.71,-0.71){};
        \draw[CliqueEmptyEdge](1)edge[]node[]{}(2);
        \draw[CliqueEmptyEdge](1)edge[]node[]{}(4);
        \draw[CliqueEmptyEdge](3)edge[]node[]{}(4);
        \draw[CliqueEdge](1)edge[]node[CliqueLabel,near start]
            {\begin{math}0\end{math}}(3);
        \draw[CliqueEmptyEdge](2)edge[]node[]{}(3);
        \draw[CliqueEdge](2)edge[]node[CliqueLabel,near end]
            {\begin{math}0\end{math}}(4);
    \end{tikzpicture}
    =
    \begin{tikzpicture}[scale=0.8,Centering]
        \node[CliquePoint](1)at(-0.43,-0.90){};
        \node[CliquePoint](2)at(-0.97,-0.22){};
        \node[CliquePoint](3)at(-0.78,0.62){};
        \node[CliquePoint](4)at(-0.00,1.00){};
        \node[CliquePoint](5)at(0.78,0.62){};
        \node[CliquePoint](6)at(0.97,-0.22){};
        \node[CliquePoint](7)at(0.43,-0.90){};
        \draw[CliqueEmptyEdge](1)edge[]node[]{}(2);
        \draw[CliqueEmptyEdge](1)edge[]node[]{}(7);
        \draw[CliqueEmptyEdge](2)edge[]node[]{}(3);
        \draw[CliqueEmptyEdge](3)edge[]node[]{}(4);
        \draw[CliqueEmptyEdge](4)edge[]node[]{}(5);
        \draw[CliqueEdge](1)edge[bend left=30]node[CliqueLabel]
            {\begin{math}0\end{math}}(6);
        \draw[CliqueEmptyEdge](4)edge[]node[]{}(5);
        \draw[CliqueEmptyEdge](5)edge[]node[CliqueLabel]{}(6);
        \draw[CliqueEmptyEdge](6)edge[]node[]{}(7);
        \draw[CliqueEdge](2)edge[bend right=30]node[CliqueLabel,near start]
            {\begin{math}0\end{math}}(4);
        \draw[CliqueEdge](3)edge[bend right=30]node[CliqueLabel,near end]
            {\begin{math}0\end{math}}(5);
    \end{tikzpicture}\,,
\end{equation}

\begin{equation} \label{equ:example_involutions_2}
    \begin{tikzpicture}[scale=0.6,Centering]
        \node[CliquePoint](1)at(-0.59,-0.81){};
        \node[CliquePoint](2)at(-0.95,0.31){};
        \node[CliquePoint](3)at(-0.00,1.00){};
        \node[CliquePoint](4)at(0.95,0.31){};
        \node[CliquePoint](5)at(0.59,-0.81){};
        \draw[CliqueEmptyEdge](1)edge[]node[]{}(2);
        \draw[CliqueEmptyEdge](1)edge[]node[]{}(5);
        \draw[CliqueEmptyEdge](2)edge[]node[]{}(3);
        \draw[CliqueEmptyEdge](3)edge[]node[]{}(4);
        \draw[CliqueEmptyEdge](1)edge[]node[]{}(2);
        \draw[CliqueEmptyEdge](4)edge[]node[]{}(5);
        \draw[CliqueEdge](1)edge[bend left=30]node[CliqueLabel]
            {\begin{math}0\end{math}}(4);
    \end{tikzpicture}
    \circ_3
    \begin{tikzpicture}[scale=0.5,Centering]
        \node[CliquePoint](1)at(-0.71,-0.71){};
        \node[CliquePoint](2)at(-0.71,0.71){};
        \node[CliquePoint](3)at(0.71,0.71){};
        \node[CliquePoint](4)at(0.71,-0.71){};
        \draw[CliqueEmptyEdge](1)edge[]node[]{}(2);
        \draw[CliqueEmptyEdge](1)edge[]node[]{}(4);
        \draw[CliqueEmptyEdge](3)edge[]node[]{}(4);
        \draw[CliqueEdge](1)edge[]node[CliqueLabel,near start]
            {\begin{math}0\end{math}}(3);
        \draw[CliqueEmptyEdge](2)edge[]node[]{}(3);
        \draw[CliqueEdge](2)edge[]node[CliqueLabel,near end]
            {\begin{math}0\end{math}}(4);
    \end{tikzpicture}
    = 0.
\end{equation}
\end{subequations}
\end{multicols}
\medskip

By seeing each solid arc $(x, y)$ of an $\Mca$-clique $\Pfr$ of
$\Deg_1\Dbb_0$ of arity $n$ as the transposition exchanging the letter
$x$ and the letter $y$, we can interpret $\Pfr$ as an involution of
$\mathfrak{S}_{n + 1}$ made of the product of these transpositions.
Hence, $\Deg_1\Dbb_0$ can be seen as an operad on involutions. Under
this point of view, the partial
compositions~\eqref{equ:example_involutions_1}
and~\eqref{equ:example_involutions_2} translate on permutations as
\vspace{-1.75em}
\begin{multicols}{2}
\begin{subequations}
\begin{equation}
    42315 \circ_2 3412 = 6452317,
\end{equation}

\begin{equation}
    42315 \circ_3 3412 = 0.
\end{equation}
\end{subequations}
\end{multicols}
\noindent Equivalently, by the Robinson-Schensted correspondence (see
for instance~\cite{Lot02}), $\Deg_1\Dbb_0$ is an operad of standard Young
tableaux. The dimensions of $\Deg_1\Dbb_0$ operad begin by
\begin{equation}
    1, 4, 10, 26, 76, 232, 764, 2620,
\end{equation}
and form, except for the first terms, Sequence~\OEIS{A000085}
of~\cite{Slo}.
Moreover, when $\# \Mca = 3$, the dimensions of
$\Deg_1\Mca$ begin by
\begin{equation}
    1, 7, 25, 81, 331, 1303, 5937, 26785,
\end{equation}
and form, except for the first terms, Sequence~\OEIS{A047974}
of~\cite{Slo}.
\medskip

Besides, any skeleton of an $\Mca$-clique of $\Deg_2\Mca$ can be seen as
a \Def{thunderstorm graph}, {\em i.e.}, a graph where connected
components are cycles or paths. Therefore, $\Deg_2\Mca$ can be seen as
an operad on such colored graphs, where the arcs of the graphs have one
color among the set $\bar{\Mca}$. When $\# \Mca = 2$, the
dimensions of this operad begin by
\begin{equation}
    1, 8, 41, 253, 1858, 15796, 152219, 1638323,
\end{equation}
and form, except for the first terms, Sequence~\OEIS{A136281}
of~\cite{Slo}.
\medskip

\subsubsection{Inclusion-free cliques}%
\label{subsubsec:quotient_Cli_M_Inf}
Let $\Rel_{\Inf\Mca}$ be the subspace of $\Cli\Mca$ generated by all
$\Mca$-cliques that are not inclusion-free.  As a quotient of graded
vector spaces,
\begin{equation}
    \Inf\Mca := \Cli\Mca/_{\Rel_{\Inf\Mca}}
\end{equation}
is the linear span of all inclusion-free $\Mca$-cliques.
\medskip

\begin{Proposition} \label{prop:quotient_Cli_M_inclusion_free}
    Let $\Mca$ be a unitary magma without nontrivial unit divisors.
    Then, the space $\Inf\Mca$ is a quotient operad of $\Cli\Mca$.
\end{Proposition}
\begin{proof}
    Since $\Mca$ has no nontrivial unit divisors, for any $\Mca$-cliques
    $\Pfr$ and $\Qfr$ of $\Cli\Mca$, each solid arc of $\Pfr$ (resp.
    $\Qfr$) gives rise to a solid arc in $\Pfr \circ_i \Qfr$, for any
    valid integer $i$. For this reason, if $\Pfr$ is an $\Mca$-clique of
    $\Rel_{\Inf\Mca}$, $\Pfr$ is not inclusion-free and each
    $\Mca$-clique obtained by a partial composition involving $\Pfr$ and
    other $\Mca$-cliques is still not inclusion-free and thus, belongs
    to $\Rel_{\Inf\Mca}$. This proves that $\Rel_{\Inf\Mca}$ is an
    operad ideal of $\Cli\Mca$ and implies the statement of the
    proposition.
\end{proof}
\medskip

For instance, in the operad $\Inf\Dbb_2$,
\vspace{-1.75em}
\begin{multicols}{2}
\begin{subequations}
\begin{equation}
    \begin{tikzpicture}[scale=0.6,Centering]
        \node[CliquePoint](1)at(-0.59,-0.81){};
        \node[CliquePoint](2)at(-0.95,0.31){};
        \node[CliquePoint](3)at(-0.00,1.00){};
        \node[CliquePoint](4)at(0.95,0.31){};
        \node[CliquePoint](5)at(0.59,-0.81){};
        \draw[CliqueEmptyEdge](1)edge[]node[]{}(2);
        \draw[CliqueEmptyEdge](1)edge[]node[]{}(5);
        \draw[CliqueEmptyEdge](2)edge[]node[]{}(3);
        \draw[CliqueEmptyEdge](3)edge[]node[]{}(4);
        \draw[CliqueEmptyEdge](4)edge[]node[]{}(5);
        \draw[CliqueEdge](1)edge[bend right=30]node[CliqueLabel,near start]
            {\begin{math}0\end{math}}(3);
        \draw[CliqueEdge](2)edge[bend right=30]node[CliqueLabel,near end]
            {\begin{math}\Dtt_1\end{math}}(4);
    \end{tikzpicture}
    \circ_4
    \begin{tikzpicture}[scale=0.5,Centering]
        \node[CliquePoint](1)at(-0.71,-0.71){};
        \node[CliquePoint](2)at(-0.71,0.71){};
        \node[CliquePoint](3)at(0.71,0.71){};
        \node[CliquePoint](4)at(0.71,-0.71){};
        \draw[CliqueEmptyEdge](1)edge[]node[]{}(4);
        \draw[CliqueEmptyEdge](3)edge[]node[]{}(4);
        \draw[CliqueEdge](1)edge[]node[CliqueLabel]
            {\begin{math}\Dtt_1\end{math}}(2);
        \draw[CliqueEdge](2)edge[]node[CliqueLabel]
            {\begin{math}0\end{math}}(3);
    \end{tikzpicture}
    =
    \begin{tikzpicture}[scale=0.8,Centering]
        \node[CliquePoint](1)at(-0.43,-0.90){};
        \node[CliquePoint](2)at(-0.97,-0.22){};
        \node[CliquePoint](3)at(-0.78,0.62){};
        \node[CliquePoint](4)at(-0.00,1.00){};
        \node[CliquePoint](5)at(0.78,0.62){};
        \node[CliquePoint](6)at(0.97,-0.22){};
        \node[CliquePoint](7)at(0.43,-0.90){};
        \draw[CliqueEmptyEdge](1)edge[]node[]{}(2);
        \draw[CliqueEmptyEdge](1)edge[]node[]{}(7);
        \draw[CliqueEmptyEdge](2)edge[]node[]{}(3);
        \draw[CliqueEmptyEdge](3)edge[]node[]{}(4);
        \draw[CliqueEmptyEdge](6)edge[]node[]{}(7);
        \draw[CliqueEdge](4)edge[]node[CliqueLabel]
            {\begin{math}\Dtt_1\end{math}}(5);
        \draw[CliqueEdge](5)edge[]node[CliqueLabel]
            {\begin{math}0\end{math}}(6);
        \draw[CliqueEdge](1)edge[bend right=30]node[CliqueLabel,near start]
            {\begin{math}0\end{math}}(3);
        \draw[CliqueEdge](2)edge[bend right=30]node[CliqueLabel,near end]
            {\begin{math}\Dtt_1\end{math}}(4);
    \end{tikzpicture}\,,
\end{equation}

\begin{equation}
    \begin{tikzpicture}[scale=0.6,Centering]
        \node[CliquePoint](1)at(-0.59,-0.81){};
        \node[CliquePoint](2)at(-0.95,0.31){};
        \node[CliquePoint](3)at(-0.00,1.00){};
        \node[CliquePoint](4)at(0.95,0.31){};
        \node[CliquePoint](5)at(0.59,-0.81){};
        \draw[CliqueEmptyEdge](1)edge[]node[]{}(2);
        \draw[CliqueEmptyEdge](1)edge[]node[]{}(5);
        \draw[CliqueEmptyEdge](2)edge[]node[]{}(3);
        \draw[CliqueEmptyEdge](3)edge[]node[]{}(4);
        \draw[CliqueEmptyEdge](4)edge[]node[]{}(5);
        \draw[CliqueEdge](1)edge[bend right=30]node[CliqueLabel,near start]
            {\begin{math}0\end{math}}(3);
        \draw[CliqueEdge](2)edge[bend right=30]node[CliqueLabel,near end]
            {\begin{math}\Dtt_1\end{math}}(4);
    \end{tikzpicture}
    \circ_3
    \begin{tikzpicture}[scale=0.5,Centering]
        \node[CliquePoint](1)at(-0.71,-0.71){};
        \node[CliquePoint](2)at(-0.71,0.71){};
        \node[CliquePoint](3)at(0.71,0.71){};
        \node[CliquePoint](4)at(0.71,-0.71){};
        \draw[CliqueEmptyEdge](1)edge[]node[]{}(4);
        \draw[CliqueEmptyEdge](3)edge[]node[]{}(4);
        \draw[CliqueEdge](1)edge[]node[CliqueLabel]
            {\begin{math}\Dtt_2\end{math}}(2);
        \draw[CliqueEdge](2)edge[]node[CliqueLabel]
            {\begin{math}0\end{math}}(3);
    \end{tikzpicture}
    = 0.
\end{equation}
\end{subequations}
\end{multicols}
\medskip

Recall that a \Def{Dyck path} of \Def{size} $n$ is a word $u$ of
$\{\Att, \Btt\}^{2n}$ such that $|u|_\Att = |u|_\Btt$ and, for each
prefix $v$ of $u$, $|v|_\Att \geq |v|_\Btt$.
\medskip

\begin{Lemma} \label{lem:bijection_Inf_M_Dyck_paths}
    Let $\Mca$ be a finite unitary magma without nontrivial unit
    divisors. For all $n \geq 2$, the set of all $\Mca$-cliques of
    $\Inf\Mca(n)$ is in one-to-one correspondence with the set of all
    Dyck paths of size $n + 1$ wherein letters $\Att$ at even positions
    are colored on~$\bar{\Mca}$. Moreover, there is a correspondence
    between these two sets that sends any $\Mca$-clique of $\Inf\Mca(n)$
    with $k$ solid edges to a Dyck path with exactly $k$ letters $\Att$
    at even positions, for any $0 \leq k \leq n$.
\end{Lemma}
\begin{proof}
    In this proof, we denote by $\Att_c$ the letter $\Att$ of a Dyck
    path colored by $c \in \bar{\Mca}$. Given an $\Mca$-clique $\Pfr$ of
    $\Inf\Mca(n)$, we decorate each vertex $x$ of $\Pfr$ by
    \begin{enumerate}[fullwidth,label={(\it\arabic*)}]
        \item \label{item:bijection_Inf_M_Dyck_paths_1}
        $\Att\Att_c$ if $x$ has one outcoming arc and no incoming arc,
        where $c$ is the label of the outcoming arc from~$x$;
        \item \label{item:bijection_Inf_M_Dyck_paths_2}
        $\Btt\Btt$ if $x$ has no outcoming arc and one incoming arc;
        \item \label{item:bijection_Inf_M_Dyck_paths_3}
        $\Btt\Att_c$ if $x$ has both one outcoming arc and one incoming
        arc, where $c$ is the label of the outcoming arc from~$x$;
        \item \label{item:bijection_Inf_M_Dyck_paths_4}
        $\Att\Btt$ otherwise.
    \end{enumerate}
    Let $\phi$ be the map sending $\Pfr$ to the word obtained by
    concatenating the decorations of the vertices of $\Pfr$ thus
    described, read from $1$ to~$n + 1$.
    \smallskip

    We show that $\phi$ is a bijection between the two sets of the
    statement of the lemma. First, observe that since $\Pfr$ is
    inclusion-free, for each vertex $y$ of $\Pfr$, there is at most one
    incoming arc to $y$ and one outcoming arc from $y$. For this reason,
    for any vertex $y$ of $\Pfr$, the total number of incoming arcs to
    vertices $x \leq y$ of $\Pfr$ is smaller than or equal to the total
    number of outcoming arcs to vertices $x \leq y$ of $\Pfr$, and the
    total number of vertices having an incoming arc is equal to the
    total number of vertices having an outcoming arc in $\Pfr$. Thus, by
    forgetting the colorations of its letters, the word $\phi(\Pfr)$ is
    a Dyck path.
    \smallskip

    Besides, given a Dyck path $u$ of size $n + 1$ wherein letters
    $\Att$ at even positions are colored on $\bar{\Mca}$, one can build
    a unique $\Mca$-clique $\Pfr$ of $\Inf\Mca(n)$ such that
    $\phi(\Pfr) = u$. Indeed, by reading the letters of $u$ two by two,
    one knows the number of outcoming and incoming arcs for each vertex
    of $\Pfr$. Since $\Pfr$ is inclusion-free, there is one unique way
    to connect these vertices by solid diagonals without creating
    inclusions of arcs. Moreover,
    by~\ref{item:bijection_Inf_M_Dyck_paths_1},
    \ref{item:bijection_Inf_M_Dyck_paths_2},
    \ref{item:bijection_Inf_M_Dyck_paths_3},
    and~\ref{item:bijection_Inf_M_Dyck_paths_4}, the colors of the
    letters $\Att$ at even positions allow to label the solid arcs of
    $\Pfr$. Hence $\phi$ is a bijection as claimed.
    \smallskip

    Finally, by definition of $\phi$, we observe that if $\Pfr$ has
    exactly $k$ solid arcs, the Dyck path $\phi(\Pfr)$ has exactly $k$
    occurrences of the letter $\Att$ at even positions, whence the whole
    statement of the lemma.
\end{proof}
\medskip

Let $\Nar(n, k)$ be the \Def{Narayana number}~\cite{Nar55} defined
for all $0 \leq k \leq n - 2$ by
\begin{equation}
    \Nar(n, k) := \frac{1}{k + 1} \binom{n - 2}{k} \binom{n - 1}{k}.
\end{equation}
The number of Dyck paths of size $n - 1$ and exactly $k$ occurrences of
the factor $\Att \Btt$ is $\Nar(n, k)$. Equivalently, this is also the
number of binary trees with $n$ leaves and exactly $k$ internal nodes
having an internal node as a left child.
\medskip

\begin{Proposition} \label{prop:dimensions_Inf_M}
    Let $\Mca$ be a finite unitary magma without nontrivial unit
    divisors. For all $n \geq 2$,
    \begin{equation} \label{equ:dimensions_Inf_M}
        \dim \Inf\Mca(n) =
        \sum_{0 \leq k \leq n}
        (m - 1)^k \; \Nar(n + 2, k),
    \end{equation}
    where $m := \# \Mca$.
\end{Proposition}
\begin{proof}
    It is known from~\cite{Sul98} that the number of Dyck paths of size
    $n + 1$ with $k$ occurrences of the letter $\Att$ at even positions
    is the Narayana number $\Nar(n + 2, k)$. Hence, by using this
    property together with Lemma~\ref{lem:bijection_Inf_M_Dyck_paths},
    we obtain that the number of inclusion-free $\Mca$-cliques of size
    $n$ with $k$ solid arcs is $(m - 1)^k \; \Nar(n + 2, k)$. Therefore,
    since an inclusion-free $\Mca$-clique of arity $n$ can have at most
    $n$ solid arcs, \eqref{equ:dimensions_Inf_M} holds.
\end{proof}
\medskip

The skeletons of the $\Mca$-cliques of $\Inf\Mca$ of arities greater
than $1$ are the graphs such that, if $\{x, y\}$ and $\{x', y'\}$ are
two arcs such that $x \leq x' < y' \leq y$, then $x = x'$ and $y = y'$.
Therefore, $\Inf\Mca$ can be seen as an operad on such colored graphs,
where the arcs of the graphs have one color among the set $\bar{\Mca}$.
Equivalently, as Lemma~\ref{lem:bijection_Inf_M_Dyck_paths} shows,
$\Inf\Mca$ can be seen as an operad of Dyck paths where letters $\Att$
at even positions are colored on~$\bar{\Mca}$.
\medskip

By Proposition~\ref{prop:dimensions_Inf_M}, when $\# \Mca = 2$, the
dimensions of $\Inf\Mca$ begin by
\begin{equation}
    1, 5, 14, 42, 132, 429, 1430, 4862,
\end{equation}
and form, except for the first terms, Sequence~\OEIS{A000108}
of~\cite{Slo}. When $\# \Mca = 3$, the dimensions of
$\Inf\Mca$ begin by
\begin{equation}
    1, 11, 45, 197, 903, 4279, 20793, 103049,
\end{equation}
and form, except for the first terms, Sequence~\OEIS{A001003}
of~\cite{Slo}. When $\# \Mca = 4$, the dimensions of
$\Inf\Mca$ begin by
\begin{equation}
    1, 19, 100, 562, 3304, 20071, 124996, 793774,
\end{equation}
and form, except for the first terms, Sequence~\OEIS{A007564}
of~\cite{Slo}.
\medskip

\subsubsection{Acyclic decorated cliques}%
\label{subsubsec:quotient_Cli_M_acyclic}
Let $\Rel_{\Acy\Mca}$ be the subspace of $\Cli\Mca$ generated by all
$\Mca$-cliques that are not acyclic. As a quotient of graded vector
spaces,
\begin{equation}
    \Acy\Mca := \Cli\Mca/_{\Rel_{\Acy\Mca}}
\end{equation}
is the linear span of all acyclic $\Mca$-cliques.
\medskip

\begin{Proposition} \label{prop:quotient_Cli_M_acyclic}
    Let $\Mca$ be a unitary magma without nontrivial unit divisors.
    Then, the space $\Acy\Mca$ is a quotient operad of $\Cli\Mca$.
\end{Proposition}
\begin{proof}
    Since $\Mca$ has no nontrivial unit divisors, for any $\Mca$-cliques
    $\Pfr$ and $\Qfr$ of $\Cli\Mca$, each solid arc of $\Pfr$ (resp.
    $\Qfr$) gives rise to a solid arc in $\Pfr \circ_i \Qfr$, for any
    valid integer $i$. For this reason, if $\Pfr$ is an $\Mca$-clique of
    $\Rel_{\Acy\Mca}$, $\Pfr$ is not acyclic and each $\Mca$-clique
    obtained by a partial composition involving $\Pfr$ and other
    $\Mca$-cliques is still not acyclic and thus, belongs to
    $\Rel_{\Acy\Mca}$. This proves that $\Rel_{\Acy\Mca}$ is an operad
    ideal of $\Cli\Mca$ and implies the statement of the proposition.
\end{proof}
\medskip

For instance, in the operad $\Acy\Dbb_2$,
\vspace{-1.75em}
\begin{multicols}{2}
\begin{subequations}
\begin{equation}
    \begin{tikzpicture}[scale=0.7,Centering]
        \node[CliquePoint](1)at(-0.59,-0.81){};
        \node[CliquePoint](2)at(-0.95,0.31){};
        \node[CliquePoint](3)at(-0.00,1.00){};
        \node[CliquePoint](4)at(0.95,0.31){};
        \node[CliquePoint](5)at(0.59,-0.81){};
        \draw[CliqueEmptyEdge](1)edge[]node[]{}(5);
        \draw[CliqueEmptyEdge](2)edge[]node[]{}(3);
        \draw[CliqueEmptyEdge](3)edge[]node[]{}(4);
        \draw[CliqueEmptyEdge](4)edge[]node[]{}(5);
        \draw[CliqueEdge](1)edge[]node[CliqueLabel]
            {\begin{math}0\end{math}}(2);
        \draw[CliqueEdge](2)edge[bend left=20]node[CliqueLabel,near start]
            {\begin{math}0\end{math}}(5);
        \draw[CliqueEdge](3)edge[bend right=20]node[CliqueLabel,near start]
            {\begin{math}\Dtt_1\end{math}}(5);
        \draw[CliqueEdge](1)edge[]node[CliqueLabel,near end]
            {\begin{math}0\end{math}}(4);
    \end{tikzpicture}
    \circ_1
    \begin{tikzpicture}[scale=0.6,Centering]
        \node[CliquePoint](1)at(-0.71,-0.71){};
        \node[CliquePoint](2)at(-0.71,0.71){};
        \node[CliquePoint](3)at(0.71,0.71){};
        \node[CliquePoint](4)at(0.71,-0.71){};
        \draw[CliqueEmptyEdge](1)edge[]node[]{}(2);
        \draw[CliqueEmptyEdge](2)edge[]node[]{}(3);
        \draw[CliqueEmptyEdge](3)edge[]node[]{}(4);
        \draw[CliqueEdge](1)edge[]node[CliqueLabel]
            {\begin{math}\Dtt_1\end{math}}(3);
        \draw[CliqueEdge](1)edge[]node[CliqueLabel]
            {\begin{math}\Dtt_1\end{math}}(4);
    \end{tikzpicture}
    =
    \begin{tikzpicture}[scale=0.9,Centering]
        \node[CliquePoint](1)at(-0.43,-0.90){};
        \node[CliquePoint](2)at(-0.97,-0.22){};
        \node[CliquePoint](3)at(-0.78,0.62){};
        \node[CliquePoint](4)at(-0.00,1.00){};
        \node[CliquePoint](5)at(0.78,0.62){};
        \node[CliquePoint](6)at(0.97,-0.22){};
        \node[CliquePoint](7)at(0.43,-0.90){};
        \draw[CliqueEmptyEdge](1)edge[]node[]{}(2);
        \draw[CliqueEmptyEdge](1)edge[]node[]{}(7);
        \draw[CliqueEmptyEdge](2)edge[]node[]{}(3);
        \draw[CliqueEmptyEdge](3)edge[]node[]{}(4);
        \draw[CliqueEmptyEdge](4)edge[]node[]{}(5);
        \draw[CliqueEmptyEdge](5)edge[]node[]{}(6);
        \draw[CliqueEmptyEdge](6)edge[]node[]{}(7);
        \draw[CliqueEdge](1)edge[bend right=20]node[CliqueLabel,near end]
            {\begin{math}\Dtt_1\end{math}}(3);
        \draw[CliqueEdge](1)edge[]node[CliqueLabel]
            {\begin{math}0\end{math}}(4);
        \draw[CliqueEdge](1)edge[bend left=30]node[CliqueLabel,near start]
            {\begin{math}0\end{math}}(6);
        \draw[CliqueEdge](4)edge[]node[CliqueLabel,near start]
            {\begin{math}0\end{math}}(7);
        \draw[CliqueEdge](5)edge[bend right=20]node[CliqueLabel,near start]
            {\begin{math}\Dtt_1\end{math}}(7);
    \end{tikzpicture}\,,
\end{equation}

\begin{equation}
    \begin{tikzpicture}[scale=0.7,Centering]
        \node[CliquePoint](1)at(-0.59,-0.81){};
        \node[CliquePoint](2)at(-0.95,0.31){};
        \node[CliquePoint](3)at(-0.00,1.00){};
        \node[CliquePoint](4)at(0.95,0.31){};
        \node[CliquePoint](5)at(0.59,-0.81){};
        \draw[CliqueEmptyEdge](1)edge[]node[]{}(5);
        \draw[CliqueEmptyEdge](2)edge[]node[]{}(3);
        \draw[CliqueEmptyEdge](3)edge[]node[]{}(4);
        \draw[CliqueEmptyEdge](4)edge[]node[]{}(5);
        \draw[CliqueEdge](1)edge[]node[CliqueLabel]
            {\begin{math}0\end{math}}(2);
        \draw[CliqueEdge](2)edge[bend left=20]node[CliqueLabel,near start]
            {\begin{math}0\end{math}}(5);
        \draw[CliqueEdge](3)edge[bend right=20]node[CliqueLabel,near start]
            {\begin{math}\Dtt_1\end{math}}(5);
        \draw[CliqueEdge](1)edge[]node[CliqueLabel,near end]
            {\begin{math}0\end{math}}(4);
    \end{tikzpicture}
    \circ_3
    \begin{tikzpicture}[scale=0.6,Centering]
        \node[CliquePoint](1)at(-0.71,-0.71){};
        \node[CliquePoint](2)at(-0.71,0.71){};
        \node[CliquePoint](3)at(0.71,0.71){};
        \node[CliquePoint](4)at(0.71,-0.71){};
        \draw[CliqueEmptyEdge](1)edge[]node[]{}(2);
        \draw[CliqueEmptyEdge](2)edge[]node[]{}(3);
        \draw[CliqueEmptyEdge](3)edge[]node[]{}(4);
        \draw[CliqueEdge](1)edge[]node[CliqueLabel]
            {\begin{math}\Dtt_2\end{math}}(3);
        \draw[CliqueEdge](1)edge[]node[CliqueLabel]
            {\begin{math}\Dtt_1\end{math}}(4);
    \end{tikzpicture}
    = 0.
\end{equation}
\end{subequations}
\end{multicols}
\medskip

The skeletons of the $\Mca$-cliques of $\Acy\Mca$ of arities greater
than $1$ are acyclic graphs or equivalently, forest of non-rooted trees.
Therefore, $\Acy\Mca$ can be seen as an operad on colored forests of
trees, where the edges of the trees of the forests have one color among
the set $\bar{\Mca}$. When $\# \Mca = 2$, the dimensions of $\Acy\Mca$
begin by
\begin{equation}
    1, 7, 38, 291, 2932, 36961, 561948, 10026505,
\end{equation}
and form, except for the first terms, Sequence~\OEIS{A001858}
of~\cite{Slo}.
\medskip

\subsection{Secondary substructures}%
\label{subsec:secondary_substructures}
Some more substructures of $\Cli\Mca$ are constructed and briefly
studied here. They are constructed by mixing some of the constructions
of the seven main substructures of $\Cli\Mca$ defined in
Section~\ref{subsec:main_substructures} in the following sense.
\medskip

For any operad $\Oca$ and operad ideals $\Rel_1$ and $\Rel_2$ of $\Oca$,
the space $\Rel_1 + \Rel_2$ is still an operad ideal of $\Oca$, and
$\Oca/_{\Rel_1 + \Rel_2}$ is a quotient of both $\Oca/_{\Rel_1}$ and
$\Oca/_{\Rel_2}$. Moreover, if $\Oca'$ is a suboperad of $\Oca$ and
$\Rel$ is an operad ideal of $\Oca$, the space $\Rel \cap \Oca'$ is an
operad ideal of $\Oca'$, and $\Oca'/_{\Rel \cap \Oca'}$ is a quotient of
$\Oca'$ and a suboperad of $\Oca/_\Rel$. For these reasons
(straightforwardly provable), we can combine the constructions of the
previous section to build a bunch of new suboperads and quotients
of~$\Cli\Mca$.
\medskip

\subsubsection{Colored white noncrossing configurations}
When $\Mca$ is a unitary magma, let
\begin{equation}
    \WNC\Mca := \Whi\Mca/_{\Rel_{\Cro_0\Mca} \cap \Whi\Mca}.
\end{equation}
The $\Mca$-cliques of $\WNC\Mca$ are white noncrossing $\Mca$-cliques.
\medskip

\begin{Proposition} \label{prop:dimensions_WNC_M}
    Let $\Mca$ be a finite unitary magma. For all $n \geq 2$,
    \begin{equation}
        \dim \WNC\Mca(n) =
        \sum_{0 \leq k \leq n - 2}
        m^k (m - 1)^{n - k - 2} \; \Nar(n, k),
    \end{equation}
    where $m := \# \Mca$.
\end{Proposition}
\begin{proof}
    From its definition, $\WNC\Mca$ can be seen as the suboperad of
    $\Cro_0 \Mca$ restricted on the linear span of all white noncrossing
    $\Mca$-cliques. For this reason,
    \begin{equation}
        \dim \WNC\Mca(n) = \frac{1}{m^{n + 1}} \dim \Cro_0 \Mca(n).
    \end{equation}
    By using the upcoming Proposition~\ref{prop:dimensions_NC_M} for an
    expression for $\dim \Cro_0\Mca(n)$, we obtain the stated result.
\end{proof}
\medskip

When $\# \Mca = 2$, the dimensions of $\WNC\Mca$ begin by
\begin{equation}
    1, 1, 3, 11, 45, 197, 903, 4279,
\end{equation}
and form Sequence~\OEIS{A001003} of~\cite{Slo}. When $\# \Mca = 3$, the
dimensions of $\WNC\Mca$ begin by
\begin{equation}
    1, 1, 5, 31, 215, 1597, 12425, 99955,
\end{equation}
and form Sequence~\OEIS{A269730} of~\cite{Slo}. Observe that these
dimensions are shifted versions the ones of the
$\gamma$-polytridendriform operads $\TDendr_\gamma$~\cite{Gir16} with
$\gamma := \# \Mca - 1$.
\medskip

\subsubsection{Colored forests of paths}
When $\Mca$ is a unitary magma without nontrivial unit divisors, let
\begin{equation}
    \Paths\Mca := \Cli\Mca/_{\Rel_{\Deg_2\Mca} + \Rel_{\Acy\Mca}}.
\end{equation}
The skeletons of the $\Mca$-cliques of $\Paths\Mca$ are forests of
non-rooted trees that are paths. Therefore, $\Paths\Mca$ can be seen as
an operad on colored such graphs, where the arcs of the graphs have one
color among the set~$\bar{\Mca}$.
\medskip

When $\# \Mca = 2$,
the dimensions of $\Paths\Mca$ begin by
\begin{equation}
    1, 7, 34, 206, 1486, 12412, 117692, 1248004,
\end{equation}
an form, except for the first terms, Sequence~\OEIS{A011800}
of~\cite{Slo}.
\medskip

\subsubsection{Colored forests}
When $\Mca$ is a unitary magma without nontrivial unit divisors, let
\begin{equation}
    \Forests\Mca := \Cli\Mca/_{\Rel_{\Cro_0\Mca} + \Rel_{\Acy\Mca}}.
\end{equation}
The skeletons of the $\Mca$-cliques of $\Forests\Mca$ are forests of
rooted trees having no arcs $\{x, y\}$ and $\{x', y'\}$ satisfying
$x < x' < y < y'$. Therefore, $\Forests\Mca$ can be seen as an operad
on such colored forests, where the edges of the forests have one color among
the set $\bar{\Mca}$. When $\# \Mca = 2$, the dimensions of
$\Forests\Mca$ begin by
\begin{equation}
    1, 7, 33, 81, 1083, 6854, 45111, 305629,
\end{equation}
and form, except for the first terms, Sequence~\OEIS{A054727},
of~\cite{Slo}.
\medskip

\subsubsection{Colored Motzkin configurations}%
\label{subsubsec:Motzkin_configurations}
When $\Mca$ is a unitary magma without nontrivial unit divisors, let
\begin{equation}
    \Motzkin\Mca := \Cli\Mca/_{\Rel_{\Cro_0\Mca} + \Rel_{\Deg_1\Mca}}.
\end{equation}
The skeletons of the $\Mca$-cliques of $\Motzkin\Mca$ are configurations
of non-intersecting chords on a circle. Equivalently, these objects are
graphs of involutions (see
Section~\ref{subsubsec:quotient_Cli_M_degrees}) having no arcs
$\{x, y\}$ and $\{x', y'\}$ satisfying $x < x' < y < y'$. These objects
are enumerated by Motzkin numbers~\cite{Mot48}. Therefore,
$\Motzkin\Mca$ can be seen as an operad on such colored graphs, where
the arcs of the graphs have one color among the set $\bar{\Mca}$. When
$\# \Mca = 2$, the dimensions of $\Motzkin\Mca$ begin by
\begin{equation}
    1, 4, 9, 21, 51, 127, 323, 835,
\end{equation}
and form, except for the first terms, Sequence~\OEIS{A001006},
of~\cite{Slo}.
\medskip

\subsubsection{Colored dissections of polygons}
When $\Mca$ is a unitary magma without nontrivial unit divisors, let
\begin{equation}
    \Diss\Mca := \Whi\Mca/_{(\Rel_{\Cro_0\Mca} + \Rel_{\Deg_1\Mca})
    \cap \Whi\Mca}.
\end{equation}
The skeletons of the $\Mca$-cliques of $\Diss\Mca$ are \Def{strict
dissections of polygons}, that are graphs of Motzkin configurations
with no arcs of the form $\{x, x + 1\}$ or $\{1, n + 1\}$, where $n + 1$
is the number of vertices of the graphs. Therefore, $\Diss\Mca$ can be
seen as an operad on such colored graphs, where the arcs of the graphs
have one color among the set $\bar{\Mca}$. When $\# \Mca = 2$, the
dimensions of $\Diss\Mca$ begin by
\begin{equation}
    1, 1, 3, 6, 13, 29, 65, 148,
\end{equation}
and form, except for the first terms, Sequence~\OEIS{A093128}
of~\cite{Slo}.
\medskip

\subsubsection{Colored Lucas configurations}
When $\Mca$ is a unitary magma without nontrivial unit divisors, let
\begin{equation}
    \Luc\Mca := \Cli\Mca/_{\Rel_{\Bub\Mca} + \Rel_{\Deg_1\Mca}}.
\end{equation}
The skeletons of the $\Mca$-cliques of $\Luc\Mca$ are graphs such that
all vertices are of degrees at most $1$ and all arcs are of the form
$\{x, x + 1\}$ or $\{1, n + 1\}$, where $n + 1$ is the number of
vertices of the graphs. Therefore, $\Luc\Mca$ can be seen as an operad
on such colored graphs, where the arcs of the graphs have one color
among the set $\bar{\Mca}$. When $\# \Mca = 2$, the dimensions of
$\Luc\Mca$ begin by
\begin{equation}
    1, 4, 7, 11, 18, 29, 47, 76,
\end{equation}
and form, except for the first terms, Sequence~\OEIS{A000032}
of~\cite{Slo}.
\medskip

\subsection{Relations between substructures}
The suboperads and quotients of $\Cli\Mca$ constructed in
Sections~\ref{subsec:main_substructures}
and~\ref{subsec:secondary_substructures} are linked by injective or
surjective operad morphisms. To establish these, we begin with the
following lemma.
\medskip

\begin{Lemma} \label{lem:inclusion_families_cliques}
    Let $\Mca$ be a unitary magma. Then,
    \begin{enumerate}[fullwidth,label={(\it\roman*)}]
        \item \label{item:inclusion_families_cliques_1}
        the space $\Rel_{\Acy\Mca}$ is a subspace of $\Rel_{\Deg_1\Mca}$;
        \item \label{item:inclusion_families_cliques_2}
        the spaces $\Rel_{\Inf\Mca}$ and $\Rel_{\Bub\Mca}$ are subspaces
        of $\Rel_{\Deg_0\Mca}$;
        \item \label{item:inclusion_families_cliques_3}
        the spaces $\Rel_{\Cro_0\Mca}$ and $\Rel_{\Deg_2\Mca}$ are
        subspaces of $\Rel_{\Bub\Mca}$;
        \item \label{item:inclusion_families_cliques_4}
        the spaces $\Rel_{\Deg_2\Mca}$ and $\Rel_{\Acy\Mca}$ are
        subspaces of $\Rel_{\Inf\Mca}$.
    \end{enumerate}
\end{Lemma}
\begin{proof}
    All the spaces appearing in the statement of the lemma are subspaces
    of $\Cli\Mca$ generated by some subfamilies of $\Mca$-cliques.
    Therefore, to prove the assertions of the lemma, we shall prove
    inclusions of adequate subfamilies of such objects.
    \smallskip

    If $\Pfr$ is an $\Mca$-clique of $\Rel_{\Acy\Mca}$, by definition,
    $\Pfr$ has a cycle formed by solid arcs. Hence, $\Pfr$ has in
    particular a solid arc and a vertex of degree $2$ or more. For this
    reason, since $\Rel_{\Deg_1\Mca}$ is the linear span of all
    $\Mca$-cliques of degrees $2$ or more, $\Pfr$ is in
    $\Rel_{\Deg_1\Mca}$. This
    implies~\ref{item:inclusion_families_cliques_1}.
    \smallskip

    If $\Pfr$ is an $\Mca$-clique of $\Rel_{\Inf\Mca}$ or
    $\Rel_{\Bub\Mca}$, by definition, $\Pfr$ has in particular a solid
    arc. Hence, since $\Rel_{\Deg_0\Mca}$ is the linear span of all
    $\Mca$-cliques with at least one vertex with a positive degree,
    $\Pfr$ is in $\Rel_{\Deg_0\Mca}$. This
    implies~\ref{item:inclusion_families_cliques_2}.
    \smallskip

    If $\Pfr$ is an $\Mca$-clique of $\Rel_{\Cro_0\Mca}$ or
    $\Rel_{\Deg_2\Mca}$, $\Pfr$ has in particular a solid diagonal.
    Indeed, when $\Pfr$ is in $\Rel_{\Cro_0\Mca}$ this property is
    immediate. When $\Pfr$ is in $\Rel_{\Deg_2\Mca}$, since $\Pfr$ has a
    vertex $x$ of degree $3$ or more, the skeleton of $\Pfr$ has three
    arcs $\{x, y_1\}$, $\{x, y_2\}$, and $\{x, y_3\}$ with
    $y_i \ne x - 1$, $y_i \ne x + 1$, and $y_i \ne |\Pfr| + 1$ for at
    least one $i \in [3]$, so that the arc
    $(\min\{x, y_i\}, \max\{x, y_i\})$ is a solid diagonal of $\Pfr$. For
    this reason, since $\Rel_{\Bub\Mca}$ is the linear span of all
    $\Mca$-cliques with at least one solid diagonal, $\Pfr$ is in
    $\Rel_{\Bub\Mca}$. This
    implies~\ref{item:inclusion_families_cliques_3}.
    \smallskip

    If $\Pfr$ is an $\Mca$-clique of $\Rel_{\Deg_2\Mca}$ or
    $\Rel_{\Acy\Mca}$, $\Pfr$ has in particular a solid arc included in
    another one. Indeed, when $\Pfr$ is in $\Rel_{\Deg_2\Mca}$, since
    $\Pfr$ has a vertex $x$ of a degree $3$ or more, the skeleton
    of $\Pfr$ has three arcs $\{x, y_1\}$, $\{x, y_2\}$, and
    $\{x, y_3\}$. One can check that for all relative orders between
    the vertices $x$, $y_1$, $y_2$, and $y_3$, one of these arcs
    includes another one, so that $\Pfr$ is not inclusion-free. When
    $\Pfr$ is in $\Rel_{\Acy\Mca}$, $\Pfr$ contains a cycle formed by
    solid arcs. Let $x_1$, $x_2$, \dots, $x_k$, $k \geq 3$, be the
    vertices of $\Pfr$ that form this cycle. We can assume without loss
    of generality that $x_1 \leq x_i$ for all $i \in [k]$ and thus, that
    $(x_1, x_2)$ and $(x_1, x_k)$ are solid arcs of $\Pfr$ being part of
    the cycle. Then, when $x_2 < x_k$, since
    $x_1 \leq x_1 < x_2 \leq x_k$, the arc $(x_1, x_k)$ includes
    $(x_1, x_2)$. Otherwise, $x_k < x_2$, and since
    $x_1 \leq x_1 < x_k \leq x_2$, the arc $(x_1, x_2)$ includes
    $(x_1, x_k)$. For these reasons, since $\Rel_{\Inf\Mca}$ is the
    linear span of all $\Mca$-cliques that are non inclusion-free,
    $\Pfr$ is in $\Rel_{\Inf\Mca}$. This
    implies~\ref{item:inclusion_families_cliques_4}.
\end{proof}
\medskip

\subsubsection{Relations between the main substructures}
Let us list and explain the morphisms between the main substructures of
$\Cli\Mca$. First, Lemma~\ref{lem:inclusion_families_cliques} implies
that there are surjective operad morphisms from $\Acy\Mca$ to
$\Deg_1\Mca$, from $\Inf\Mca$ to $\Deg_0\Mca$, from $\Bub\Mca$ to
$\Deg_0\Mca$, from $\Cro_0\Mca$ to $\Bub\Mca$, from $\Deg_2\Mca$ to
$\Bub\Mca$, from $\Deg_2\Mca$ to $\Inf\Mca$, and from $\Acy\Mca$
to~$\Inf\Mca$. Second, when $B$, $E$, and $D$ are subsets of $\Mca$ such
that $\Unit_\Mca \in B$, $\Unit_\Mca \in E$, and $E \Op B \subseteq D$,
$\Whi\Mca$ is a suboperad of $\Lab_{B,E,D}\Mca$. Finally, there is a
surjective operad morphism from $\Whi\Mca$ to the associative operad
$\As$ sending any $\Mca$-clique $\Pfr$ of $\Whi\Mca$ to the unique basis
element of $\As$ of the same arity as the one of~$\Pfr$. The relations
between the main suboperads and quotients of $\Cli\Mca$ built here are
summarized in the diagram of operad morphisms of
Figure~\ref{fig:diagram_main_operads}.
\begin{figure}[ht]
    \centering
    \scalebox{.73}{
    \begin{tikzpicture}[xscale=1.2,yscale=1.1,Centering]
        \node[text=BrickRed](CliM)at(8,10)
            {\begin{math}\Cli\Mca\end{math}};
        \node[text=RoyalPurple](AcyM)at(4,8)
            {\begin{math}\Acy\Mca\end{math}};
        \node[text=RoyalPurple](DegkM)at(6,8)
            {\begin{math}\Deg_k\Mca\end{math}};
        \node[text=RoyalBlue](CrokM)at(10,8)
            {\begin{math}\Cro_k\Mca\end{math}};
        \node[text=RoyalBlue](LabM)at(12,8)
            {\begin{math}\Lab_{B, E, D}\Mca\end{math}};
        \node[text=RoyalPurple](Deg2M)at(6,6)
            {\begin{math}\Deg_2\Mca\end{math}};
        \node[text=RoyalBlue](Cro0M)at(10,6)
            {\begin{math}\Cro_0\Mca\end{math}};
        \node[text=RoyalPurple](InfM)at(4,4)
            {\begin{math}\Inf\Mca\end{math}};
        \node[text=RoyalPurple](Deg1M)at(6,4)
            {\begin{math}\Deg_1\Mca\end{math}};
        \node[text=RoyalBlue](BubM)at(8,4)
            {\begin{math}\Bub\Mca\end{math}};
        \node[text=RoyalBlue](WhiM)at(12,4)
            {\begin{math}\Whi\Mca\end{math}};
        \node[text=RoyalPurple](Deg0M)at(8,2)
            {\begin{math}\Deg_0\Mca\end{math}};
        \draw[Surjection](CliM)--(AcyM);
        \draw[Surjection](CliM)--(DegkM);
        \draw[Surjection](CliM)to[bend right=15](CrokM);
        \draw[Injection](CrokM)to[bend right=15](CliM);
        \draw[Injection](LabM)--(CliM);
        \draw[Surjection](AcyM)--(InfM);
        \draw[Surjection](AcyM)--(Deg1M);
        \draw[Surjection](DegkM)--(Deg2M);
        \draw[Surjection](CrokM)to[bend right=15](Cro0M);
        \draw[Injection](Cro0M)to[bend right=15](CrokM);
        \draw[Injection](WhiM)--(LabM);
        \draw[Surjection](Deg2M)--(InfM);
        \draw[Surjection](Deg2M)--(Deg1M);
        \draw[Surjection](Deg2M)--(BubM);
        \draw[Surjection](Cro0M)--(BubM);
        \draw[Surjection](InfM)--(Deg0M);
        \draw[Surjection](Deg1M)--(Deg0M);
        \draw[Surjection](BubM)--(Deg0M);
        \draw[Surjection](WhiM)--(Deg0M);
    \end{tikzpicture}}
    \vspace{-.5em}
    \caption{\footnotesize
    The diagram of the main suboperads and quotients of $\Cli\Mca$.
    Arrows~$\rightarrowtail$ (resp.~$\twoheadrightarrow$) are injective
    (resp. surjective) operad morphisms. Here, $\Mca$ is a unitary magma
    without nontrivial unit divisors, $k$ is a positive integer, and
    $B$, $E$, and $D$ are subsets of $\Mca$ such that
    $\Unit_\Mca \in B$, $\Unit_\Mca \in E$, and $E \Op B \subseteq D$.}
    \label{fig:diagram_main_operads}
\end{figure}
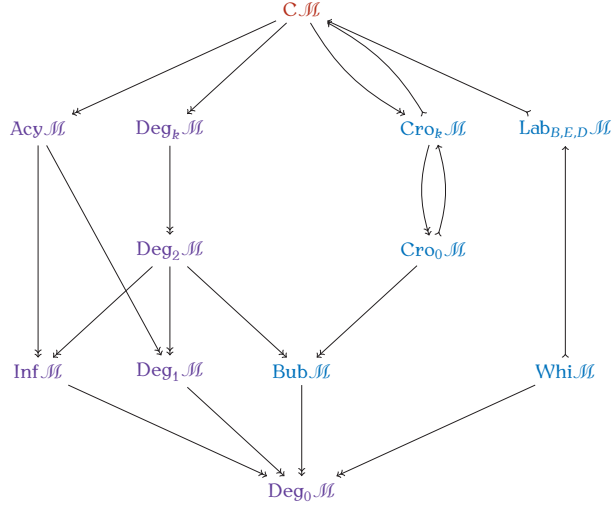
\medskip

\subsubsection{Relations between the secondary and main substructures}
Let us now list and explain the morphisms between the secondary and main
substructures of $\Cli\Mca$. First, immediately from their definitions,
$\WNC\Mca$ is a suboperad of $\Cro_0\Mca$ and a quotient of~$\Whi\Mca$,
$\Paths\Mca$ is both a quotient of $\Deg_2\Mca$ and $\Acy\Mca$,
$\Forests\Mca$ is both a quotient of $\Cro_0\Mca$ and $\Acy\Mca$,
$\Motzkin\Mca$ is both a quotient of $\Cro_0\Mca$ and $\Deg_1\Mca$,
$\Diss\Mca$ is a suboperad of $\Motzkin\Mca$ and a quotient of
$\WNC\Mca$, and $\Luc\Mca$ is both a quotient of $\Bub\Mca$ and
$\Deg_1\Mca$. Moreover, since by
Lemma~\ref{lem:inclusion_families_cliques}, $\Rel_{\Acy\Mca}$ is a
subspace of $\Rel_{\Deg_1\Mca}$, $\Rel_{\Deg_2\Mca}$ and
$\Rel_{\Acy\Mca}$ are subspaces of $\Rel_{\Inf\Mca}$, and
$\Rel_{\Cro_0\Mca}$ is a subspace of $\Rel_{\Bub\Mca}$, we respectively
have that $\Rel_{\Deg_2\Mca} + \Rel_{\Acy\Mca}$ is a subspace of both
$\Rel_{\Deg_1\Mca}$ and $\Rel_{\Inf\Mca}$,
$\Rel_{\Cro_0\Mca} + \Rel_{\Acy\Mca}$ is a subspace of
$\Rel_{\Cro_0\Mca} + \Rel_{\Deg_1\Mca}$, and
$\Rel_{\Cro_0\Mca} + \Rel_{\Deg_1\Mca}$ is a subspace of
$\Rel_{\Bub\Mca} + \Rel_{\Deg_1\Mca}$. For these reasons, there are
surjective operad morphisms from $\Paths\Mca$ to $\Deg_1\Mca$, from
$\Paths\Mca$ to~$\Inf\Mca$, from $\Forests\Mca$ to~$\Motzkin\Mca$, and
from $\Motzkin\Mca$ to~$\Luc\Mca$. The relations between the secondary
suboperads and quotients of $\Cli\Mca$ built here are summarized in the
diagram of operad morphisms of
Figure~\ref{fig:diagram_secondary_operads}.
\begin{figure}[ht]
    \centering
    \scalebox{.73}{
    \begin{tikzpicture}[xscale=1.45,yscale=1.1,Centering]
        \node[text=BrickRed](CliM)at(13,3)
            {\begin{math}\Cli\Mca\end{math}};
        \node[text=RoyalPurple](Deg0M)at(14,-8)
            {\begin{math}\Deg_0\Mca\end{math}};
        \node[text=RoyalBlue](Cro0M)at(12,0)
            {\begin{math}\Cro_0\Mca\end{math}};
        \node[text=RoyalBlue](WhiM)at(10,0)
            {\begin{math}\Whi\Mca\end{math}};
        \node[text=RoyalPurple](AcyM)at(14,0)
            {\begin{math}\Acy\Mca\end{math}};
        \node[text=RoyalPurple](Deg2M)at(16,0)
            {\begin{math}\Deg_2\Mca\end{math}};
        \node[text=ForestGreen](WNCM)at(10,-2)
            {\begin{math}\WNC\Mca\end{math}};
        \node[text=ForestGreen](ForM)at(13,-2)
            {\begin{math}\Forests\Mca\end{math}};
        \node[text=ForestGreen](PatM)at(15,-2)
            {\begin{math}\Paths\Mca\end{math}};
        \node[text=RoyalBlue](BubM)at(17,-2)
            {\begin{math}\Bub\Mca\end{math}};
        \node[text=RoyalPurple](Deg1M)at(14,-4)
            {\begin{math}\Deg_1\Mca\end{math}};
        \node[text=RoyalPurple](InfM)at(16,-4)
            {\begin{math}\Inf\Mca\end{math}};
        \node[text=ForestGreen](MotM)at(13,-5)
            {\begin{math}\Motzkin\Mca\end{math}};
        \node[text=ForestGreen](DisM)at(10,-7)
            {\begin{math}\Diss\Mca\end{math}};
        \node[text=ForestGreen](LucM)at(17,-6)
            {\begin{math}\Luc\Mca\end{math}};
        \draw[Injection](WNCM)--(Cro0M);
        \draw[Injection](DisM)--(MotM);
        \draw[Surjection](Cro0M)--(ForM);
        \draw[Surjection](Deg2M)--(BubM);
        \draw[Surjection](Deg2M)--(PatM);
        \draw[Surjection](AcyM)--(PatM);
        \draw[Surjection](AcyM)--(ForM);
        \draw[Surjection](PatM)--(InfM);
        \draw[Surjection](PatM)--(Deg1M);
        \draw[Surjection](ForM)--(MotM);
        \draw[Surjection](MotM)--(LucM);
        \draw[Surjection](WNCM)--(DisM);
        \draw[Surjection](WhiM)--(WNCM);
        \draw[Surjection](Deg1M)--(MotM);
        \draw[Surjection](BubM)--(LucM);
        \draw[Surjection](DisM)--(Deg0M);
        \draw[Surjection](LucM)--(Deg0M);
        \draw[Injection](WhiM)--(CliM);
        \draw[Injection](Cro0M)to[bend right=15](CliM);
        \draw[Surjection](CliM)to[bend right=15](Cro0M);
        \draw[Surjection](CliM)--(AcyM);
        \draw[Surjection](CliM)--(Deg2M);
        \draw[Surjection](InfM)--(Deg0M);
    \end{tikzpicture}}
    \vspace{-.5em}
    \caption{\footnotesize
    The diagram of the secondary suboperads and quotients of $\Cli\Mca$
    together with some of their related main suboperads and quotients of
    $\Cli\Mca$. Arrows~$\rightarrowtail$ (resp.~$\twoheadrightarrow$)
    are injective (resp. surjective) operad morphisms. Here, $\Mca$ is a
    unitary magma without nontrival unit divisors.}
    \label{fig:diagram_secondary_operads}
\end{figure}
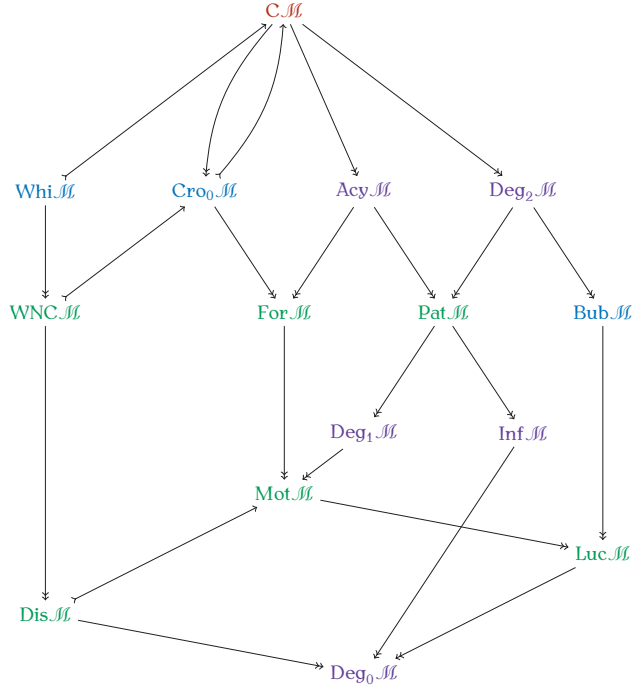
\medskip

\section{Operads of noncrossing decorated cliques}%
\label{sec:operad_noncrossing}
We perform here a complete study of the suboperad $\Cro_0\Mca$ of
noncrossing $\Mca$-cliques defined in
Section~\ref{subsubsec:quotient_Cli_M_crossings}. For simplicity, this
operad is denoted in the sequel as $\NC\Mca$ and named as the
\Def{noncrossing $\Mca$-clique operad}. The process giving from any
unitary magma $\Mca$ the operad $\NC\Mca$ is called the
\Def{noncrossing clique construction}.
\medskip

\subsection{General properties}
To study $\NC\Mca$, we begin by establishing the fact that $\NC\Mca$
inherits from some properties of~$\Cli\Mca$. Then, we shall describe a
realization of $\NC\Mca$ in terms of decorated Schröder trees, compute a
minimal generating set of $\NC\Mca$, and compute its dimensions.
\medskip

First of all, we call \Def{fundamental basis} of $\NC\Mca$ the
fundamental basis of $\Cli\Mca$ restricted on noncrossing
$\Mca$-cliques. By definition of $\NC\Mca$ and by
Proposition~\ref{prop:quotient_Cli_M_crossings}, the partial composition
$\Pfr \circ_i \Qfr$ of two noncrossing $\Mca$-cliques $\Pfr$ and $\Qfr$
in $\NC\Mca$ is equal to the partial composition $\Pfr \circ_i \Qfr$ in
$\Cli\Mca$. Therefore, the fundamental basis of $\NC\Mca$ is a
set-operad basis.
\medskip

\subsubsection{First properties}

\begin{Proposition} \label{prop:inherited_properties_NC_M}
    Let $\Mca$ be a unitary magma. Then,
    \begin{enumerate}[fullwidth,label={(\it\roman*)}]
        \item \label{item:inherited_properties_NC_M_1}
        the associative elements of $\NC\Mca$ are the ones of $\Cli\Mca$;
        \item \label{item:inherited_properties_NC_M_2}
        the group of symmetries of $\NC\Mca$ contains the map
        $\Returned$ (defined by~\eqref{equ:returned_map_Cli_M}) and all
        the maps $\Cli\theta$ where $\theta$ are unitary magma
        automorphisms of $\Mca$;
        \item \label{item:inherited_properties_NC_M_3}
        the fundamental basis of $\NC\Mca$ is a basic set-operad basis
        if and only if $\Mca$ is right cancellable;
        \item \label{item:inherited_properties_NC_M_4}
        the map $\rho$ (defined by~\eqref{equ:rotation_map_Cli_M}) is a
        rotation map of $\NC\Mca$ endowing it with a cyclic operad
        structure.
    \end{enumerate}
\end{Proposition}
\begin{proof}
    First, since by Proposition~\ref{prop:quotient_Cli_M_crossings}
    $\NC\Mca$ is a suboperad of $\Cli\Mca$, each associative element
    of $\NC\Mca$ is an associative element of $\Cli\Mca$. Moreover,
    since all $\Mca$-bubbles are in $\NC\Mca$ and, by
    Proposition~\ref{prop:associative_elements_Cli_M}, all associative
    elements of $\Cli\Mca$ are linear combinations of $\Mca$-bubbles,
    each associative element of $\Cli\Mca$ belongs to $\NC\Mca$.
    Whence~\ref{item:inherited_properties_NC_M_1}. Besides, since for
    any noncrossing $\Mca$-clique $\Pfr$, $\Returned(\Pfr)$ (resp.
    $\rho(\Pfr)$) is still noncrossing, by
    Proposition~\ref{prop:symmetries_Cli_M} (resp.
    Proposition~\ref{prop:cyclic_Cli_M}),
    \ref{item:inherited_properties_NC_M_2} (resp.
    \ref{item:inherited_properties_NC_M_4}) holds. Finally, since again
    by Proposition~\ref{prop:quotient_Cli_M_crossings}, $\NC\Mca$ is a
    suboperad of $\Cli\Mca$, Proposition~\ref{prop:basic_Cli_M} and
    the fact that $\NC\Mca(2) = \Cli\Mca(2)$
    imply~\ref{item:inherited_properties_NC_M_3}.
\end{proof}
\medskip

\subsubsection{Treelike expressions on bubbles}
\label{subsubsec:treelike_bubbles}
Let $\Pfr$ be a noncrossing $\Mca$-clique or arity $n \geq 2$, and
$(x, y)$ be a diagonal or the base of $\Pfr$. Consider the path
$(x = z_1, z_2, \dots, z_k, z_{k + 1} = y)$ in $\Pfr$ such that
$k \geq 2$, for all $i \in [k + 1]$, $x \leq z_i \leq y$, and for all
$i \in [k]$, $z_{i + 1}$ is the greatest vertex of $\Pfr$ so that
$(z_i, z_{i + 1})$ is a solid diagonal or a (non-necessarily solid) edge
of $\Pfr$. The \Def{area} of $\Pfr$ adjacent to $(x, y)$ is the
$\Mca$-bubble $\Qfr$ of arity $k$ whose base is labeled by $\Pfr(x, y)$
and $\Qfr_i = \Pfr(z_i, z_{i + 1})$ for all $i \in [k]$. From a
geometric point of view, $\Qfr$ is the unique maximal component of
$\Pfr$ adjacent to the arc $(x, y)$, without solid diagonals, and
bounded by solid diagonals or edges of~$\Pfr$. For instance, for the
noncrossing $\Z$-clique
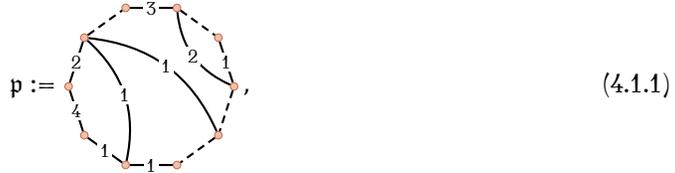
\begin{equation}
    \Pfr :=
    \begin{tikzpicture}[scale=1.1,Centering]
        \node[CliquePoint](1)at(-0.31,-0.95){};
        \node[CliquePoint](2)at(-0.81,-0.59){};
        \node[CliquePoint](3)at(-1.00,-0.00){};
        \node[CliquePoint](4)at(-0.81,0.59){};
        \node[CliquePoint](5)at(-0.31,0.95){};
        \node[CliquePoint](6)at(0.31,0.95){};
        \node[CliquePoint](7)at(0.81,0.59){};
        \node[CliquePoint](8)at(1.00,0.00){};
        \node[CliquePoint](9)at(0.81,-0.59){};
        \node[CliquePoint](10)at(0.31,-0.95){};
        \draw[CliqueEdge](1)edge[]node[CliqueLabel]
            {\begin{math}1\end{math}}(2);
        \draw[CliqueEdge](1)edge[]node[CliqueLabel]
            {\begin{math}1\end{math}}(10);
        \draw[CliqueEdge](2)edge[]node[CliqueLabel]
            {\begin{math}4\end{math}}(3);
        \draw[CliqueEdge](1)edge[bend right=30]node[CliqueLabel]
            {\begin{math}1\end{math}}(4);
        \draw[CliqueEdge](3)edge[]node[CliqueLabel]
            {\begin{math}2\end{math}}(4);
        \draw[CliqueEmptyEdge](4)edge[]node[]{}(5);
        \draw[CliqueEdge](5)edge[]node[CliqueLabel]
            {\begin{math}3\end{math}}(6);
        \draw[CliqueEdge](4)edge[bend left=30]node[CliqueLabel]
            {\begin{math}1\end{math}}(9);
        \draw[CliqueEmptyEdge](6)edge[]node[]{}(7);
        \draw[CliqueEdge](6)edge[bend right=30]node[CliqueLabel]
            {\begin{math}2\end{math}}(8);
        \draw[CliqueEdge](7)edge[]node[CliqueLabel]
            {\begin{math}1\end{math}}(8);
        \draw[CliqueEmptyEdge](8)edge[]node[]{}(9);
        \draw[CliqueEmptyEdge](9)edge[]node[]{}(10);
    \end{tikzpicture}\,,
\end{equation}
the path associated with the diagonal $(4, 9)$ of $\Pfr$ is
$(4, 5, 6, 8, 9)$. For this reason, the area of $\Pfr$ adjacent to
$(4, 9)$ is the $\Z$-bubble
\begin{equation}
    \begin{tikzpicture}[scale=.6,Centering]
        \node[CliquePoint](1)at(-0.59,-0.81){};
        \node[CliquePoint](2)at(-0.95,0.31){};
        \node[CliquePoint](3)at(-0.00,1.00){};
        \node[CliquePoint](4)at(0.95,0.31){};
        \node[CliquePoint](5)at(0.59,-0.81){};
        \draw[CliqueEmptyEdge](1)edge[]node[]{}(2);
        \draw[CliqueEdge](1)edge[]node[CliqueLabel]
            {\begin{math}1\end{math}}(5);
        \draw[CliqueEdge](2)edge[]node[CliqueLabel]
            {\begin{math}3\end{math}}(3);
        \draw[CliqueEdge](3)edge[]node[CliqueLabel]
            {\begin{math}2\end{math}}(4);
        \draw[CliqueEmptyEdge](4)edge[]node[]{}(5);
    \end{tikzpicture}\,.
\end{equation}
\medskip

\begin{Proposition} \label{prop:unique_decomposition_NC_M}
    Let $\Mca$ be a unitary magma and $\Pfr$ be a noncrossing
    $\Mca$-clique of arity greater than $1$. Then, there is a unique
    $\Mca$-bubble $\Qfr$ with a maximal arity $k \geq 2$ such that
    $\Pfr = \Qfr \circ [\Rfr_1, \dots, \Rfr_k]$, where each $\Rfr_i$,
    $i \in [k]$, is a noncrossing $\Mca$-clique with a base labeled
    by~$\Unit_\Mca$.
\end{Proposition}
\begin{proof}
    Let $\Qfr'$ be the area of $\Pfr$ adjacent to its base and $k'$ be
    the arity of $\Qfr'$. By definition of the partial composition of
    $\NC\Mca$, for all $\Mca$-cliques $\Ufr$, $\Ufr'$, $\Ufr_1$, and
    $\Ufr_2$, if $\Ufr = \Ufr' \circ_i \Ufr_1 = \Ufr' \circ_i \Ufr_2$
    and $\Ufr_1$ and $\Ufr_2$ have bases labeled by $\Unit_\Mca$, then
    $\Ufr_1 = \Ufr_2$. This implies in particular that there are unique
    noncrossing $\Mca$-cliques $\Rfr'_i$, $i \in [k']$, with bases
    labeled by $\Unit_\Mca$ such that
    $\Pfr = \Qfr' \circ \left[\Rfr'_1, \dots, \Rfr'_{k'}\right]$.
    Finally, the fact that $\Qfr'$ is the area of $\Pfr$ adjacent to its
    base implies the maximality for the arity of $\Qfr'$. The statement
    of the proposition follows.
\end{proof}
\medskip

Consider the map
\begin{equation}
    \BubbleTree :
    \NC\Mca \to \Free\left(\Vect\left(\Bubbles_\Mca\right)\right)
\end{equation}
defined linearly an recursively by $\BubbleTree(\UnitClique) := \Leaf$
and, for any noncrossing $\Mca$-clique $\Pfr$ of arity greater than~$1$,
by
\begin{equation}
    \BubbleTree(\Pfr) :=
    \Corolla(\Qfr) \circ
    \left[\BubbleTree(\Rfr_1), \dots, \BubbleTree(\Rfr_k)\right],
\end{equation}
where $\Pfr = \Qfr \circ [\Rfr_1, \dots, \Rfr_k]$ is the unique
decomposition of $\Pfr$ stated in
Proposition~\ref{prop:unique_decomposition_NC_M}. We call
$\BubbleTree(\Pfr)$ the \Def{bubble tree} of~$\Pfr$. For instance,
in~$\NC\Z$,
\begin{equation} \label{equ:bubble_tree_example}
    \begin{tikzpicture}[scale=1.3,Centering]
        \node[CliquePoint](1)at(-0.31,-0.95){};
        \node[CliquePoint](2)at(-0.81,-0.59){};
        \node[CliquePoint](3)at(-1.00,-0.00){};
        \node[CliquePoint](4)at(-0.81,0.59){};
        \node[CliquePoint](5)at(-0.31,0.95){};
        \node[CliquePoint](6)at(0.31,0.95){};
        \node[CliquePoint](7)at(0.81,0.59){};
        \node[CliquePoint](8)at(1.00,0.00){};
        \node[CliquePoint](9)at(0.81,-0.59){};
        \node[CliquePoint](10)at(0.31,-0.95){};
        \draw[CliqueEdge](1)edge[]node[CliqueLabel]
            {\begin{math}1\end{math}}(2);
        \draw[CliqueEdge](1)edge[bend left=20]node[CliqueLabel]
            {\begin{math}2\end{math}}(5);
        \draw[CliqueEdge](1)edge[]node[CliqueLabel]
            {\begin{math}1\end{math}}(10);
        \draw[CliqueEdge](2)edge[]node[CliqueLabel]
            {\begin{math}4\end{math}}(3);
        \draw[CliqueEdge](2)edge[bend right=20]node[CliqueLabel]
            {\begin{math}1\end{math}}(4);
        \draw[CliqueEdge](3)edge[]node[CliqueLabel]
            {\begin{math}2\end{math}}(4);
        \draw[CliqueEmptyEdge](4)edge[]node[]{}(5);
        \draw[CliqueEdge](5)edge[]node[CliqueLabel]
            {\begin{math}3\end{math}}(6);
        \draw[CliqueEdge](5)edge[bend right=30]node[CliqueLabel]
            {\begin{math}3\end{math}}(9);
        \draw[CliqueEdge](5)edge[bend right=30]node[CliqueLabel]
            {\begin{math}1\end{math}}(10);
        \draw[CliqueEmptyEdge](6)edge[]node[]{}(7);
        \draw[CliqueEdge](6)edge[bend right=30]node[CliqueLabel]
            {\begin{math}2\end{math}}(9);
        \draw[CliqueEdge](7)edge[]node[CliqueLabel]
            {\begin{math}1\end{math}}(8);
        \draw[CliqueEmptyEdge](8)edge[]node[]{}(9);
        \draw[CliqueEmptyEdge](9)edge[]node[]{}(10);
    \end{tikzpicture}
    \quad \xmapsto{\; \BubbleTree \;} \quad
    \begin{tikzpicture}[xscale=.4,yscale=.34,Centering]
        \node[](0)at(0.00,-6.00){};
        \node[](11)at(9.00,-12.00){};
        \node[](12)at(10.00,-12.00){};
        \node[](14)at(12.00,-6.00){};
        \node[](2)at(1.00,-9.00){};
        \node[](4)at(3.00,-9.00){};
        \node[](5)at(4.00,-6.00){};
        \node[](7)at(6.00,-9.00){};
        \node[](9)at(8.00,-12.00){};
        \node[](1)at(2.00,-3.00){
            \begin{tikzpicture}[scale=0.4]
                \node[CliquePoint](1_1)at(-0.71,-0.71){};
                \node[CliquePoint](1_2)at(-0.71,0.71){};
                \node[CliquePoint](1_3)at(0.71,0.71){};
                \node[CliquePoint](1_4)at(0.71,-0.71){};
                \draw[CliqueEdge](1_1)edge[]node[CliqueLabel]
                    {\begin{math}1\end{math}}(1_2);
                \draw[CliqueEmptyEdge](1_1)edge[]node[]{}(1_4);
                \draw[CliqueEdge](1_2)edge[]node[CliqueLabel]
                    {\begin{math}1\end{math}}(1_3);
                \draw[CliqueEmptyEdge](1_3)edge[]node[]{}(1_4);
            \end{tikzpicture}};
        \node[](10)at(9.00,-9.00){
            \begin{tikzpicture}[scale=0.4]
                \node[CliquePoint](10_1)at(-0.71,-0.71){};
                \node[CliquePoint](10_2)at(-0.71,0.71){};
                \node[CliquePoint](10_3)at(0.71,0.71){};
                \node[CliquePoint](10_4)at(0.71,-0.71){};
                \draw[CliqueEmptyEdge](10_1)edge[]node[]{}(10_2);
                \draw[CliqueEmptyEdge](10_1)edge[]node[]{}(10_4);
                \draw[CliqueEdge](10_2)edge[]node[CliqueLabel]
                    {\begin{math}1\end{math}}(10_3);
                \draw[CliqueEmptyEdge](10_3)edge[]node[]{}(10_4);
            \end{tikzpicture}};
        \node[](13)at(11.00,-3.00){
            \begin{tikzpicture}[scale=0.3]
                \node[CliquePoint](13_1)at(-0.87,-0.50){};
                \node[CliquePoint](13_2)at(-0.00,1.00){};
                \node[CliquePoint](13_3)at(0.87,-0.50){};
                \draw[CliqueEdge](13_1)edge[]node[CliqueLabel]
                    {\begin{math}3\end{math}}(13_2);
                \draw[CliqueEmptyEdge](13_1)edge[]node[]{}(13_3);
                \draw[CliqueEmptyEdge](13_2)edge[]node[]{}(13_3);
            \end{tikzpicture}};
        \node[](3)at(2.00,-6.00){
            \begin{tikzpicture}[scale=0.3]
                \node[CliquePoint](3_1)at(-0.87,-0.50){};
                \node[CliquePoint](3_2)at(-0.00,1.00){};
                \node[CliquePoint](3_3)at(0.87,-0.50){};
                \draw[CliqueEdge](3_1)edge[]node[CliqueLabel]
                    {\begin{math}4\end{math}}(3_2);
                \draw[CliqueEmptyEdge](3_1)edge[]node[]{}(3_3);
                \draw[CliqueEdge](3_2)edge[]node[CliqueLabel]
                    {\begin{math}2\end{math}}(3_3);
            \end{tikzpicture}};
        \node[](6)at(5.00,0.00){
            \begin{tikzpicture}[scale=0.3]
                \node[CliquePoint](6_1)at(-0.87,-0.50){};
                \node[CliquePoint](6_2)at(-0.00,1.00){};
                \node[CliquePoint](6_3)at(0.87,-0.50){};
                \draw[CliqueEdge](6_1)edge[]node[CliqueLabel]
                    {\begin{math}2\end{math}}(6_2);
                \draw[CliqueEdge](6_1)edge[]node[CliqueLabel]
                    {\begin{math}1\end{math}}(6_3);
                \draw[CliqueEdge](6_2)edge[]node[CliqueLabel]
                    {\begin{math}1\end{math}}(6_3);
            \end{tikzpicture}};
        \node[](8)at(7.00,-6.00){
            \begin{tikzpicture}[scale=0.3]
                \node[CliquePoint](8_1)at(-0.87,-0.50){};
                \node[CliquePoint](8_2)at(-0.00,1.00){};
                \node[CliquePoint](8_3)at(0.87,-0.50){};
                \draw[CliqueEdge](8_1)edge[]node[CliqueLabel]
                    {\begin{math}3\end{math}}(8_2);
                \draw[CliqueEmptyEdge](8_1)edge[]node[]{}(8_3);
                \draw[CliqueEdge](8_2)edge[]node[CliqueLabel]
                    {\begin{math}2\end{math}}(8_3);
            \end{tikzpicture}};
        \draw[Edge](0)--(1);
        \draw[Edge](1)--(6);
        \draw[Edge](10)--(8);
        \draw[Edge](11)--(10);
        \draw[Edge](12)--(10);
        \draw[Edge](13)--(6);
        \draw[Edge](14)--(13);
        \draw[Edge](2)--(3);
        \draw[Edge](3)--(1);
        \draw[Edge](4)--(3);
        \draw[Edge](5)--(1);
        \draw[Edge](7)--(8);
        \draw[Edge](8)--(13);
        \draw[Edge](9)--(10);
        \node(r)at(5.00,2.5){};
        \draw[Edge](r)--(6);
    \end{tikzpicture}\,.
\end{equation}
\medskip

\begin{Lemma} \label{lem:map_NC_M_bubble_tree_treelike_expression}
    Let $\Mca$ be a unitary magma. For any noncrossing $\Mca$-clique
    $\Pfr$, $\BubbleTree(\Pfr)$ is a treelike expression on
    $\Bubbles_\Mca$ of~$\Pfr$.
\end{Lemma}
\begin{proof}
    We proceed by induction on the arity $n$ of $\Pfr$. If $n = 1$,
    since $\Pfr = \UnitClique$ and $\BubbleTree(\UnitClique) = \Leaf$,
    the statement of the lemma immediately holds. Otherwise, one has
    \begin{math}
        \BubbleTree(\Pfr) =
        \Corolla(\Qfr) \circ
        \left[\BubbleTree(\Rfr_1), \dots, \BubbleTree(\Rfr_k)\right]
    \end{math}
    where $\Pfr$ uniquely decomposes as
    $\Pfr = \Qfr \circ [\Rfr_1, \dots, \Rfr_k]$ under the conditions
    stated by  Proposition~\ref{prop:unique_decomposition_NC_M}. By
    definition of areas and of the map $\BubbleTree$, $\Qfr$ is an
    $\Mca$-bubble. Moreover, by induction hypothesis, any
    $\BubbleTree(\Rfr_i)$, $i \in [k]$, is a treelike expression on
    $\Bubbles_\Mca$ of $\Rfr_i$. Hence, $\BubbleTree(\Pfr)$ is a
    treelike expression on $\Bubbles_\Mca$ of~$\Pfr$.
\end{proof}
\medskip

\begin{Proposition} \label{prop:map_NC_M_bubble_tree}
    Let $\Mca$ be a unitary magma. Then, the map $\BubbleTree$ is
    injective and the image of $\BubbleTree$ is the linear span of all
    syntax trees $\Tfr$ on $\Bubbles_\Mca$ such that
    \begin{enumerate}[fullwidth,label={(\it\roman*)}]
        \item \label{item:map_NC_M_bubble_tree_1}
        the root of $\Tfr$ is labeled by an $\Mca$-bubble;
        \item \label{item:map_NC_M_bubble_tree_2}
        the internal nodes of $\Tfr$ different from the root are
        labeled by $\Mca$-bubbles whose bases are labeled
        by~$\Unit_\Mca$;
        \item \label{item:map_NC_M_bubble_tree_3}
        if $x$ and $y$ are two internal nodes of $\Tfr$ such that $y$ is
        the $i$th child of $x$, the $i$th edge of the bubble labeling
        $x$ is solid.
    \end{enumerate}
\end{Proposition}
\begin{proof}
    First of all, since by definition $\BubbleTree$ sends a basis
    element of $\NC\Mca$ to a basis element of
    $\Free\left(\Vect\left(\Bubbles_\Mca\right)\right)$, it is
    sufficient to show that $\BubbleTree$ is injective as a map from
    $\Cliques_\Mca$ to the set of syntax trees on $\Bubbles_\Mca$ to
    establish that it is an injective linear map. For this, we proceed
    by induction on the arity $n$. If $n = 1$, since
    $\BubbleTree(\UnitClique) = \Leaf$ and $\NC\Mca(1)$ is of dimension
    $1$, $\BubbleTree$ is injective. Assume now that $\Pfr$ and $\Pfr'$
    are two noncrossing $\Mca$-cliques of arity $n$ such that
    $\BubbleTree(\Pfr) = \BubbleTree(\Pfr')$. Hence, $\Pfr$ (resp.
    $\Pfr'$) uniquely decompose as
    $\Pfr = \Qfr \circ [\Rfr_1, \dots, \Rfr_k]$ (resp.
    $\Pfr' = \Qfr' \circ [\Rfr'_1, \dots, \Rfr'_k]$) as stated by
    Proposition~\ref{prop:unique_decomposition_NC_M} and
    \begin{equation}
        \BubbleTree(\Pfr)
        = \Corolla(\Qfr) \circ
        \left[\BubbleTree(\Rfr_1), \dots, \BubbleTree(\Rfr_k)\right]
        = \Corolla(\Qfr') \circ
        \left[\BubbleTree(\Rfr'_1), \dots, \BubbleTree(\Rfr'_k)\right]
        = \BubbleTree(\Pfr').
    \end{equation}
    Now, because by definition of areas, all bases of the $\Rfr_i$ and
    $\Rfr'_i$, $i \in [k]$, are labeled by $\Unit_\Mca$, this implies
    that $\Qfr = \Qfr'$. Therefore, we have
    $\BubbleTree(\Rfr_i) = \BubbleTree(\Rfr'_i)$ for all $i \in [k]$, so
    that, by induction hypothesis, $\Rfr_i = \Rfr'_i$ for all
    $i \in [k]$. Hence, $\BubbleTree$ is injective.
    \smallskip

    The definition of $\BubbleTree$ together with
    Proposition~\ref{prop:unique_decomposition_NC_M} lead to the fact
    that for any noncrossing $\Mca$-clique $\Pfr$, the syntax tree
    $\BubbleTree(\Pfr)$ satisfies~\ref{item:map_NC_M_bubble_tree_1},
    \ref{item:map_NC_M_bubble_tree_2},
    and~\ref{item:map_NC_M_bubble_tree_3}. Conversely, let $\Tfr$ be a
    syntax tree satisfying~\ref{item:map_NC_M_bubble_tree_1},
    \ref{item:map_NC_M_bubble_tree_2},
    and~\ref{item:map_NC_M_bubble_tree_3}. Let us show by structural
    induction on $\Tfr$ that there is a noncrossing $\Mca$-clique $\Pfr$
    such that $\BubbleTree(\Pfr) = \Tfr$. If $\Tfr = \Leaf$, the
    property holds because $\BubbleTree(\UnitClique) = \Leaf$.
    Otherwise, one has $\Tfr = \Sfr \circ [\Ufr_1, \dots, \Ufr_k]$ where
    $\Sfr$ is a syntax tree of degree $1$ and the $\Ufr_i$, $i \in [k]$,
    are syntax trees. Since $\Tfr$
    satisfies~\ref{item:map_NC_M_bubble_tree_1},
    \ref{item:map_NC_M_bubble_tree_2},
    and~\ref{item:map_NC_M_bubble_tree_3}, the trees $\Sfr$ and
    $\Ufr_i$, $i \in [k]$, satisfy the same three properties. Therefore,
    by induction hypothesis, there are noncrossing $\Mca$-cliques $\Qfr$
    and $\Rfr_i$, $i \in [k]$, such that $\BubbleTree(\Qfr) = \Sfr$ and
    $\BubbleTree(\Rfr_i) = \Ufr_i$. Set now $\Pfr$ as the noncrossing
    $\Mca$-clique $\Qfr \circ [\Rfr_1, \dots, \Rfr_k]$. By definition of
    the map $\BubbleTree$ and the unique decomposition stated in
    Proposition~\ref{prop:unique_decomposition_NC_M} for $\Pfr$, one
    obtains that $\BubbleTree(\Pfr) = \Tfr$.
\end{proof}
\medskip

Observe that $\BubbleTree$ is not an operad morphism. Indeed,
\begin{equation} \label{equ:bubble_tree_not_morphism}
    \BubbleTree\left(
        \begin{tikzpicture}[scale=0.3,Centering]
        \node[CliquePoint](1)at(-0.87,-0.50){};
        \node[CliquePoint](2)at(-0.00,1.00){};
        \node[CliquePoint](3)at(0.87,-0.50){};
        \draw[CliqueEmptyEdge](1)edge[]node[]{}(2);
        \draw[CliqueEmptyEdge](1)edge[]node[]{}(3);
        \draw[CliqueEmptyEdge](2)edge[]node[]{}(3);
    \end{tikzpicture}
    \circ_1
        \begin{tikzpicture}[scale=0.3,Centering]
        \node[CliquePoint](1)at(-0.87,-0.50){};
        \node[CliquePoint](2)at(-0.00,1.00){};
        \node[CliquePoint](3)at(0.87,-0.50){};
        \draw[CliqueEmptyEdge](1)edge[]node[]{}(2);
        \draw[CliqueEmptyEdge](1)edge[]node[]{}(3);
        \draw[CliqueEmptyEdge](2)edge[]node[]{}(3);
    \end{tikzpicture}
    \right)
    =
    \begin{tikzpicture}[scale=0.4,Centering]
        \node[](0)at(0.00,-2.00){};
        \node[](2)at(1.00,-2.00){};
        \node[](3)at(2.00,-2.00){};
        \node[](1)at(1.00,0.00){
            \begin{tikzpicture}[scale=0.4]
                \node[CliquePoint](1_1)at(-0.71,-0.71){};
                \node[CliquePoint](1_2)at(-0.71,0.71){};
                \node[CliquePoint](1_3)at(0.71,0.71){};
                \node[CliquePoint](1_4)at(0.71,-0.71){};
                \draw[CliqueEmptyEdge](1_1)edge[]node[]{}(1_2);
                \draw[CliqueEmptyEdge](1_1)edge[]node[]{}(1_4);
                \draw[CliqueEmptyEdge](1_2)edge[]node[]{}(1_3);
                \draw[CliqueEmptyEdge](1_3)edge[]node[]{}(1_4);
            \end{tikzpicture}};
        \draw[Edge](0)--(1);
        \draw[Edge](2)--(1);
        \draw[Edge](3)--(1);
        \node(r)at(1.00,1.50){};
        \draw[Edge](r)--(1);
    \end{tikzpicture}
    \enspace \ne
    \begin{tikzpicture}[scale=0.4,Centering]
        \node[](0)at(0.00,-3.33){};
        \node[](2)at(2.00,-3.33){};
        \node[](4)at(4.00,-1.67){};
        \node[](1)at(1.00,-1.67){
            \begin{tikzpicture}[scale=0.3]
                \node[CliquePoint](1_1)at(-0.87,-0.50){};
                \node[CliquePoint](1_2)at(-0.00,1.00){};
                \node[CliquePoint](1_3)at(0.87,-0.50){};
                \draw[CliqueEmptyEdge](1_1)edge[]node[]{}(1_2);
                \draw[CliqueEmptyEdge](1_1)edge[]node[]{}(1_3);
                \draw[CliqueEmptyEdge](1_2)edge[]node[]{}(1_3);
            \end{tikzpicture}};
        \node[](3)at(3.00,0.00){
            \begin{tikzpicture}[scale=0.3]
                \node[CliquePoint](3_1)at(-0.87,-0.50){};
                \node[CliquePoint](3_2)at(-0.00,1.00){};
                \node[CliquePoint](3_3)at(0.87,-0.50){};
                \draw[CliqueEmptyEdge](3_1)edge[]node[]{}(3_2);
                \draw[CliqueEmptyEdge](3_1)edge[]node[]{}(3_3);
                \draw[CliqueEmptyEdge](3_2)edge[]node[]{}(3_3);
            \end{tikzpicture}};
        \draw[Edge](0)--(1);
        \draw[Edge](1)--(3);
        \draw[Edge](2)--(1);
        \draw[Edge](4)--(3);
        \node(r)at(3.00,1.25){};
        \draw[Edge](r)--(3);
    \end{tikzpicture}
    =
    \BubbleTree\left(
    \begin{tikzpicture}[scale=0.3,Centering]
        \node[CliquePoint](1)at(-0.87,-0.50){};
        \node[CliquePoint](2)at(-0.00,1.00){};
        \node[CliquePoint](3)at(0.87,-0.50){};
        \draw[CliqueEmptyEdge](1)edge[]node[]{}(2);
        \draw[CliqueEmptyEdge](1)edge[]node[]{}(3);
        \draw[CliqueEmptyEdge](2)edge[]node[]{}(3);
    \end{tikzpicture}
    \right)
    \circ_1
    \BubbleTree\left(
    \begin{tikzpicture}[scale=0.3,Centering]
        \node[CliquePoint](1)at(-0.87,-0.50){};
        \node[CliquePoint](2)at(-0.00,1.00){};
        \node[CliquePoint](3)at(0.87,-0.50){};
        \draw[CliqueEmptyEdge](1)edge[]node[]{}(2);
        \draw[CliqueEmptyEdge](1)edge[]node[]{}(3);
        \draw[CliqueEmptyEdge](2)edge[]node[]{}(3);
    \end{tikzpicture}
    \right).
\end{equation}
Observe that~\eqref{equ:bubble_tree_not_morphism} holds for all unitary
magmas $\Mca$ since $\Unit_\Mca$ is always idempotent.
\medskip

\subsubsection{Realization in terms of decorated Schröder trees}%
\label{subsubsec:M_Schroder_trees}
Recall that a \Def{Schröder tree} is a rooted planar tree such that all
internal nodes have at least two children. An \Def{$\Mca$-Schröder tree}
$\Tfr$ is a Schröder tree such that each edge connecting two internal
nodes is labeled on $\bar{\Mca}$, each edge connecting an internal node
an a leaf is labeled on $\Mca$, and the outgoing edge from the root of
$\Tfr$ is labeled on $\Mca$ (see~\eqref{equ:example_M_Schroder_tree} for
an example of a $\Z$-Schröder tree).
\medskip

From the description of the image of the map $\BubbleTree$ provided by
Proposition~\ref{prop:map_NC_M_bubble_tree}, any bubble tree $\Tfr$
of a noncrossing $\Mca$-clique $\Pfr$ of arity $n$ can be encoded by an
$\Mca$-Schröder tree $\Sfr$ with $n$ leaves. Indeed, this
$\Mca$-Schröder tree is obtained by considering each internal node $x$
of $\Tfr$ and by labeling the edge connecting $x$ and its $i$th child by
the label of the $i$th edge of the $\Mca$-bubble labeling $x$. The
outgoing edge from the root of $\Sfr$ is labeled by the label of the
base of the $\Mca$-bubble labeling the root of $\Tfr$. For instance, the
bubble tree of~\eqref{equ:bubble_tree_example} is encoded by the
$\Z$-Schröder tree
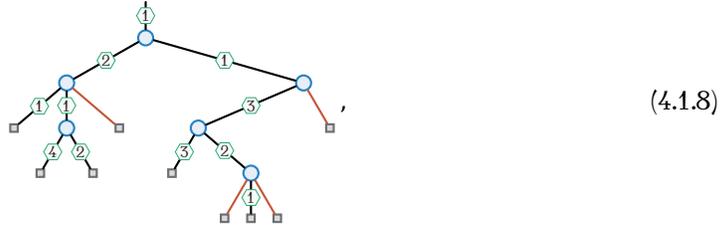
\begin{equation} \label{equ:example_M_Schroder_tree}
    \begin{tikzpicture}[xscale=.35,yscale=.2,Centering]
        \node[Leaf](0)at(0.00,-6.00){};
        \node[Leaf](11)at(9.00,-12.00){};
        \node[Leaf](12)at(10.00,-12.00){};
        \node[Leaf](14)at(12.00,-6.00){};
        \node[Leaf](2)at(1.00,-9.00){};
        \node[Leaf](4)at(3.00,-9.00){};
        \node[Leaf](5)at(4.00,-6.00){};
        \node[Leaf](7)at(6.00,-9.00){};
        \node[Leaf](9)at(8.00,-12.00){};
        \node[Node](1)at(2.00,-3.00){};
        \node[Node](10)at(9.00,-9.00){};
        \node[Node](13)at(11.00,-3.00){};
        \node[Node](3)at(2.00,-6.00){};
        \node[Node](6)at(5.00,0.00){};
        \node[Node](8)at(7.00,-6.00){};
        \draw[Edge](0)edge[]node[EdgeLabel]{\begin{math}1\end{math}}(1);
        \draw[Edge](1)edge[]node[EdgeLabel]{\begin{math}2\end{math}}(6);
        \draw[Edge](10)edge[]node[EdgeLabel]{\begin{math}2\end{math}}(8);
        \draw[Edge](11)edge[]node[EdgeLabel]{\begin{math}1\end{math}}(10);
        \draw[Edge](12)--(10);
        \draw[Edge](13)edge[]node[EdgeLabel]{\begin{math}1\end{math}}(6);
        \draw[Edge](14)--(13);
        \draw[Edge](2)edge[]node[EdgeLabel]{\begin{math}4\end{math}}(3);
        \draw[Edge](3)edge[]node[EdgeLabel]{\begin{math}1\end{math}}(1);
        \draw[Edge](4)edge[]node[EdgeLabel]{\begin{math}2\end{math}}(3);
        \draw[Edge](5)--(1);
        \draw[Edge](7)edge[]node[EdgeLabel]{\begin{math}3\end{math}}(8);
        \draw[Edge](8)edge[]node[EdgeLabel]{\begin{math}3\end{math}}(13);
        \draw[Edge](9)--(10);
        \node(r)at(5.00,3){};
        \draw[Edge](r)edge[]node[EdgeLabel]{\begin{math}1\end{math}}(6);
    \end{tikzpicture}\,,
\end{equation}
where the labels of the edges are drawn in the hexagons and where
unlabeled edges are implicitly labeled by $\Unit_\Mca$. We shall use
these drawing conventions in the sequel. As a side remark, observe that
the $\Mca$-Schröder tree encoding a noncrossing $\Mca$-clique $\Pfr$ and
the dual tree of $\Pfr$ (in the usual meaning) have the same underlying
unlabeled tree.
\medskip

This encoding of noncrossing $\Mca$-cliques by bubble trees is
reversible and hence, one can interpret $\NC\Mca$ as an operad on the
linear span of all $\Mca$-Schröder trees. Hence, through this
interpretation, if $\Sfr$ and $\Tfr$ are two $\Mca$-Schröder trees and
$i$ is a valid integer, the tree $\Sfr \circ_i \Tfr$ is computed by
grafting the root of $\Tfr$ to the $i$th leaf of $\Sfr$. Then, by
denoting by $b$ the label of the edge adjacent to the root of $\Tfr$ and
by $a$ the label of the edge adjacent to the $i$th leaf of $\Sfr$, we
have two cases to consider, depending on the value of $c := a \Op b$. If
$c \ne \Unit_\Mca$, we label the edge connecting $\Sfr$ and $\Tfr$ by
$c$. Otherwise, when $c = \Unit_\Mca$, we contract the edge connecting
$\Sfr$ and $\Tfr$ by merging the root of $\Tfr$ and the father of the
$i$th leaf of $\Sfr$ (see
Figure~\ref{fig:composition_NC_M_Schroder_trees}).
\begin{figure}[ht]
    \centering
    \captionsetup[subfigure]{width=5cm}
    \subfloat[]
    [The expression $\Sfr \circ_i \Tfr$ to compute. The displayed
    leaf is the $i$th one of~$\Sfr$.]{
        \begin{tikzpicture}[xscale=1,yscale=.55,Centering]
            \node[Subtree](1)at(0,0){\begin{math}\Sfr'\end{math}};
            \node[Node](2)at(0,-1){};
            \node[Leaf](3)at(0,-2){};
            \draw[Edge](1)--(2);
            \draw[Edge](2)edge[]node[EdgeLabel]
                {\begin{math}a\end{math}}(3);
            \node[below of=2,font=\footnotesize]{\begin{math}i\end{math}};
        \end{tikzpicture}
        \quad $\circ_i$ \quad
        \begin{tikzpicture}[xscale=.9,yscale=.8,Centering]
            \node[Node](1)at(0,0){};
            \node(r)at(0,1){};
            \node[Subtree](2)at(-1,-1){\begin{math}\Tfr_1\end{math}};
            \node[Subtree](3)at(1,-1){\begin{math}\Tfr_k\end{math}};
            \draw[Edge](1)edge[]node[EdgeLabel]
                {\begin{math}b\end{math}}(r);
            \draw[Edge](1)--(2);
            \draw[Edge](1)--(3);
            \node at(0,-1){\begin{math}\dots\end{math}};
        \end{tikzpicture}
    \label{subfig:composition_NC_M_Schroder_trees_1}}

    \captionsetup[subfigure]{width=3cm}
    \subfloat[]
    [The resulting tree when $a \Op b \ne \Unit_\Mca$.]{
        \begin{tikzpicture}[xscale=.9,yscale=.68,Centering]
            \node[Subtree](1)at(0,-.5){\begin{math}\Sfr'\end{math}};
            \node[Node](2)at(0,-1.5){};
            \draw[Edge](1)--(2);
            \node[Node](4)at(0,-3){};
            \node[Subtree](6)at(-1,-4){\begin{math}\Tfr_1\end{math}};
            \node[Subtree](7)at(1,-4){\begin{math}\Tfr_k\end{math}};
            \draw[Edge](4)--(6);
            \draw[Edge](4)--(7);
            \node at(0,-4){\begin{math}\dots\end{math}};
            \draw[Edge](4)edge[]node[EdgeLabel]
                {\begin{math}a \Op b\end{math}}(2);
        \end{tikzpicture}
    \label{subfig:composition_NC_M_Schroder_trees_2}}
    \qquad
    \qquad
    \qquad
    \subfloat[]
    [The resulting tree when $a \Op b = \Unit_\Mca$.]{
        \begin{tikzpicture}[xscale=.9,yscale=.7,Centering]
            \node[Subtree](1)at(0,-2){\begin{math}\Sfr'\end{math}};
            \node[Node](4)at(0,-3){};
            \draw[Edge](1)--(4);
            \node[Subtree](6)at(-1,-4){\begin{math}\Tfr_1\end{math}};
            \node[Subtree](7)at(1,-4){\begin{math}\Tfr_k\end{math}};
            \draw[Edge](4)--(6);
            \draw[Edge](4)--(7);
            \node at(0,-4){\begin{math}\dots\end{math}};
        \end{tikzpicture}
    \label{subfig:composition_NC_M_Schroder_trees_3}}
    \caption{\footnotesize
    The partial composition of $\NC\Mca$ realized on $\Mca$-Schröder
    trees. Here, the two
    cases~\protect\subref{subfig:composition_NC_M_Schroder_trees_2}
    and~\protect\subref{subfig:composition_NC_M_Schroder_trees_3}
    for the computation of $\Sfr \circ_i \Tfr$
    are shown, where $\Sfr$ and $\Tfr$ are two $\Mca$-Schröder trees.
    In these drawings, the triangles denote subtrees.}
    \label{fig:composition_NC_M_Schroder_trees}
\end{figure}
For instance, in $\NC\N_3$, one has the two partial compositions
\begin{subequations}
\begin{equation}
    \begin{tikzpicture}[xscale=.28,yscale=.2,Centering]
        \node[Leaf](0)at(0.00,-5.33){};
        \node[Leaf](2)at(2.00,-5.33){};
        \node[Leaf](4)at(3.00,-5.33){};
        \node[Leaf](6)at(5.00,-5.33){};
        \node[Leaf](7)at(6.00,-2.67){};
        \node[Node](1)at(1.00,-2.67){};
        \node[Node](3)at(4.00,0.00){};
        \node[Node](5)at(4.00,-2.67){};
        \draw[Edge](0)--(1);
        \draw[Edge](1)edge[]node[EdgeLabel]{\begin{math}1\end{math}}(3);
        \draw[Edge](2)edge[]node[EdgeLabel]{\begin{math}1\end{math}}(1);
        \draw[Edge](4)edge[]node[EdgeLabel]{\begin{math}2\end{math}}(5);
        \draw[Edge](5)edge[]node[EdgeLabel]{\begin{math}1\end{math}}(3);
        \draw[Edge](6)--(5);
        \draw[Edge](7)--(3);
        \node(r)at(4.00,2.75){};
        \draw[Edge](r)edge[]node[EdgeLabel]{\begin{math}2\end{math}}(3);
    \end{tikzpicture}
    \circ_2
    \begin{tikzpicture}[xscale=.3,yscale=.27,Centering]
        \node[Leaf](0)at(0.00,-1.67){};
        \node[Leaf](2)at(2.00,-3.33){};
        \node[Leaf](4)at(4.00,-3.33){};
        \node[Node](1)at(1.00,0.00){};
        \node[Node](3)at(3.00,-1.67){};
        \draw[Edge](0)--(1);
        \draw[Edge](2)--(3);
        \draw[Edge](3)edge[]node[EdgeLabel]{\begin{math}1\end{math}}(1);
        \draw[Edge](4)edge[]node[EdgeLabel]{\begin{math}2\end{math}}(3);
        \node(r)at(1.00,2){};
        \draw[Edge](r)edge[]node[EdgeLabel]{\begin{math}1\end{math}}(1);
    \end{tikzpicture}
    =
    \begin{tikzpicture}[xscale=.28,yscale=.22,Centering]
        \node[Leaf](0)at(0.00,-4.80){};
        \node[Leaf](10)at(9.00,-4.80){};
        \node[Leaf](11)at(10.00,-2.40){};
        \node[Leaf](2)at(2.00,-7.20){};
        \node[Leaf](4)at(4.00,-9.60){};
        \node[Leaf](6)at(6.00,-9.60){};
        \node[Leaf](8)at(7.00,-4.80){};
        \node[Node](1)at(1.00,-2.40){};
        \node[Node](3)at(3.00,-4.80){};
        \node[Node](5)at(5.00,-7.20){};
        \node[Node](7)at(8.00,0.00){};
        \node[Node](9)at(8.00,-2.40){};
        \draw[Edge](0)--(1);
        \draw[Edge](1)edge[]node[EdgeLabel]{\begin{math}1\end{math}}(7);
        \draw[Edge](10)--(9);
        \draw[Edge](11)--(7);
        \draw[Edge](2)--(3);
        \draw[Edge](3)edge[]node[EdgeLabel]{\begin{math}2\end{math}}(1);
        \draw[Edge](4)--(5);
        \draw[Edge](5)edge[]node[EdgeLabel]{\begin{math}1\end{math}}(3);
        \draw[Edge](6)edge[]node[EdgeLabel]{\begin{math}2\end{math}}(5);
        \draw[Edge](8)edge[]node[EdgeLabel]{\begin{math}2\end{math}}(9);
        \draw[Edge](9)edge[]node[EdgeLabel]{\begin{math}1\end{math}}(7);
        \node(r)at(8.00,2.5){};
        \draw[Edge](r)edge[]node[EdgeLabel]{\begin{math}2\end{math}}(7);
    \end{tikzpicture}\,,
\end{equation}
\begin{equation}
    \begin{tikzpicture}[xscale=.28,yscale=.2,Centering]
        \node[Leaf](0)at(0.00,-5.33){};
        \node[Leaf](2)at(2.00,-5.33){};
        \node[Leaf](4)at(3.00,-5.33){};
        \node[Leaf](6)at(5.00,-5.33){};
        \node[Leaf](7)at(6.00,-2.67){};
        \node[Node](1)at(1.00,-2.67){};
        \node[Node](3)at(4.00,0.00){};
        \node[Node](5)at(4.00,-2.67){};
        \draw[Edge](0)--(1);
        \draw[Edge](1)edge[]node[EdgeLabel]{\begin{math}1\end{math}}(3);
        \draw[Edge](2)edge[]node[EdgeLabel]{\begin{math}1\end{math}}(1);
        \draw[Edge](4)edge[]node[EdgeLabel]{\begin{math}2\end{math}}(5);
        \draw[Edge](5)edge[]node[EdgeLabel]{\begin{math}1\end{math}}(3);
        \draw[Edge](6)--(5);
        \draw[Edge](7)--(3);
        \node(r)at(4.00,2.75){};
        \draw[Edge](r)edge[]node[EdgeLabel]{\begin{math}2\end{math}}(3);
    \end{tikzpicture}
    \circ_3
    \begin{tikzpicture}[xscale=.3,yscale=.27,Centering]
        \node[Leaf](0)at(0.00,-1.67){};
        \node[Leaf](2)at(2.00,-3.33){};
        \node[Leaf](4)at(4.00,-3.33){};
        \node[Node](1)at(1.00,0.00){};
        \node[Node](3)at(3.00,-1.67){};
        \draw[Edge](0)--(1);
        \draw[Edge](2)--(3);
        \draw[Edge](3)edge[]node[EdgeLabel]{\begin{math}1\end{math}}(1);
        \draw[Edge](4)edge[]node[EdgeLabel]{\begin{math}2\end{math}}(3);
        \node(r)at(1.00,2){};
        \draw[Edge](r)edge[]node[EdgeLabel]{\begin{math}1\end{math}}(1);
    \end{tikzpicture}
    =
    \begin{tikzpicture}[xscale=.28,yscale=.22,Centering]
        \node[Leaf](0)at(0.00,-5.50){};
        \node[Leaf](10)at(8.00,-2.75){};
        \node[Leaf](2)at(2.00,-5.50){};
        \node[Leaf](4)at(3.00,-5.50){};
        \node[Leaf](6)at(4.00,-8.25){};
        \node[Leaf](8)at(6.00,-8.25){};
        \node[Leaf](9)at(7.00,-5.50){};
        \node[Node](1)at(1.00,-2.75){};
        \node[Node](3)at(5.00,0.00){};
        \node[Node](5)at(5.00,-2.75){};
        \node[Node](7)at(5.00,-5.50){};
        \draw[Edge](0)--(1);
        \draw[Edge](1)edge[]node[EdgeLabel]{\begin{math}1\end{math}}(3);
        \draw[Edge](10)--(3);
        \draw[Edge](2)edge[]node[EdgeLabel]{\begin{math}1\end{math}}(1);
        \draw[Edge](4)--(5);
        \draw[Edge](5)edge[]node[EdgeLabel]{\begin{math}1\end{math}}(3);
        \draw[Edge](6)--(7);
        \draw[Edge](7)edge[]node[EdgeLabel]{\begin{math}1\end{math}}(5);
        \draw[Edge](8)edge[]node[EdgeLabel]{\begin{math}2\end{math}}(7);
        \draw[Edge](9)--(5);
        \node(r)at(5.00,2.5){};
        \draw[Edge](r)edge[]node[EdgeLabel]{\begin{math}2\end{math}}(3);
    \end{tikzpicture}\,.
\end{equation}
\end{subequations}
\medskip

In the sequel, we shall indifferently see $\NC\Mca$ as an operad on
noncrossing $\Mca$-cliques or on $\Mca$-Schröder trees.
\medskip

\subsubsection{Minimal generating set}

\begin{Proposition} \label{prop:generating_set_NC_M}
    Let $\Mca$ be a unitary magma. The set $\Triangles_\Mca$ of all
    $\Mca$-triangles is a minimal generating set of~$\NC\Mca$.
\end{Proposition}
\begin{proof}
    We start by showing by induction on the arity that the suboperad
    $(\NC\Mca)^{\Triangles_\Mca}$ of $\NC\Mca$ generated by
    $\Triangles_\Mca$ is $\NC\Mca$. It is immediately true in arity $1$.
    Let $\Pfr$ be a noncrossing $\Mca$-clique of arity $n \geq 2$.
    Proposition~\ref{prop:unique_decomposition_NC_M} says in
    particular that we can express $\Pfr$ as
    $\Pfr = \Qfr \circ [\Rfr_1, \dots, \Rfr_k]$ where $\Qfr$ is an
    $\Mca$-bubble of arity $k \geq 2$ and the $\Rfr_i$, $i \in [k]$, are
    noncrossing $\Mca$-cliques. Since $\Qfr$ is an $\Mca$-bubble, it
    can be expressed as
    \begin{equation}\label{equ:generating_set_NC_M}
        \Qfr =
        \TriangleXEX{\Qfr_0}{\Unit_\Mca}{\Qfr_k}
        \circ_1
        \TriangleEEX{\Unit_\Mca}{\Unit_\Mca}{\Qfr_{k - 1}}
        \circ_1 \dots \circ_1
        \TriangleEEX{\Unit_\Mca}{\Unit_\Mca}{\Qfr_3}
        \circ_1
        \TriangleEXX{\Unit_\Mca}{\Qfr_1}{\Qfr_2}\,.
    \end{equation}
    Observe that in~\eqref{equ:generating_set_NC_M}, brackets are not
    necessary since $\circ_1$ is associative. Since $k \geq 2$, the
    arities of each $\Rfr_i$, $i \in [k]$, are smaller than the one of
    $\Pfr$. For this reason, by induction hypothesis, each $\Rfr_i$
    belongs to $(\NC\Mca)^{\Triangles_\Mca}$. Moreover,
    since~\eqref{equ:generating_set_NC_M} shows an expression of $\Qfr$
    by partial compositions of $\Mca$-triangles, $\Qfr$ also belongs to
    $(\NC\Mca)^{\Triangles_\Mca}$. This implies that it is also the case
    for $\Pfr$. Hence, $\NC\Mca$ is generated by~$\Triangles_\Mca$.
    \smallskip

    Finally, due to the fact that the partial composition of two
    $\Mca$-triangles is an $\Mca$-clique of arity $3$, if $\Pfr$ is an
    $\Mca$-triangle, $\Pfr$ cannot be expressed as a partial composition
    of $\Mca$-triangles. Moreover, since the space $\NC\Mca(1)$ is
    trivial, these arguments imply that $\Triangles_\Mca$ is a minimal
    generating set of~$\NC\Mca$.
\end{proof}
\medskip

Proposition~\ref{prop:generating_set_NC_M} also says that $\NC\Mca$ is
the smallest suboperad of $\Cli\Mca$ that contains all $\Mca$-triangles
and that $\NC\Mca$ is the biggest binary suboperad of~$\Cli\Mca$.
\medskip

\subsubsection{Dimensions}
We now use the notion of bubble trees introduced in
Section~\ref{subsubsec:treelike_bubbles} to compute the dimensions
of~$\NC\Mca$.
\medskip

\begin{Proposition} \label{prop:Hilbert_series_NC_M}
    Let $\Mca$ be a finite unitary magma. The Hilbert series
    $\Hilbert_{\NC\Mca}(t)$ of $\NC\Mca$ satisfies
    \begin{equation} \label{equ:Hilbert_series_NC_M}
        t + \left(m^3 - 2m^2 + 2m - 1\right)t^2
        + \left(2m^2t - 3mt + 2t - 1\right) \Hilbert_{\NC\Mca}(t)
        + \left(m - 1\right) \Hilbert_{\NC\Mca}(t)^2
        = 0,
    \end{equation}
    where $m := \# \Mca$.
\end{Proposition}
\begin{proof}
    By Proposition~\ref{prop:map_NC_M_bubble_tree}, the set of
    noncrossing $\Mca$-cliques is in one-to-one correspondence with the
    set of the syntax trees on $\Bubbles_\Mca$ that
    satisfy~\ref{item:map_NC_M_bubble_tree_1},
    \ref{item:map_NC_M_bubble_tree_2},
    and~\ref{item:map_NC_M_bubble_tree_3}. Let us call $T(t)$ the
    generating series of these trees and $S(t)$ the generating series of
    these trees with the extra condition that the roots are labeled by
    $\Mca$-bubbles whose bases are labeled by~$\Unit_\Mca$. Immediately
    from its description, $S(t)$ satisfies
    \begin{equation} \label{equ:Hilbert_series_NC_M_1}
        S(t) = t + \sum_{n \geq 2} \left((m - 1) S(t) + t\right)^n,
    \end{equation}
    and $T(t)$ satisfies
    \begin{equation} \label{equ:Hilbert_series_NC_M_2}
        T(t) = t + m (S(t) - t).
    \end{equation}
    As the set of all noncrossing $\Mca$-cliques forms the fundamental
    basis of $\NC\Mca$, one has $\Hilbert_{\NC\Mca}(t) = T(t)$. We
    eventually obtain~\eqref{equ:Hilbert_series_NC_M}
    from~\eqref{equ:Hilbert_series_NC_M_1}
    and~\eqref{equ:Hilbert_series_NC_M_2} by a direct computation.
\end{proof}
\medskip

We deduce from Proposition~\ref{prop:Hilbert_series_NC_M} that the
Hilbert series of $\NC\Mca$ satisfies
\begin{equation} \label{equ:Hilbert_series_NC_M_function}
    \Hilbert_{\NC\Mca}(t) =
    \frac{1 - (2m^2 - 3m + 2)t - \sqrt{1 - 2(2m^2 - m)t + m^2t^2}}
    {2(m - 1)},
\end{equation}
where $m := \# \Mca \ne 1$.
\medskip

By using Narayana numbers, whose definition is recalled in
Section~\ref{subsubsec:quotient_Cli_M_Inf}, one can state the following
result.
\medskip

\begin{Proposition} \label{prop:dimensions_NC_M}
    Let $\Mca$ be a finite unitary magma. For all $n \geq 2$,
    \begin{equation} \label{equ:dimensions_NC_M}
        \dim \NC\Mca(n) =
        \sum_{0 \leq k \leq n - 2}
            m^{n + k + 1} (m - 1)^{n - k - 2} \;
            \Nar(n, k),
    \end{equation}
    where $m := \# \Mca$.
\end{Proposition}
\begin{proof}
    As shown by Proposition~\ref{prop:map_NC_M_bubble_tree}, each
    noncrossing $\Mca$-clique $\Pfr$ of $\NC\Mca(n)$ can be encoded by a
    unique syntax tree $\BubbleTree(\Pfr)$ on $\Bubbles_\Mca$ satisfying
    some conditions. Moreover,
    Proposition~\ref{prop:generating_set_NC_M} shows that any
    noncrossing $\Mca$-clique can be expressed (not necessarily in a
    unique way) as partial compositions of several $\Mca$-triangles. By
    combining these two results, we obtain that any noncrossing
    $\Mca$-clique $\Pfr$ can be encoded by a syntax tree on
    $\Triangles_\Mca$ obtained from $\BubbleTree(\Pfr)$ by replacing any
    of its nodes $\Sfr$ of arity $\ell \geq 3$ by left comb binary
    syntax trees $\Sfr'$ on $\Triangles_\Mca$ satisfying
    \begin{equation} \label{equ:dimensions_NC_M_demo}
        \Sfr' :=
        \Corolla\left(\Qfr^1\right) \circ_1
        \Corolla\left(\Qfr^2\right) \circ_1 \dots \circ_1
        \Corolla\left(\Qfr^{\ell - 1}\right),
    \end{equation}
    where the $\Qfr^i$, $i \in [\ell - 1]$, are the unique
    $\Mca$-triangles such that for any $i \in [2, \ell - 1]$, the base
    of $\Qfr^i$ is labeled by $\Unit_\Mca$, for any $j \in [\ell - 2]$,
    the first edge of $\Qfr^j$ is labeled by $\Unit_\Mca$, and
    $\Eval(\Sfr') = \Eval(\Sfr)$. Observe that
    in~\eqref{equ:dimensions_NC_M_demo}, brackets are not necessary
    since $\circ_1$ is associative. Therefore, $\Pfr$ can be encoded in
    a unique way as a binary syntax tree $\Tfr$ on $\Triangles_\Mca$
    satisfying the following restrictions:
    \begin{enumerate}[fullwidth,label={(\it\roman*)}]
        \item the $\Mca$-triangles labeling the internal nodes of
        $\Tfr$ which are not the root have bases labeled by $\Unit_\Mca$;
        \item if $x$ and $y$ are two internal nodes of $\Tfr$ such that
        $y$ is the right child of $x$, the second edge of the bubble
        labeling $x$ is solid.
    \end{enumerate}
    \smallskip

    To establish~\eqref{equ:dimensions_NC_M}, since the set of all
    noncrossing $\Mca$-cliques forms the fundamental basis of $\NC\Mca$,
    we now have to count these binary trees. Consider a binary tree
    $\Tfr$ of arity $n \geq 2$ with exactly $k \in [0, n - 2]$ internal
    nodes having an internal node as a left child. There are $m$ ways to
    label the base of the $\Mca$-triangle labeling the root of $\Tfr$,
    $m^k$ ways to label the first edges of the $\Mca$-triangles labeling
    the internal nodes of $\Tfr$ that have an internal node as left
    child, $m^n$ ways to label the first (resp. second) edges of the
    $\Mca$-triangles labeling the internal nodes of $\Tfr$ having a leaf
    as left (resp. right) child, and, since there are exactly
    $n - k - 2$ internal nodes of $\Tfr$ having an internal node as a
    right child, there are $(m - 1)^{n - k - 2}$ ways to label the
    second edges of the $\Mca$-triangles labeling these internal nodes.
    Now, since $\Nar(n, k)$ counts the binary trees with $n$ leaves and
    exactly $k$ internal nodes having an internal node as a left child,
    and a binary tree with $n$ leaves can have at most $n - 2$ internal
    nodes having an internal node as left child,
    \eqref{equ:dimensions_NC_M} follows.
\end{proof}
\medskip

We can use Proposition~\ref{prop:dimensions_NC_M} to compute the
first dimensions of $\NC\Mca$. For instance, depending on
$m := \# \Mca$, we have the following sequences of dimensions:
\begin{subequations}
\begin{equation}
    1, 1, 1, 1, 1, 1, 1, 1,
    \qquad m = 1,
\end{equation}
\begin{equation}
    1, 8, 48, 352, 2880, 25216, 231168, 2190848,
    \qquad m = 2,
\end{equation}
\begin{equation}
    1, 27, 405, 7533, 156735, 349263, 81520425, 1967414265,
    \qquad m = 3,
\end{equation}
\begin{equation}
    1, 64, 1792, 62464, 2437120, 101859328, 4459528192, 201889939456.
    \qquad m = 4,
\end{equation}
\end{subequations}
The second one forms, except for the first terms,
Sequence~\OEIS{A054726} of~\cite{Slo}. The last two sequences are not
listed in~\cite{Slo} at this time.
\medskip

\subsection{Presentation and Koszulity}
The aim of this section is to establish a presentation by generators and
relations of $\NC\Mca$. For this, we will define an adequate rewrite
rule on the set of the syntax trees on $\Triangles_\Mca$ and prove that
it admits the required properties.
\medskip

\subsubsection{Space of relations}
\label{subsubsec:space_of_relations_NC_M}
Let $\Rel_{\NC\Mca}$ be the subspace of
$\Free\left(\Vect\left(\Triangles_\Mca\right)\right)(3)$ generated by
the elements
\begin{subequations}
\begin{equation} \label{equ:relation_1_NC_M}
    \Corolla\left(\Triangle{\Pfr_0}{\Pfr_1}{\Pfr_2}\right)
    \circ_1
    \Corolla\left(\Triangle{\Qfr_0}{\Qfr_1}{\Qfr_2}\right)
    -
    \Corolla\left(\Triangle{\Pfr_0}{\Rfr_1}{\Pfr_2}\right)
    \circ_1
    \Corolla\left(\Triangle{\Rfr_0}{\Qfr_1}{\Qfr_2}\right),
    \qquad
    \mbox{if } \Pfr_1 \Op \Qfr_0 = \Rfr_1 \Op \Rfr_0 \ne \Unit_\Mca,
\end{equation}
\begin{equation} \label{equ:relation_2_NC_M}
    \Corolla\left(\Triangle{\Pfr_0}{\Pfr_1}{\Pfr_2}\right)
    \circ_1
    \Corolla\left(\Triangle{\Qfr_0}{\Qfr_1}{\Qfr_2}\right)
    -
    \Corolla\left(\Triangle{\Pfr_0}{\Qfr_1}{\Rfr_2}\right)
    \circ_2
    \Corolla\left(\Triangle{\Rfr_0}{\Qfr_2}{\Pfr_2}\right),
    \qquad
    \mbox{if } \Pfr_1 \Op \Qfr_0 = \Rfr_2 \Op \Rfr_0 = \Unit_\Mca,
\end{equation}
\begin{equation} \label{equ:relation_3_NC_M}
    \Corolla\left(\Triangle{\Pfr_0}{\Pfr_1}{\Pfr_2}\right)
    \circ_2
    \Corolla\left(\Triangle{\Qfr_0}{\Qfr_1}{\Qfr_2}\right)
    -
    \Corolla\left(\Triangle{\Pfr_0}{\Pfr_1}{\Rfr_2}\right)
    \circ_2
    \Corolla\left(\Triangle{\Rfr_0}{\Qfr_1}{\Qfr_2}\right),
    \qquad
    \mbox{if } \Pfr_2 \Op \Qfr_0 = \Rfr_2 \Op \Rfr_0 \ne \Unit_\Mca,
\end{equation}
\end{subequations}
where $\Pfr$, $\Qfr$, and $\Rfr$ are $\Mca$-triangles.
\medskip

\begin{Lemma} \label{lem:quadratic_relations_NC_M}
    Let $\Mca$ be a unitary magma, and $\Sfr$ and $\Tfr$ be two syntax
    trees of arity $3$ on $\Triangles_\Mca$. Then, $\Sfr - \Tfr$ belongs
    to $\Rel_{\NC\Mca}$ if and only if $\Eval(\Sfr) = \Eval(\Tfr)$.
\end{Lemma}
\begin{proof}
    Assume first that $\Sfr - \Tfr$ belongs to $\Rel_{\NC\Mca}$. Then,
    $\Sfr - \Tfr$ is a linear combination of elements of the
    form~\eqref{equ:relation_1_NC_M}, \eqref{equ:relation_2_NC_M},
    and~\eqref{equ:relation_3_NC_M}. Now, observe that if $\Pfr$,
    $\Qfr$, and $\Rfr$ are three $\Mca$-triangles,
    \begin{enumerate}[fullwidth,label=(\alph*)]
        \item when
        $\delta := \Pfr_1 \Op \Qfr_0 = \Rfr_1 \Op \Rfr_0 \ne \Unit_\Mca$,
        we have
        \begin{equation}
            \Eval\left(
                \Corolla\left(\Triangle{\Pfr_0}{\Pfr_1}{\Pfr_2}\right)
                \circ_1
                \Corolla\left(\Triangle{\Qfr_0}{\Qfr_1}{\Qfr_2}\right)
            \right)
            =
            \SquareRight{\Qfr_1}{\Qfr_2}{\Pfr_2}{\Pfr_0}{\delta}
            =
            \Eval\left(
                \Corolla\left(\Triangle{\Pfr_0}{\Rfr_1}{\Pfr_2}\right)
                \circ_1
                \Corolla\left(\Triangle{\Rfr_0}{\Qfr_1}{\Qfr_2}\right)
            \right),
        \end{equation}
        \item when $\Pfr_1 \Op \Qfr_0 = \Rfr_2 \Op \Rfr_0 = \Unit_\Mca$,
        we have
        \begin{equation}
            \Eval\left(
                \Corolla\left(\Triangle{\Pfr_0}{\Pfr_1}{\Pfr_2}\right)
                \circ_1
                \Corolla\left(\Triangle{\Qfr_0}{\Qfr_1}{\Qfr_2}\right)
            \right)
            =
            \SquareN{\Qfr_1}{\Qfr_2}{\Pfr_2}{\Pfr_0}
            =
            \Eval\left(
                \Corolla\left(\Triangle{\Pfr_0}{\Qfr_1}{\Rfr_2}\right)
                \circ_2
                \Corolla\left(\Triangle{\Rfr_0}{\Qfr_2}{\Pfr_2}\right)
            \right),
        \end{equation}
        \item when
        $\delta := \Pfr_2 \Op \Qfr_0 = \Rfr_2 \Op \Rfr_0 \ne \Unit_\Mca$,
        we have
        \begin{equation}
            \Eval\left(
                \Corolla\left(\Triangle{\Pfr_0}{\Pfr_1}{\Pfr_2}\right)
                \circ_2
                \Corolla\left(\Triangle{\Qfr_0}{\Qfr_1}{\Qfr_2}\right)
            \right)
            =
            \SquareLeft{\Pfr_1}{\Qfr_1}{\Qfr_2}{\Pfr_0}{\delta}
            =
            \Eval\left(
                \Corolla\left(\Triangle{\Pfr_0}{\Pfr_1}{\Rfr_2}\right)
                \circ_2
                \Corolla\left(\Triangle{\Rfr_0}{\Qfr_1}{\Qfr_2}\right)
            \right).
        \end{equation}
    \end{enumerate}
    This shows that all evaluations in $\NC\Mca$
    of~\eqref{equ:relation_1_NC_M}, \eqref{equ:relation_2_NC_M},
    and~\eqref{equ:relation_3_NC_M} are equal to zero. Therefore,
    $\Eval(\Sfr - \Tfr) = 0$ and hence, one has
    $\Eval(\Sfr) - \Eval(\Tfr) = 0$ and, as expected,
    $\Eval(\Sfr) = \Eval(\Tfr)$.
    \smallskip

    Let us now assume that $\Eval(\Sfr) = \Eval(\Tfr)$ and let
    $\Rfr := \Eval(\Sfr)$. As $\Sfr$ is of arity $3$, $\Rfr$ also is
    of arity~$3$ and thus,
    \begin{equation}\label{equ:quadratic_relations_NC_M_demo_1}
        \Rfr \in \left\{
        \SquareRight{\Qfr_1}{\Qfr_2}{\Pfr_2}{\Pfr_0}{\delta},
        \SquareN{\Qfr_1}{\Qfr_2}{\Pfr_2}{\Pfr_0},
        \SquareLeft{\Pfr_1}{\Qfr_1}{\Qfr_2}{\Pfr_0}{\delta} :
        \Pfr, \Qfr \in \Triangles_\Mca,
        \delta \in \bar{\Mca}
        \right\}.
    \end{equation}
    Now, by definition of the partial composition of $\NC\Mca$, if
    $\Rfr$ has the form of the  first (resp. second, third) noncrossing
    $\Mca$-clique appearing
    in~\eqref{equ:quadratic_relations_NC_M_demo_1}, $\Sfr$ and $\Tfr$
    are respectively of the form of the first and second syntax trees
    of~\eqref{equ:relation_1_NC_M} (resp. \eqref{equ:relation_2_NC_M},
    \eqref{equ:relation_3_NC_M}). Hence, in all cases, $\Sfr - \Tfr$ is
    in~$\Rel_{\NC\Mca}$.
\end{proof}
\medskip

\begin{Proposition} \label{prop:dimensions_relations_NC_M}
    Let $\Mca$ be a finite unitary magma. Then, the dimension of the
    space $\Rel_{\NC\Mca}$ satisfies
    \begin{equation}
        \dim \Rel_{\NC\Mca} = 2m^6 - 2m^5 + m^4,
    \end{equation}
    where $m := \# \Mca$.
\end{Proposition}
\begin{proof}
    For any $x \in \Mca$, let $f(x)$ be the number of ordered pairs
    $(y, z) \in \Mca^2$ such that $x = y \Op z$. Since $\Mca$ is finite,
    $f : \Mca \to \N$ is a well-defined map.
    \smallskip

    Let $\RelEq$ be the equivalence relation on the set of the syntax
    trees on $\Triangles_\Mca$ of arity $3$ satisfying
    $\Sfr \RelEq \Tfr$ if $\Sfr$ and $\Tfr$ are two such syntax trees
    satisfying $\Eval(\Sfr) = \Eval(\Tfr)$. Let also $C$ be the set of
    all noncrossing $\Mca$-cliques of arity $3$. For any $\Rfr \in C$,
    we denote by $[\Rfr]_\RelEq$ the set of all syntax trees $\Sfr$
    satisfying $\Eval(\Sfr) = \Rfr$.
    Proposition~\ref{prop:generating_set_NC_M} says in particular that
    any $\Rfr \in C$ can be obtained by a partial composition of two
    $\Mca$-triangles, and hence, all $[\Rfr]_\RelEq$ are nonempty sets
    and thus, are $\RelEq$-equivalence classes.
    \smallskip

    Moreover, by Lemma~\ref{lem:quadratic_relations_NC_M}, for any
    syntax trees $\Sfr$ and $\Tfr$, one has $\Sfr \RelEq \Tfr$ if and
    only if $\Sfr - \Tfr$ is in $\Rel_{\NC\Mca}$. For this reason, the
    dimension of $\Rel_{\NC\Mca}$ is linked with the cardinalities of
    all $\RelEq$-equivalence classes by
    \begin{equation} \label{equ:dimensions_relations_NC_M_demo_1}
        \dim \Rel_{\NC\Mca} =
        \sum_{\Rfr \in C} \# [\Rfr]_\RelEq - 1.
    \end{equation}
    Let us compute~\eqref{equ:dimensions_relations_NC_M_demo_1} by
    enumerating each $\RelEq$-equivalence class $[\Rfr]_\RelEq$.
    \smallskip

    Observe that since $\Rfr$ is of arity $3$, it can be of three
    different forms according to the presence of a solid diagonal.
    \begin{enumerate}[fullwidth,label=(\alph*)]
        \item If
        \begin{equation}
            \Rfr = \SquareRight{\Qfr_1}{\Qfr_2}{\Pfr_2}{\Pfr_0}{\delta}
        \end{equation}
        for some $\Pfr_0, \Pfr_2, \Qfr_1, \Qfr_2 \in \Mca$ and
        $\delta \in \bar{\Mca}$, to have $\Sfr \in [\Rfr]_\RelEq$, we
        necessarily have
        \begin{equation}
            \Sfr =
            \Corolla\left(\Triangle{\Pfr_0}{\Pfr_1}{\Pfr_2}\right)
            \circ_1
            \Corolla\left(\Triangle{\Qfr_0}{\Qfr_1}{\Qfr_2}\right)
        \end{equation}
        where $\Pfr_1, \Qfr_0 \in \Mca$ and $\Pfr_1 \Op \Qfr_0 = \delta$.
        Hence, $\# [\Rfr]_\RelEq = f(\delta)$.
        \item If
        \begin{equation}
            \Rfr = \SquareN{\Qfr_1}{\Qfr_2}{\Pfr_2}{\Pfr_0}
        \end{equation}
        for some $\Pfr_0, \Pfr_2, \Qfr_1, \Qfr_2 \in \Mca$, to have
        $\Sfr \in [\Rfr]_\RelEq$, we necessarily have
        \begin{equation}
            \Sfr \in \left\{
            \Corolla\left(\Triangle{\Pfr_0}{\Pfr_1}{\Pfr_2}\right)
            \circ_1
            \Corolla\left(\Triangle{\Qfr_0}{\Qfr_1}{\Qfr_2}\right),
            \Corolla\left(\Triangle{\Pfr_0}{\Qfr_1}{\Rfr_2}\right)
            \circ_2
            \Corolla\left(\Triangle{\Rfr_0}{\Qfr_2}{\Pfr_2}\right)
            \right\}
        \end{equation}
        where $\Pfr_1, \Qfr_0, \Rfr_0, \Rfr_2 \in \Mca$,
        $\Pfr_1 \Op \Qfr_0 = \Unit_\Mca$, and
        $\Rfr_2 \Op \Rfr_0 = \Unit_\Mca$. Hence,
        $\# [\Rfr]_\RelEq = 2f(\Unit_\Mca)$.
        \item Otherwise,
        \begin{equation}
            \Rfr = \SquareLeft{\Pfr_1}{\Qfr_1}{\Qfr_2}{\Pfr_0}{\delta}
        \end{equation}
        for some $\Pfr_0, \Pfr_1, \Qfr_1, \Qfr_2 \in \Mca$ and
        $\delta \in \bar{\Mca}$, and to have $\Sfr \in [\Rfr]_\RelEq$,
        we necessarily have
        \begin{equation}
            \Sfr =
            \Corolla\left(\Triangle{\Pfr_0}{\Pfr_1}{\Pfr_2}\right)
            \circ_2
            \Corolla\left(\Triangle{\Qfr_0}{\Qfr_1}{\Qfr_2}\right)
        \end{equation}
        where $\Pfr_2, \Qfr_0 \in \Mca$ and $\Pfr_2 \Op \Qfr_0 = \delta$.
        Hence, $\# [\Rfr]_{\RelEq} = f(\delta)$.
    \end{enumerate}
    Therefore, by using the fact that
    \begin{equation}
        \sum_{\delta \in \Mca} f(\delta) = m^2,
    \end{equation}
    from~\eqref{equ:dimensions_relations_NC_M_demo_1} we obtain
    \begin{equation}\begin{split}
        \dim \Rel_{\NC\Mca} & =
        \left(
            \sum_{\substack{
                \Pfr_0, \Pfr_2, \Qfr_1, \Qfr_2 \in \Mca \\
                \delta \in \bar{\Mca}
            }}
            f(\delta) - 1
        \right)
        +
        \left(
            \sum_{\Pfr_0, \Pfr_2, \Qfr_1, \Qfr_2 \in \Mca}
            2f(\Unit_\Mca) - 1
        \right)
        +
        \left(
            \sum_{\substack{
                \Pfr_0, \Pfr_1, \Qfr_1, \Qfr_2 \in \Mca \\
                \delta \in \bar{\Mca}
            }}
            f(\delta) - 1
            \right) \\
        & = m^4
        \left(
            2 \left(
                \sum_{\delta \in \bar{\Mca}}
                f(\delta) - 1
            \right)
            + 2 f(\Unit_\Mca) - 1
        \right) \\
        & = 2m^6 - 2m^5 + m^4,
    \end{split}\end{equation}
    establishing the statement of the proposition.
\end{proof}
\medskip

Observe that, by Proposition~\ref{prop:dimensions_relations_NC_M}, the
dimension of $\Rel_{\NC\Mca}$ only depends on the cardinality of $\Mca$
and not on its operation~$\Op$.
\medskip

\subsubsection{Rewrite rule}
\label{subsubsec:rewrite_rule_NC_M}
Let $\Rew$ be the rewrite rule on the set of the syntax trees on
$\Triangles_\Mca$ satisfying
\begin{subequations}
\begin{equation} \label{equ:rewrite_1_NC_M}
    \Corolla\left(\Triangle{\Pfr_0}{\Pfr_1}{\Pfr_2}\right)
    \circ_1
    \Corolla\left(\Triangle{\Qfr_0}{\Qfr_1}{\Qfr_2}\right)
    \Rew
    \Corolla\left(\Triangle{\Pfr_0}{\delta}{\Pfr_2}\right)
    \circ_1
    \Corolla\left(\TriangleEXX{\Unit_\Mca}{\Qfr_1}{\Qfr_2}\right),
    \qquad
    \mbox{if } \Qfr_0 \ne \Unit_\Mca,
    \mbox{ where } \delta := \Pfr_1 \Op \Qfr_0,
\end{equation}
\begin{equation} \label{equ:rewrite_2_NC_M}
    \Corolla\left(\Triangle{\Pfr_0}{\Pfr_1}{\Pfr_2}\right)
    \circ_1
    \Corolla\left(\Triangle{\Qfr_0}{\Qfr_1}{\Qfr_2}\right)
    \Rew
    \Corolla\left(\TriangleXXE{\Pfr_0}{\Qfr_1}{\Unit_\Mca}\right)
    \circ_2
    \Corolla\left(\TriangleEXX{\Unit_\Mca}{\Qfr_2}{\Pfr_2}\right),
    \qquad
    \mbox{if } \Pfr_1 \Op \Qfr_0 = \Unit_\Mca,
\end{equation}
\begin{equation} \label{equ:rewrite_3_NC_M}
    \Corolla\left(\Triangle{\Pfr_0}{\Pfr_1}{\Pfr_2}\right)
    \circ_2
    \Corolla\left(\Triangle{\Qfr_0}{\Qfr_1}{\Qfr_2}\right)
    \Rew
    \Corolla\left(\Triangle{\Pfr_0}{\Pfr_1}{\delta}\right)
    \circ_2
    \Corolla\left(\TriangleEXX{\Unit_\Mca}{\Qfr_1}{\Qfr_2}\right),
    \qquad
    \mbox{if } \Qfr_0 \ne \Unit_\Mca,
    \mbox{ where } \delta := \Pfr_2 \Op \Qfr_0,
\end{equation}
\end{subequations}
where $\Pfr$ and $\Qfr$ are $\Mca$-triangles.
\medskip

\begin{Lemma} \label{lem:equivalence_relation_rewrite_rule_NC_M}
    Let $\Mca$ be a unitary magma. Then, the vector space induced by the
    rewrite rule $\Rew$ is $\Rel_{\NC\Mca}$.
\end{Lemma}
\begin{proof}
    Let $\Sfr$ and $\Tfr$ be two syntax trees on $\Triangles_\Mca$ such
    that $\Sfr \Rew \Tfr$. We have three cases to consider depending on
    the form of $\Sfr$ and~$\Tfr$.
    \begin{enumerate}[fullwidth,label=(\alph*)]
        \item if $\Sfr$ (resp. $\Tfr$) is of the form described by the
        left (resp. right) member of~\eqref{equ:rewrite_1_NC_M}, we have
        \begin{equation}
            \Eval(\Sfr)
            = \SquareRight{\Qfr_1}{\Qfr_2}{\Pfr_2}{\Pfr_0}{\delta}
            = \Eval(\Tfr),
        \end{equation}
        where $\delta := \Pfr_1 \Op \Qfr_0$.
        \item If $\Sfr$ (resp. $\Tfr$) is of the form described by the
        left (resp. right) member of~\eqref{equ:rewrite_2_NC_M}, we have
        \begin{equation}
            \Eval(\Sfr)
            = \SquareN{\Qfr_1}{\Qfr_2}{\Pfr_2}{\Pfr_0}
            = \Eval(\Tfr).
        \end{equation}
        \item Otherwise, $\Sfr$ (resp. $\Tfr$) is of the form described
        by the left (resp. right) member of~\eqref{equ:rewrite_3_NC_M}.
        We have
        \begin{equation}
            \Eval(\Sfr)
            = \SquareLeft{\Pfr_1}{\Qfr_1}{\Qfr_2}{\Pfr_0}{\delta}
            = \Eval(\Tfr),
        \end{equation}
        where $\delta := \Pfr_2 \Op \Qfr_0$.
    \end{enumerate}
    Therefore, by Lemma~\ref{lem:quadratic_relations_NC_M} we have
    $\Sfr - \Tfr \in \Rel_{\NC\Mca}$ for each case. This leads to the
    fact that $\Sfr \RewTransSym \Tfr$ implies
    $\Sfr - \Tfr \in \Rel_{\NC\Mca}$, and shows that the space induced
    by $\Rew$ is a subspace of~$\Rel_{\NC\Mca}$.
    \smallskip

    Let us now assume that $\Sfr$ and $\Tfr$ are two syntax trees on
    $\Triangles_\Mca$ such that $\Sfr - \Tfr$ is a generator of
    $\Rel_{\NC\Mca}$ among~\eqref{equ:relation_1_NC_M},
    \eqref{equ:relation_2_NC_M}, and~\eqref{equ:relation_3_NC_M}.
    \begin{enumerate}[fullwidth,label=(\alph*)]
        \item If $\Sfr$ (resp. $\Tfr$) is of the form described by the
        left (resp. right) member of~\eqref{equ:relation_1_NC_M}, we
        have by~\eqref{equ:rewrite_1_NC_M},
        \begin{equation}
            \Sfr \RewTrans
            \Corolla\left(\Triangle{\Pfr_0}{\delta}{\Pfr_2}\right)
            \circ_1
            \Corolla\left(\TriangleEXX{\Unit_\Mca}{\Qfr_1}{\Qfr_2}\right)
            \quad \mbox{and} \quad
            \Tfr \RewTrans
            \Corolla\left(\Triangle{\Pfr_0}{\delta'}{\Pfr_2}\right)
            \circ_1
            \Corolla\left(\TriangleEXX{\Unit_\Mca}{\Qfr_1}{\Qfr_2}\right),
        \end{equation}
        where $\delta := \Pfr_1 \Op \Qfr_0$ and
        $\delta' := \Rfr_1 \Op \Rfr_0$. Since
        by~\eqref{equ:relation_1_NC_M}, $\delta = \delta'$, we obtain
        that $\Sfr \RewTransSym \Tfr$.
        \item If $\Sfr$ (resp. $\Tfr$) is of the form described by the
        left (resp. right) member of~\eqref{equ:relation_2_NC_M}, we
        have by~\eqref{equ:rewrite_2_NC_M} and
        by~\eqref{equ:rewrite_3_NC_M},
        \begin{equation}
            \Sfr \Rew
            \Corolla\left(\TriangleXXE{\Pfr_0}{\Qfr_1}{\Unit_\Mca}\right)
            \circ_2
            \Corolla\left(\TriangleEXX{\Unit_\Mca}{\Qfr_2}{\Pfr_2}\right)
            \quad \mbox{and} \quad
            \Tfr \RewTrans
            \Corolla\left(\TriangleXXE{\Pfr_0}{\Qfr_1}{\Unit_\Mca}\right)
            \circ_2
            \Corolla\left(\TriangleEXX{\Unit_\Mca}{\Qfr_2}{\Pfr_2}\right).
        \end{equation}
        We obtain that $\Sfr \RewTransSym \Tfr$.
        \item Otherwise, $\Sfr$ (resp. $\Tfr$) is of the form described
        by the left (resp. right) member of~\eqref{equ:relation_3_NC_M}.
        We have by~\eqref{equ:rewrite_3_NC_M},
        \begin{equation}
            \Sfr \RewTrans
            \Corolla\left(\Triangle{\Pfr_0}{\Pfr_1}{\delta}\right)
            \circ_2
            \Corolla\left(\TriangleEXX{\Unit_\Mca}{\Qfr_1}{\Qfr_2}\right)
            \quad \mbox{and} \quad
            \Tfr \RewTrans
            \Corolla\left(\Triangle{\Pfr_0}{\Pfr_1}{\delta'}\right)
            \circ_2
            \Corolla\left(\TriangleEXX{\Unit_\Mca}{\Qfr_1}{\Qfr_2}\right),
        \end{equation}
        where $\delta := \Pfr_2 \Op \Qfr_0$ and
        $\delta' := \Rfr_2 \Op \Rfr_0$. Since
        by~\eqref{equ:relation_3_NC_M}, $\delta = \delta'$, we obtain
        that $\Sfr \RewTransSym \Tfr$.
    \end{enumerate}
    Hence, for each case, we have $\Sfr \RewTransSym \Tfr$. This shows
    that $\Rel_{\NC\Mca}$ is a subspace of the space induced by $\Rew$.
    The statement of the lemma follows.
\end{proof}
\medskip

\begin{Lemma} \label{lem:rewrite_rule_NC_M_terminating}
    For any unitary magma $\Mca$, the rewrite rule $\Rew$ is terminating.
\end{Lemma}
\begin{proof}
    By denoting by $T_n$ the set of all syntax trees on
    $\Triangles_\Mca$ of arity $n$, let $\phi : T_n \to \N^2$ be the map
    defined in the following way. For any syntax tree $\Tfr$ of $T_n$,
    $\phi(\Tfr) := (\alpha, \beta)$ where $\alpha$ is the sum, for all
    internal nodes $x$ of $\Tfr$, of the number of internal nodes in the
    left subtree of $x$, and $\beta$ is the number of internal nodes of
    $\Tfr$ labeled by an $\Mca$-triangle whose base is not labeled by
    $\Unit_\Mca$. Let $\Sfr$ and $\Tfr$ be two syntax trees of $T_3$
    such that $\Sfr \Rew \Tfr$. Due to the definition of $\Rew$, we have
    three configurations to explore. In what follows,
    $\eta : \Mca \to \N$ is the map satisfying $\eta(a) := 0$ if
    $a = \Unit_\Mca$ and $\eta(a) := 1$ otherwise.
    \begin{enumerate}[fullwidth,label=(\alph*)]
        \item If $\Sfr$ (resp. $\Tfr$) is of the form described by the
        left (resp. right) member of~\eqref{equ:rewrite_1_NC_M}, we
        have, by denoting by $\leq$ the lexicographic order on $\N^2$,
        \begin{equation}\begin{split}
            \phi\left(
            \Corolla\left(\Triangle{\Pfr_0}{\Pfr_1}{\Pfr_2}\right)
            \circ_1
            \Corolla\left(\Triangle{\Qfr_0}{\Qfr_1}{\Qfr_2}\right)
            \right)
            & =
            \left(1, \eta(\Pfr_0) + 1\right) \\
            & >
            \left(1, \eta(\Pfr_0)\right)
            =
            \phi\left(
            \Corolla\left(\Triangle{\Pfr_0}{\delta}{\Pfr_2}\right)
            \circ_1
            \Corolla\left(\TriangleEXX{\Unit_\Mca}{\Qfr_1}{\Qfr_2}\right)
            \right),
        \end{split}\end{equation}
        where $\delta := \Pfr_1 \Op \Qfr_0$.
        \item If $\Sfr$ (resp. $\Tfr$) is of the form described by the
        left (resp. right) member of~\eqref{equ:rewrite_2_NC_M}, we
        have
        \begin{equation}\begin{split}
            \phi\left(
            \Corolla\left(\Triangle{\Pfr_0}{\Pfr_1}{\Pfr_2}\right)
            \circ_1
            \Corolla\left(\Triangle{\Qfr_0}{\Qfr_1}{\Qfr_2}\right)
            \right)
            & =
            \left(1, \eta(\Pfr_0) + \eta(\Qfr_0)\right) \\
            & >
            \left(0, \eta(\Pfr_0)\right)
            =
            \phi\left(
            \Corolla\left(\TriangleXXE{\Pfr_0}{\Qfr_1}{\Unit_\Mca}\right)
            \circ_2
            \Corolla\left(\TriangleEXX{\Unit_\Mca}{\Qfr_2}{\Pfr_2}\right)
            \right).
        \end{split}\end{equation}
        \item Otherwise, $\Sfr$ (resp. $\Tfr$) is of the form described
        by the left (resp. right) member of~\eqref{equ:rewrite_3_NC_M}.
        We have
        \begin{equation}\begin{split}
            \phi\left(
            \Corolla\left(\Triangle{\Pfr_0}{\Pfr_1}{\Pfr_2}\right)
            \circ_2
            \Corolla\left(\Triangle{\Qfr_0}{\Qfr_1}{\Qfr_2}\right)
            \right)
            & =
            \left(0, \eta(\Pfr_0) + 1\right) \\
            & >
            \left(0, \eta(\Pfr_0)\right)
            =
            \phi\left(
            \Corolla\left(\Triangle{\Pfr_0}{\Pfr_1}{\delta}\right)
            \circ_2
            \Corolla\left(\TriangleEXX{\Unit_\Mca}{\Qfr_1}{\Qfr_2}\right)
            \right),
        \end{split}\end{equation}
        where $\delta := \Pfr_2 \Op \Qfr_0$.
    \end{enumerate}
    Therefore, for all syntax trees $\Sfr$ and $\Tfr$ such that
    $\Sfr \Rew \Tfr$, $\phi(\Sfr) > \phi(\Tfr)$. This implies
    that for all syntax trees $\Sfr$ and $\Tfr$ such that
    $\Sfr \ne \Tfr$ and $\Sfr \RewTrans \Tfr$,
    $\phi(\Sfr) > \phi(\Tfr)$. Since $(0, 0)$ is the smallest element of
    $\N^2$ with respect to the lexicographic order $\leq$, the statement
    of the lemma follows.
\end{proof}
\medskip

\begin{Lemma} \label{lem:rewrite_rule_NC_M_normal_forms}
    Let $\Mca$ be a unitary magma. The set of the normal forms of the
    rewrite rule $\Rew$ is the set of the syntax trees $\Tfr$ on
    $\Triangles_\Mca$ such that, for any internal nodes $x$ and $y$ of
    $\Tfr$ where $y$ is a child of $x$,
    \begin{enumerate}[fullwidth,label={(\it\roman*)}]
        \item \label{item:rewrite_rule_NC_M_normal_forms_1}
        the base of the $\Mca$-triangle labeling $y$ is labeled
        by~$\Unit_\Mca$;
        \item \label{item:rewrite_rule_NC_M_normal_forms_2}
        if $y$ is a left child of $x$, the first edge of the
        $\Mca$-triangle labeling $x$ is not labeled by~$\Unit_\Mca$.
    \end{enumerate}
\end{Lemma}
\begin{proof}
    By Lemma~\ref{lem:rewrite_rule_NC_M_terminating}, $\Rew$ is
    terminating. Therefore, $\Rew$ admits normal forms, which are by
    definition the syntax trees on $\Triangles_\Mca$ that cannot be
    rewritten by~$\Rew$.
    \smallskip

    Let $\Tfr$ be a normal form of $\Rew$. The fact that $\Tfr$
    satisfies~\ref{item:rewrite_rule_NC_M_normal_forms_1} is an
    immediate consequence of the fact that $\Tfr$ avoids the patterns
    appearing as left members of~\eqref{equ:rewrite_1_NC_M}
    and~\eqref{equ:rewrite_3_NC_M}. Moreover, since $\Tfr$ avoids the
    patterns appearing as left members of~\eqref{equ:rewrite_2_NC_M},
    one cannot have $\Pfr_1 \Op \Qfr_0 = \Unit_\Mca$, where $\Pfr$
    (resp. $\Qfr$) is the label of $x$ (resp. $y$). Since
    by~\ref{item:rewrite_rule_NC_M_normal_forms_1},
    $\Qfr_0 = \Unit_\Mca$, we necessarily have $\Pfr_1 \ne \Unit_\Mca$.
    Hence, $\Tfr$ satisfies~\ref{item:rewrite_rule_NC_M_normal_forms_2}.
    \smallskip

    Conversely, if $\Tfr$ is a syntax tree on $\Triangles_\Mca$
    satisfying~\ref{item:rewrite_rule_NC_M_normal_forms_1}
    and~\ref{item:rewrite_rule_NC_M_normal_forms_2}, a direct inspection
    shows that one cannot rewrite $\Tfr$ by $\Rew$. Therefore, $\Tfr$ is
    a normal form of~$\Rew$.
\end{proof}
\medskip

\begin{Lemma} \label{lem:generating_series_normal_forms_rewrite_rule_NC_M}
    Let $\Mca$ be a finite unitary magma. The generating series of the
    normal forms of the rewrite rule $\Rew$ is the Hilbert series
    $\Hilbert_{\NC\Mca(t)}$ of~$\NC\Mca$.
\end{Lemma}
\begin{proof}
    First, since by Lemma~\ref{lem:rewrite_rule_NC_M_terminating},
    $\Rew$ is terminating, and since for any $n \geq 1$, due to the
    finiteness of $\Mca$, there are finitely many syntax trees on
    $\Triangles_\Mca$ of arity $n$, the generating series $T(t)$ of the
    normal forms of $\Rew$ is well-defined.
    \smallskip

    Let $S(t)$ be the generating series of the normal forms of $\Rew$
    such that the bases of the $\Mca$-triangles labeling the roots are
    labeled by $\Unit_\Mca$. Immediately from the description of the
    normal forms of $\Rew$ provided by
    Lemma~\ref{lem:rewrite_rule_NC_M_normal_forms}, we obtain that
    $S(t)$ satisfies
    \begin{equation}
        S(t) = t + mtS(t) + (m - 1)m S(t)^2.
    \end{equation}
    Again by
    Lemma~\ref{lem:rewrite_rule_NC_M_normal_forms}, we have
    \begin{equation}
        T(t) = t + m(S(t) - t).
    \end{equation}
    A direct computation shows that $T(t)$ satisfies the algebraic
    equation
    \begin{equation}
        t + \left(m^3 - 2m^2 + 2m - 1\right)t^2
        + \left(2m^2t - 3mt + 2t - 1\right) T(t)
        + \left(m - 1\right) T(t)^2
        = 0.
    \end{equation}
    Hence, by Proposition~\ref{prop:Hilbert_series_NC_M}, we observe
    that $T(t) = \Hilbert_{\NC\Mca(t)}$.
\end{proof}
\medskip

\begin{Lemma} \label{lem:rewrite_rule_NC_M_confluent}
    For any finite unitary magma $\Mca$, the rewrite rule $\Rew$ is
    confluent.
\end{Lemma}
\begin{proof}
    By contradiction, assume that $\Rew$ is not confluent. Since by
    Lemma~\ref{lem:rewrite_rule_NC_M_terminating}, $\Rew$ is
    terminating, there is an integer $n \geq 1$ and two normal forms
    $\Tfr$ and $\Tfr'$ of $\Rew$ of arity $n$ such that $\Tfr \ne \Tfr'$
    and $\Tfr \CRewTransSym \Tfr'$. Now,
    Lemma~\ref{lem:quadratic_relations_NC_M} together with
    Lemma~\ref{lem:equivalence_relation_rewrite_rule_NC_M} imply that
    $\Eval(\Tfr) = \Eval(\Tfr')$. By
    Proposition~\ref{prop:generating_set_NC_M}, the map
    \begin{math}
        \Eval :
        \Free\left(\Vect\left(\Triangles_\Mca\right)\right)
        \to \NC\Mca
    \end{math}
    is surjective, leading to the fact that the number of normal forms
    of $\Rew$ of arity $n$ is greater than the number of noncrossing
    $\Mca$-cliques of arity $n$. However, by
    Lemma~\ref{lem:generating_series_normal_forms_rewrite_rule_NC_M},
    there are as many normal forms of $\Rew$ of arity $n$ as noncrossing
    $\Mca$-cliques of arity $n$. This raises a contradiction and proves
    the statement of the lemma.
\end{proof}
\medskip

\subsubsection{Presentation and Koszulity}
The results of Sections~\ref{subsubsec:space_of_relations_NC_M}
and~\ref{subsubsec:rewrite_rule_NC_M} are finally used here to state a
presentation of $\NC\Mca$ and the fact that $\NC\Mca$ is a Koszul
operad.
\medskip

\begin{Theorem} \label{thm:presentation_NC_M}
    Let $\Mca$ be a finite unitary magma. Then, $\NC\Mca$ admits the
    presentation~$\left(\Triangles_\Mca, \Rel_{\NC\Mca}\right)$.
\end{Theorem}
\begin{proof}
    First, since by Lemmas~\ref{lem:rewrite_rule_NC_M_terminating}
    and~\ref{lem:rewrite_rule_NC_M_confluent}, $\Rew$ is a convergent
    rewrite rule, and since by
    Lemma~\ref{lem:equivalence_relation_rewrite_rule_NC_M}, the space
    induced by $\Rew$ is $\Rel_{\NC\Mca}$, we can regard the underlying
    space of the quotient operad
    \begin{equation}
        \Oca :=
        \Free\left(
        \Vect\left(\Triangles_\Mca\right)\right)
        /_{\langle\Rel_{\NC\Mca}\rangle}
    \end{equation}
    as the linear span of all normal forms of $\Rew$. Moreover, as a
    consequence of Lemma~\ref{lem:quadratic_relations_NC_M}, the map
    \begin{math}
        \phi : \Oca \to \NC\Mca
    \end{math}
    defined linearly for any normal form $\Tfr$ of $\Rew$ by
    $\phi(\Tfr) := \Eval(\Tfr)$ is an operad morphism. Now, by
    Proposition~\ref{prop:generating_set_NC_M}, $\phi$ is surjective.
    Moreover, by
    Lemma~\ref{lem:generating_series_normal_forms_rewrite_rule_NC_M},
    we obtain that the dimensions of the spaces $\Oca(n)$, $n \geq 1$,
    are the ones of $\NC\Mca(n)$. Hence, $\phi$ is an operad isomorphism
    and the statement of the theorem follows.
\end{proof}
\medskip

Let us use Theorem~\ref{thm:presentation_NC_M} to express the
presentations of the operads $\NC\N_2$ and $\NC\Dbb_0$.
The operad $\NC\N_2$ is generated by
\begin{equation}
    \Triangles_{\N_2} =
    \left\{
    \TriangleEEE{}{}{},
    \TriangleEXE{}{1}{},
    \TriangleEEX{}{}{1},
    \TriangleEXX{}{1}{1},
    \TriangleXEE{1}{}{},
    \TriangleXXE{1}{1}{},
    \TriangleXEX{1}{}{1},
    \Triangle{1}{1}{1}
    \right\},
\end{equation}
and these generators are subjected exactly to the nontrivial relations
\begin{subequations}
\begin{equation}
    \TriangleXEX{a}{}{b_3}
    \circ_1
    \Triangle{1}{b_1}{b_2}
    =
    \Triangle{a}{1}{b_3}
    \circ_1
    \TriangleEXX{}{b_1}{b_2},
    \qquad
    a, b_1, b_2, b_3 \in \N_2,
\end{equation}
\begin{equation}
    \Triangle{a}{1}{b_3}
    \circ_1
    \Triangle{1}{b_1}{b_2}
    =
    \TriangleXEX{a}{}{b_3}
    \circ_1
    \TriangleEXX{}{b_1}{b_2}
    =
    \TriangleXXE{a}{b_1}{}
    \circ_2
    \TriangleEXX{}{b_2}{b_3}
    =
    \Triangle{a}{b_1}{1}
    \circ_2
    \Triangle{1}{b_2}{b_3},
    \qquad
    a, b_1, b_2, b_3 \in \N_2,
\end{equation}
\begin{equation}
    \TriangleXXE{a}{b_1}{}
    \circ_2
    \Triangle{1}{b_2}{b_3}
    =
    \Triangle{a}{b_1}{1}
    \circ_2
    \TriangleEXX{}{b_2}{b_3},
    \qquad
    a, b_1, b_2, b_3 \in \N_2.
\end{equation}
\end{subequations}
On the other hand, the operad $\NC\Dbb_0$ is generated by
\begin{equation}
    \Triangles_{\Dbb_0} =
    \left\{
    \TriangleEEE{}{}{},
    \TriangleEXE{}{0}{},
    \TriangleEEX{}{}{0},
    \TriangleEXX{}{0}{0},
    \TriangleXEE{0}{}{},
    \TriangleXXE{0}{0}{},
    \TriangleXEX{0}{}{0},
    \Triangle{0}{0}{0}
    \right\},
\end{equation}
and these generators are subjected exactly to the nontrivial relations
\begin{subequations}
\begin{equation}
    \TriangleXEX{a}{}{b_3}
    \circ_1
    \Triangle{0}{b_1}{b_2}
    =
    \Triangle{a}{0}{b_3}
    \circ_1
    \Triangle{0}{b_1}{b_2}
    =
    \Triangle{a}{0}{b_3}
    \circ_1
    \TriangleEXX{}{b_1}{b_2},
    \qquad
    a, b_1, b_2, b_3 \in \Dbb_0,
\end{equation}
\begin{equation}
    \TriangleXEX{a}{}{b_3}
    \circ_1
    \TriangleEXX{}{b_1}{b_2}
    =
    \TriangleXXE{a}{b_1}{}
    \circ_2
    \TriangleEXX{}{b_2}{b_3},
    \qquad
    a, b_1, b_2, b_3 \in \Dbb_0,
\end{equation}
\begin{equation}
    \TriangleXXE{a}{b_1}{}
    \circ_2
    \Triangle{0}{b_2}{b_3}
    =
    \Triangle{a}{b_1}{0}
    \circ_2
    \Triangle{0}{b_2}{b_3}
    =
    \Triangle{a}{b_1}{0}
    \circ_2
    \TriangleEXX{}{b_2}{b_3},
    \qquad
    a, b_1, b_2, b_3 \in \Dbb_0.
\end{equation}
\end{subequations}
\medskip

\begin{Theorem} \label{thm:Koszul_NC_M}
    For any finite unitary magma $\Mca$, $\NC\Mca$ is Koszul and the set
    of the normal forms of $\Rew$ forms a Poincaré-Birkhoff-Witt basis
    of~$\NC\Mca$.
\end{Theorem}
\begin{proof}
    By Lemma~\ref{lem:equivalence_relation_rewrite_rule_NC_M} and
    Theorem~\ref{thm:presentation_NC_M}, the rewrite rule $\Rew$ is an
    orientation of the space of relations $\Rel_{\NC\Mca}$ of
    $\NC\Mca$. Moreover, by
    Lemmas~\ref{lem:rewrite_rule_NC_M_terminating}
    and~\ref{lem:rewrite_rule_NC_M_confluent}, this rewrite rule is
    convergent. Therefore, by Lemma~\ref{lem:koszulity_criterion_pbw},
    $\NC\Mca$ is Koszul. Finally, the set of the normal forms of $\Rew$
    described by Lemma~\ref{lem:rewrite_rule_NC_M_normal_forms} is, by
    definition, a Poincaré-Birkhoff-Witt basis of~$\NC\Mca$.
\end{proof}
\medskip

\subsection{Suboperads generated by bubbles}
In this section, we consider suboperads of $\NC\Mca$ generated by finite
sets of $\Mca$-bubbles. We assume here that $\Mca$ is endowed with an
arbitrary total order so that $\Mca = \{x_0, x_1, \dots\}$ with
$x_0 = \Unit_\Mca$.
\medskip

\subsubsection{Treelike expressions on bubbles}
Let $B$ and $E$ be two subsets of $\Mca$. We denote by
$\Bubbles_\Mca^{B, E}$ the set of all $\Mca$-bubbles $\Pfr$ such that
the bases of $\Pfr$ are labeled on $B$ and all edges of $\Pfr$ are
labeled on $E$. Moreover, we say that $\Mca$ is
\Def{$(E, B)$-quasi-injective} if for all $x, x' \in E$ and
$y, y' \in B$, $x \Op y = x' \Op y' \ne \Unit_\Mca$ implies $x = x'$ and
$y = y'$.
\medskip

\begin{Lemma} \label{lem:treelike_expression_suboperad_bubbles}
    Let $\Mca$ be a unitary magma, and $B$ and $E$ be two subsets of
    $\Mca$. If $\Mca$ is $(E, B)$-quasi-injective, then any
    $\Mca$-clique admits at most one treelike expression on
    $\Bubbles_\Mca^{B, E}$ of a minimal degree.
\end{Lemma}
\begin{proof}
    Assume that $\Pfr$ is an $\Mca$-clique admitting a treelike
    expression on $\Bubbles_\Mca^{B, E}$. This implies that the base of
    $\Pfr$ is labeled on $B$, all solid diagonals of $\Pfr$ are labeled
    on $B \Op E$, and all edges of $\Pfr$ are labeled on $E$. By
    Proposition~\ref{prop:unique_decomposition_NC_M} and
    Lemma~\ref{lem:map_NC_M_bubble_tree_treelike_expression}, the tree
    $\Tfr := \BubbleTree(\Pfr)$ is a treelike expression of $\Pfr$ on
    $\Bubbles_\Mca$ of a minimal degree. Now, observe that $\Tfr$ is not
    necessarily a syntax tree on $\Bubbles_\Mca^{B, E}$ as required
    since some of its internal nodes can be labeled by bubbles that do
    not belong to $\Bubbles_\Mca^{B, E}$. Since $\Mca$ is
    $(E, B)$-quasi-injective, there is one unique way to relabel the
    internal nodes of $\Tfr$ by bubbles of $\Bubbles_\Mca^{B, E}$ to
    obtain a syntax tree on $\Bubbles_\Mca^{B, E}$ such that
    $\Eval(\Tfr') = \Eval(\Tfr)$. By construction, $\Tfr'$ satisfies the
    properties of the statement of the lemma.
\end{proof}
\medskip

\subsubsection{Dimensions}%
\label{subsubsec:dimensions_suboperads_triangles}
Let $G$ be a set of $\Mca$-bubbles and
$\Xi := \left\{\xi_{x_0}, \xi_{x_1}, \dots\right\}$ be a set of
noncommutative variables. Given $x_i \in \Mca$, let
$\SeriesBubbles_{x_i}$ be the series of
$\N \langle\langle \Xi \rangle\rangle$ defined by
\begin{equation} \label{equ:series_bubbles}
    \SeriesBubbles_{x_i}\left(\xi_{x_0}, \xi_{x_1}, \dots\right) :=
    \sum_{\substack{
        \Pfr \in \Bubbles_\Mca^G \\
        \Pfr \ne \UnitClique}}
    \enspace
    \prod_{i \in [|\Pfr|]}
    \xi_{\Pfr_i},
\end{equation}
where $\Bubbles_\Mca^G$ is the set of all $\Mca$-bubbles that can be
obtained by partial compositions of elements of $G$. Observe
from~\eqref{equ:series_bubbles} that a noncommutative monomial
$u \in \Xi^{\geq 2}$ appears in $\SeriesBubbles_{x_i}$ with $1$ as
coefficient if and only if there is in the suboperad of $\NC\Mca$
generated by $G$ an $\Mca$-bubble with a base labeled by $x_i$ and with
$u$ as border.
\medskip

Let also for any $x_i \in \Mca$, the series $\SeriesElements_{x_i}$ of
$\N \langle\langle t \rangle\rangle$ defined by
\begin{equation}
    \SeriesElements_{x_i}(t) :=
    \SeriesBubbles_{x_i}\left(t + \bar \SeriesElements_{x_0}(t),
    t + \bar \SeriesElements_{x_1}(t), \dots\right),
\end{equation}
where for any $x_i \in \Mca$,
\begin{equation} \label{equ:bar_series_elements_based}
    \bar \SeriesElements_{x_i}(t) :=
    \sum_{\substack{
        x_j \in \Mca \\
        x_i \Op x_j \ne \Unit_\Mca
    }}
    \SeriesElements_{x_j}(t).
\end{equation}
\medskip

\begin{Proposition} \label{prop:suboperads_NC_M_triangles_dimensions}
    Let $\Mca$ be a unitary magma and $G$ be a finite set of
    $\Mca$-bubbles such that, by denoting by $B$ (resp. $E$) the set of
    the labels of the bases (resp. edges) of the elements of $G$, $\Mca$
    is $(E, B)$-quasi-injective. Then, the Hilbert series
    $\Hilbert_{(\NC\Mca)^G}(t)$ of the suboperad of $\NC\Mca$ generated
    by $G$ satisfies
    \begin{equation} \label{equ:suboperads_NC_M_triangles_dimensions}
        \Hilbert_{(\NC\Mca)^G}(t) =
        t +
        \sum_{x_i \in \Mca} \SeriesElements_{x_i}(t).
    \end{equation}
\end{Proposition}
\begin{proof}
    By Lemma~\ref{lem:treelike_expression_suboperad_bubbles}, any
    $\Mca$-clique of $(\NC\Mca)^G$ admits exactly one treelike
    expression on $\Mca$-bubbles of $(\NC\Mca)^G$ of a minimal degree.
    For this reason, and as a consequence of the
    definition~\eqref{equ:bar_series_elements_based} of the series
    $\bar \SeriesElements_{x_i}(t)$, $x_i \in \Mca$, the series
    $\SeriesElements_{x_i}(t)$ is the generating series of all
    $\Mca$-cliques of $(\NC\Mca)^G$ different from $\UnitClique$ and
    with a base labeled by $x_i \in \Mca$. Therefore, the
    expression~\eqref{equ:suboperads_NC_M_triangles_dimensions} for the
    Hilbert series of $(\NC\Mca)^G$ follows.
\end{proof}
\medskip

As a side remark,
Proposition~\ref{prop:suboperads_NC_M_triangles_dimensions} can be
proved by using the notion of bubble decompositions of operads developed
in~\cite{CG14}. This result provides a practical method to compute the
dimensions of some suboperads $(\NC\Mca)^G$ of $\NC\Mca$ by describing
the series~\eqref{equ:series_bubbles} of the bubbles of
$\Bubbles_\Mca^G$. This result implies also, when $G$ satisfies the
requirement of
Proposition~\ref{prop:suboperads_NC_M_triangles_dimensions}, that the
Hilbert series of $(\NC\Mca)^G$ is algebraic.
\medskip

\subsubsection{First example~: a cubic suboperad}
Consider the suboperad of $\NC\Ebb_2$ generated by
\begin{equation}
    G := \left\{
        \TriangleXEX{\Ett_1}{}{\Ett_1},
        \TriangleXEX{\Ett_2}{}{\Ett_2}
    \right\}.
\end{equation}
Computer experiments show that the generators of $(\NC\Ebb_2)^G$ are
not subjected to any quadratic relation but are subjected to the
four cubic nontrivial relations
\begin{subequations}
\begin{equation} \label{equ:rel_suboperad_1_1}
    \TriangleXEX{\Ett_1}{}{\Ett_1}
    \circ_2
    \left(
    \TriangleXEX{\Ett_1}{}{\Ett_1}
    \circ_2
    \TriangleXEX{\Ett_1}{}{\Ett_1}
    \right)
    =
    \TriangleXEX{\Ett_1}{}{\Ett_1}
    \circ_2
    \left(
    \TriangleXEX{\Ett_2}{}{\Ett_2}
    \circ_2
    \TriangleXEX{\Ett_1}{}{\Ett_1}
    \right),
\end{equation}
\begin{equation} \label{equ:rel_suboperad_1_2}
    \TriangleXEX{\Ett_1}{}{\Ett_1}
    \circ_2
    \left(
    \TriangleXEX{\Ett_1}{}{\Ett_1}
    \circ_2
    \TriangleXEX{\Ett_2}{}{\Ett_2}
    \right)
    =
    \TriangleXEX{\Ett_1}{}{\Ett_1}
    \circ_2
    \left(
    \TriangleXEX{\Ett_2}{}{\Ett_2}
    \circ_2
    \TriangleXEX{\Ett_2}{}{\Ett_2}
    \right),
\end{equation}
\begin{equation} \label{equ:rel_suboperad_1_3}
    \TriangleXEX{\Ett_2}{}{\Ett_2}
    \circ_2
    \left(
    \TriangleXEX{\Ett_1}{}{\Ett_1}
    \circ_2
    \TriangleXEX{\Ett_1}{}{\Ett_1}
    \right)
    =
    \TriangleXEX{\Ett_2}{}{\Ett_2}
    \circ_2
    \left(
    \TriangleXEX{\Ett_2}{}{\Ett_2}
    \circ_2
    \TriangleXEX{\Ett_1}{}{\Ett_1}
    \right),
\end{equation}
\begin{equation} \label{equ:rel_suboperad_1_4}
    \TriangleXEX{\Ett_2}{}{\Ett_2}
    \circ_2
    \left(
    \TriangleXEX{\Ett_1}{}{\Ett_1}
    \circ_2
    \TriangleXEX{\Ett_2}{}{\Ett_2}
    \right)
    =
    \TriangleXEX{\Ett_2}{}{\Ett_2}
    \circ_2
    \left(
    \TriangleXEX{\Ett_2}{}{\Ett_2}
    \circ_2
    \TriangleXEX{\Ett_2}{}{\Ett_2}
    \right).
\end{equation}
\end{subequations}
Hence, $(\NC\Ebb_2)^G$ is not a quadratic operad. Moreover, it is
possible to prove that this operad does not admit any other nontrivial
relations between its generators. This can be performed by defining
a rewrite rule on the syntax trees on $G$, consisting in rewriting
the left patterns of~\eqref{equ:rel_suboperad_1_1},
\eqref{equ:rel_suboperad_1_2}, \eqref{equ:rel_suboperad_1_3},
and~\eqref{equ:rel_suboperad_1_4} into their respective right patterns,
and by checking that this rewrite rule admits the required properties
(like the ones establishing the presentation of $\NC\Mca$ by
Theorem~\ref{thm:presentation_NC_M}). The existence of this nonquadratic
operad shows that $\NC\Mca$ contains nonquadratic suboperads
even if it is quadratic.
\medskip

One can prove by induction on the arity that the set of bubbles of
$(\NC\Ebb_2)^G$ is the set $B_1 \sqcup B_2$ where $B_1$ (resp. $B_2$)
is the set of all  bubbles whose bases are labeled by $\Ett_1$ (resp.
$\Ett_2$) and the border is $\Unit \Ett_1$ (resp. $\Unit \Ett_2$), or
$\Unit \Unit \Unit^* \Ett_1$, or $\Unit \Unit \Unit^* \Ett_2$. Hence, we
obtain that
\begin{subequations}
\begin{equation}
    \SeriesBubbles_\Unit
    \left(\xi_\Unit, \xi_{\Ett_1}, \xi_{\Ett_2}\right)
    = 0,
\end{equation}
\begin{equation}
    \SeriesBubbles_{\Ett_1}
    \left(\xi_\Unit, \xi_{\Ett_1}, \xi_{\Ett_2}\right) =
    \frac{\xi_\Unit}{1 - \xi_\Unit}
    \left(\xi_{\Ett_1} + \xi_\Unit \xi_{\Ett_2}\right)
    =
    \SeriesBubbles_{\Ett_2}
    \left(\xi_\Unit, \xi_{\Ett_2}, \xi_{\Ett_1}\right),
\end{equation}
\end{subequations}
\noindent Moreover, one can check that $G$ satisfies the conditions
required by Proposition~\ref{prop:suboperads_NC_M_triangles_dimensions}.
We hence have
\vspace{-1.75em}
\begin{multicols}{2}
\begin{subequations}
\begin{equation}
    \bar\SeriesElements_\Unit(t)
    = \SeriesElements_{\Ett_1}(t) + \SeriesElements_{\Ett_2}(t),
\end{equation}

\begin{equation}
    \bar\SeriesElements_{\Ett_1}(t)
    = \SeriesElements_\Unit(t)
    = \bar\SeriesElements_{\Ett_2}(t),
\end{equation}
\end{subequations}
\end{multicols}
\noindent and
\begin{subequations}
\begin{equation}
    \SeriesElements_\Unit(t) = 0,
\end{equation}
\begin{equation}
    \SeriesElements_{\Ett_1}(t) =
    \SeriesBubbles_{\Ett_1}(
        t + \SeriesElements_{\Ett_1}(t) + \SeriesElements_{\Ett_2}(t),
        t, t)
    =
    \SeriesBubbles_{\Ett_2}(
        t + \SeriesElements_{\Ett_1}(t) + \SeriesElements_{\Ett_2}(t),
        t, t)
    =
    \SeriesElements_{\Ett_2}(t),
\end{equation}
\end{subequations}
By Proposition~\ref{prop:suboperads_NC_M_triangles_dimensions},
the Hilbert series of $(\NC\Ebb_2)^G$ satisfies
\begin{equation}
    \Hilbert_{(\NC\Ebb_2)^G}(t) = t + \SeriesElements_\Unit(t) +
    \SeriesElements_{\Ett_1}(t) +
    \SeriesElements_{\Ett_2}(t) = t + 2 \SeriesElements_{\Ett_1}(t),
\end{equation}
and, by a straightforward computation, we obtain that this series
satisfies the algebraic equation
\begin{equation}
    t + (t - 1) \Hilbert_{(\NC\Ebb_2)^G}(t) +
    (2t + 1)\Hilbert_{(\NC\Ebb_2)^G}(t)^2 = 0.
\end{equation}
The first dimensions of $(\NC\Ebb_2)^G$ are
\begin{equation}
    1, 2, 8, 36, 180, 956, 5300, 30316,
\end{equation}
and form Sequence~\OEIS{A129148} of~\cite{Slo}.
\medskip

\subsubsection{Second example~: a suboperad of Motzkin paths}
Consider the suboperad of $\NC\Dbb_0$ generated by
\begin{equation}
    G := \left\{
        \TriangleEEE{}{}{},
        \SquareMotz
    \right\}.
\end{equation}
Computer experiments show that the generators of $(\NC\Dbb_0)^G$ are
subjected to four quadratic nontrivial relations
\begin{subequations}
\begin{equation} \label{equ:rel_suboperad_2_1}
    \TriangleEEE{}{}{} \circ_1 \TriangleEEE{}{}{}
    =
    \TriangleEEE{}{}{} \circ_2 \TriangleEEE{}{}{}\,,
\end{equation}
\begin{equation} \label{equ:rel_suboperad_2_2}
    \SquareMotz \circ_1 \TriangleEEE{}{}{}
    =
    \TriangleEEE{}{}{} \circ_2 \SquareMotz\,,
\end{equation}
\begin{equation} \label{equ:rel_suboperad_2_3}
    \TriangleEEE{}{}{} \circ_1 \SquareMotz
    =
    \SquareMotz \circ_2 \TriangleEEE{}{}{}\,,
\end{equation}
\begin{equation} \label{equ:rel_suboperad_2_4}
    \SquareMotz \circ_1 \SquareMotz
    =
    \SquareMotz \circ_3 \SquareMotz\,.
\end{equation}
\end{subequations}
It is possible to prove that this operad does not admit any other
nontrivial relations between its generators. This can be performed by
defining a rewrite rule on the syntax trees on $G$, consisting in
rewriting the left patterns of~\eqref{equ:rel_suboperad_2_1},
\eqref{equ:rel_suboperad_2_2}, \eqref{equ:rel_suboperad_2_3},
and~\eqref{equ:rel_suboperad_2_4} into their respective right patterns,
and by checking that this rewrite rule admits the required properties
(like the ones establishing the presentation of $\NC\Mca$ by
Theorem~\ref{thm:presentation_NC_M}).
\medskip

One can prove by induction on the arity that the set of bubbles of
$(\NC\Dbb_0)^G$ is the set of all bubbles whose bases are labeled by
$\Unit$ and borders are words of $\{\Unit, 0\}^{\geq 2}$ such that each
occurrence of $0$ has a $\Unit$ immediately at its left and a $\Unit$
immediately at its right. Hence, we obtain that
\vspace{-1.75em}
\begin{multicols}{2}
\begin{subequations}
\begin{equation}
    \SeriesBubbles_\Unit\left(\xi_\Unit, \xi_0\right)
    = \frac{1}{1 - \xi_\Unit - \xi_\Unit \xi_0} \xi_\Unit - \xi_\Unit,
\end{equation}

\begin{equation}
    \SeriesBubbles_0\left(\xi_\Unit, \xi_0\right) = 0.
\end{equation}
\end{subequations}
\end{multicols}
\noindent Moreover, one can check that $G$ satisfies the conditions
required by Proposition~\ref{prop:suboperads_NC_M_triangles_dimensions}.
We hence have
\vspace{-1.75em}
\begin{multicols}{2}
\begin{subequations}
\begin{equation}
    \bar\SeriesElements_\Unit(t) = \SeriesElements_0(t),
\end{equation}

\begin{equation}
    \bar\SeriesElements_0(t)
    = \SeriesElements_\Unit(t) + \SeriesElements_0(t),
\end{equation}
\end{subequations}
\end{multicols}
\noindent and
\vspace{-1.75em}
\begin{multicols}{2}
\begin{subequations}
\begin{equation}
    \SeriesElements_\Unit(t) =
        \SeriesBubbles_\Unit\left(t,
        t + \SeriesElements_\Unit(t)\right),
\end{equation}

\begin{equation}
    \SeriesElements_0(t) = 0.
\end{equation}
\end{subequations}
\end{multicols}
\noindent By Proposition~\ref{prop:suboperads_NC_M_triangles_dimensions},
the Hilbert series of $(\NC\Dbb_0)^G$ satisfies
\begin{equation}
    \Hilbert_{(\NC\Dbb_0)^G}(t) = t + \SeriesElements_\Unit(t),
\end{equation}
and, by a straightforward computation, we obtain that this series
satisfies the algebraic equation
\begin{equation}
    t + (t - 1) \Hilbert_{(\NC\Dbb_0)^G}(t)
    + t\Hilbert_{(\NC\Dbb_0)^G}(t)^2 = 0.
\end{equation}
The first dimensions of $(\NC\Dbb_0)^G$ are
\begin{equation}
    1, 1, 2, 4, 9, 21, 51, 127,
\end{equation}
and form Sequence~\OEIS{A001006} of~\cite{Slo}. The operad
$(\NC\Dbb_0)^G$ has the same presentation by generators and relations
(and thus, the same Hilbert series) as the operad $\Motz$ defined
in~\cite{Gir15}, involving Motzkin paths. Hence, $(\NC\Dbb_0)^G$ and
$\Motz$ are two isomorphic operads. Note in passing that these two
operads are not isomorphic to the operad $\Motzkin\Dbb_0$ constructed in
Section~\ref{subsubsec:Motzkin_configurations} and involving Motzkin
configurations. Indeed, the sequence of the dimensions of this last
operad is a shifted version of the one of $(\NC\Dbb_0)^G$ and~$\Motz$.
\medskip

\subsection{Algebras over the noncrossing clique operads}
We begin by briefly describing $\NC\Mca$-algebras in terms of relations
between their operations and the free $\NC\Mca$-algebras over one
generator. We continue this section by providing two ways to construct
(non-necessarily free) $\NC\Mca$-algebras. The first one takes as input
an associative algebra endowed with endofunctions satisfying some
conditions, and the second one takes as input a monoid.
\medskip

\subsubsection{Relations}
From the presentation of $\NC\Mca$ established by
Theorem~\ref{thm:presentation_NC_M}, any $\NC\Mca$-algebra is a vector
space $\Alg$ endowed with binary linear operations
\begin{equation}
    \TriangleOp{\Pfr_0}{\Pfr_1}{\Pfr_2} :
    \Alg \otimes \Alg \to \Alg,
    \qquad
    \Pfr \in \Triangles_\Mca,
\end{equation}
satisfying, for all $a_1, a_2, a_3 \in \Alg$, the relations
\begin{subequations}
\begin{equation} \label{equ:relation_NC_M_algebras_1}
    \left(a_1 \TriangleOp{\Qfr_0}{\Qfr_1}{\Qfr_2} a_2\right)
    \TriangleOp{\Pfr_0}{\Pfr_1}{\Pfr_2} a_3
    =
    \left(a_1 \TriangleOp{\Rfr_0}{\Qfr_1}{\Qfr_2} a_2\right)
    \TriangleOp{\Pfr_0}{\Rfr_1}{\Pfr_2} a_3,
    \qquad
    \mbox{if } \Pfr_1 \Op \Qfr_0 = \Rfr_1 \Op \Rfr_0 \ne \Unit_\Mca,
\end{equation}
\begin{equation} \label{equ:relation_NC_M_algebras_2}
    \left(a_1 \TriangleOp{\Qfr_0}{\Qfr_1}{\Qfr_2} a_2\right)
    \TriangleOp{\Pfr_0}{\Pfr_1}{\Pfr_2} a_3
    =
    a_1 \TriangleOp{\Pfr_0}{\Qfr_1}{\Rfr_2}
    \left(a_2 \TriangleOp{\Rfr_0}{\Qfr_2}{\Pfr_2} a_3 \right),
    \qquad
    \mbox{if } \Pfr_1 \Op \Qfr_0 = \Rfr_2 \Op \Rfr_0 = \Unit_\Mca,
\end{equation}
\begin{equation} \label{equ:relation_NC_M_algebras_3}
    a_1 \TriangleOp{\Pfr_0}{\Pfr_1}{\Pfr_2}
    \left( a_2 \TriangleOp{\Qfr_0}{\Qfr_1}{\Qfr_2} a_3 \right)
    =
    a_1 \TriangleOp{\Pfr_0}{\Pfr_1}{\Rfr_2}
    \left( a_2 \TriangleOp{\Rfr_0}{\Qfr_1}{\Qfr_2} a_3 \right),
    \qquad
    \mbox{if } \Pfr_2 \Op \Qfr_0 = \Rfr_2 \Op \Rfr_0 \ne \Unit_\Mca,
\end{equation}
\end{subequations}
where $\Pfr$, $\Qfr$, and $\Rfr$ are $\Mca$-triangles.
Remark that $\Mca$ has to be finite because
Theorem~\ref{thm:presentation_NC_M} requires this property as premise.
\medskip

\subsubsection{Free algebras over one generator}
From the realization of $\NC\Mca$ coming from its definition as a
suboperad of $\Cli\Mca$, the free $\NC\Mca$-algebra over one generator
is the linear span $\NC\Mca$ of all noncrossing $\Mca$-cliques endowed
with the linear operations
\begin{equation}
    \TriangleOp{\Pfr_O}{\Pfr_1}{\Pfr_2} :
    \NC\Mca(n) \otimes \NC\Mca(m) \to \NC\Mca(n + m),
    \qquad
    \Pfr \in \Triangles_\Mca,
    n, m \geq 1,
\end{equation}
defined, for any noncrossing $\Mca$-cliques $\Qfr$ and $\Rfr$, by
\begin{equation} \label{equ:free_algebra_MT_M_product}
    \Qfr \TriangleOp{\Pfr_0}{\Pfr_1}{\Pfr_2} \Rfr
    :=
    \left(\Triangle{\Pfr_0}{\Pfr_1}{\Pfr_2}
    \circ_2 \Rfr\right) \circ_1 \Qfr.
\end{equation}
In terms of $\Mca$-Schröder trees (see
Section~\ref{subsubsec:M_Schroder_trees}),
\eqref{equ:free_algebra_MT_M_product} is the $\Mca$-Schröder tree
obtained by grafting the $\Mca$-Schröder trees of $\Qfr$ and $\Rfr$
respectively as left and right children of a binary corolla having its
edge adjacent to the root labeled by $\Pfr_0$, its first edge labeled by
$\Pfr_1 \Op \Qfr_0$, and second edge labeled by $\Pfr_2 \Op \Rfr_0$,
and by contracting each of these two edges when labeled by~$\Unit_\Mca$.
For instance, in the free $\NC\N_3$-algebra, we have
\begin{subequations}
\begin{equation}
    \begin{tikzpicture}[xscale=.34,yscale=.28,Centering]
        \node[Leaf](0)at(0.00,-2.00){};
        \node[Leaf](2)at(1.00,-4.00){};
        \node[Leaf](4)at(3.00,-4.00){};
        \node[Leaf](5)at(4.00,-2.00){};
        \node[Node](1)at(2.00,0.00){};
        \node[Node](3)at(2.00,-2.00){};
        \draw[Edge](0)edge[]node[EdgeLabel]{\begin{math}1\end{math}}(1);
        \draw[Edge](2)--(3);
        \draw[Edge](3)edge[]node[EdgeLabel]{\begin{math}2\end{math}}(1);
        \draw[Edge](4)edge[]node[EdgeLabel]{\begin{math}1\end{math}}(3);
        \draw[Edge](5)edge[]node[EdgeLabel]{\begin{math}1\end{math}}(1);
        \node(r)at(2.00,2){};
        \draw[Edge](r)edge[]node[EdgeLabel]{\begin{math}2\end{math}}(1);
    \end{tikzpicture}
    \enspace \TriangleOp{1}{2}{1} \enspace
    \begin{tikzpicture}[xscale=.32,yscale=.25,Centering]
        \node[Leaf](0)at(0.00,-4.00){};
        \node[Leaf](2)at(1.00,-4.00){};
        \node[Leaf](3)at(2.00,-4.00){};
        \node[Leaf](5)at(4.00,-2.00){};
        \node[Node](1)at(1.00,-2.00){};
        \node[Node](4)at(3.00,0.00){};
        \draw[Edge](0)--(1);
        \draw[Edge](1)edge[]node[EdgeLabel]{\begin{math}1\end{math}}(4);
        \draw[Edge](2)--(1);
        \draw[Edge](3)edge[]node[EdgeLabel]{\begin{math}2\end{math}}(1);
        \draw[Edge](5)edge[]node[EdgeLabel]{\begin{math}2\end{math}}(4);
        \node(r)at(3.00,1.25){};
        \draw[Edge](r)--(4);
    \end{tikzpicture}
    \enspace = \enspace
    \begin{tikzpicture}[xscale=.32,yscale=.18,Centering]
        \node[Leaf](0)at(0.00,-6.50){};
        \node[Leaf](10)at(8.00,-9.75){};
        \node[Leaf](12)at(10.00,-6.50){};
        \node[Leaf](2)at(1.00,-9.75){};
        \node[Leaf](4)at(3.00,-9.75){};
        \node[Leaf](5)at(4.00,-6.50){};
        \node[Leaf](7)at(6.00,-9.75){};
        \node[Leaf](9)at(7.00,-9.75){};
        \node[Node](1)at(2.00,-3.25){};
        \node[Node](11)at(9.00,-3.25){};
        \node[Node](3)at(2.00,-6.50){};
        \node[Node](6)at(5.00,0.00){};
        \node[Node](8)at(7.00,-6.50){};
        \draw[Edge](0)edge[]node[EdgeLabel]{\begin{math}1\end{math}}(1);
        \draw[Edge](1)edge[]node[EdgeLabel]{\begin{math}1\end{math}}(6);
        \draw[Edge](10)edge[]node[EdgeLabel]{\begin{math}2\end{math}}(8);
        \draw[Edge](11)edge[]node[EdgeLabel]{\begin{math}1\end{math}}(6);
        \draw[Edge](12)edge[]node[EdgeLabel]{\begin{math}2\end{math}}(11);
        \draw[Edge](2)--(3);
        \draw[Edge](3)edge[]node[EdgeLabel]{\begin{math}2\end{math}}(1);
        \draw[Edge](4)edge[]node[EdgeLabel]{\begin{math}1\end{math}}(3);
        \draw[Edge](5)edge[]node[EdgeLabel]{\begin{math}1\end{math}}(1);
        \draw[Edge](7)--(8);
        \draw[Edge](8)edge[]node[EdgeLabel]{\begin{math}1\end{math}}(11);
        \draw[Edge](9)--(8);
        \node(r)at(5.00,3){};
        \draw[Edge](r)edge[]node[EdgeLabel]{\begin{math}1\end{math}}(6);
    \end{tikzpicture}\,,
\end{equation}
\begin{equation}
    \begin{tikzpicture}[xscale=.34,yscale=.28,Centering]
        \node[Leaf](0)at(0.00,-2.00){};
        \node[Leaf](2)at(1.00,-4.00){};
        \node[Leaf](4)at(3.00,-4.00){};
        \node[Leaf](5)at(4.00,-2.00){};
        \node[Node](1)at(2.00,0.00){};
        \node[Node](3)at(2.00,-2.00){};
        \draw[Edge](0)edge[]node[EdgeLabel]{\begin{math}1\end{math}}(1);
        \draw[Edge](2)--(3);
        \draw[Edge](3)edge[]node[EdgeLabel]{\begin{math}2\end{math}}(1);
        \draw[Edge](4)edge[]node[EdgeLabel]{\begin{math}1\end{math}}(3);
        \draw[Edge](5)edge[]node[EdgeLabel]{\begin{math}1\end{math}}(1);
        \node(r)at(2.00,2){};
        \draw[Edge](r)edge[]node[EdgeLabel]{\begin{math}2\end{math}}(1);
    \end{tikzpicture}
    \enspace \TriangleOp{1}{1}{1} \enspace
    \begin{tikzpicture}[xscale=.32,yscale=.25,Centering]
        \node[Leaf](0)at(0.00,-4.00){};
        \node[Leaf](2)at(1.00,-4.00){};
        \node[Leaf](3)at(2.00,-4.00){};
        \node[Leaf](5)at(4.00,-2.00){};
        \node[Node](1)at(1.00,-2.00){};
        \node[Node](4)at(3.00,0.00){};
        \draw[Edge](0)--(1);
        \draw[Edge](1)edge[]node[EdgeLabel]{\begin{math}1\end{math}}(4);
        \draw[Edge](2)--(1);
        \draw[Edge](3)edge[]node[EdgeLabel]{\begin{math}2\end{math}}(1);
        \draw[Edge](5)edge[]node[EdgeLabel]{\begin{math}2\end{math}}(4);
        \node(r)at(3.00,1.25){};
        \draw[Edge](r)--(4);
    \end{tikzpicture}
    \enspace = \enspace
    \begin{tikzpicture}[xscale=.33,yscale=.21,Centering]
        \node[Leaf](0)at(0.00,-3.00){};
        \node[Leaf](1)at(1.00,-6.00){};
        \node[Leaf](11)at(10.00,-6.00){};
        \node[Leaf](3)at(3.00,-6.00){};
        \node[Leaf](5)at(5.00,-3.00){};
        \node[Leaf](6)at(6.00,-9.00){};
        \node[Leaf](8)at(7.00,-9.00){};
        \node[Leaf](9)at(8.00,-9.00){};
        \node[Node](10)at(9.00,-3.00){};
        \node[Node](2)at(2.00,-3.00){};
        \node[Node](4)at(4.00,0.00){};
        \node[Node](7)at(7.00,-6.00){};
        \draw[Edge](0)edge[]node[EdgeLabel]{\begin{math}1\end{math}}(4);
        \draw[Edge](1)--(2);
        \draw[Edge](10)edge[]node[EdgeLabel]{\begin{math}1\end{math}}(4);
        \draw[Edge](11)edge[]node[EdgeLabel]{\begin{math}2\end{math}}(10);
        \draw[Edge](2)edge[]node[EdgeLabel]{\begin{math}2\end{math}}(4);
        \draw[Edge](3)edge[]node[EdgeLabel]{\begin{math}1\end{math}}(2);
        \draw[Edge](5)edge[]node[EdgeLabel]{\begin{math}1\end{math}}(4);
        \draw[Edge](6)--(7);
        \draw[Edge](7)edge[]node[EdgeLabel]{\begin{math}1\end{math}}(10);
        \draw[Edge](8)--(7);
        \draw[Edge](9)edge[]node[EdgeLabel]{\begin{math}2\end{math}}(7);
        \node(r)at(4.00,3){};
        \draw[Edge](r)edge[]node[EdgeLabel]{\begin{math}1\end{math}}(4);
    \end{tikzpicture}\,,
\end{equation}
\begin{equation}
    \begin{tikzpicture}[xscale=.34,yscale=.28,Centering]
        \node[Leaf](0)at(0.00,-2.00){};
        \node[Leaf](2)at(1.00,-4.00){};
        \node[Leaf](4)at(3.00,-4.00){};
        \node[Leaf](5)at(4.00,-2.00){};
        \node[Node](1)at(2.00,0.00){};
        \node[Node](3)at(2.00,-2.00){};
        \draw[Edge](0)edge[]node[EdgeLabel]{\begin{math}1\end{math}}(1);
        \draw[Edge](2)--(3);
        \draw[Edge](3)edge[]node[EdgeLabel]{\begin{math}2\end{math}}(1);
        \draw[Edge](4)edge[]node[EdgeLabel]{\begin{math}1\end{math}}(3);
        \draw[Edge](5)edge[]node[EdgeLabel]{\begin{math}1\end{math}}(1);
        \node(r)at(2.00,2){};
        \draw[Edge](r)edge[]node[EdgeLabel]{\begin{math}2\end{math}}(1);
    \end{tikzpicture}
    \enspace \TriangleOp{1}{2}{0} \enspace
    \begin{tikzpicture}[xscale=.32,yscale=.25,Centering]
        \node[Leaf](0)at(0.00,-4.00){};
        \node[Leaf](2)at(1.00,-4.00){};
        \node[Leaf](3)at(2.00,-4.00){};
        \node[Leaf](5)at(4.00,-2.00){};
        \node[Node](1)at(1.00,-2.00){};
        \node[Node](4)at(3.00,0.00){};
        \draw[Edge](0)--(1);
        \draw[Edge](1)edge[]node[EdgeLabel]{\begin{math}1\end{math}}(4);
        \draw[Edge](2)--(1);
        \draw[Edge](3)edge[]node[EdgeLabel]{\begin{math}2\end{math}}(1);
        \draw[Edge](5)edge[]node[EdgeLabel]{\begin{math}2\end{math}}(4);
        \node(r)at(3.00,1.25){};
        \draw[Edge](r)--(4);
    \end{tikzpicture}
    \enspace = \enspace
    \begin{tikzpicture}[xscale=.32,yscale=.18,Centering]
        \node[Leaf](0)at(0.00,-6.00){};
        \node[Leaf](10)at(7.00,-6.00){};
        \node[Leaf](11)at(8.00,-3.00){};
        \node[Leaf](2)at(1.00,-9.00){};
        \node[Leaf](4)at(3.00,-9.00){};
        \node[Leaf](5)at(4.00,-6.00){};
        \node[Leaf](7)at(5.00,-6.00){};
        \node[Leaf](9)at(6.00,-6.00){};
        \node[Node](1)at(2.00,-3.00){};
        \node[Node](3)at(2.00,-6.00){};
        \node[Node](6)at(6.00,0.00){};
        \node[Node](8)at(6.00,-3.00){};
        \draw[Edge](0)edge[]node[EdgeLabel]{\begin{math}1\end{math}}(1);
        \draw[Edge](1)edge[]node[EdgeLabel]{\begin{math}1\end{math}}(6);
        \draw[Edge](10)edge[]node[EdgeLabel]{\begin{math}2\end{math}}(8);
        \draw[Edge](11)edge[]node[EdgeLabel]{\begin{math}2\end{math}}(6);
        \draw[Edge](2)--(3);
        \draw[Edge](3)edge[]node[EdgeLabel]{\begin{math}2\end{math}}(1);
        \draw[Edge](4)edge[]node[EdgeLabel]{\begin{math}1\end{math}}(3);
        \draw[Edge](5)edge[]node[EdgeLabel]{\begin{math}1\end{math}}(1);
        \draw[Edge](7)--(8);
        \draw[Edge](8)edge[]node[EdgeLabel]{\begin{math}1\end{math}}(6);
        \draw[Edge](9)--(8);
        \node(r)at(6.00,3){};
        \draw[Edge](r)edge[]node[EdgeLabel]{\begin{math}1\end{math}}(6);
    \end{tikzpicture}\,,
\end{equation}
\begin{equation}
    \begin{tikzpicture}[xscale=.34,yscale=.28,Centering]
        \node[Leaf](0)at(0.00,-2.00){};
        \node[Leaf](2)at(1.00,-4.00){};
        \node[Leaf](4)at(3.00,-4.00){};
        \node[Leaf](5)at(4.00,-2.00){};
        \node[Node](1)at(2.00,0.00){};
        \node[Node](3)at(2.00,-2.00){};
        \draw[Edge](0)edge[]node[EdgeLabel]{\begin{math}1\end{math}}(1);
        \draw[Edge](2)--(3);
        \draw[Edge](3)edge[]node[EdgeLabel]{\begin{math}2\end{math}}(1);
        \draw[Edge](4)edge[]node[EdgeLabel]{\begin{math}1\end{math}}(3);
        \draw[Edge](5)edge[]node[EdgeLabel]{\begin{math}1\end{math}}(1);
        \node(r)at(2.00,2){};
        \draw[Edge](r)edge[]node[EdgeLabel]{\begin{math}2\end{math}}(1);
    \end{tikzpicture}
    \enspace \TriangleOp{1}{1}{0} \enspace
    \begin{tikzpicture}[xscale=.32,yscale=.25,Centering]
        \node[Leaf](0)at(0.00,-4.00){};
        \node[Leaf](2)at(1.00,-4.00){};
        \node[Leaf](3)at(2.00,-4.00){};
        \node[Leaf](5)at(4.00,-2.00){};
        \node[Node](1)at(1.00,-2.00){};
        \node[Node](4)at(3.00,0.00){};
        \draw[Edge](0)--(1);
        \draw[Edge](1)edge[]node[EdgeLabel]{\begin{math}1\end{math}}(4);
        \draw[Edge](2)--(1);
        \draw[Edge](3)edge[]node[EdgeLabel]{\begin{math}2\end{math}}(1);
        \draw[Edge](5)edge[]node[EdgeLabel]{\begin{math}2\end{math}}(4);
        \node(r)at(3.00,1.25){};
        \draw[Edge](r)--(4);
    \end{tikzpicture}
    \enspace = \enspace
    \begin{tikzpicture}[xscale=.32,yscale=.2,Centering]
        \node[Leaf](0)at(0.00,-3.67){};
        \node[Leaf](1)at(1.00,-7.33){};
        \node[Leaf](10)at(8.00,-3.67){};
        \node[Leaf](3)at(3.00,-7.33){};
        \node[Leaf](5)at(4.00,-3.67){};
        \node[Leaf](6)at(5.00,-7.33){};
        \node[Leaf](8)at(6.00,-7.33){};
        \node[Leaf](9)at(7.00,-7.33){};
        \node[Node](2)at(2.00,-3.67){};
        \node[Node](4)at(4.00,0.00){};
        \node[Node](7)at(6.00,-3.67){};
        \draw[Edge](0)edge[]node[EdgeLabel]{\begin{math}1\end{math}}(4);
        \draw[Edge](1)--(2);
        \draw[Edge](10)edge[]node[EdgeLabel]{\begin{math}2\end{math}}(4);
        \draw[Edge](2)edge[]node[EdgeLabel]{\begin{math}2\end{math}}(4);
        \draw[Edge](3)edge[]node[EdgeLabel]{\begin{math}1\end{math}}(2);
        \draw[Edge](5)edge[]node[EdgeLabel]{\begin{math}1\end{math}}(4);
        \draw[Edge](6)--(7);
        \draw[Edge](7)edge[]node[EdgeLabel]{\begin{math}1\end{math}}(4);
        \draw[Edge](8)--(7);
        \draw[Edge](9)edge[]node[EdgeLabel]{\begin{math}2\end{math}}(7);
        \node(r)at(4.00,3){};
        \draw[Edge](r)edge[]node[EdgeLabel]{\begin{math}1\end{math}}(4);
    \end{tikzpicture}\,.
\end{equation}
\end{subequations}
\medskip

\subsubsection{From associative algebras}
Let $\Alg$ be an associative algebra with associative product denoted by
$\OpAssoc$, and
\begin{equation} \label{equ:compatible_set}
    \omega_x : \Alg \to \Alg, \qquad x \in \Mca,
\end{equation}
be a family of linear maps, not necessarily associative algebra
morphisms, indexed by the elements of $\Mca$. We say that $\Alg$
together with this family~\eqref{equ:compatible_set} of maps is
\Def{$\Mca$-compatible} if
\begin{equation} \label{equ:compatible_magma_on_algebra_1}
    \omega_{\Unit_\Mca} = \Id_\Alg
\end{equation}
where $\Id_\Alg$ is the identity map on $\Alg$, and
\begin{equation} \label{equ:compatible_magma_on_algebra_2}
    \omega_x \circ \omega_y = \omega_{x \Op y},
\end{equation}
for all $x, y \in \Mca$. Let us now use $\Mca$-compatible associative
algebras to construct $\NC\Mca$-algebras.
\medskip

\begin{Theorem} \label{thm:NC_M_algebras}
    Let $\Mca$ be a finite unitary magma and $\Alg$ be an
    $\Mca$-compatible associative algebra. The vector space $\Alg$
    endowed with the binary linear operations
    \begin{equation} \label{equ:NC_M_algebras}
        \TriangleOp{\Pfr_0}{\Pfr_1}{\Pfr_2} :
        \Alg \otimes \Alg \to \Alg,
        \qquad
        \Pfr \in \Triangles_\Mca,
    \end{equation}
    defined for each $\Mca$-triangle $\Pfr$ and any $a_1, a_2 \in \Alg$
    by
    \begin{equation} \label{equ:NC_M_algebras_def}
        a_1 \TriangleOp{\Pfr_0}{\Pfr_1}{\Pfr_2} a_2
        :=
        \omega_{\Pfr_0}\left(\omega_{\Pfr_1}\left(a_1\right)
        \OpAssoc \omega_{\Pfr_2}\left(a_2\right)\right),
    \end{equation}
    is an $\NC\Mca$-algebra.
\end{Theorem}
\begin{proof}
    Let us prove that the operations~\eqref{equ:NC_M_algebras} satisfy
    Relations~\eqref{equ:relation_NC_M_algebras_1},
    \eqref{equ:relation_NC_M_algebras_2},
    and~\eqref{equ:relation_NC_M_algebras_3} of $\NC\Mca$-algebras.
    Since $\Mca$ is finite, this amounts to show that these operations
    endow $\Alg$ with an $\NC\Mca$-algebra structure. For this, let
    $a_1$, $a_2$, and $a_3$ be three elements of $\Alg$, and $\Pfr$,
    $\Qfr$, and $\Rfr$ be three $\Mca$-triangles.
    \begin{enumerate}[fullwidth,label=(\alph*)]
        \item When
        $\Pfr_1 \Op \Qfr_0 = \Rfr_1 \Op \Rfr_0 \ne \Unit_\Mca$, since
        by~\eqref{equ:compatible_magma_on_algebra_2},
        $\omega_{\Pfr_1} \circ \omega_{\Qfr_0} =
        \omega_{\Rfr_1} \circ \omega_{\Rfr_0}$,
        \begin{equation}\begin{split}
            \left(a_1 \TriangleOp{\Qfr_0}{\Qfr_1}{\Qfr_2} a_2\right)
            \TriangleOp{\Pfr_0}{\Pfr_1}{\Pfr_2} a_3
            & =
            \omega_{\Qfr_0}(
            \omega_{\Qfr_1}(a_1) \OpAssoc
            \omega_{\Qfr_2}(a_2))
            \TriangleOp{\Pfr_0}{\Pfr_1}{\Pfr_2} a_3 \\
            & =
            \omega_{\Pfr_0}(
            \omega_{\Pfr_1}(\omega_{\Qfr_0}(
            \omega_{\Qfr_1}(a_1) \OpAssoc
            \omega_{\Qfr_2}(a_2)))
            \OpAssoc
            \omega_{\Pfr_2}(a_3)) \\
            & =
            \omega_{\Pfr_0}(
            \omega_{\Rfr_1}(\omega_{\Rfr_0}(
            \omega_{\Qfr_1}(a_1) \OpAssoc
            \omega_{\Qfr_2}(a_2)))
            \OpAssoc
            \omega_{\Pfr_2}(a_3)) \\
            & =
            \left(a_1 \TriangleOp{\Rfr_0}{\Qfr_1}{\Qfr_2} a_2\right)
            \TriangleOp{\Pfr_0}{\Rfr_1}{\Pfr_2} a_3,
        \end{split}\end{equation}
        so that~\eqref{equ:relation_NC_M_algebras_1} holds.
        \item When $\Pfr_1 \Op \Qfr_0 = \Rfr_2 \Op \Rfr_0 = \Unit_\Mca$,
        since by~\eqref{equ:compatible_magma_on_algebra_1},
        $\omega_{\Pfr_1} \circ \omega_{\Qfr_0} =
        \omega_{\Rfr_2} \circ \omega_{\Rfr_0} = \Id_\Alg$ and since
        $\OpAssoc$ is associative,
        \begin{equation}\begin{split}
            \left(a_1 \TriangleOp{\Qfr_0}{\Qfr_1}{\Qfr_2} a_2\right)
            \TriangleOp{\Pfr_0}{\Pfr_1}{\Pfr_2} a_3
            & =
            \omega_{\Qfr_0}(\omega_{\Qfr_1}(a_1)
            \OpAssoc \omega_{\Qfr_2}(a_2))
            \TriangleOp{\Pfr_0}{\Pfr_1}{\Pfr_2} a_3 \\
            & =
            \omega_{\Pfr_0}(
            (\omega_{\Pfr_1}(
            \omega_{\Qfr_0}(\omega_{\Qfr_1}(a_1)
            \OpAssoc \omega_{\Qfr_2}(a_2))))
            \OpAssoc
            \omega_{\Pfr_2}(a_3)) \\
            & =
            \omega_{\Pfr_0}(
            \omega_{\Qfr_1}(a_1) \OpAssoc \omega_{\Qfr_2}(a_2)
            \OpAssoc
            \omega_{\Pfr_2}(a_3)) \\
            & =
            \omega_{\Pfr_0}(\omega_{\Qfr_1}(a_1) \OpAssoc
            (\omega_{\Qfr_2}(a_2) \OpAssoc
            \omega_{\Pfr_2}(a_3))) \\
            & =
            \omega_{\Pfr_0}(\omega_{\Qfr_1}(a_1) \OpAssoc
            \omega_{\Rfr_2}(\omega_{\Rfr_0}(\omega_{\Qfr_2}(a_2) \OpAssoc
            \omega_{\Pfr_2}(a_3)))) \\
            & = a_1 \TriangleOp{\Pfr_0}{\Qfr_1}{\Rfr_2}
            \left(a_2 \TriangleOp{\Rfr_0}{\Qfr_2}{\Pfr_2} a_3 \right),
        \end{split}\end{equation}
        so that~\eqref{equ:relation_NC_M_algebras_2} holds.
        \item When
        $\Pfr_2 \Op \Qfr_0 = \Rfr_2 \Op \Rfr_0 \ne \Unit_\Mca$,
        since by~\eqref{equ:compatible_magma_on_algebra_2},
        $\omega_{\Pfr_2} \circ \omega_{\Qfr_0} =
        \omega_{\Rfr_2} \circ \omega_{\Rfr_0}$,
        \begin{equation}\begin{split}
            a_1 \TriangleOp{\Pfr_0}{\Pfr_1}{\Pfr_2}
            \left( a_2 \TriangleOp{\Qfr_0}{\Qfr_1}{\Qfr_2} a_3 \right)
            & =
            a_1 \TriangleOp{\Pfr_0}{\Pfr_1}{\Pfr_2}
            \omega_{\Qfr_0}
            (\omega_{\Qfr_1}(a_2) \OpAssoc \omega_{\Qfr_2}(a_3)) \\
            & =
            \omega_{\Pfr_0}(
            \omega_{\Pfr_1}(a_1)
            \OpAssoc
            \omega_{\Pfr_2}(
            \omega_{\Qfr_0}
            (\omega_{\Qfr_1}(a_2) \OpAssoc \omega_{\Qfr_2}(a_3)))) \\
            & =
            \omega_{\Pfr_0}(
            \omega_{\Pfr_1}(a_1)
            \OpAssoc
            \omega_{\Rfr_2}(
            \omega_{\Rfr_0}
            (\omega_{\Qfr_1}(a_2) \OpAssoc \omega_{\Qfr_2}(a_3)))) \\
            & =
            a_1 \TriangleOp{\Pfr_0}{\Pfr_1}{\Rfr_2}
            \left( a_2 \TriangleOp{\Rfr_0}{\Qfr_1}{\Qfr_2} a_3 \right),
        \end{split}\end{equation}
        so that~\eqref{equ:relation_NC_M_algebras_3} holds.
    \end{enumerate}
    Therefore, $\Alg$ is an $\NC\Mca$-algebra.
\end{proof}
\medskip

By Theorem~\ref{thm:NC_M_algebras}, $\Alg$ has the structure of an
$\NC\Mca$-algebra. Hence, there is a left action~$\cdot$ of the operad
$\NC\Mca$ on the tensor algebra of $\Alg$ of the form
\begin{equation}
    \cdot : \NC\Mca(n) \otimes \Alg^{\otimes n} \to \Alg,
    \qquad n \geq 1,
\end{equation}
whose definition comes from the ones of the
operations~\eqref{equ:NC_M_algebras} and
Relation~\eqref{equ:algebra_over_operad}. We describe here an algorithm
to compute the action of any element of $\NC\Mca$ of arity $n$ on
tensors $a_1 \otimes \dots \otimes a_n$ of $\Alg^{\otimes n}$. First, if
$\Bfr$ is an $\Mca$-bubble of arity $n$,
\begin{equation} \label{equ:NC_M_algebras_action_bubbles}
    \Bfr \cdot \left(a_1 \otimes \dots \otimes a_n\right)
    = \omega_{\Bfr_0}\left(\prod_{i \in [n]}
    \omega_{\Bfr_i}\left(a_i\right) \right),
\end{equation}
where the product
of~\eqref{equ:NC_M_algebras_action_bubbles} denotes the iterated
version of the associative product $\OpAssoc$ of~$\Alg$. When $\Pfr$ is
a noncrossing $\Mca$-clique of arity $n$, $\Pfr$ acts recursively
on $a_1 \otimes \dots \otimes a_n$ as follows. One has
\begin{equation}
    \Pfr \cdot a_1 = a_1
\end{equation}
when $\Pfr = \UnitClique$, and
\begin{equation} \label{equ:NC_M_algebras_action_cliques}
    \Pfr \cdot \left(a_1 \otimes \dots \otimes a_n\right) =
    \Bfr \cdot \left(
        \left(\Rfr_1 \cdot \left(a_1 \otimes \dots
                \otimes a_{|\Rfr_1|}\right)\right)
        \otimes \dots \otimes
        \left(\Rfr_k \cdot
            \left(a_{|\Rfr_1| + \dots + |\Rfr_{k - 1}| + 1}
            \otimes \dots  \otimes a_n\right)\right)
    \right),
\end{equation}
where, by setting $\Tfr$ as the bubble tree $\BubbleTree(\Pfr)$ of
$\Pfr$ (see Section~\ref{subsubsec:treelike_bubbles}), $\Bfr$ and
$\Rfr_1$, \dots, $\Rfr_k$ are the unique $\Mca$-bubble and noncrossing
$\Mca$-cliques such that
\begin{math}
    \Tfr = \Corolla(\Bfr)
    \circ [\BubbleTree(\Rfr_1), \dots, \BubbleTree(\Rfr_k)].
\end{math}
\medskip

Here are few examples of the construction provided by
Theorem~\ref{thm:NC_M_algebras}.
\begin{description}[fullwidth]
    \item[Noncommutative polynomials and selected concatenation]
    Let us consider the unitary magma $\Sbb_\ell$ of all subsets of
    $[\ell]$ with the union as product. Let
    $A := \{a_j : j \in [\ell]\}$ be an alphabet of noncommutative
    letters. We define on the associative algebra
    $\K \langle A \rangle$ of polynomials on $A$ the linear maps
    \begin{equation}
        \omega_S : \K \langle A \rangle \to \K \langle A \rangle,
        \qquad S \in \Sbb_\ell,
    \end{equation}
    as follows. For any $u \in A^*$ and $S \in \Sbb_\ell$, we set
    \begin{equation}
        \omega_S(u) :=
        \begin{cases}
            u & \mbox{if } |u|_{a_j} \geq 1 \mbox{ for all } j \in S, \\
            0 & \mbox{otherwise}.
        \end{cases}
    \end{equation}
    Since, for all $u \in A^*$, $\omega_{\emptyset}(u) = u$ and
    $(\omega_S \circ \omega_{S'})(u) = \omega_{S \cup S'}(u)$, and
    $\emptyset$ is the unit of $\Sbb_\ell$, we obtain from
    Theorem~\ref{thm:NC_M_algebras} that the
    operations~\eqref{equ:NC_M_algebras} endow $ \K \langle A \rangle$
    with an $\NC\Sbb_\ell$-algebra structure. For instance, when
    $\ell := 3$, one has
    \begin{subequations}
    \begin{equation}
        \left(a_1 + a_1 a_3 + a_2 a_2\right)
        \;
        \begin{tikzpicture}[scale=.52,Centering]
            \node[shape=coordinate](1)at(0,1){};
            \node[shape=coordinate](2)at(0.87,-0.5){};
            \node[shape=coordinate](3)at(-0.87,-0.5){};
            \draw[draw=Sepia!90](1)edge[]node[CliqueLabel,font=\tiny]
                {\begin{math}\{2\}\end{math}}(2);
            \draw[draw=Sepia!90](1)edge[]node[CliqueLabel,font=\tiny]
                {\begin{math}\{1\}\end{math}}(3);
            \draw[draw=Sepia!90](2)edge[]node[CliqueLabel,font=\tiny]
                {\begin{math}\{2, 3\}\end{math}}(3);
        \end{tikzpicture}
        \;
        \left(1 + a_3 + a_2 a_1\right)
        =
        a_1 a_3 a_2 a_1,
    \end{equation}
    \begin{equation}
        \left(a_1 + a_1 a_3 + a_2 a_2\right)
        \;
        \begin{tikzpicture}[scale=.52,Centering]
            \node[shape=coordinate](1)at(0,1){};
            \node[shape=coordinate](2)at(0.87,-0.5){};
            \node[shape=coordinate](3)at(-0.87,-0.5){};
            \draw[draw=Sepia!90](1)edge[]node[CliqueLabel,font=\tiny]
                {\begin{math}\emptyset\end{math}}(2);
            \draw[draw=Sepia!90](1)edge[]node[CliqueLabel,font=\tiny]
                {\begin{math}\{1\}\end{math}}(3);
            \draw[draw=Sepia!90](2)edge[]node[CliqueLabel,font=\tiny]
                {\begin{math}\{1, 3\}\end{math}}(3);
        \end{tikzpicture}
        \;
        \left(1 + a_3 + a_2 a_1\right)
        =
        2 \; a_1 a_3 + a_1 a_3 a_3 +
        a_1 a_3 a_2 a_1.
    \end{equation}
    \end{subequations}
    Besides, to compute the action
    \begin{equation} \label{equ:example_action_S_clique}
        \begin{tikzpicture}[scale=1.3,Centering]
            \node[CliquePoint](1)at(-0.34,-0.94){};
            \node[CliquePoint](2)at(-0.87,-0.50){};
            \node[CliquePoint](3)at(-0.98,0.17){};
            \node[CliquePoint](4)at(-0.64,0.77){};
            \node[CliquePoint](5)at(-0.00,1.00){};
            \node[CliquePoint](6)at(0.64,0.77){};
            \node[CliquePoint](7)at(0.98,0.17){};
            \node[CliquePoint](8)at(0.87,-0.50){};
            \node[CliquePoint](9)at(0.34,-0.94){};
            \draw[CliqueEmptyEdge](1)edge[]node[CliqueLabel]{}(2);
            \draw[CliqueEdge](1)edge[bend left=30]node[CliqueLabel,near start]
                {\begin{math}\{1\}\end{math}}(7);
            \draw[CliqueEmptyEdge](1)edge[]node[CliqueLabel]{}(9);
            \draw[CliqueEdge](2)edge[]node[CliqueLabel]
                {\begin{math}\{1\}\end{math}}(3);
            \draw[CliqueEdge](2)edge[bend right=30]node[CliqueLabel]
                {\begin{math}\{1\}\end{math}}(4);
            \draw[CliqueEdge](3)edge[]node[CliqueLabel]
                {\begin{math}\{2\}\end{math}}(4);
            \draw[CliqueEdge](4)edge[bend right=30]node[CliqueLabel]
                {\begin{math}\{1, 2\}\end{math}}(6);
            \draw[CliqueEdge](4)edge[]node[CliqueLabel]
                {\begin{math}\{2\}\end{math}}(5);
            \draw[CliqueEmptyEdge](5)edge[]node[CliqueLabel]{}(6);
            \draw[CliqueEdge](6)edge[]node[CliqueLabel]
                {\begin{math}\{3\}\end{math}}(7);
            \draw[CliqueEmptyEdge](7)edge[]node[CliqueLabel]{}(8);
            \draw[CliqueEdge](7)edge[]node[CliqueLabel]
                {\begin{math}\{1, 2\}\end{math}}(9);
            \draw[CliqueEmptyEdge](8)edge[]node[CliqueLabel]{}(9);
        \end{tikzpicture}
        \cdot
        (f \otimes f \otimes f \otimes f \otimes f \otimes f \otimes f
        \otimes f)
    \end{equation}
    where $f := a_1 + a_2 + a_3$, we use the above algorithm
    and~\eqref{equ:NC_M_algebras_action_bubbles}
    and~\eqref{equ:NC_M_algebras_action_cliques}. By presenting the
    computation on the bubble tree of the noncrossing $\Sbb_3$-clique
    of~\eqref{equ:example_action_S_clique}, we obtain
    \begin{equation}
        \begin{tikzpicture}[xscale=.8,yscale=.6,Centering]
            \node(0)at(0.00,-6.50){};
            \node(1)at(1.00,-9.75){};
            \node(10)at(10.00,-6.50){};
            \node(12)at(12.00,-6.50){};
            \node(3)at(3.00,-9.75){};
            \node(5)at(5.00,-9.75){};
            \node(7)at(7.00,-9.75){};
            \node(8)at(8.00,-6.50){};
            \node(11)at(11.00,-3.25){
                \begin{tikzpicture}[scale=.55]
                    \node[CliquePoint](A1)at(-0.87,-0.50){};
                    \node[CliquePoint](A2)at(-0.00,1.00){};
                    \node[CliquePoint](A3)at(0.87,-0.50){};
                    \draw[CliqueEmptyEdge](A1)edge[]node[CliqueLabel]
                        {}(A2);
                    \draw[CliqueEmptyEdge](A1)edge[]node[CliqueLabel]
                        {}(A3);
                    \draw[CliqueEmptyEdge](A2)edge[]node[CliqueLabel]
                        {}(A3);
                \end{tikzpicture}
            };
            \node(2)at(2.00,-6.50){
                \begin{tikzpicture}[scale=.55]
                    \node[CliquePoint](B1)at(-0.87,-0.50){};
                    \node[CliquePoint](B2)at(-0.00,1.00){};
                    \node[CliquePoint](B3)at(0.87,-0.50){};
                    \draw[CliqueEdge](B1)edge[]node[CliqueLabel]
                        {\begin{math}\{1\}\end{math}}(B2);
                    \draw[CliqueEmptyEdge](B1)edge[]node[CliqueLabel]
                        {}(B3);
                    \draw[CliqueEdge](B2)edge[]node[CliqueLabel]
                        {\begin{math}\{2\}\end{math}}(B3);
                \end{tikzpicture}
            };
            \node(4)at(4.00,-3.25){
                \begin{tikzpicture}[scale=.75]
                \node[CliquePoint](C1)at(-0.59,-0.81){};
                \node[CliquePoint](C2)at(-0.95,0.31){};
                \node[CliquePoint](C3)at(-0.00,1.00){};
                \node[CliquePoint](C4)at(0.95,0.31){};
                \node[CliquePoint](C5)at(0.59,-0.81){};
                \draw[CliqueEmptyEdge](C1)edge[]node[CliqueLabel]{}(C2);
                \draw[CliqueEmptyEdge](C1)edge[]node[CliqueLabel]{}(C5);
                \draw[CliqueEdge](C2)edge[]node[CliqueLabel]
                    {\begin{math}\{1\}\end{math}}(C3);
                \draw[CliqueEdge](C3)edge[]node[CliqueLabel]
                    {\begin{math}\{1, 2\}\end{math}}(C4);
                \draw[CliqueEdge](C4)edge[]node[CliqueLabel]
                    {\begin{math}\{3\}\end{math}}(C5);
                \end{tikzpicture}
            };
            \node(6)at(6.00,-6.50){
                \begin{tikzpicture}[scale=.55]
                    \node[CliquePoint](D1)at(-0.87,-0.50){};
                    \node[CliquePoint](D2)at(-0.00,1.00){};
                    \node[CliquePoint](D3)at(0.87,-0.50){};
                    \draw[CliqueEdge](D1)edge[]node[CliqueLabel]
                        {\begin{math}\{2\}\end{math}}(D2);
                    \draw[CliqueEmptyEdge](D1)edge[]node[CliqueLabel]
                        {}(D3);
                    \draw[CliqueEmptyEdge](D2)edge[]node[CliqueLabel]
                        {}(D3);
                \end{tikzpicture}
            };
            \node(9)at(9.00,0.00){
                \begin{tikzpicture}[scale=.55]
                    \node[CliquePoint](E1)at(-0.87,-0.50){};
                    \node[CliquePoint](E2)at(-0.00,1.00){};
                    \node[CliquePoint](E3)at(0.87,-0.50){};
                    \draw[CliqueEdge](E1)edge[]node[CliqueLabel]
                        {\begin{math}\{1\}\end{math}}(E2);
                    \draw[CliqueEmptyEdge](E1)edge[]node[CliqueLabel]
                        {}(E3);
                    \draw[CliqueEdge](E2)edge[]node[CliqueLabel]
                        {\begin{math}\{1, 2\}\end{math}}(E3);
                \end{tikzpicture}
            };
            \draw(0)edge[Edge]node[CliqueLabel]
                {\begin{math}
                    \textcolor{NavyBlue}{f}
                \end{math}}(4);
            \draw(1)edge[Edge]node[CliqueLabel]
                {\begin{math}
                    \textcolor{NavyBlue}{f}
                \end{math}}(2);
            \draw(10)edge[Edge]node[CliqueLabel]
                {\begin{math}
                    \textcolor{NavyBlue}{f}
                \end{math}}(11);
            \draw(11)edge[Edge]node[CliqueLabel]
                {\begin{math}
                    \textcolor{NavyBlue}{(a_1 + a_2 + a_3)^2}
                \end{math}}(9);
            \draw(12)edge[Edge]node[CliqueLabel]
                {\begin{math}
                    \textcolor{NavyBlue}{f}
                \end{math}}(11);
            \draw(2)edge[Edge]node[CliqueLabel]
                {\begin{math}
                    \textcolor{NavyBlue}{a_1 a_2}
                \end{math}}(4);
            \draw(3)edge[Edge]node[CliqueLabel]
                {\begin{math}
                    \textcolor{NavyBlue}{f}
                \end{math}}(2);
            \draw(4)edge[Edge]node[CliqueLabel]
                {\begin{math}
                    \textcolor{NavyBlue}
                    {(a_1 + a_2 + a_3) a_1 a_2 a_2 a_1 a_3}
                \end{math}}(9);
            \draw(5)edge[Edge]node[CliqueLabel]
                {\begin{math}
                    \textcolor{NavyBlue}{f}
                \end{math}}(6);
            \draw(6)edge[Edge]node[CliqueLabel]
                {\begin{math}
                    \textcolor{NavyBlue}{a_2 (a_1 + a_2 + a_3)\qquad}
                \end{math}}(4);
            \draw(7)edge[Edge]node[CliqueLabel]
                {\begin{math}
                    \textcolor{NavyBlue}{f}
                \end{math}}(6);
            \draw(8)edge[Edge]node[CliqueLabel]
                {\begin{math}
                    \textcolor{NavyBlue}{f}
                \end{math}}(4);
            \node(r)at(9.00,2){};
            \draw(r)edge[Edge]node[CliqueLabel]
                {\begin{math}
                    \textcolor{NavyBlue}
                    {(a_1 + a_2 + a_3) a_1 a_2 a_2 a_1 a_3
                    (a_1 a_2 + a_2 a_1)}
                \end{math}}(9);
        \end{tikzpicture}\,,
    \end{equation}
    so that~\eqref{equ:example_action_S_clique} is equal to the
    polynomial
    \begin{math}
        (a_1 + a_2 + a_3) a_1 a_2 a_2 a_1 a_3 (a_1 a_2 + a_2 a_1).
    \end{math}
    \medskip

    \item[Noncommutative polynomials and constant term product]
    Consider here the unitary magma $\Dbb_0$. Let
    $A := \{a_1, a_2, \dots\}$ be an infinite alphabet of noncommutative
    letters. We define on the associative algebra $\K \langle A \rangle$
    of polynomials on $A$ the linear maps
    \begin{equation}
        \omega_\Unit, \omega_0 :
        \K \langle A \rangle \to \K \langle A \rangle,
    \end{equation}
    as follows. For any $u \in A^*$, we set $\omega_\Unit(u) := u$, and
    \begin{equation}
        \omega_0(u) :=
        \begin{cases}
            1 & \mbox{if } u = \epsilon, \\
            0 & \mbox{otherwise}.
        \end{cases}
    \end{equation}
    In other terms, $\omega_0(f)$ is the constant term, denoted by
    $f(0)$, of the polynomial $f \in \K \langle A \rangle$. Since
    $\omega_\Unit$ is the identity map on $\K \langle A \rangle$ and,
    for all $u \in A^*$,
    \begin{equation}
        (\omega_0 \circ \omega_0)(f)
        = (f(0))(0)
        = f(0)
        = \omega_0(f),
    \end{equation}
    we obtain from
    Theorem~\ref{thm:NC_M_algebras} that the
    operations~\eqref{equ:NC_M_algebras} endow $\K \langle A \rangle$
    with a $\NC\Dbb_0$-algebra structure. For instance, for all
    polynomials $f_1$ and $f_2$ of $\K \langle A \rangle$, we have
    \vspace{-1.75em}
    \begin{multicols}{2}
    \begin{subequations}
    \begin{equation}
        f_1 \TriangleOp{\Unit}{\Unit}{\Unit} f_2 = f_1 f_2,
    \end{equation}
    \begin{equation}
        f_1 \TriangleOp{0}{\Unit}{\Unit} f_2 = (f_1 f_2)(0)
        = f_1(0) \; f_2(0),
    \end{equation}

    \begin{equation} \label{equ:constant_term_product_example_1}
        f_1 \TriangleOp{\Unit}{0}{\Unit} f_2 = f_1(0) \; f_2,
    \end{equation}
    \begin{equation} \label{equ:constant_term_product_example_2}
        f_1 \TriangleOp{\Unit}{\Unit}{0} f_2 = f_1 \; (f_2(0)).
    \end{equation}
    \end{subequations}
    \end{multicols}
    \noindent From~\eqref{equ:constant_term_product_example_1}
    and~\eqref{equ:constant_term_product_example_2}, when
    $f_1(0) = 1 = f_2(0)$,
    \begin{equation}
        f_1
        \; \left(
        \TriangleOp{\Unit}{0}{\Unit} + \TriangleOp{\Unit}{\Unit}{0}
        \right) \;
        f_2
        =
        f_1(0) \; f_2 + f_1 \; (f_2(0))
        = f_1 + f_2.
    \end{equation}
\end{description}
\medskip

\subsubsection{From monoids}
If $\Mca$ is a monoid, with binary associative operation $\Op$ and unit
$\Unit_\Mca$, we denote by $\K \langle \Mca^* \rangle$ the space of all
noncommutative polynomials on $\Mca$, seen as an alphabet, with
coefficients in $\K$. This space can be endowed with an
$\NC\Mca$-algebra structure as follows.
\medskip

For any $x \in \Mca$ and any word $w \in \Mca^*$, let
\begin{equation}
    x * w := (x \Op w_1) \dots (x \Op w_{|w|}).
\end{equation}
This operation $*$ is linearly extended on the right on
$\K \langle \Mca^* \rangle$.
\medskip

\begin{Proposition} \label{prop:NC_M_algebras_monoid_polynomials}
    Let $\Mca$ be a finite monoid. The vector space
    $\K \langle \Mca^* \rangle$ endowed with the binary linear
    operations
    \begin{equation} \label{equ:NC_M_algebras_monoid_polynomials_op}
        \TriangleOp{\Pfr_0}{\Pfr_1}{\Pfr_2} :
        \K \langle \Mca^* \rangle \otimes \K \langle \Mca^* \rangle
        \to \K \langle \Mca^* \rangle,
        \qquad \Pfr \in \Triangles_\Mca,
    \end{equation}
    defined for each $\Mca$-triangle $\Pfr$ and any
    $f_1, f_2 \in \K \langle \Mca^* \rangle$ by
    \begin{equation} \label{equ:NC_M_algebras_monoid_polynomials}
        f_1 \TriangleOp{\Pfr_0}{\Pfr_1}{\Pfr_2} f_2 :=
        \Pfr_0 * \left(\left(\Pfr_1 * f_1\right) \;
        \left(\Pfr_2 * f_2\right)\right),
    \end{equation}
    is an $\NC\Mca$-algebra.
\end{Proposition}
\begin{proof}
    This follows from Theorem~\ref{thm:NC_M_algebras} as a particular
    case of the general construction it provides. Indeed,
    $\K \langle \Mca^* \rangle$ is an associative algebra for the
    concatenation product of words. Moreover, by defining maps
    $\omega_x : \K \langle \Mca^* \rangle \to \K \langle \Mca^* \rangle$,
    $x \in \Mca$, linearly by $\omega_x(u) := x * u$ for any word
    $u \in \Mca^*$, we obtain, since $\Mca$ is a monoid, that this
    family of maps satisfies~\eqref{equ:compatible_magma_on_algebra_1}
    and~\eqref{equ:compatible_magma_on_algebra_2}. Now, since the
    definition~\eqref{equ:NC_M_algebras_monoid_polynomials} is the
    specialization of the definition~\eqref{equ:NC_M_algebras_def} in
    this particular case, the statement of the proposition follows.
\end{proof}
\medskip

Here are few examples of the construction provided by
Proposition~\ref{prop:NC_M_algebras_monoid_polynomials}.
\begin{description}[fullwidth]
    \item[Words and double shifted concatenation]
    Consider the monoid $\N_\ell$ for an $\ell \geq 1$. By
    Proposition~\ref{prop:NC_M_algebras_monoid_polynomials}, the
    operations~\eqref{equ:NC_M_algebras_monoid_polynomials_op} endow
    $\K \left\langle \N_\ell^* \right\rangle$ with a structure of an
    $\NC\N_\ell$-algebra. For instance, in
    $\K \left\langle \N_4^* \right\rangle$, one has
    \begin{equation}
        0211 \; \TriangleOp{1}{2}{0} \; 312 = 3100023.
    \end{equation}
    \medskip

    \item[Words and erasing concatenation]
    Consider the monoid $\Dbb_\ell$ for an $\ell \geq 0$. By
    Proposition~\ref{prop:NC_M_algebras_monoid_polynomials},
    the operations~\eqref{equ:NC_M_algebras_monoid_polynomials_op}
    endow $\K \langle \Dbb_\ell^* \rangle$ with a structure of an
    $\NC\Dbb_\ell$-algebra. For instance, for all words $u$ and $v$
    of $\Dbb_\ell^*$, we have
    \vspace{-1.75em}
    \begin{multicols}{2}
    \begin{subequations}
    \begin{equation}
        u \TriangleOp{\Unit}{\Unit}{\Unit} v = u v,
    \end{equation}
    \begin{equation}
        u \TriangleOp{\Dtt_i}{\Unit}{\Unit} v = (u v)_{\Dtt_i},
    \end{equation}

    \begin{equation}
        u \TriangleOp{0}{\Unit}{\Unit} v = 0^{|u| + |v|},
    \end{equation}
    \begin{equation}
        u \TriangleOp{\Unit}{\Dtt_i}{\Dtt_j} v = u_{\Dtt_i} \; v_{\Dtt_j},
    \end{equation}
    \end{subequations}
    \end{multicols}
    \noindent where, for any word $w$ of $\Dbb_\ell^*$ and any element
    $\Dtt_j$ of $\Dbb_\ell$, $j \in [\ell]$, $w_{\Dtt_j}$ is the word
    obtained by replacing each occurrence of $\Unit$ by $\Dtt_j$ and
    each occurrence of $\Dtt_i$, $i \in [\ell]$, by $0$ in~$w$.
\end{description}
\medskip

\subsection{Koszul dual}\label{subsec:dual_NC_M}
Since by Theorem~\ref{thm:presentation_NC_M}, the operad $\NC\Mca$ is
binary and quadratic, this operad admits a Koszul dual $\NC\Mca^!$. We
end the study of $\NC\Mca$ by collecting the main properties
of~$\NC\Mca^!$.
\medskip

\subsubsection{Presentation}
Let  $\Rel_{\NC\Mca}^!$ be the subspace of
$\Free\left(\Vect\left(\Triangles_\Mca\right)\right)(3)$ generated by
the elements
\begin{subequations}
\begin{equation} \label{equ:relation_1_NC_M_dual}
    \sum_{\substack{
        \Pfr_1, \Qfr_0 \in \Mca \\
        \Pfr_1 \Op \Qfr_0 = \delta
    }}
    \Corolla\left(\Triangle{\Pfr_0}{\Pfr_1}{\Pfr_2}\right)
    \circ_1
    \Corolla\left(\Triangle{\Qfr_0}{\Qfr_1}{\Qfr_2}\right),
    \qquad
    \Pfr_0, \Pfr_2, \Qfr_1, \Qfr_2 \in \Mca,
    \delta \in \bar{\Mca},
\end{equation}
\begin{equation} \label{equ:relation_2_NC_M_dual}
    \sum_{\substack{
        \Pfr_1, \Qfr_0 \in \Mca \\
        \Pfr_1 \Op \Qfr_0 = \Unit_\Mca
    }}
    \Corolla\left(\Triangle{\Pfr_0}{\Pfr_1}{\Pfr_2}\right)
    \circ_1
    \Corolla\left(\Triangle{\Qfr_0}{\Qfr_1}{\Qfr_2}\right)
    -
    \Corolla\left(\Triangle{\Pfr_0}{\Qfr_1}{\Pfr_1}\right)
    \circ_2
    \Corolla\left(\Triangle{\Qfr_0}{\Qfr_2}{\Pfr_2}\right),
    \qquad
    \Pfr_0, \Pfr_2, \Qfr_1, \Qfr_2 \in \Mca,
\end{equation}
\begin{equation} \label{equ:relation_3_NC_M_dual}
    \sum_{\substack{
        \Pfr_2, \Qfr_0 \in \Mca \\
        \Pfr_2 \Op \Qfr_0 = \delta
    }}
    \Corolla\left(\Triangle{\Pfr_0}{\Pfr_1}{\Pfr_2}\right)
    \circ_2
    \Corolla\left(\Triangle{\Qfr_0}{\Qfr_1}{\Qfr_2}\right),
    \qquad
    \Pfr_0, \Pfr_1, \Qfr_1, \Qfr_2 \in \Mca,
    \delta \in \bar{\Mca},
\end{equation}
\end{subequations}
where $\Pfr$ and $\Qfr$ are $\Mca$-triangles.
\medskip

\begin{Proposition} \label{prop:presentation_dual_NC_M}
    Let $\Mca$ be a finite unitary magma. Then, the Koszul dual
    $\NC\Mca^!$ of $\NC\Mca$ admits the presentation
    $\left(\Triangles_\Mca, \Rel_{\NC\Mca}^!\right)$.
\end{Proposition}
\begin{proof}
    Let
    \begin{equation}
        f :=
        \sum_{\Tfr \in T_3} \lambda_\Tfr \Tfr
    \end{equation}
    be a generic element of $\Rel_{\NC\Mca}^!$, where $T_3$ is the set
    of all syntax trees on $\Triangles_\Mca$ or arity $3$ and the
    $\lambda_\Tfr$ are coefficients of $\K$. By definition of Koszul
    duality of operads, $\langle r, f \rangle = 0$ for all
    $r \in \Rel_{\NC\Mca}$, where $\langle -, - \rangle$ is the scalar
    product defined in~\eqref{equ:scalar_product_koszul}. Then, since
    $\Rel_{\NC\Mca}$ is the subspace of
    $\Free\left(\Vect\left(\Triangles_\Mca\right)\right)(3)$ generated
    by~\eqref{equ:relation_1_NC_M},
    \eqref{equ:relation_2_NC_M}, and~\eqref{equ:relation_3_NC_M}, one
    has
    \begin{subequations}
    \begin{equation} \label{prop:presentation_dual_NC_M_proof_1}
        \lambda_{\Corolla\left(\Triangle{\Pfr_0}{\Pfr_1}{\Pfr_2}\right)
            \circ_1
            \Corolla\left(\Triangle{\Qfr_0}{\Qfr_1}{\Qfr_2}\right)}
        -
        \lambda_{\Corolla\left(\Triangle{\Pfr_0}{\Rfr_1}{\Pfr_2}\right)
            \circ_1
            \Corolla\left(\Triangle{\Rfr_0}{\Qfr_1}{\Qfr_2}\right)}
        = 0,
        \qquad
        \Pfr_1 \Op \Qfr_0 = \Rfr_1 \Op \Rfr_0 \ne \Unit_\Mca,
    \end{equation}
    \begin{equation} \label{prop:presentation_dual_NC_M_proof_2}
        \lambda_{\Corolla\left(\Triangle{\Pfr_0}{\Pfr_1}{\Pfr_2}\right)
            \circ_1
            \Corolla\left(\Triangle{\Qfr_0}{\Qfr_1}{\Qfr_2}\right)}
        +
        \lambda_{\Corolla\left(\Triangle{\Pfr_0}{\Qfr_1}{\Rfr_2}\right)
            \circ_2
            \Corolla\left(\Triangle{\Rfr_0}{\Qfr_2}{\Pfr_2}\right)}
        = 0,
        \qquad
        \Pfr_1 \Op \Qfr_0 = \Rfr_2 \Op \Rfr_0 = \Unit_\Mca,
    \end{equation}
    \begin{equation} \label{prop:presentation_dual_NC_M_proof_3}
        \lambda_{\Corolla\left(\Triangle{\Pfr_0}{\Pfr_1}{\Pfr_2}\right)
            \circ_2
            \Corolla\left(\Triangle{\Qfr_0}{\Qfr_1}{\Qfr_2}\right)}
        -
        \lambda_{\Corolla\left(\Triangle{\Pfr_0}{\Pfr_1}{\Rfr_2}\right)
            \circ_2
            \Corolla\left(\Triangle{\Rfr_0}{\Qfr_1}{\Qfr_2}\right)}
        = 0,
        \qquad
        \Pfr_2 \Op \Qfr_0 = \Rfr_2 \Op \Rfr_0 \ne \Unit_\Mca,
    \end{equation}
    \end{subequations}
    where $\Pfr$, $\Qfr$, and $\Rfr$ are $\Mca$-triangles. This implies
    that $f$ is of the form
    \begin{equation}\begin{split}
        f & =
        \sum_{\substack{
            \Pfr_0, \Pfr_2, \Qfr_1, \Qfr_2 \in \Mca \\
            \delta \in \bar{\Mca}
        }}
        \lambda^{(1)}_{\Pfr_0, \Pfr_2, \Qfr_1, \Qfr_2, \delta}
        \sum_{\substack{
            \Pfr_1, \Qfr_0 \in \Mca \\
            \Pfr_1 \Op \Qfr_0 = \delta
        }}
        \Corolla\left(\Triangle{\Pfr_0}{\Pfr_1}{\Pfr_2}\right)
        \circ_1
        \Corolla\left(\Triangle{\Qfr_0}{\Qfr_1}{\Qfr_2}\right) \\
        & +
        \sum_{\Pfr_0, \Pfr_2, \Qfr_1, \Qfr_2 \in \Mca}
        \lambda^{(2)}_{\Pfr_0, \Pfr_2, \Qfr_1, \Qfr_2}
        \left(
        \sum_{\substack{
            \Pfr_1, \Qfr_0 \in \Mca \\
            \Pfr_1 \Op \Qfr_0 = \Unit_\Mca
        }}
        \Corolla\left(\Triangle{\Pfr_0}{\Pfr_1}{\Pfr_2}\right)
        \circ_1
        \Corolla\left(\Triangle{\Qfr_0}{\Qfr_1}{\Qfr_2}\right)
        -
        \Corolla\left(\Triangle{\Pfr_0}{\Qfr_1}{\Pfr_1}\right)
        \circ_2
        \Corolla\left(\Triangle{\Qfr_0}{\Qfr_2}{\Pfr_2}\right)
        \right) \\
        & +
        \sum_{\substack{
            \Pfr_0, \Pfr_1, \Qfr_1, \Qfr_2 \in \Mca \\
            \delta \in \bar{\Mca}
        }}
        \lambda^{(3)}_{\Pfr_0, \Pfr_1, \Qfr_1, \Qfr_2, \delta}
        \sum_{\substack{
            \Pfr_2, \Qfr_0 \in \Mca \\
            \Pfr_2 \Op \Qfr_0 = \delta
        }}
        \Corolla\left(\Triangle{\Pfr_0}{\Pfr_1}{\Pfr_2}\right)
        \circ_2
        \Corolla\left(\Triangle{\Qfr_0}{\Qfr_1}{\Qfr_2}\right),
    \end{split}\end{equation}
    where, for any $\Mca$-triangles $\Pfr$ and $\Qfr$ and any
    $\delta \in \bar{\Mca}$, the
    $\lambda^{(1)}_{\Pfr_0, \Pfr_2, \Qfr_1, \Qfr_0, \delta}$,
    $\lambda^{(2)}_{\Pfr_0, \Pfr_2, \Qfr_1, \Qfr_2}$, and
    $\lambda^{(3)}_{\Pfr_1, \Pfr_0, \Qfr_1, \Qfr_2, \delta}$ are
    coefficients of $\K$. Therefore, $f$ belongs to the space generated
    by~\eqref{equ:relation_1_NC_M_dual},
    \eqref{equ:relation_2_NC_M_dual},
    and~\eqref{equ:relation_3_NC_M_dual}. Finally, since the
    coefficients of each of these relations
    satisfy~\eqref{prop:presentation_dual_NC_M_proof_1},
    \eqref{prop:presentation_dual_NC_M_proof_2},
    and~\eqref{prop:presentation_dual_NC_M_proof_3}, the statement of
    the proposition follows.
\end{proof}
\medskip

Let us use Proposition~\ref{prop:presentation_dual_NC_M} to express the
presentations of the operads $\NC\N_2^!$ and $\NC\Dbb_0^!$. The operad
$\NC\N_2^!$ is generated by
\begin{equation}
    \Triangles_{\N_2} =
    \left\{
    \TriangleEEE{}{}{},
    \TriangleEXE{}{1}{},
    \TriangleEEX{}{}{1},
    \TriangleEXX{}{1}{1},
    \TriangleXEE{1}{}{},
    \TriangleXXE{1}{1}{},
    \TriangleXEX{1}{}{1},
    \Triangle{1}{1}{1}
    \right\},
\end{equation}
and these generators are subjected exactly to the nontrivial relations
\begin{subequations}
\begin{equation}
    \TriangleXEX{a}{}{b_3} \circ_1 \Triangle{1}{b_1}{b_2}
    +
    \Triangle{a}{1}{b_3} \circ_1 \TriangleEXX{}{b_1}{b_2}
    = 0,
    \qquad
    a, b_1, b_2, b_3 \in \N_2,
\end{equation}
\begin{equation}
    \Triangle{a}{1}{b_3} \circ_1 \Triangle{1}{b_1}{b_2}
    +
    \TriangleXEX{a}{}{b_3} \circ_1 \TriangleEXX{}{b_1}{b_2}
    =
    \TriangleXXE{a}{b_1}{} \circ_2 \TriangleEXX{}{b_2}{b_3}
    +
    \Triangle{a}{b_1}{1} \circ_2 \Triangle{1}{b_2}{b_3},
    \qquad
    a, b_1, b_2, b_3 \in \N_2,
\end{equation}
\begin{equation}
    \TriangleXXE{a}{b_1}{} \circ_2 \Triangle{1}{b_2}{b_3}
    +
    \Triangle{a}{b_1}{1} \circ_2 \TriangleEXX{}{b_2}{b_3}
    = 0,
    \qquad
    a, b_1, b_2, b_3 \in \N_2.
\end{equation}
\end{subequations}
On the other hand, the operad $\NC\Dbb_0^!$ is generated by
\begin{equation}
    \Triangles_{\Dbb_0} =
    \left\{
    \TriangleEEE{}{}{},
    \TriangleEXE{}{0}{},
    \TriangleEEX{}{}{0},
    \TriangleEXX{}{0}{0},
    \TriangleXEE{0}{}{},
    \TriangleXXE{0}{0}{},
    \TriangleXEX{0}{}{0},
    \Triangle{0}{0}{0}
    \right\},
\end{equation}
and these generators are subjected exactly to the nontrivial relations
\begin{subequations}
\begin{equation}
    \TriangleXEX{a}{}{b_3} \circ_1 \Triangle{0}{b_1}{b_2}
    +
    \Triangle{a}{0}{b_3} \circ_1 \Triangle{0}{b_1}{b_2}
    +
    \Triangle{a}{0}{b_3} \circ_1 \TriangleEXX{}{b_1}{b_2}
    = 0,
    \qquad
    a, b_1, b_2, b_3 \in \Dbb_0,
\end{equation}
\begin{equation}
    \TriangleXEX{a}{}{b_3} \circ_1 \TriangleEXX{}{b_1}{b_2}
    =
    \TriangleXXE{a}{b_1}{} \circ_2 \TriangleEXX{}{b_2}{b_3},
    \qquad
    a, b_1, b_2, b_3 \in \Dbb_0,
\end{equation}
\begin{equation}
    \TriangleXXE{a}{b_1}{} \circ_2 \Triangle{0}{b_2}{b_3}
    +
    \Triangle{a}{b_1}{0} \circ_2 \Triangle{0}{b_2}{b_3}
    +
    \Triangle{a}{b_1}{0} \circ_2 \TriangleEXX{}{b_2}{b_3}
    = 0,
    \qquad
    a, b_1, b_2, b_3 \in \Dbb_0.
\end{equation}
\end{subequations}
\medskip

\begin{Proposition} \label{prop:dimensions_relations_NC_M_dual}
    Let $\Mca$ be a finite unitary magma. Then, the dimension of the
    space $\Rel_{\NC\Mca}^!$ satisfies
    \begin{equation} \label{equ:dimensions_relations_NC_M_dual}
        \dim \Rel_{\NC\Mca}^! = 2m^5 - m^4,
    \end{equation}
    where $m := \# \Mca$.
\end{Proposition}
\begin{proof}
    To compute the dimension of the space of relations
    $\Rel_{\NC\Mca}^!$ of $\NC\Mca^!$, we consider the presentation of
    $\NC\Mca^!$ provided by
    Proposition~\ref{prop:presentation_dual_NC_M}. Consider the space
    $\Rel_1$ generated by the family consisting in the
    elements~\eqref{equ:relation_1_NC_M_dual}. Since this family is
    linearly independent and each of its element is totally specified by
    a tuple
    \begin{math}
        (\Pfr_0, \Pfr_2, \Qfr_1, \Qfr_2, \delta)
        \in \Mca^4 \times \bar{\Mca},
    \end{math}
    we obtain
    \begin{equation}
        \dim \Rel_1 = m^4 (m - 1).
    \end{equation}
    For the same reason, the dimension of the space $\Rel_3$
    generated by the elements~\eqref{equ:relation_3_NC_M_dual}
    satisfies $\dim \Rel_3 = \dim \Rel_1$. Now, let $\Rel_2$ be
    the space generated by the
    elements~\eqref{equ:relation_2_NC_M_dual}. Since this family is
    linearly independent and each of its elements is totally specified
    by a tuple
    \begin{math}
        (\Pfr_0, \Pfr_2, \Qfr_1, \Qfr_2) \in \Mca^4,
    \end{math}
    we obtain
    \begin{equation}
        \dim \Rel_2 = m^4.
    \end{equation}
    Therefore, since
    \begin{equation}
        \Rel_{\NC\Mca}^! =
        \Rel_1 \oplus \Rel_2 \oplus \Rel_3,
    \end{equation}
    we obtain the stated
    formula~\eqref{equ:dimensions_relations_NC_M_dual} by summing the
    dimensions of $\Rel_1$, $\Rel_2$, and $\Rel_3$.
\end{proof}
\medskip

Observe that, by Propositions~\ref{prop:dimensions_relations_NC_M}
and~\ref{prop:dimensions_relations_NC_M_dual}, we have
\begin{equation}\begin{split}
    \dim \Rel_{\NC\Mca} + \dim \Rel_{\NC\Mca}^!
        & = 2m^6 - 2m^5 + m^4
          + 2m^5 - m^4 \\
        & = 2m^6 \\
        & = \dim \Free\left(\Vect\left(\Triangles_\Mca\right)\right)(3),
\end{split}\end{equation}
as expected by Koszul duality, where $m := \# \Mca$.
\medskip

\subsubsection{Dimensions}

\begin{Proposition} \label{prop:Hilbert_series_NC_M_dual}
    Let $\Mca$ be a finite unitary magma. The Hilbert series
    $\Hilbert_{\NC\Mca^!}(t)$ of $\NC\Mca^!$ satisfies
    \begin{equation} \label{equ:Hilbert_series_NC_M_dual}
        t + (m - 1)t^2
        + \left(2m^2t - 3mt + 2t -1\right)\Hilbert_{\NC\Mca^!}(t)
        + \left(m^3 - 2m^2 + 2m - 1\right)\Hilbert_{\NC\Mca^!}(t)^2 = 0,
    \end{equation}
    where $m := \# \Mca$.
\end{Proposition}
\begin{proof}
    Let $G(t)$ be the generating series such that $G(-t)$
    satisfies~\eqref{equ:Hilbert_series_NC_M_dual}. Therefore, $G(t)$
    satisfies
    \begin{equation} \label{equ:Hilbert_series_NC_M_dual_1}
        -t + (m - 1)t^2
        + \left(-2m^2t + 3mt - 2t -1\right)G(t)
        + \left(m^3 - 2m^2 + 2m - 1\right)G(t)^2 = 0,
    \end{equation}
    and, by solving~\eqref{equ:Hilbert_series_NC_M_dual_1} as a
    quadratic equation where $t$ is the unknown, we obtain
    \begin{equation} \label{equ:Hilbert_series_NC_M_dual_2}
        t =
        \frac{1 + (2m^2 - 3m + 2)G(t)
        - \sqrt{1 + 2(2m^2 - m)G(t)
        + m^2G(t)^2}}{2(m - 1)}.
    \end{equation}
    Moreover, by Proposition~\ref{prop:Hilbert_series_NC_M}
    and~\eqref{equ:Hilbert_series_NC_M_function}, by setting
    $F(t) := \Hilbert_{\NC\Mca}(-t)$, we have
    \begin{equation} \label{equ:Hilbert_series_NC_M_dual_3}
       F(G(t)) =
        \frac{1 + (2m^2 - 3m + 2)G(t)
        - \sqrt{1 + 2(2m^2 - m)G(t)
        + m^2G(t)^2}}{2(m - 1)}
        = t,
    \end{equation}
    showing that $F(t)$ and $G(t)$ are the inverses for each other for
    series composition.
    \smallskip

    Now, since by Theorem~\ref{thm:Koszul_NC_M}, $\NC\Mca$ is a Koszul
    operad, the Hilbert series of $\NC\Mca$ and $\NC\Mca^!$
    satisfy~\eqref{equ:Hilbert_series_Koszul_operads}. Therefore,
    \eqref{equ:Hilbert_series_NC_M_dual_3} implies that the Hilbert
    series of $\NC\Mca^!$ is the series $\Hilbert_{\NC\Mca^!}(t)$,
    satisfying the stated
    relation~\eqref{equ:Hilbert_series_NC_M_dual}.
\end{proof}
\medskip

We deduce from Proposition~\ref{prop:Hilbert_series_NC_M_dual} that the
Hilbert series of $\NC\Mca^!$ satisfies
\begin{equation} \label{equ:Hilbert_function_NC_M_dual}
    \Hilbert_{\NC\Mca^!}(t) =
    \frac{1 - (2m^2 - 3m + 2)t -\sqrt{1 - 2(2m^3 - 2m^2 + m)t +m^2t^2}}
    {2(m^3 - 2m^2 + 2m - 1)},
\end{equation}
where $m := \# \Mca \ne 1$.
\medskip

\begin{Proposition} \label{prop:dimensions_NC_M_dual}
    Let $\Mca$ be a finite unitary magma. For all $n \geq 2$,
    \begin{equation} \label{equ:dimensions_NC_M_dual}
        \dim \NC\Mca^!(n)
        =
        \sum_{0 \leq k \leq n - 2}
        m^{n + 1}
        (m(m - 1) + 1)^k (m (m - 1))^{n - k - 2}
        \; \Nar(n, k).
    \end{equation}
\end{Proposition}
\begin{proof}
    The proof consists in enumerating dual $\Mca$-cliques, introduced in
    upcoming Section~\ref{subsubsec:basis_Cli_M_dual}. Indeed, by
    Proposition~\ref{prop:elements_NC_M_dual}, $\dim \NC\Mca^!(n)$ is
    equal to the number of dual $\Mca$-cliques of arity $n$. The
    expression for $\dim \NC\Mca^!(n)$ claimed
    by~\eqref{equ:dimensions_NC_M_dual} can be proved by using similar
    arguments as the ones intervening in the proof of
    Proposition~\ref{prop:dimensions_NC_M} for the
    expression~\eqref{equ:dimensions_NC_M} of $\dim \NC\Mca(n)$.
\end{proof}
\medskip

We can use Proposition~\ref{prop:dimensions_NC_M_dual} to compute the
first dimensions of $\NC\Mca^!$. For instance, depending on
$m := \# \Mca$, we have the following sequences of dimensions:
\begin{subequations}
\begin{equation}
    1, 1, 1, 1, 1, 1, 1, 1,
    \qquad m = 1,
\end{equation}
\begin{equation}
    1, 8, 80, 992, 13760, 204416, 3180800, 51176960,
    \qquad m = 2,
\end{equation}
\begin{equation}
    1, 27, 1053, 51273, 2795715, 163318599, 9994719033, 632496651597,
    \qquad m = 3,
\end{equation}
\begin{equation}
    1, 64, 6400, 799744, 111923200, 16782082048, 2636161024000,
    428208345579520,
    \qquad m = 4.
\end{equation}
\end{subequations}
The second one is Sequence~\OEIS{A234596} of~\cite{Slo}. The last two
sequences are not listed in~\cite{Slo} at this time. It is worthwhile
to observe that the dimensions of $\NC\Mca^!$ when $\# \Mca = 2$ are the
ones of the operad~$\BNC$ of bicolored noncrossing configurations (see
Section~\ref{subsec:operad_bnc}).
\medskip

\subsubsection{Basis} \label{subsubsec:basis_Cli_M_dual}
To describe a basis of $\NC\Mca^!$, we introduce the following sort
of $\Mca$-decorated cliques.
A \Def{dual $\Mca$-clique} is an $\Mca^2$-clique such that its base
and its edges are labeled by pairs $(a, a) \in \Mca^2$, and all solid
diagonals are labeled by pairs $(a, b) \in \Mca^2$ with $a \ne b$.
Observe that a non-solid diagonal of a dual $\Mca$-clique is
labeled by $(\Unit_\Mca, \Unit_\Mca)$. All definitions about
$\Mca$-cliques of Section~\ref{subsec:decorated_cliques} remain
valid for dual $\Mca$-cliques. For example,
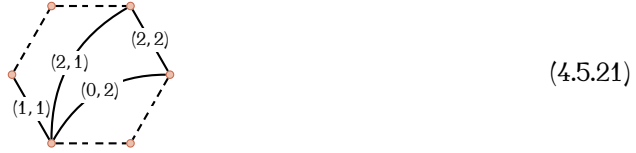
\begin{equation}
    \begin{tikzpicture}[scale=1.05,Centering]
        \node[CliquePoint](1)at(-0.50,-0.87){};
        \node[CliquePoint](2)at(-1.00,-0.00){};
        \node[CliquePoint](3)at(-0.50,0.87){};
        \node[CliquePoint](4)at(0.50,0.87){};
        \node[CliquePoint](5)at(1.00,0.00){};
        \node[CliquePoint](6)at(0.50,-0.87){};
        \draw[CliqueEdge](1)edge[]node[CliqueLabel]
            {\begin{math}(1, 1)\end{math}}(2);
        \draw[CliqueEdge](1)edge[bend left=30]node[CliqueLabel]
            {\begin{math}(0, 2)\end{math}}(5);
        \draw[CliqueEmptyEdge](1)edge[]node[CliqueLabel]{}(6);
        \draw[CliqueEdge](1)edge[bend left=30]node[CliqueLabel]
            {\begin{math}(2, 1)\end{math}}(4);
        \draw[CliqueEmptyEdge](3)edge[]node[CliqueLabel]{}(4);
        \draw[CliqueEdge](4)edge[]node[CliqueLabel]
            {\begin{math}(2, 2)\end{math}}(5);
        \draw[CliqueEmptyEdge](2)edge[]node[CliqueLabel]{}(3);
        \draw[CliqueEmptyEdge](5)edge[]node[CliqueLabel]{}(6);
    \end{tikzpicture}
\end{equation}
is a noncrossing dual $\N_3$-clique.
\medskip

\begin{Proposition} \label{prop:elements_NC_M_dual}
    Let $\Mca$ be a finite unitary magma. The underlying graded vector
    space of $\NC\Mca^!$ is the linear span of all noncrossing dual
    $\Mca$-cliques.
\end{Proposition}
\begin{proof}
    The statement of the proposition is equivalent to the fact that the
    generating series of noncrossing dual $\Mca$-cliques is the Hilbert
    series $\Hilbert_{\NC\Mca^!}(t)$ of $\NC\Mca^!$. From the definition
    of dual $\Mca$-cliques, we obtain that the set of the dual
    $\Mca$-cliques of arity $n$, $n \geq 1$, is in bijection with the
    set of the $\Mca^2$-Schröder trees of arity $n$ having the outgoing
    edges from the root and the edges connecting internal nodes with
    leaves labeled by pairs $(a, a) \in \Mca^2$, and the edges
    connecting two internal nodes labeled by pairs $(a, b) \in \Mca^2$
    with $a \ne b$. The map $\BubbleTree$ defined in
    Section~\ref{subsubsec:treelike_bubbles} (see also
    Section~\ref{subsubsec:M_Schroder_trees}) realizes such a bijection.
    Let $T(t)$ be the generating series of these $\Mca^2$-Schröder
    trees, and let $S(t)$ be the generating series of the
    $\Mca^2$-Schröder trees of arities greater than $1$ and such that
    the outgoing edges from the roots and the edges connecting two
    internal nodes are labeled by pairs $(a, b) \in \Mca^2$ with
    $a \ne b$, and the edges connecting internal nodes with leaves are
    labeled by pairs $(a, a) \in \Mca^2$. From the description of these
    trees, one has
    \begin{equation}
        S(t) = m (m - 1) \frac{(mt + S(t))^2}{1 - mt - S},
    \end{equation}
    where $m := \# \Mca$. Moreover, when $m \ne 1$,  $T(t)$ satisfies
    \begin{equation}
        T(t) = t + \frac{S(t)}{m - 1},
    \end{equation}
    and we obtain that $T(t)$
    admits~\eqref{equ:Hilbert_function_NC_M_dual} as solution. Then, by
    Proposition~\ref{prop:Hilbert_series_NC_M_dual}, when $m \ne 1$,
    this implies the statement of the proposition. When $m = 1$, it
    follows from Proposition~\ref{prop:presentation_dual_NC_M} that
    $\NC\Mca^!$ is isomorphic to the associative operad $\As$. Hence, in
    this case, $\dim \NC\Mca^!(n) = 1$ for all $n \geq 1$. Since there
    is exactly one dual $\Mca$-clique of arity $n$ for any $n \geq 1$,
    the statement of the proposition is satisfied.
\end{proof}
\medskip

Proposition~\ref{prop:elements_NC_M_dual} gives a combinatorial
description of the elements of $\NC\Mca^!$. Nevertheless, we do not know
for the time being a partial composition on the linear span of these
elements providing a realization of~$\NC\Mca^!$.
\medskip

\section{Concrete constructions}%
\label{sec:concrete_constructions}
The clique construction provides alternative definitions of known
operads. We explore here the cases of the operad $\NCP$ of based
noncrossing trees, the operad $\FF_4$ of formal fractions, the operad
$\BNC$ of bicolored noncrossing configurations and, the operads $\MT$
and $\DMT$ of multi-tildes and double multi-tildes.
\medskip

\subsection{Rational functions and related operads}
We use here the (noncrossing) clique construction to interpret few
operads related to the operad $\RatFct$ of rational functions (see
Section~\ref{subsubsec:rational_functions}).
\medskip

\subsubsection{Dendriform and based noncrossing tree operads}
The \Def{operad of based noncrossing trees} $\NCP$ is an operad
introduced in~\cite{Cha07}. This operad is generated by two binary
elements $\GDendr$ and $\DDendr$ subjected to exactly one quadratic
nontrivial relation. The algebras over $\NCP$ are
\Def{$\LOp$-algebras} and have been studied in~\cite{Ler11}.
We do not describe $\NCP$ in details here because this is not essential
for the sequel. We just explain how to construct $\NCP$ through the
clique construction and interpret a known link between $\NCP$ and the
dendriform operad through the rational functions associated with
$\Z$-cliques (see Section~\ref{subsubsec:rational_functions}).
\medskip

Let $\Oca_{\NCP}$ be the suboperad of $\Cli\Z$ generated by
\begin{equation}
    \left\{\TriangleEXE{}{-1}{},\TriangleEEX{}{}{-1}\right\}.
\end{equation}
By using Proposition~\ref{prop:suboperads_NC_M_triangles_dimensions}, we
find that the Hilbert series $\Hilbert_{\Oca_{\NCP}}(t)$ of
$\Oca_{\NCP}$ satisfies
\begin{equation} \label{equ:Hilbert_series_L_operad}
    t - \Hilbert_{\Oca_{\NCP}}(t)
    + 2 \Hilbert_{\Oca_{\NCP}}(t)^2
    - \Hilbert_{\Oca_{\NCP}}(t)^3 = 0.
\end{equation}
The first dimensions of $\Oca$ are
\begin{equation}
    1, 2, 7, 30, 143, 728, 3876, 21318,
\end{equation}
and form Sequence~\OEIS{A006013} of~\cite{Slo}. Moreover, one can see
that
\begin{equation} \label{equ:relation_L_operad}
    \TriangleEEX{}{}{-1} \circ_1 \TriangleEXE{}{-1}{}
    =
    \TriangleEXE{}{-1}{} \circ_2 \TriangleEEX{}{}{-1},
\end{equation}
is the only nontrivial relation of degree $2$ between the generators
of~$\Oca_{\NCP}$.
\medskip

\begin{Proposition} \label{prop:construction_NCP}
    The operad $\Oca_{\NCP}$ is isomorphic to the operad $\NCP$.
\end{Proposition}
\begin{proof}
    Let $\phi : \Oca_{\NCP}(2) \to \NCP(2)$ be the linear map satisfying
    \vspace{-1.75em}
    \begin{multicols}{2}
    \begin{subequations}
    \begin{equation}
        \phi\left(\TriangleEEX{}{}{-1}\right) = \GDendr,
    \end{equation}

    \begin{equation}
        \phi\left(\TriangleEXE{}{-1}{}\right) = \DDendr,
    \end{equation}
    \end{subequations}
    \end{multicols}
    \noindent where $\GDendr$ and $\DDendr$ are the two binary
    generators of $\NCP$. In~\cite{Cha07}, a presentation of $\NCP$ is
    described wherein its generators satisfy one nontrivial relation of
    degree $2$. This relation can be obtained by replacing each
    $\Z$-clique appearing in~\eqref{equ:relation_L_operad} by its image
    by $\phi$. For this reason, $\phi$ uniquely extends into an operad
    morphism. Moreover, because the image of $\phi$ contains all the
    generators of $\NCP$, this morphism is surjective. Finally, the
    Hilbert series of $\NCP$
    satisfies~\eqref{equ:Hilbert_series_L_operad}, so that $\Oca_{\NCP}$
    and $\NCP$ have the same dimensions. Therefore, $\phi$ is an operad
    isomorphism.
\end{proof}
\medskip

Loday as shown in~\cite{Lod10} that the suboperad of $\RatFct$
generated by the rational functions $f_1(u_1, u_2) := u_1^{-1}$ and
$f_2(u_1, u_2) := u_2^{-1}$ is isomorphic to the dendriform operad
$\Dendr$~\cite{Lod01}. This operad is generated by two binary elements
$\GDendr$ and $\DDendr$ which are subjected to three quadratic
nontrivial relations. An isomorphism between $\Dendr$ and the suboperad
of $\RatFct$ generated by $f_1$ and $f_2$ sends $\GDendr$ to $f_2$ and
$\DDendr$ to $f_1$. By Theorem~\ref{thm:rat_fct_cliques}, $\Frac_\Id$ is
an operad morphism from $\Cli\Z$ to $\RatFct$. Hence, the restriction of
$\Frac_\Id$ on $\Oca_{\NCP}$ is also an operad morphism from
$\Oca_{\NCP}$ to $\RatFct$. Moreover, since
\vspace{-1.75em}
\begin{multicols}{2}
\begin{subequations}
\begin{equation}
    \Frac_\Id\left(
        \TriangleEXE{}{-1}{}
    \right)
    = \frac{1}{u_1} = f_1,
\end{equation}

\begin{equation}
    \Frac_\Id\left(
        \TriangleEEX{}{}{-1}
    \right)
    = \frac{1}{u_2} = f_2,
\end{equation}
\end{subequations}
\end{multicols}
\noindent the map $\Frac_\Id$ is a surjective operad morphism from
$\Oca_{\NCP}$ to~$\Dendr$.
\medskip

\subsubsection{Operad of formal fractions}%
\label{subsubsec:operad_ff}
The \Def{operad of formal fractions} $\FF$ is an operad introduced
in~\cite{CHN16}. Its elements of arity $n \geq 1$ are fractions whose
numerators and denominators are formal products of subsets of $[n]$. For
instance,
\begin{equation}
    \frac{\{1, 3, 4\} \{2\} \{4, 6\}}{\{2, 3, 5\} \{4\}}
\end{equation}
is an element of arity $6$ of $\FF$. We do not describe the partial
composition of this operad since its knowledge is not essential for
the sequel. The operad $\FF$ admits a suboperad $\FF_4$, defined
as the binary suboperad of $\FF$ generated by
\begin{equation} \label{equ:generators_FF4}
    \left\{
        \frac{1}{\{1\} \{1, 2\}},
        \frac{1}{\{2\} \{1, 2\}},
        \frac{1}{\{1, 2\}},
        \frac{1}{\{1\} \{2\}}
    \right\}.
\end{equation}
We explain here how to construct $\FF_4$ through the clique construction.
\medskip

Let $\Oca_{\FF_4}$ be the suboperad of $\Cli\Z$ generated by
\begin{equation}
    \left\{
        \Triangle{-1}{-1}{1}, \Triangle{-1}{1}{-1},
        \Triangle{-1}{1}{1}, \Triangle{1}{-1}{-1}
    \right\}.
\end{equation}
By using Proposition~\ref{prop:suboperads_NC_M_triangles_dimensions},
we find that the Hilbert series $\Hilbert_{\Oca_{\FF_4}}(t)$ of
$\Oca_{\FF_4}$ satisfies
\begin{equation} \label{equ:Hilbert_series_FF4}
    t + (2t - 1) \Hilbert_{\Oca_{\FF_4}}(t)
    + 2 \Hilbert_{\Oca_{\FF_4}}(t)^2 = 0.
\end{equation}
The first dimensions of $\Oca_{\FF_4}$ are
\begin{equation}
    1, 4, 24, 176, 1440, 12608, 115584, 1095424,
\end{equation}
and form Sequence~\OEIS{A156017} of~\cite{Slo}. Moreover, by computer
exploration, we obtain the list
\vspace{-1.75em}
\begin{multicols}{2}
\begin{subequations}
\begin{equation} \label{equ:relation_FF4_1}
    \Triangle{-1}{1}{-1} \circ_1 \Triangle{-1}{-1}{1}
    =
    \Triangle{-1}{-1}{1} \circ_2 \Triangle{-1}{1}{-1},
\end{equation}
\begin{equation}
    \Triangle{-1}{1}{1} \circ_1 \Triangle{-1}{1}{1}
    =
    \Triangle{-1}{1}{1} \circ_2 \Triangle{-1}{1}{1},
\end{equation}
\begin{equation}
    \Triangle{-1}{1}{1} \circ_1 \Triangle{-1}{-1}{1}
    =
    \Triangle{-1}{-1}{1} \circ_2 \Triangle{-1}{1}{1},
\end{equation}
\begin{equation}
    \Triangle{-1}{1}{1} \circ_1 \Triangle{-1}{1}{-1}
    =
    \Triangle{-1}{1}{1} \circ_2 \Triangle{-1}{-1}{1},
\end{equation}

\begin{equation}
    \Triangle{-1}{1}{-1} \circ_1 \Triangle{-1}{1}{1}
    =
    \Triangle{-1}{1}{1} \circ_2 \Triangle{-1}{1}{-1},
\end{equation}
\begin{equation}
    \Triangle{1}{-1}{-1} \circ_1 \Triangle{1}{-1}{-1}
    =
    \Triangle{1}{-1}{-1} \circ_2 \Triangle{1}{-1}{-1},
\end{equation}
\begin{equation}
    \Triangle{-1}{1}{-1} \circ_1 \Triangle{-1}{1}{-1}
    =
    \Triangle{-1}{1}{-1} \circ_2 \Triangle{1}{-1}{-1},
\end{equation}
\begin{equation} \label{equ:relation_FF4_8}
    \Triangle{-1}{-1}{1} \circ_1 \Triangle{1}{-1}{-1}
    =
    \Triangle{-1}{-1}{1} \circ_2 \Triangle{-1}{-1}{1},
\end{equation}
\end{subequations}
\end{multicols}
\noindent of all nontrivial relations of degree $2$ between the
generators of~$\Oca_{\FF_4}$.
\medskip

\begin{Proposition} \label{prop:construction_FF4}
    The operad $\Oca_{\FF_4}$ is isomorphic to the operad $\FF_4$.
\end{Proposition}
\begin{proof}
    Let $\phi : \Oca_{\FF_4}(2) \to \FF_4(2)$ be the linear map
    satisfying
    \vspace{-1.75em}
    \begin{multicols}{2}
    \begin{subequations}
    \begin{equation}
        \phi\left(\Triangle{-1}{-1}{1}\right)
        = \frac{1}{\{1\} \{1, 2\}},
    \end{equation}
    \begin{equation}
        \phi\left(\Triangle{-1}{1}{-1}\right)
        = \frac{1}{\{2\} \{1, 2\}},
    \end{equation}

    \begin{equation}
        \phi\left(\Triangle{-1}{1}{1}\right)
        = \frac{1}{\{1, 2\}},
    \end{equation}
    \begin{equation}
        \phi\left(\Triangle{1}{-1}{-1}\right)
        = \frac{1}{\{1\} \{2\}}.
    \end{equation}
    \end{subequations}
    \end{multicols}
    \noindent In~\cite{CHN16}, a presentation of $\FF_4$ is described
    wherein its generators satisfy eight nontrivial relations of degree
    $2$. These relations can be obtained by replacing each $\Z$-clique
    appearing in~\eqref{equ:relation_FF4_1}--\eqref{equ:relation_FF4_8}
    by its image by $\phi$. For this reason, $\phi$ uniquely extends
    into an operad morphism. Moreover, because the image of $\phi$
    contains all the generators of $\FF_4$, this morphism is surjective.
    Finally, again by~\cite{CHN16}, the Hilbert series of $\FF_4$
    satisfies~\eqref{equ:Hilbert_series_FF4}, so that $\Oca_{\FF_4}$
    and $\FF_4$ have the same dimensions. Therefore, $\phi$ is an
    operad isomorphism.
\end{proof}
\medskip

Proposition~\ref{prop:construction_FF4} shows hence that the operad
$\FF_4$ can be built through the construction $\Cli$. Observe also that,
as a consequence of Proposition~\ref{prop:construction_FF4}, all
suboperads of $\FF_4$ defined in~\cite{CHN16} that are generated by a
subset of~\eqref{equ:generators_FF4} can be constructed by the clique
construction.
\medskip

\subsection{Operad of bicolored noncrossing configurations}
\label{subsec:operad_bnc}
The \Def{operad of bicolored noncrossing configurations} $\BNC$ is an
operad defined in~\cite{CG14}. For any $n \geq 2$, $\BNC(n)$ is the
linear span of all \Def{bicolored noncrossing configurations}, where
such objects are regular polygons $\Cfr$ with $n + 1$ edges and such
that any arc of $\Cfr$ is blue, red, or uncolored, no blue or red arc
crosses another blue or red arc, and all red arcs are diagonals. These
objects can be seen as particular cliques, so that all definitions of
Section~\ref{subsec:cliques} remain valid here. For instance,
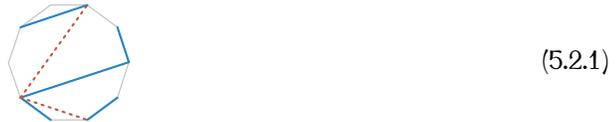
\begin{equation}
    \begin{tikzpicture}[scale=.8,Centering]
        \node[shape=coordinate](0)at(-0.3,-0.95){};
        \node[shape=coordinate](1)at(-0.8,-0.58){};
        \node[shape=coordinate](2)at(-1.,-0.){};
        \node[shape=coordinate](3)at(-0.8,0.59){};
        \node[shape=coordinate](4)at(-0.3,0.96){};
        \node[shape=coordinate](5)at(0.31,0.96){};
        \node[shape=coordinate](6)at(0.81,0.59){};
        \node[shape=coordinate](7)at(1.,0.01){};
        \node[shape=coordinate](8)at(0.81,-0.58){};
        \node[shape=coordinate](9)at(0.31,-0.95){};
        \draw[CliqueEdgeGray](0)--(1);
        \draw[CliqueEdgeGray](1)--(2);
        \draw[CliqueEdgeGray](2)--(3);
        \draw[CliqueEdgeGray](3)--(4);
        \draw[CliqueEdgeGray](4)--(5);
        \draw[CliqueEdgeGray](5)--(6);
        \draw[CliqueEdgeGray](6)--(7);
        \draw[CliqueEdgeGray](7)--(8);
        \draw[CliqueEdgeGray](8)--(9);
        \draw[CliqueEdgeGray](9)--(0);
        \draw[CliqueEdgeBlue](0)--(1);
        \draw[CliqueEdgeBlue](1)--(7);
        \draw[CliqueEdgeBlue](3)--(5);
        \draw[CliqueEdgeBlue](6)--(7);
        \draw[CliqueEdgeBlue](8)--(9);
        \draw[CliqueEdgeRed](1)--(5);
        \draw[CliqueEdgeRed](1)--(9);
    \end{tikzpicture}
\end{equation}
is a bicolored noncrossing configuration of arity~$9$ (blue arcs are
drawn as continuous segments and red arcs, as dashed ones). Moreover,
$\BNC(1)$ is the linear span of the singleton containing the only
polygon of arity $1$ with its only arc is uncolored. The partial
composition of $\BNC$ is defined, in a geometric way, as follows. For
any bicolored noncrossing configurations $\Cfr$ and $\Dfr$ of respective
arities $n$ and $m$, and $i \in [n]$, the bicolored noncrossing
configuration $\Cfr \circ_i \Dfr$ is obtained by gluing the base of
$\Dfr$ onto the $i$th edge of $\Cfr$, and then,
\begin{enumerate}[fullwidth,label=(\alph*)]
    \item if the base of $\Dfr$ and the $i$th edge of $\Cfr$ are both
    uncoloured, the arc $(i, i + m)$ of $\Cfr \circ_i \Dfr$ becomes red;
    \item if the base of $\Dfr$ and the $i$th edge of $\Cfr$ are both
    blue, the arc $(i, i + m)$ of $\Cfr \circ_i \Dfr$ becomes blue;
    \item otherwise, the base of $\Dfr$ and the $i$th necessarily have
    different colors; in this case, the arc $(i, i + m)$ of
    $\Cfr \circ_i \Dfr$ is uncolored.
\end{enumerate}
For example,
\begin{subequations}
\begin{equation}
    \begin{tikzpicture}[scale=.7,Centering]
        \node[shape=coordinate](0)at(-0.49,-0.86){};
        \node[shape=coordinate](1)at(-1.,-0.){};
        \node[shape=coordinate](2)at(-0.5,0.87){};
        \node[shape=coordinate](3)at(0.5,0.87){};
        \node[shape=coordinate](4)at(1.,0.01){};
        \node[shape=coordinate](5)at(0.51,-0.86){};
        \draw[CliqueEdgeGray](0)--(1);
        \draw[CliqueEdgeGray](1)--(2);
        \draw[CliqueEdgeGray](2)--(3);
        \draw[CliqueEdgeGray](3)--(4);
        \draw[CliqueEdgeGray](4)--(5);
        \draw[CliqueEdgeGray](5)--(0);
        \draw[CliqueEdgeBlue](0)--(1);
        \draw[CliqueEdgeBlue](1)--(2);
        \draw[CliqueEdgeBlue](4)--(5);
        \draw[CliqueEdgeBlue](0)--(5);
        \draw[CliqueEdgeRed](3)--(5);
    \end{tikzpicture}
    \enspace \circ_3 \enspace
    \begin{tikzpicture}[scale=.7,Centering]
        \node[shape=coordinate](0)at(-0.58,-0.8){};
        \node[shape=coordinate](1)at(-0.95,0.31){};
        \node[shape=coordinate](2)at(-0.,1.){};
        \node[shape=coordinate](3)at(0.96,0.31){};
        \node[shape=coordinate](4)at(0.59,-0.8){};
        \draw[CliqueEdgeGray](0)--(1);
        \draw[CliqueEdgeGray](1)--(2);
        \draw[CliqueEdgeGray](2)--(3);
        \draw[CliqueEdgeGray](3)--(4);
        \draw[CliqueEdgeGray](4)--(0);
        \draw[CliqueEdgeBlue](0)--(1);
        \draw[CliqueEdgeBlue](1)--(4);
        \draw[CliqueEdgeBlue](3)--(4);
        \draw[CliqueEdgeRed](2)--(4);
    \end{tikzpicture}
    \enspace = \enspace
    \begin{tikzpicture}[scale=.7,Centering]
        \node[shape=coordinate](0)at(-0.34,-0.93){};
        \node[shape=coordinate](1)at(-0.86,-0.5){};
        \node[shape=coordinate](2)at(-0.98,0.18){};
        \node[shape=coordinate](3)at(-0.64,0.77){};
        \node[shape=coordinate](4)at(-0.,1.){};
        \node[shape=coordinate](5)at(0.65,0.77){};
        \node[shape=coordinate](6)at(0.99,0.18){};
        \node[shape=coordinate](7)at(0.87,-0.49){};
        \node[shape=coordinate](8)at(0.35,-0.93){};
        \draw[CliqueEdgeGray](0)--(1);
        \draw[CliqueEdgeGray](1)--(2);
        \draw[CliqueEdgeGray](2)--(3);
        \draw[CliqueEdgeGray](3)--(4);
        \draw[CliqueEdgeGray](4)--(5);
        \draw[CliqueEdgeGray](5)--(6);
        \draw[CliqueEdgeGray](6)--(7);
        \draw[CliqueEdgeGray](7)--(8);
        \draw[CliqueEdgeGray](8)--(0);
        \draw[CliqueEdgeBlue](0)--(1);
        \draw[CliqueEdgeBlue](1)--(2);
        \draw[CliqueEdgeBlue](2)--(3);
        \draw[CliqueEdgeBlue](3)--(6);
        \draw[CliqueEdgeBlue](5)--(6);
        \draw[CliqueEdgeBlue](7)--(8);
        \draw[CliqueEdgeBlue](0)--(8);
        \draw[CliqueEdgeRed](2)--(6);
        \draw[CliqueEdgeRed](4)--(6);
        \draw[CliqueEdgeRed](6)--(8);
    \end{tikzpicture}\,,
\end{equation}
\begin{equation}
    \begin{tikzpicture}[scale=.7,Centering]
        \node[shape=coordinate](0)at(-0.49,-0.86){};
        \node[shape=coordinate](1)at(-1.,-0.){};
        \node[shape=coordinate](2)at(-0.5,0.87){};
        \node[shape=coordinate](3)at(0.5,0.87){};
        \node[shape=coordinate](4)at(1.,0.01){};
        \node[shape=coordinate](5)at(0.51,-0.86){};
        \draw[CliqueEdgeGray](0)--(1);
        \draw[CliqueEdgeGray](1)--(2);
        \draw[CliqueEdgeGray](2)--(3);
        \draw[CliqueEdgeGray](3)--(4);
        \draw[CliqueEdgeGray](4)--(5);
        \draw[CliqueEdgeGray](5)--(0);
        \draw[CliqueEdgeBlue](1)--(2);
        \draw[CliqueEdgeBlue](1)--(5);
        \draw[CliqueEdgeBlue](2)--(3);
        \draw[CliqueEdgeBlue](3)--(4);
        \draw[CliqueEdgeBlue](4)--(5);
        \draw[CliqueEdgeRed](1)--(4);
    \end{tikzpicture}
    \enspace \circ_5 \enspace
    \begin{tikzpicture}[scale=.7,Centering]
        \node[shape=coordinate](0)at(-0.58,-0.8){};
        \node[shape=coordinate](1)at(-0.95,0.31){};
        \node[shape=coordinate](2)at(-0.,1.){};
        \node[shape=coordinate](3)at(0.96,0.31){};
        \node[shape=coordinate](4)at(0.59,-0.8){};
        \draw[CliqueEdgeGray](0)--(1);
        \draw[CliqueEdgeGray](1)--(2);
        \draw[CliqueEdgeGray](2)--(3);
        \draw[CliqueEdgeGray](3)--(4);
        \draw[CliqueEdgeGray](4)--(0);
        \draw[CliqueEdgeBlue](0)--(1);
        \draw[CliqueEdgeBlue](1)--(3);
        \draw[CliqueEdgeBlue](1)--(4);
        \draw[CliqueEdgeBlue](2)--(3);
        \draw[CliqueEdgeBlue](3)--(4);
        \draw[CliqueEdgeBlue](0)--(4);
    \end{tikzpicture}
    \enspace = \enspace
    \begin{tikzpicture}[scale=.7,Centering]
        \node[shape=coordinate](0)at(-0.34,-0.93){};
        \node[shape=coordinate](1)at(-0.86,-0.5){};
        \node[shape=coordinate](2)at(-0.98,0.18){};
        \node[shape=coordinate](3)at(-0.64,0.77){};
        \node[shape=coordinate](4)at(-0.,1.){};
        \node[shape=coordinate](5)at(0.65,0.77){};
        \node[shape=coordinate](6)at(0.99,0.18){};
        \node[shape=coordinate](7)at(0.87,-0.49){};
        \node[shape=coordinate](8)at(0.35,-0.93){};
        \draw[CliqueEdgeGray](0)--(1);
        \draw[CliqueEdgeGray](1)--(2);
        \draw[CliqueEdgeGray](2)--(3);
        \draw[CliqueEdgeGray](3)--(4);
        \draw[CliqueEdgeGray](4)--(5);
        \draw[CliqueEdgeGray](5)--(6);
        \draw[CliqueEdgeGray](6)--(7);
        \draw[CliqueEdgeGray](7)--(8);
        \draw[CliqueEdgeGray](8)--(0);
        \draw[CliqueEdgeBlue](1)--(2);
        \draw[CliqueEdgeBlue](1)--(8);
        \draw[CliqueEdgeBlue](2)--(3);
        \draw[CliqueEdgeBlue](3)--(4);
        \draw[CliqueEdgeBlue](4)--(5);
        \draw[CliqueEdgeBlue](4)--(8);
        \draw[CliqueEdgeBlue](5)--(7);
        \draw[CliqueEdgeBlue](5)--(8);
        \draw[CliqueEdgeBlue](6)--(7);
        \draw[CliqueEdgeBlue](7)--(8);
        \draw[CliqueEdgeRed](1)--(4);
    \end{tikzpicture}\,,
\end{equation}
\begin{equation}
    \begin{tikzpicture}[scale=.7,Centering]
        \node[shape=coordinate](0)at(-0.50,-0.87){};
        \node[shape=coordinate](1)at(-1.00,-0.00){};
        \node[shape=coordinate](2)at(-0.50,0.87){};
        \node[shape=coordinate](3)at(0.50,0.87){};
        \node[shape=coordinate](4)at(1.00,0.00){};
        \node[shape=coordinate](5)at(0.50,-0.87){};
        \draw[CliqueEdgeGray](0)--(1);
        \draw[CliqueEdgeGray](1)--(2);
        \draw[CliqueEdgeGray](2)--(3);
        \draw[CliqueEdgeGray](3)--(4);
        \draw[CliqueEdgeGray](4)--(5);
        \draw[CliqueEdgeGray](5)--(0);
        \draw[CliqueEdgeBlue](0)--(1);
        \draw[CliqueEdgeBlue](1)--(2);
        \draw[CliqueEdgeBlue](2)--(5);
        \draw[CliqueEdgeBlue](3)--(5);
        \draw[CliqueEdgeBlue](4)--(5);
        \draw[CliqueEdgeRed](0)--(2);
    \end{tikzpicture}
    \enspace \circ_3 \enspace
    \begin{tikzpicture}[scale=.7,Centering]
        \node[shape=coordinate](0)at(-0.59,-0.81){};
        \node[shape=coordinate](1)at(-0.95,0.31){};
        \node[shape=coordinate](2)at(-0.00,1.00){};
        \node[shape=coordinate](3)at(0.95,0.31){};
        \node[shape=coordinate](4)at(0.59,-0.81){};
        \draw[CliqueEdgeGray](0)--(1);
        \draw[CliqueEdgeGray](1)--(2);
        \draw[CliqueEdgeGray](2)--(3);
        \draw[CliqueEdgeGray](3)--(4);
        \draw[CliqueEdgeGray](4)--(0);
        \draw[CliqueEdgeBlue](0)--(4);
        \draw[CliqueEdgeBlue](0)--(1);
        \draw[CliqueEdgeBlue](1)--(2);
        \draw[CliqueEdgeBlue](2)--(3);
        \draw[CliqueEdgeBlue](3)--(4);
        \draw[CliqueEdgeRed](1)--(3);
        \draw[CliqueEdgeRed](1)--(4);
    \end{tikzpicture}
    \enspace = \enspace
    \begin{tikzpicture}[scale=.7,Centering]
        \node[shape=coordinate](0)at(-0.34,-0.94){};
        \node[shape=coordinate](1)at(-0.87,-0.50){};
        \node[shape=coordinate](2)at(-0.98,0.17){};
        \node[shape=coordinate](3)at(-0.64,0.77){};
        \node[shape=coordinate](4)at(-0.00,1.00){};
        \node[shape=coordinate](5)at(0.64,0.77){};
        \node[shape=coordinate](6)at(0.98,0.17){};
        \node[shape=coordinate](7)at(0.87,-0.50){};
        \node[shape=coordinate](8)at(0.34,-0.94){};
        \draw[CliqueEdgeGray](0)--(1);
        \draw[CliqueEdgeGray](1)--(2);
        \draw[CliqueEdgeGray](2)--(3);
        \draw[CliqueEdgeGray](3)--(4);
        \draw[CliqueEdgeGray](4)--(5);
        \draw[CliqueEdgeGray](5)--(6);
        \draw[CliqueEdgeGray](6)--(7);
        \draw[CliqueEdgeGray](7)--(8);
        \draw[CliqueEdgeGray](8)--(0);
        \draw[CliqueEdgeBlue](0)--(1);
        \draw[CliqueEdgeBlue](1)--(2);
        \draw[CliqueEdgeBlue](2)--(3);
        \draw[CliqueEdgeBlue](2)--(8);
        \draw[CliqueEdgeBlue](3)--(4);
        \draw[CliqueEdgeBlue](4)--(5);
        \draw[CliqueEdgeBlue](5)--(6);
        \draw[CliqueEdgeBlue](6)--(8);
        \draw[CliqueEdgeBlue](7)--(8);
        \draw[CliqueEdgeRed](0)--(2);
        \draw[CliqueEdgeRed](3)--(5);
        \draw[CliqueEdgeRed](3)--(6);
    \end{tikzpicture}\,.
\end{equation}
\end{subequations}
\medskip

Let us now consider the unitary magma
$\Mca_{\BNC} := \{\Unit, \Att, \Btt\}$ wherein operation $\Op$ is
defined by the Cayley table
\begin{small}
\begin{equation}
    \begin{tabular}{c||c|c|c|}
        $\Op$ & \; $\Unit$ \; & \; $\Att$ \; & \; $\Btt$ \; \\ \hline \hline
        $\Unit$ & $\Unit$ & $\Att$ & $\Btt$ \\ \hline
        $\Att$ & $\Att$ & $\Att$ & $\Unit$ \\ \hline
        $\Btt$ & $\Btt$ & $\Unit$ & $\Btt$
    \end{tabular}\, .
\end{equation}
\end{small}
In other words, $\Mca_{\BNC}$ is the unitary magma wherein $\Att$ and
$\Btt$ are idempotent, and $\Att \Op \Btt = \Unit = \Btt \Op \Att$.
Observe that $\Mca_{\BNC}$ is a commutative unitary magma, but, since
\begin{equation}
    (\Btt \Op \Att) \Op \Att = \Unit \Op \Att = \Att
    \ne
    \Btt = \Btt \Op \Unit = \Btt \Op (\Att \Op \Att),
\end{equation}
the operation $\Op$ is not associative.
\medskip

Let $\phi : \BNC \to \NC\Mca_{\BNC}$ be the linear map defined in the
following way. For any bicolored noncrossing configuration $\Cfr$,
$\phi(\Cfr)$ is the noncrossing $\Mca_{\BNC}$-clique of $\NC\Mca_{\BNC}$
obtained by replacing all blue arcs of $\Cfr$ by arcs labeled by $\Att$,
all red diagonals of $\Cfr$ by diagonals labeled by $\Btt$, all
uncolored edges and bases of $\Cfr$ by edges labeled by $\Btt$, and all
uncolored diagonals of $\Cfr$ by diagonals labeled by $\Unit$. For
instance,
\begin{equation}
    \phi\left(
    \begin{tikzpicture}[scale=.7,Centering]
        \node[shape=coordinate](0)at(-0.49,-0.86){};
        \node[shape=coordinate](1)at(-1.,-0.){};
        \node[shape=coordinate](2)at(-0.5,0.87){};
        \node[shape=coordinate](3)at(0.5,0.87){};
        \node[shape=coordinate](4)at(1.,0.01){};
        \node[shape=coordinate](5)at(0.51,-0.86){};
        \draw[CliqueEdgeGray](0)--(1);
        \draw[CliqueEdgeGray](1)--(2);
        \draw[CliqueEdgeGray](2)--(3);
        \draw[CliqueEdgeGray](3)--(4);
        \draw[CliqueEdgeGray](4)--(5);
        \draw[CliqueEdgeGray](5)--(0);
        \draw[CliqueEdgeBlue](1)--(2);
        \draw[CliqueEdgeBlue](1)--(3);
        \draw[CliqueEdgeBlue](0)--(5);
        \draw[CliqueEdgeRed](1)--(4);
    \end{tikzpicture}
    \right)
    \enspace = \enspace
    \begin{tikzpicture}[scale=.7,Centering]
        \node[CliquePoint](0)at(-0.49,-0.86){};
        \node[CliquePoint](1)at(-1.,-0.){};
        \node[CliquePoint](2)at(-0.5,0.87){};
        \node[CliquePoint](3)at(0.5,0.87){};
        \node[CliquePoint](4)at(1.,0.01){};
        \node[CliquePoint](5)at(0.51,-0.86){};
        \draw[CliqueEdge](0)edge[]node[CliqueLabel]
            {\begin{math}\Btt\end{math}}(1);
        \draw[CliqueEdge](1)edge[]node[CliqueLabel]
            {\begin{math}\Att\end{math}}(2);
        \draw[CliqueEdge](2)edge[]node[CliqueLabel]
            {\begin{math}\Btt\end{math}}(3);
        \draw[CliqueEdge](3)edge[]node[CliqueLabel]
            {\begin{math}\Btt\end{math}}(4);
        \draw[CliqueEdge](4)edge[]node[CliqueLabel]
            {\begin{math}\Btt\end{math}}(5);
        \draw[CliqueEdge](0)edge[]node[CliqueLabel]
            {\begin{math}\Att\end{math}}(5);
        \draw[CliqueEdge](1)edge[bend right=30]node[CliqueLabel]
            {\begin{math}\Btt\end{math}}(4);
        \draw[CliqueEdge](1)edge[bend right=30]node[CliqueLabel]
            {\begin{math}\Att\end{math}}(3);
    \end{tikzpicture}\,.
\end{equation}
\medskip

\begin{Proposition} \label{prop:construction_BNC}
    The linear span of $\UnitClique$ together with all noncrossing
    $\Mca_{\BNC}$-cliques without edges nor bases labeled by $\Unit$
    forms a suboperad of $\NC\Mca_{\BNC}$ isomorphic to~$\BNC$.
    Moreover, $\phi$ is an isomorphism between these two operads.
\end{Proposition}
\begin{proof}
    Let us denote by $\Oca_\BNC$ the subspace of $\NC\Mca_{\BNC}$
    described in the statement of the proposition. First of all, its
    follows from the definition of the partial composition of
    $\NC\Mca_{\BNC}$ that $\Oca_\BNC$ is closed under the partial
    composition operation. Hence, and since $\Oca_\BNC$ contains the
    unit of $\NC\Mca_{\BNC}$, $\Oca_\BNC$ is an operad. Second, observe
    that the image of $\phi$ is the underlying space of $\Oca_\BNC$ and,
    from the definition of the partial composition of $\BNC$, one can
    check that $\phi$ is an operad morphism. Finally, since $\phi$ is a
    bijection from $\BNC$ to $\Oca_\BNC$, the statement of the
    proposition follows.
\end{proof}
\medskip

Proposition~\ref{prop:construction_BNC} shows hence that the operad
$\BNC$ can be built through the noncrossing clique construction.
Moreover, observe that in~\cite{CG14}, an automorphism of $\BNC$ called
\Def{complementary} is considered. The complementary of a bicolored
noncrossing configuration is an involution acting by modifying the
colors of some arcs of its arcs. Under our setting, this automorphism
translates simply as the map $\Cli\theta : \Oca_\BNC \to \Oca_\BNC$
where $\Oca_\BNC$ is the operad isomorphic to $\BNC$ described in the
statement of Proposition~\ref{prop:construction_BNC} and
$\theta : \Mca_{\BNC} \to \Mca_{\BNC}$ is the unitary magma automorphism
of $\Mca_{\BNC}$ satisfying $\theta(\Unit) = \Unit$,
$\theta(\Att) = \Btt$, and~$\theta(\Btt) = \Att$.
\medskip

Besides, it is shown in~\cite{CG14} that the set of all bicolored
noncrossing configurations of arity $2$ is a minimal generating set of
$\BNC$. Thus, by Proposition~\ref{prop:construction_BNC}, the set
\begin{equation}
    \left\{
        \Triangle{\Att}{\Att}{\Att},
        \Triangle{\Att}{\Att}{\Btt},
        \Triangle{\Att}{\Btt}{\Att},
        \Triangle{\Att}{\Btt}{\Btt},
        \Triangle{\Btt}{\Att}{\Att},
        \Triangle{\Btt}{\Att}{\Btt},
        \Triangle{\Btt}{\Btt}{\Att},
        \Triangle{\Btt}{\Btt}{\Btt}
    \right\}
\end{equation}
is a minimal generating set of the suboperad $\Oca_\BNC$ of
$\NC\Mca_{\BNC}$ isomorphic to $\BNC$. As a consequence, all the
suboperads of $\BNC$ defined in~\cite{CG14} which are generated by a
subset of the set of the generators of $\BNC$ can be constructed by the
noncrossing clique construction. This includes, among others, the
magmatic operad, the free operad on two binary generators, the operad of
noncrossing plants~\cite{Cha07}, the dipterous operad~\cite{LR03,Zin12},
and the $2$-associative operad~\cite{LR06,Zin12}.
\medskip

\subsection{Operads from language theory}
We provide constructions of two operads coming from formal language
theory by using the clique construction.
\medskip

\subsubsection{Multi-tildes}
Multi-tildes are operators introduced in~\cite{CCM11} in the context of
formal language theory as a convenient way to express regular languages.
A \Def{multi-tilde} is a pair $(n, \Sfr)$ where $n$ is a positive
integer and $\Sfr$ is a subset of $\{(x, y) \in [n]^2 : x \leq y\}$. The
\Def{arity} of the multi-tilde $(n, \Sfr)$ is $n$. As shown
in~\cite{LMN13}, the linear span of all multi-tildes admits a very
natural structure of an operad. This operad, denoted by $\MT$, is
defined as follows. For any $n \geq 1$, $\MT(n)$ is the linear span of
all multi-tildes of arity $n$ and the partial composition
$(n, \Sfr) \circ_i (m, \Tfr)$, $i \in [n]$, of two multi-tildes
$(n, \Sfr)$ and $(m, \Tfr)$ is defined linearly by
\begin{equation}
    (n, \Sfr) \circ_i (m, \Tfr) :=
    \left(n + m - 1,
    \left\{\Shift_{i, m}(x, y) : (x, y) \in \Sfr\right\}
    \cup \left\{\Shift_{0, i}(x, y) : (x, y) \in \Tfr\right\}\right),
\end{equation}
where
\begin{equation}
    \Shift_{j, p}(x, y) :=
    \begin{cases}
        (x, y) & \mbox{if } y \leq i - 1, \\
        (x, y + p - 1) & \mbox{if } x \leq i \leq y, \\
        (x + p - 1, y + p - 1) & \mbox{otherwise}.
    \end{cases}
\end{equation}
For instance, one has
\begin{subequations}
\begin{equation} \label{equ:example_composition_MT_1}
    (5, \{(1, 5), (2, 4), (4, 5)\}) \circ_4 (6, \{(2, 2), (4, 6)\}) \\
    = (10, \{(1, 10), (2, 9), (4, 10), (5, 5), (7, 9)\}),
\end{equation}
\begin{equation} \label{equ:example_composition_MT_2}
    (5, \{(1, 5), (2, 4), (4, 5)\}) \circ_5 (6, \{(2, 2), (4, 6)\}) \\
    = (10, \{(1, 10), (2, 4), (4, 10), (6, 6), (8, 10)\}).
\end{equation}
\end{subequations}
Observe that the multi-tilde $(1, \emptyset)$ is the unit of~$\MT$.
\medskip

Let $\phi : \MT \to \Cli\Dbb_0$ be the map linearly defined as follows.
For any multi-tilde $(n, \Sfr)$ different from $(1, \{(1, 1)\})$,
$\phi((n, \Sfr))$ is the $\Dbb_0$-clique of arity $n$ defined, for any
$1 \leq x < y \leq n + 1$, by
\begin{equation} \label{equ:isomorphism_MT_Cli_M}
    \phi((n, \Sfr))(x, y) :=
    \begin{cases}
        0 & \mbox{if } (x, y - 1) \in \Sfr, \\
        \Unit & \mbox{otherwise}.
    \end{cases}
\end{equation}
For instance,
\begin{equation}
    \phi((5, \{(1, 5), (2, 4), (4, 5)\}))
    =
    \begin{tikzpicture}[scale=.7,Centering]
        \node[CliquePoint](1)at(-0.50,-0.87){};
        \node[CliquePoint](2)at(-1.00,-0.00){};
        \node[CliquePoint](3)at(-0.50,0.87){};
        \node[CliquePoint](4)at(0.50,0.87){};
        \node[CliquePoint](5)at(1.00,0.00){};
        \node[CliquePoint](6)at(0.50,-0.87){};
        \draw[CliqueEdge](1)edge[]node[CliqueLabel]
            {\begin{math}0\end{math}}(6);
        \draw[CliqueEmptyEdge](1)edge[]node[CliqueLabel]{}(2);
        \draw[CliqueEmptyEdge](2)edge[]node[CliqueLabel]{}(3);
        \draw[CliqueEdge](2)edge[]node[CliqueLabel,near start]
            {\begin{math}0\end{math}}(5);
        \draw[CliqueEmptyEdge](3)edge[]node[CliqueLabel]{}(4);
        \draw[CliqueEdge](4)edge[bend right=30]node[CliqueLabel,near start]
            {\begin{math}0\end{math}}(6);
        \draw[CliqueEmptyEdge](4)edge[]node[CliqueLabel]{}(5);
        \draw[CliqueEmptyEdge](5)edge[]node[CliqueLabel]{}(6);
    \end{tikzpicture}\,.
\end{equation}
\medskip

\begin{Proposition} \label{prop:construction_MT}
    The operad $\Cli\Dbb_0$ is isomorphic to the suboperad
    of $\MT$ consisting in the linear span of all multi-tildes except
    the nontrivial multi-tilde $(1, \{(1, 1)\})$ of arity $1$. Moreover,
    $\phi$ is an isomorphism between these two operads.
\end{Proposition}
\begin{proof}
    A direct consequence of the
    definition~\eqref{equ:isomorphism_MT_Cli_M} of $\phi$ is that this
    map is an isomorphism of vector spaces. Moreover, it follows from
    the definitions of the partial compositions of $\MT$ and
    $\Cli\Dbb_0$ that $\phi$ is an operad morphism.
\end{proof}
\medskip

By Proposition~\ref{prop:construction_MT}, one can interpret the partial
compositions~\eqref{equ:example_composition_MT_1}
and~\eqref{equ:example_composition_MT_2} of multi-tildes as partial
compositions of $\Dbb_0$-cliques. This give respectively
\begin{subequations}
\begin{equation}
    \begin{tikzpicture}[scale=.7,Centering]
        \node[CliquePoint](1)at(-0.50,-0.87){};
        \node[CliquePoint](2)at(-1.00,-0.00){};
        \node[CliquePoint](3)at(-0.50,0.87){};
        \node[CliquePoint](4)at(0.50,0.87){};
        \node[CliquePoint](5)at(1.00,0.00){};
        \node[CliquePoint](6)at(0.50,-0.87){};
        \draw[CliqueEdge](1)edge[]node[CliqueLabel]
            {\begin{math}0\end{math}}(6);
        \draw[CliqueEmptyEdge](1)edge[]node[CliqueLabel]{}(2);
        \draw[CliqueEmptyEdge](2)edge[]node[CliqueLabel]{}(3);
        \draw[CliqueEdge](2)edge[]node[CliqueLabel,near start]
            {\begin{math}0\end{math}}(5);
        \draw[CliqueEmptyEdge](3)edge[]node[CliqueLabel]{}(4);
        \draw[CliqueEdge](4)edge[bend right=30]node[CliqueLabel,near start]
            {\begin{math}0\end{math}}(6);
        \draw[CliqueEmptyEdge](4)edge[]node[CliqueLabel]{}(5);
        \draw[CliqueEmptyEdge](5)edge[]node[CliqueLabel]{}(6);
    \end{tikzpicture}
    \enspace \circ_4 \enspace
    \begin{tikzpicture}[scale=.8,Centering]
        \node[CliquePoint](1)at(-0.43,-0.90){};
        \node[CliquePoint](2)at(-0.97,-0.22){};
        \node[CliquePoint](3)at(-0.78,0.62){};
        \node[CliquePoint](4)at(-0.00,1.00){};
        \node[CliquePoint](5)at(0.78,0.62){};
        \node[CliquePoint](6)at(0.97,-0.22){};
        \node[CliquePoint](7)at(0.43,-0.90){};
        \draw[CliqueEmptyEdge](1)edge[]node[CliqueLabel]{}(2);
        \draw[CliqueEmptyEdge](1)edge[]node[CliqueLabel]{}(7);
        \draw[CliqueEdge](2)edge[]node[CliqueLabel]
            {\begin{math}0\end{math}}(3);
        \draw[CliqueEmptyEdge](3)edge[]node[CliqueLabel]{}(4);
        \draw[CliqueEmptyEdge](4)edge[]node[CliqueLabel]{}(5);
        \draw[CliqueEdge](4)edge[bend right=30]node[CliqueLabel]
            {\begin{math}0\end{math}}(7);
        \draw[CliqueEmptyEdge](5)edge[]node[CliqueLabel]{}(6);
        \draw[CliqueEmptyEdge](6)edge[]node[CliqueLabel]{}(7);
    \end{tikzpicture}
    \enspace = \enspace
    \begin{tikzpicture}[scale=1.1,Centering]
        \node[CliquePoint](1)at(-0.28,-0.96){};
        \node[CliquePoint](2)at(-0.76,-0.65){};
        \node[CliquePoint](3)at(-0.99,-0.14){};
        \node[CliquePoint](4)at(-0.91,0.42){};
        \node[CliquePoint](5)at(-0.54,0.84){};
        \node[CliquePoint](6)at(-0.00,1.00){};
        \node[CliquePoint](7)at(0.54,0.84){};
        \node[CliquePoint](8)at(0.91,0.42){};
        \node[CliquePoint](9)at(0.99,-0.14){};
        \node[CliquePoint](10)at(0.76,-0.65){};
        \node[CliquePoint](11)at(0.28,-0.96){};
        \draw[CliqueEdge](1)edge[]node[CliqueLabel]
            {\begin{math}0\end{math}}(11);
        \draw[CliqueEdge](2)edge[bend left=30]node[CliqueLabel,near start]
            {\begin{math}0\end{math}}(10);
        \draw[CliqueEdge](4)edge[bend left=30]node[CliqueLabel]
            {\begin{math}0\end{math}}(11);
        \draw[CliqueEdge](5)edge[]node[CliqueLabel]
            {\begin{math}0\end{math}}(6);
        \draw[CliqueEdge](7)edge[bend right=30]node[CliqueLabel]
            {\begin{math}0\end{math}}(10);
        \draw[CliqueEmptyEdge](1)edge[]node[CliqueLabel]{}(2);
        \draw[CliqueEmptyEdge](2)edge[]node[CliqueLabel]{}(3);
        \draw[CliqueEmptyEdge](3)edge[]node[CliqueLabel]{}(4);
        \draw[CliqueEmptyEdge](4)edge[]node[CliqueLabel]{}(5);
        \draw[CliqueEmptyEdge](6)edge[]node[CliqueLabel]{}(7);
        \draw[CliqueEmptyEdge](7)edge[]node[CliqueLabel]{}(8);
        \draw[CliqueEmptyEdge](8)edge[]node[CliqueLabel]{}(9);
        \draw[CliqueEmptyEdge](9)edge[]node[CliqueLabel]{}(10);
        \draw[CliqueEmptyEdge](10)edge[]node[CliqueLabel]{}(11);
    \end{tikzpicture}\,,
\end{equation}
\begin{equation}
    \begin{tikzpicture}[scale=.7,Centering]
        \node[CliquePoint](1)at(-0.50,-0.87){};
        \node[CliquePoint](2)at(-1.00,-0.00){};
        \node[CliquePoint](3)at(-0.50,0.87){};
        \node[CliquePoint](4)at(0.50,0.87){};
        \node[CliquePoint](5)at(1.00,0.00){};
        \node[CliquePoint](6)at(0.50,-0.87){};
        \draw[CliqueEdge](1)edge[]node[CliqueLabel]
            {\begin{math}0\end{math}}(6);
        \draw[CliqueEmptyEdge](1)edge[]node[CliqueLabel]{}(2);
        \draw[CliqueEmptyEdge](2)edge[]node[CliqueLabel]{}(3);
        \draw[CliqueEdge](2)edge[]node[CliqueLabel,near start]
            {\begin{math}0\end{math}}(5);
        \draw[CliqueEmptyEdge](3)edge[]node[CliqueLabel]{}(4);
        \draw[CliqueEdge](4)edge[bend right=30]node[CliqueLabel,near start]
            {\begin{math}0\end{math}}(6);
        \draw[CliqueEmptyEdge](4)edge[]node[CliqueLabel]{}(5);
        \draw[CliqueEmptyEdge](5)edge[]node[CliqueLabel]{}(6);
    \end{tikzpicture}
    \enspace \circ_5 \enspace
    \begin{tikzpicture}[scale=.8,Centering]
        \node[CliquePoint](1)at(-0.43,-0.90){};
        \node[CliquePoint](2)at(-0.97,-0.22){};
        \node[CliquePoint](3)at(-0.78,0.62){};
        \node[CliquePoint](4)at(-0.00,1.00){};
        \node[CliquePoint](5)at(0.78,0.62){};
        \node[CliquePoint](6)at(0.97,-0.22){};
        \node[CliquePoint](7)at(0.43,-0.90){};
        \draw[CliqueEmptyEdge](1)edge[]node[CliqueLabel]{}(2);
        \draw[CliqueEmptyEdge](1)edge[]node[CliqueLabel]{}(7);
        \draw[CliqueEdge](2)edge[]node[CliqueLabel]
            {\begin{math}0\end{math}}(3);
        \draw[CliqueEmptyEdge](3)edge[]node[CliqueLabel]{}(4);
        \draw[CliqueEmptyEdge](4)edge[]node[CliqueLabel]{}(5);
        \draw[CliqueEdge](4)edge[bend right=30]node[CliqueLabel]
            {\begin{math}0\end{math}}(7);
        \draw[CliqueEmptyEdge](5)edge[]node[CliqueLabel]{}(6);
        \draw[CliqueEmptyEdge](6)edge[]node[CliqueLabel]{}(7);
    \end{tikzpicture}
    \enspace = \enspace
    \begin{tikzpicture}[scale=1.1,Centering]
        \node[CliquePoint](1)at(-0.28,-0.96){};
        \node[CliquePoint](2)at(-0.76,-0.65){};
        \node[CliquePoint](3)at(-0.99,-0.14){};
        \node[CliquePoint](4)at(-0.91,0.42){};
        \node[CliquePoint](5)at(-0.54,0.84){};
        \node[CliquePoint](6)at(-0.00,1.00){};
        \node[CliquePoint](7)at(0.54,0.84){};
        \node[CliquePoint](8)at(0.91,0.42){};
        \node[CliquePoint](9)at(0.99,-0.14){};
        \node[CliquePoint](10)at(0.76,-0.65){};
        \node[CliquePoint](11)at(0.28,-0.96){};
        \draw[CliqueEdge](1)edge[]node[CliqueLabel]
            {\begin{math}0\end{math}}(11);
        \draw[CliqueEdge](2)edge[bend right=30]node[CliqueLabel,near start]
            {\begin{math}0\end{math}}(5);
        \draw[CliqueEdge](4)edge[bend left=30]node[CliqueLabel]
            {\begin{math}0\end{math}}(11);
        \draw[CliqueEdge](6)edge[]node[CliqueLabel]
            {\begin{math}0\end{math}}(7);
        \draw[CliqueEdge](8)edge[bend right=30]node[CliqueLabel]
            {\begin{math}0\end{math}}(11);
        \draw[CliqueEmptyEdge](1)edge[]node[CliqueLabel]{}(2);
        \draw[CliqueEmptyEdge](2)edge[]node[CliqueLabel]{}(3);
        \draw[CliqueEmptyEdge](3)edge[]node[CliqueLabel]{}(4);
        \draw[CliqueEmptyEdge](4)edge[]node[CliqueLabel]{}(5);
        \draw[CliqueEmptyEdge](5)edge[]node[CliqueLabel]{}(6);
        \draw[CliqueEmptyEdge](7)edge[]node[CliqueLabel]{}(8);
        \draw[CliqueEmptyEdge](8)edge[]node[CliqueLabel]{}(9);
        \draw[CliqueEmptyEdge](9)edge[]node[CliqueLabel]{}(10);
        \draw[CliqueEmptyEdge](10)edge[]node[CliqueLabel]{}(11);
    \end{tikzpicture}\,.
\end{equation}
\end{subequations}
\medskip

\subsubsection{Double multi-tildes}
Double multi-tildes are natural generalizations of multi-tildes,
introduced in~\cite{GLMN16}. A \Def{double multi-tilde} is a triple
$(n, \Sfr, \Tfr)$ where $(n, \Tfr)$ and $(n, \Sfr)$ are both multi-tildes
of the same arity $n$. The \Def{arity} of the double multi-tilde
$(n, \Sfr, \Tfr)$ is $n$. As shown in~\cite{GLMN16}, the linear span of
all double multi-tildes admits a structure of an operad. This operad,
denoted by $\DMT$, is defined as follows. For any $n \geq 1$, $\DMT(n)$
is the linear span of all double multi-tildes of arity $n$ and the
partial composition $(n, \Sfr, \Tfr) \circ_i (m, \Ufr, \Vfr)$,
$i \in [n]$, of two double multi-tildes $(n, \Sfr, \Tfr)$ and
$(m, \Ufr, \Vfr)$ is defined linearly by
\begin{equation} \label{equ:partial_composition_DMT}
    (n, \Sfr, \Tfr) \circ_i (m, \Ufr, \Vfr) :=
    (n, \Sfr \circ_i \Ufr, \Tfr \circ_i \Vfr),
\end{equation}
where the two partial compositions $\circ_i$ of the right member
of~\eqref{equ:partial_composition_DMT} are the ones of~$\MT$. We can
observe that $\DMT$ is isomorphic to the Hadamard product $\MT * \MT$.
For instance, one has
\begin{equation} \label{equ:example_composition_DMT}
    (3, \{(2, 2)\}, \{(1, 2), (1, 3)\})
    \circ_2
    (2, \{(1, 1)\}, \{(1, 2)\})
    =
    (4, \{(2, 2), (2, 3)\}, \{(1, 3), (1, 4), (2, 3)\}).
\end{equation}
The unit of $\DMT$ is $(1, \emptyset, \emptyset)$.
\medskip

Consider now the operad $\Cli\Dbb_0^2$ and let
$\phi : \DMT \to \Cli\Dbb_0^2$ be the map linearly defined as follows.
The image by $\phi$ of $(1, \emptyset, \emptyset)$ is the unit of
$\Cli\Dbb_0^2$ and, for any double multi-tilde $(n, \Sfr, \Tfr)$ of
arity $n \geq 2$, $\phi((n, \Sfr, \Tfr))$ is the $\Dbb_0^2$-clique of
arity $n$ defined, for any $1 \leq x < y \leq n + 1$, by
\begin{equation} \label{equ:isomorphism_DMT_Cli_M}
    \phi((n, \Sfr, \Tfr))(x, y) :=
    \begin{cases}
        (0, \Unit) & \mbox{if } (x, y - 1) \in \Sfr
            \mbox{ and } (x, y - 1) \notin \Tfr, \\
        (\Unit, 0) & \mbox{if } (x, y - 1) \notin \Sfr
            \mbox{ and } (x, y - 1) \in \Tfr, \\
        (0, 0) & \mbox{if } (x, y - 1) \in \Sfr
            \mbox{ and } (x, y - 1) \in \Tfr, \\
        (\Unit, \Unit) & \mbox{otherwise}.
    \end{cases}
\end{equation}
For instance,
\begin{equation}
    \phi((4, \{(2, 2), (2, 3)\}, \{(1, 3), (1, 4), (2, 3)\}))
    =
    \begin{tikzpicture}[scale=1.0,Centering]
        \node[CliquePoint](1)at(-0.59,-0.81){};
        \node[CliquePoint](2)at(-0.95,0.31){};
        \node[CliquePoint](3)at(-0.00,1.00){};
        \node[CliquePoint](4)at(0.95,0.31){};
        \node[CliquePoint](5)at(0.59,-0.81){};
        \draw[CliqueEmptyEdge](1)edge[]node[CliqueLabel]{}(2);
        \draw[CliqueEdge](1)edge[bend left=30]node[CliqueLabel]
            {\begin{math}(\Unit, 0)\end{math}}(4);
        \draw[CliqueEdge](1)edge[]node[CliqueLabel]
            {\begin{math}(\Unit, 0)\end{math}}(5);
        \draw[CliqueEdge](2)edge[]node[CliqueLabel]
            {\begin{math}(0, \Unit)\end{math}}(3);
        \draw[CliqueEdge](2)edge[]node[CliqueLabel]
            {\begin{math}(0, 0)\end{math}}(4);
        \draw[CliqueEmptyEdge](3)edge[]node[CliqueLabel]{}(4);
        \draw[CliqueEmptyEdge](4)edge[]node[CliqueLabel]{}(5);
    \end{tikzpicture}\,.
\end{equation}
\medskip

\begin{Proposition} \label{prop:construction_DMT}
    The operad $\Cli\Dbb_0^2$ is isomorphic to the suboperad of $\DMT$
    consisting in the linear span of all double multi-tildes except the
    three nontrivial double multi-tildes of arity $1$. Moreover, $\phi$
    is an isomorphism between these two operads.
\end{Proposition}
\begin{proof}
    There are two ways to prove the first assertion of the statement of
    the proposition. On the one hand, this property follows from
    Proposition~\ref{prop:Cli_M_Cartesian_product} and
    Proposition~\ref{prop:construction_MT}. On the other hand, the whole
    statement of the proposition is a direct consequence of the
    definition~\eqref{equ:isomorphism_DMT_Cli_M} of $\phi$, showing that
    $\phi$ is an isomorphism of vector spaces, and, from the definitions
    of the partial compositions of $\DMT$ and $\Cli\Dbb_0^2$ showing
    that $\phi$ is an operad morphism.
\end{proof}
\medskip

By Proposition~\ref{prop:construction_DMT}, one can interpret the
partial composition~\eqref{equ:example_composition_DMT} of double
multi-tildes as a partial composition of $\Dbb_0^2$-cliques. This gives
\begin{equation}
    \begin{tikzpicture}[scale=.7,Centering]
        \node[CliquePoint](1)at(-0.71,-0.71){};
        \node[CliquePoint](2)at(-0.71,0.71){};
        \node[CliquePoint](3)at(0.71,0.71){};
        \node[CliquePoint](4)at(0.71,-0.71){};
        \draw[CliqueEmptyEdge](1)edge[]node[CliqueLabel]{}(2);
        \draw[CliqueEdge](1)edge[]node[CliqueLabel]
            {\begin{math}(\Unit, 0)\end{math}}(3);
        \draw[CliqueEdge](1)edge[]node[CliqueLabel]
            {\begin{math}(\Unit, 0)\end{math}}(4);
        \draw[CliqueEdge](2)edge[]node[CliqueLabel]
            {\begin{math}(0, \Unit)\end{math}}(3);
        \draw[CliqueEmptyEdge](3)edge[]node[CliqueLabel]{}(4);
    \end{tikzpicture}
    \enspace \circ_2 \enspace
    \begin{tikzpicture}[scale=.6,Centering]
        \node[CliquePoint](1)at(-0.87,-0.50){};
        \node[CliquePoint](2)at(-0.00,1.00){};
        \node[CliquePoint](3)at(0.87,-0.50){};
        \draw[CliqueEdge](1)edge[]node[CliqueLabel]
            {\begin{math}(0, \Unit)\end{math}}(2);
        \draw[CliqueEdge](1)edge[]node[CliqueLabel]
            {\begin{math}(\Unit, 0)\end{math}}(3);
        \draw[CliqueEmptyEdge](2)edge[]node[CliqueLabel]{}(3);
    \end{tikzpicture}
    \enspace = \enspace
    \begin{tikzpicture}[scale=1.0,Centering]
        \node[CliquePoint](1)at(-0.59,-0.81){};
        \node[CliquePoint](2)at(-0.95,0.31){};
        \node[CliquePoint](3)at(-0.00,1.00){};
        \node[CliquePoint](4)at(0.95,0.31){};
        \node[CliquePoint](5)at(0.59,-0.81){};
        \draw[CliqueEmptyEdge](1)edge[]node[CliqueLabel]{}(2);
        \draw[CliqueEdge](1)edge[bend left=30]node[CliqueLabel]
            {\begin{math}(\Unit, 0)\end{math}}(4);
        \draw[CliqueEdge](1)edge[]node[CliqueLabel]
            {\begin{math}(\Unit, 0)\end{math}}(5);
        \draw[CliqueEdge](2)edge[]node[CliqueLabel]
            {\begin{math}(0, \Unit)\end{math}}(3);
        \draw[CliqueEdge](2)edge[]node[CliqueLabel]
            {\begin{math}(0, 0)\end{math}}(4);
        \draw[CliqueEmptyEdge](3)edge[]node[CliqueLabel]{}(4);
        \draw[CliqueEmptyEdge](4)edge[]node[CliqueLabel]{}(5);
    \end{tikzpicture}\,.
\end{equation}
\medskip

\section*{Conclusion and perspectives}
This works presents and study the clique construction $\Cli$, producing
operads from unitary magmas. We have seen that $\Cli$ has many both
algebraic and combinatorial properties. Among its most notable ones,
$\Cli\Mca$ admits several quotients involving combinatorial families of
decorated cliques, admits a binary and quadratic suboperad $\NC\Mca$
which is a Koszul, and contains a lot of already studied and classic
operads. Besides, in the course of this work, whose text is already long
enough, we have put aside a bunch of questions. Let us address these here.
\smallskip

When $\Mca$ is a $\Z$-graded unitary magma, a link between $\Cli\Mca$
and the operad of rational functions $\RatFct$~\cite{Lod10} has been
developed in Section~\ref{subsubsec:rational_functions} by means of a
morphism $\Frac_\theta$ between these two operads. We have observed that
$\Frac_\theta$ is not injective (see~\eqref{equ:frac_not_injective_1}
and~\eqref{equ:frac_not_injective_2}). A description of the kernel of
$\Frac_\theta$, even when $\Mca$ is the unitary magma $\Z$, seems not
easy to obtain. Trying to obtain this description is a first perspective
of this work.
\smallskip

Here is a second perspective. In Section~\ref{sec:quotients_suboperads},
we have defined and briefly studied some quotients and suboperads of
$\Cli\Mca$. In particular, we have considered the quotient $\Deg_1\Mca$
of $\Cli\Mca$, involving $\Mca$-cliques of degrees at most $1$. As
mentioned, $\Deg_1\Dbb_0$ is an operad defined on the linear span of
involutions (except the nontrivial involution of $\mathfrak{S}_2$). A
complete study of this operad seems worthwhile, including a description
of a minimal generating set, a presentation by generators and relations,
a description of its partial composition on the $\Hsf$-basis and on the
$\Ksf$-basis, and a realization of this operad in terms of standard
Young tableaux.
\smallskip

The last question we develop here concerns the Koszul dual $\NC\Mca^!$
of $\NC\Mca$. Section~\ref{subsec:dual_NC_M} contains results about this
operad, like a description of its presentation and a formula for its
dimensions. We have also established the fact that, as graded vector
spaces, $\NC\Mca^!$ is isomorphic to the linear span of all noncrossing
dual $\Mca$-cliques. To obtain a realization of $\NC\Mca^!$, it is now
enough to endow this last space with an adequate partial composition.
This is the last perspective we address here.
\medskip

\bibliographystyle{alpha}
\bibliography{Bibliography}

\end{document}